\documentclass[a4paper,11pt,leqno]{article}
\usepackage{amssymb, amsmath, amsthm, graphicx, color, epsfig, amsfonts}

\newtheorem{thm}{Theorem}[section]
\newtheorem{lem}[thm]{Lemma}
\newtheorem{cor}[thm]{Corollary}
\newtheorem{prop}[thm]{Proposition}
\newtheorem{defn}{Definition}[section]
\newtheorem{rem}{Remark}[section]

\numberwithin{equation}{section}

\renewcommand{\a}{\alpha}
\renewcommand{\b}{\beta}
\newcommand{\e}{\varepsilon}
\newcommand{\de}{\delta}
\newcommand{\fa}{\varphi}
\newcommand{\ga}{\gamma}

\newcommand{\la}{\lambda}
\newcommand{\si}{\sigma}

\renewcommand{\t}{\tau}
\newcommand{\om}{\omega}
\newcommand{\De}{\Delta}
\newcommand{\Ga}{\Gamma}
\newcommand{\La}{\Lambda}
\newcommand{\Om}{\Omega}
\newcommand{\lan}{\langle}
\newcommand{\ran}{\rangle}

\newcommand{\U}{v}
\newcommand{\V}{\psi}
\newcommand{\W}{\xi}

\def\R{{\mathbb{R}}}
\def\N{{\mathbb{N}}}
\def\Z{{\mathbb{Z}}}
\def\T{{\mathbb{T}}}

\allowdisplaybreaks

\title{Schauder estimate for \\ quasilinear discrete PDEs
of parabolic type}

\author{Tadahisa Funaki$^\ast$ and Sunder Sethuraman$^\dagger$}

\date{\today } 
\begin{document}
\maketitle

\begin{abstract} 
We investigate quasilinear discrete PDEs $\partial_t u = \De^N \fa(u)+ Kf(u)$
of reaction-diffusion type with nonlinear diffusion term 
defined on an $n$-dimensional unit torus discretized with mesh size $\tfrac1N$ for
$N\in \N$, where $\De^N$ is the discrete Laplacian, $\varphi$ is a strictly 
increasing $C^5$ function and $f$ is a $C^1$ function.
We establish $L^\infty$ bounds and space-time H\"older estimates, 
both uniform in $N$, of the first and second spatial discrete derivatives of the solutions.
In the equation, $K>0$ is a large constant and we show how these estimates depend
on $K$.  The motivation for this work stems originally from the study of hydrodynamic
scaling limits of interacting particle systems.

Our method is a two steps approach in terms of the H\"older estimate and Schauder
estimate, which is known for continuous parabolic PDEs.  We first show the discrete
H\"older estimate uniform in $N$ for the solutions of the associated linear discrete
PDEs with continuous coefficients, based on the Nash estimate.
We next establish the discrete Schauder estimate for linear discrete PDEs 
with uniform H\"older coefficients.  The link between discrete and continuous 
settings is given by the polylinear interpolations.  Since this operation has a 
non-local nature, the method requires proper modifications.

We also discuss another method based on the study of the 
corresponding fundamental solutions.
\end{abstract}

\footnote{ \hskip -6.5mm
$^\ast$Department of Mathematics,
Waseda University,
3-4-1 Okubo, Shinjuku-ku,
Tokyo 169-8555, Japan. \\
e-mail: funaki@ms.u-tokyo.ac.jp \\
$^\dagger$Department of Mathematics,
University of Arizona,
621 N.\ Santa Rita Ave.,
Tucson, AZ 85750, USA.
e-mail: sethuram@math.arizona.edu

%Running head: Schauder estimate for discrete PDEs
}



\tableofcontents

\section{Introduction}\label{sec:GS}

In this article, we investigate quasilinear discrete PDEs (partial differential equations)
defined on an $n$-dimensional unit torus $\T^n \cong [0,1)^n$ with periodic boundary 
discretized with mesh size $\tfrac1N$ for $N\in \N$.  The goal is to establish
gradient estimates for the solutions, especially to show $L^\infty$ bounds and also space-time 
H\"older estimates, both uniform in $N$, of the first and second spatial discrete 
derivatives of the solutions.  As a model equation, we take a type of discrete
reaction-diffusion equation with nonlinear diffusion term, which is a variant of
Allen-Cahn equation;  see \eqref{eq:2.1-X} below.  Such a type of equation arises in 
the study of the hydrodynamic scaling limit of interacting particle systems on the $n$-dimensional
periodic integer lattice of size $N$; see \cite{EFHPS-1}, \cite{EFHPS-2} and \cite{FMST}, which
motivate the present article.
We expect that one can 
extend our results to a wide class of quasilinear discrete PDEs.

Our method is based on the two steps approach due to the H\"older estimate and Schauder
estimate, which is well-known and originally exploited for continuous linear and
quasilinear parabolic PDEs.  We first show the discrete H\"older estimate
uniform in $N$ for the solutions of linear discrete
PDEs with continuous coefficients, based on the so-called Nash estimate.
This applies to our quasilinear equations.  Once the H\"older estimate is
in hand, the second step is to improve
the regularity by establishing the discrete Schauder estimate for linear discrete PDEs 
with uniform H\"older coefficients, cf.\ Lieberman \cite{Li96} in the continuous setting.
The link between discrete and continuous settings will be given by the 
polylinear interpolations.  Since this operation has a non-local nature, the method requires
proper modifications.

Another method known to derive gradient estimates is based on the study of the 
corresponding fundamental solutions.  We will also briefly discuss this approach.

\subsection{Quasilinear discrete PDEs}

Let $\T_N^n\equiv (\Z/N\Z)^n=\{1,2,\ldots,N\}^n$ be
the integer lattice of size $N$ with periodic boundary.
Let $u=u^N(t,\tfrac{x}N), x \in \T_N^n$, be a solution of the quasilinear discrete PDE 
\begin{align}  \label{eq:2.1-X}
\partial_t u = \De^N \fa(u)+ Kf(u), \quad \tfrac{x}N \in \tfrac1N\T_N^n,
\end{align}
where $\De^N$ is the discrete Laplacian on $\frac1N\T_N^n$ (see 
\eqref{eq:Lap-N} for precise definition), $\varphi$ is a strictly increasing $C^5$ 
function on $\R$ and $f$ is a $C^1$ function on $\R$ satisfying that
$u\cdot f(u)<0$ for $u\in \R$ with large enough $|u|$.  In our problem, we are interested in how the estimates depend on $K>0$, a large constant.
The goal is to give both $L^\infty$ and H\"older estimates on first and second 
discrete derivatives $\nabla_e^N u^N(t,\tfrac{x}N)$ and $\nabla_{e_1}^N \nabla_{e_2}^Nu^N(t,\tfrac{x}N)$
of the solution uniformly in $N$; see \eqref{eq:2.nablaeN} for the definition of
discrete derivatives $\nabla^N = \{\nabla_e^N\}_{e\in \Z^n: |e|=1}$.
Roughly, in a stationary situation, for example for a traveling wave type solution,
the second derivative would behave as
$\De^N \fa(u^N) = O(K)$ so that one would expect $\nabla_e^Nu^N= O(\sqrt{K})$
and $\nabla_{e_1}^N \nabla_{e_2}^N u^N = O(K)$, but actually our results are weaker than this, given the quasilinear nature of the problem.

Note that the equation \eqref{eq:2.1-X} is a system of ODEs.  Given an
initial value $u^N(0,\tfrac{x}N)$ which is uniformly bounded in $N, x$ and $K$,
by the comparison argument for the nonlinear Laplacian $\De^N\fa(u)$ and
noting the condition on $f$, one can show that the solution $u$ exists uniquely
and globally in time.  In addition, one can show that there exist $-\infty<u_-<u_+<\infty$
such that
\begin{align} \label{eq:1.u_pm}
u_- \le u^N (t,\tfrac{x}N)\le u_+
\end{align}
holds for all $t\ge 0$ and $x \in \T_N^n$; see Section \ref{sec:1-comparison}. 
In particular, if $f(0)>0$ and $u^N(0,\frac{x}N)$
is uniformly positive, one can take $u_->0$ so that, in this case, we can discuss 
\eqref{eq:2.1-X} for $f$ and $\fa$ defined only on $[0,\infty)$.  Such an equation
appears in the setting of the particle system, in which
the positive solution $u^N$ describes the macroscopic particle density.

\subsection{Linear discrete PDEs}

The following  linear discrete PDE \eqref{eq:1.2-linear}
of divergence form is associated to the nonlinear equation \eqref{eq:2.1-X}.  
To describe the equation, we introduce some notation.  For $0<c_-<c_+<\infty$,
let $\mathcal{A}(c_-,c_+)$ be the class of functions $a=\{a_{x,e}\}_{x\in \T_N^n,
e\in \Z^n: |e|=1}$ satisfying the symmetry, uniform positivity (nondegeneracy)
and boundedness conditions:
\begin{align} \label{eq:a-1-N}
& a_{x,e}=a_{x+e,-e}, \\
\label{eq:a-2-N}
& c_-\le a_{x,e}\le c_+.
\end{align}
For $a\in \mathcal{A}(c_-,c_+)$,
we consider the linear discrete diffusion operator $L_a^N$ with coefficient $a$
defined by
\begin{align} \label{eq:1-L_a}
L_a^N  u(\tfrac{x}N) & = N^2 \sum_{e: |e|=1} a_{x,e} \{u(\tfrac{x+e}N)-u(\tfrac{x}N)\}\\
&= - \tfrac12\sum_{e:|e|=1} \nabla^{N,*}_e \big(a_{x,e}\nabla^N_e u \big)(\tfrac{x}N),
\notag
\end{align}
where we have used \eqref{eq:a-1-N} for the second line; see \eqref{eq:2.4}
for $\nabla^{N,*}_e$ and also \eqref{eq:2.6} below.  In particular, if $a\equiv 1$,
then $L_a^N= \De^N$.  Note that the operator $L^N_a$ is symmetric: $(L^N_a)^* = L^N_a$
with respect to the natural $L^2$-inner product for functions on $\tfrac1N \T_N^n$;
see \eqref{eq:2.5} below.  For given $a(t) \in \mathcal{A}(c_-,c_+)$ and a bounded
function $g(t,\tfrac{x}N)$, both continuous in $t\ge 0$, we consider
the following linear discrete PDE for $u = u^N(t,\tfrac{x}N)$:
\begin{align}  \label{eq:1.2-linear}
\partial_t u = L_{a(t)}^N u+ g(t,\tfrac{x}N), \quad \tfrac{x}N \in \tfrac1N\T_N^n.
\end{align}
Note that \eqref{eq:1.2-linear} has a unique solution globally in time.

\subsection {Relation between the quasilinear equation \eqref{eq:2.1-X} and the linear equation
\eqref{eq:1.2-linear}}  \label{sec:1.3}

Once we a priori  know the solution $u=u^N$ of \eqref{eq:2.1-X}, the nonlinear equation \eqref{eq:2.1-X}
is reduced to the linear equation \eqref{eq:1.2-linear} by taking $a(t)\equiv
a^N(t)$ and $g(t,\tfrac{x}N) \equiv g^N(t,\tfrac{x}N)$ as follows.
The nonlinear Laplacian $\De^N\fa(u)$ is regarded as the linear operator
$L_{a(u^N(t))}^Nu$, i.e., 
\begin{align}  \label{eq:De-Lta}
\De^N\fa(u^N(t))(\tfrac{x}N) = L_{a(u^N(t))}^N  u^N(t,\tfrac{x}N)
\end{align}
holds, if we define $a(u)=\{a_{x,e}(u)\}$ for $u=\{u(\tfrac{x}N)\}$ as
a discrete gradient of $\fa(u)$ in direction $e$ determined by
\begin{align}  \label{eq:a_xe}
a_{x,e}(u) := \left\{  \begin{aligned}
& \frac{\fa(u(\tfrac{x+e}N))-\fa(u(\tfrac{x}N))}{u(\tfrac{x+e}N)-u(\tfrac{x}N)},
\quad \text{if }\;  u(\tfrac{x+e}N)\not=u(\tfrac{x}N), \\
& \fa'(u(\tfrac{x}N)),
\hskip 25.5mm \text{if }\; u(\tfrac{x+e}N)=u(\tfrac{x}N).
\end{aligned} \right.
\end{align}
Note that $a(u)$ satisfies the symmetry condition \eqref{eq:a-1-N}
and,  by \eqref{eq:1.u_pm}, $a(u^N(t))\in \mathcal{A}(c_-,c_+)$ holds with
\begin{align} \label{eq:c_pm-N}
c_-= \min_{u\in [u_-,u_+]} \fa'(u), \quad c_+= \max_{u\in [u_-,u_+]} \fa'(u).
\end{align}
Moreover, we take $g(t,\tfrac{x}N)=K f(u^N(t,\tfrac{x}N))$.  Note that
$g$ is bounded by \eqref{eq:1.u_pm}.

Note also that, from \eqref{eq:De-Lta} and \eqref{eq:a_xe}
combined with the `mean-value' lemma (cf.\ Lemma \ref{lem:mvt}), the
H\"older estimate of the solution of \eqref{eq:2.1-X}, shown by the Nash
bound, implies a like estimate for the coefficient $a(u^N(t))$ when we view \eqref{eq:2.1-X} 
as a linear equation \eqref{eq:1.2-linear}.

\subsection{Overview of the main results}

Our main results are the $L^\infty$ bounds and the space-time H\"older estimates,
both uniform in $N$,  for the first and second discrete derivatives, 
$\nabla_e^N u^N(t,\tfrac{x}N)$ and $\nabla_{e_1}^N \nabla_{e_2}^N u^N(t,\tfrac{x}N)$,
of the solution of the equation \eqref{eq:2.1-X}.
To illustrate part of our main estimates, let us introduce discrete $C^k$-norms
for $u^N=\{u^N(\tfrac{x}N)\}$ and for $k=0,1,2,\ldots$ by
\begin{align} \label{eq:norm-uN}
\|u^N\|_{C_N^k} := \sum_{i=0}^k \sum_{e_1, \ldots, e_i} \max_{x\in \T_N^n}
\big|\nabla_{e_1}^N \cdots \nabla_{e_i}^N u^N(\tfrac{x}N)\big|
\end{align}
where when $k=0$, the norm reduces to $\|u^N\|_\infty= \|u^N\|_{L^\infty(\T_N^n)}$.

To derive the estimates for the first derivatives, we first study the linear equation 
\eqref{eq:1.2-linear} with bounded $g(t)$ and the coefficient $a(t)$, which satisfies
the uniform $\a$-H\"older estimate; see the assumption (A.2) in Section \ref{sec:4.1}.
This assumption is reasonable by the last comment in Section \ref{sec:1.3}.
We show the uniform estimate on the first derivative of the solution $u^N(t)$
of \eqref{eq:1.2-linear}:
\begin{align}\label{eq:main-1}
\|u^N(t)\|_{C_N^1} \le \frac{C}{\sqrt{t}}, \quad 0<t\le T,
\end{align}
where $C=C(c_\pm)$ depends on $c_\pm$, H\"older constant and exponent $\a
\in (0,1)$ of $a(t)$, $\|g\|_\infty$ and $\|u^N(\cdot)\|_\infty = \sup_{t\in [0,T]}\|u^N(t)\|_\infty$; 
see \eqref{eq:4.5} of Theorem \ref{thm:Li-T4.8}.  We note that the constant
$C$ is linearly growing in $\|g\|_\infty$ and $\|u^N(\cdot)\|_\infty$, in particular,
in $\|u^N(0)\|_\infty$ due to the linearity of the equation \eqref{eq:1.2-linear}.
A further H\"older estimate on $\nabla_e^N u^N(t,\tfrac{x}N)$ in $(t,\tfrac{x}N)$
is given in  \eqref{eq:4.4}:
\begin{align}  \label{eq:1.14=holder}
|\nabla_e^Nu^N(t_1,\tfrac{x_1}N) - \nabla_e^Nu^N(t_2,\tfrac{x_2}N)|
\le C\frac{ \big| t_1-t_2 \big|^{\frac{\a}2}
+ \big| \tfrac{x_1}N - \tfrac{x_2}N \big|^\a }{(t_1\wedge t_2)^{\frac{1+\a}2}},
\quad 0<t_1, t_2\le T.
\end{align}

By analyzing how $C$ depends on $a(t)$ in \eqref{eq:main-1},
we deduce the following estimate for the solution of the quasilinear equation \eqref{eq:2.1-X}:
\begin{align}\label{eq:main-2}
\|u^N(t)\|_{C_N^1} \le \frac{C(K^{\frac1\si}+1)}{\sqrt{t}}, \quad 0<t\le T,
\end{align}
where $\si\in (0,1)$ is the H\"older exponent in Nash estimate; 
see Corollary \ref{cor:K^3}.  The nonlinearity of the above estimate in $K$ is due to that
of the equation \eqref{eq:2.1-X}.  The constant $C$ in \eqref{eq:main-1}
depends on the H\"older property of the coefficient $a(t)$, given in Corollary \ref{cor:2.3} 
applied to the equation \eqref{eq:2.1-X}, in terms of the Nash estimate.

The equation \eqref{eq:2.1-X} can also be expressed in non-divergence form 
for the variable $\V=\fa(u)$:
\begin{align}
\label{nondivergenceform}
\partial_t \V &= \fa'(u)\big\{ \De^N \V + K f(u)\} \\
&= \fa'(\fa^{-1}(\V))\big\{ \De^N \V + K f(\fa^{-1}(\V))\}.
\notag
\end{align}
To derive $L^\infty$ and H\"older estimates on 
the second derivatives $\nabla_{e_1}^N \nabla_{e_2}^N u^N(t)$, it turns out to be efficient to use 
the equation \eqref{nondivergenceform}.
As above, we first consider a linear system of equations associated to the equation
satisfied by $\xi=\nabla^N\psi \equiv \{\nabla_e^N\psi\}_e$ and derive the estimate for
$\|\xi\|_{C_N^1}$ as well as the H\"older estimate for $\nabla_e^N\xi$; see the equation
\eqref{eq:we}, Theorem \ref{second_Schauder_thm} and Corollary \ref{alternative_rmk}.

These results are applied to the nonlinear equation \eqref{nondivergenceform} and we obtain
the uniform estimate on the second derivatives of the solution $u^N(t)$ of \eqref{eq:2.1-X}:
\begin{align}\label{eq:main-3}
\|u^N(t)\|_{C_N^2} \le \frac{C(K^{\frac2\si}+1)}{t},\quad 0<t\le T;
\end{align}
see \eqref{eq:5.29} in Corollary \ref{second_Schauder_cor} combining with
\eqref{eq:main-2}.  The H\"older estimate on 
$\nabla_{e_1}^N \nabla_{e_2}^N u^N(t,\tfrac{x}N)$ in $(t,\tfrac{x}N)$
is also given in \eqref{eq:holder-uu} of this corollary.

The Schauder estimates \eqref{eq:main-1}, \eqref{eq:main-2} and \eqref{eq:main-3}
improve the regularity, but these yield a diverging factor near $t=0$, which is 
inherited from the H\"older estimate we find for $u^N(t)$ in Theorem \ref{thm:Holder-u^N}
and therefore for $a(u^N(t))$ defined by \eqref{eq:a_xe}.  To avoid this, we show that
such a diverging factor does not appear in the $L^\infty$ and H\"older estimates
if we assume regularity for the initial value; see  Section \ref{sec:Regt=0},
Theorems \ref{thm:4.2} and \ref{second_Schauder_time_thm}.
Based on these observations, we obtain
\begin{align}\label{eq:main-4}
& \|u^N(t)\|_{C_N^1} \le CK_0^{\frac1\si}, \quad t\in [0,T],\\
\label{eq:main-5}
& \|u^N(t)\|_{C_N^2} \le C\big[\bar K_0^{\frac{2}{\si}} (1+\bar C_0^{24}) + 
\bar K_0^{\frac{2}{\si}-1} \bar C_0^{48}\big]
,\quad t\in [0,T],
\end{align}
in \eqref{eq:cor4.4-2} and \eqref{eq:5:36} in
Corollaries \ref{extended_cor} and \ref{second_Schauder_time_cor} by
improving \eqref{eq:main-2} and \eqref{eq:main-3} under the assumptions that 
$C_0=\sup_N \|u^N(0)\|_{C_N^2}<\infty$ and $\bar C_0=\sup_N \|u^N(0)\|_{C_N^4}<\infty$,
respectively, where we set
\begin{equation}  \label{eq:K_0}
K_0 = K + C_0 +1\quad \text{ and } \quad \bar K_0 = K + \bar C_0 +1.
\end{equation}
These estimates clarify the dependence on the initial values.  Such a clarification is useful, since,
as we mentioned above, we may consider the situation when $\nabla^Nu^N(0)
= O(\sqrt{K})$ among others.  An intermediate estimate on the second derivatives
\begin{align}\label{eq:main-inter}
\|u^N(t)\|_{C_N^2} \le \frac{C}{\sqrt{t}} K_0^{\frac{2}{\si}}, \quad t\in [0,T]
\end{align}
is given in Corollary \ref{cor:5.7} under the assumption $C_0<\infty$,
milder than that for \eqref{eq:main-5}.
We also have an improvement of the H\"older estimates; see \eqref{eq:cor4.4-3},
\eqref{cor 5.9 holder} and \eqref{cor 5.10 holder}.

The estimate on $\|u^N(t)\|_{C_N^2}$ can be obtained via the parametrix method 
applied for the fundamental solution associated with the operator $L_{a(u^N(t))}
-\partial_t$; see Proposition \ref{prop.6.4}.  But the result obtained in this way
is much weaker than \eqref{eq:main-5} in terms of $K$ so that this approach has a disadvantage. 
We note that, in the semilinear case with $\fa(u)=cu$ with $c>0$,  
the gradient estimates of the discrete heat kernel leads to 
the bound on $\|u^N(t)\|_{C_N^1}$, which is however reasonable in $K$; see  
\cite{FT}, Proposition 4.3 and Remark \ref{rem:2.sigma} below.

\subsection{Methods and contents of the article}

In the continuous setting, the Schauder estimate for linear parabolic PDEs is well-known; 
see Lieberman \cite{Li96} Theorems 4.8 and 6.48, cf.\ \cite{GT} for elliptic PDEs.
The basic requirement to derive the
Schauder estimate is the H\"older continuity of the coefficients.

In Section \ref{sec:Nash}, as we mentioned above, we show the discrete H\"older estimate 
for the solution $u^N$ of the equation \eqref{eq:1.2-linear} uniformly in $N$.  We recall
Nash continuity estimate for \eqref{eq:1.2-linear} with $g\equiv 0$ and then apply
Duhamel's formula to treat the term $g$.  This applies to the 
quasilinear equation \eqref{eq:2.1-X}.  Such a step is fundamental to move to the next step
to establish the discrete Schauder estimate.  In this context, we also state the maximum principle.
See \cite{Li96} Theorem 6.29 in the continuous setting.

In Section \ref{sec:Schauder-1}, recalling the last comment in Section \ref{sec:1.3},
we work on the linear equation \eqref{eq:1.2-linear} with $\a$-H\"older coefficient $a(t)$
and bounded $g(t)$ in general and derive both $L^\infty$ and H\"older estimates 
for $\nabla^N u$.  The results are applied to the quasilinear equation \eqref{eq:2.1-X} 
and we obtain \eqref{eq:main-1}, \eqref{eq:1.14=holder} and \eqref{eq:main-2}.

Section \ref{sec:Li-Schauder} is a preparatory section.  Some H\"older seminorms
unweighted or weighted near $t=0$ are introduced.  To apply the continuous method,
especially two basic lemmas explained later, we introduce 
a polylinear interpolation as a link between discrete and continuous 
settings by embedding discrete functions $u^N$ into continuous ones $\tilde u^N$;
see \eqref{eq:poli-A} below.  However, the discrete setting has a minimum unit
of size $\frac1N$ and the continuous PDE method does not cover this regime.
We have to treat distances less than this minimum unit from different viewpoint.

With the $\a$-H\"older assumption (A.2) on $a(t)$ in mind, we couple and compare
the solution $u^N$ of \eqref{eq:1.2-linear} with the solution $v=v^N$ of 
a simpler equation called the discrete heat equation
\begin{align}  \label{eq:dHeq}
\partial_t v = \De_a^N v,
\end{align}
where $\De_a^N := L_a^N$ with a coefficient $a=\{a_e\}_{|e|=1}$, which is
constant in space and time, satisfying $a_e=a_{-e}$ and $c_-\le a_e \le c_+$;
see \eqref{eq:Lap-a-N} and \eqref{eq:heat-D} below.

Differently from the continuous setting, the interior estimates of the polylinear 
interpolation $\tilde v^N$ of $v^N$  
on the parabolic ball $Q(r)$ of size $r>0$ (see \eqref{Q_defn}), given the behavior of
$\U^N$ on a wider discrete parabolic ball $Q_N(R)$,
can be discussed only under the condition $r+\tfrac{\sqrt{n}}N\le R$.
We call this gap the \lq\lq non-local nature'' of our discrete problem.
One reason for this gap is to guarantee the definability of the polylinear interpolation 
$\tilde v^N$ of $v^N$ on $Q(r)$.
The other is to derive the interior estimate, severing the influence from 
the outside area.  Thus, we need to consider a band area, that is, a gap region
$Q_N(R)\setminus Q(r)$, which is unnecessary in the continuous setting.

The main idea is to apply two basic lemmas used in \cite{Li96} in suitably modified forms.
The first is that an integral estimate for oscillation of $F$ (we take $F=\nabla^N\tilde u^N$)
on $Q(r)$  implies H\"older estimates for $F$ (cf.\ Lemma \ref{lem:Li-L4.3-G}),
and the other is an iteration lemma (cf.\ Lemmas \ref{lem:Li-L4.6} and \ref{lem:Li-L4.6-N}).

To derive the integral estimate required
for  Lemma \ref{lem:Li-L4.3-G}, as we mentioned, we couple $u^N$ with the 
solution $v^N$ of the discrete heat equation \eqref{eq:dHeq} with properly chosen $a$
and derive the discrete energy inequality for $\nabla^N(u^N-v^N)$ based on the summation 
by parts formula; see Lemmas \ref{lem:D-energy-a-Holder} and \ref{lem:estimate-D-energy}.
Here, the H\"older property of $a(t)$ plays a role.  The oscillation estimate for
$\nabla^N \tilde v^N$ in a smaller domain by that in a wider domain is prepared
in Proposition \ref{lem:Li-L4.5}, especially \eqref{eq:Li-L4.5-3}.  This is shown essentially
by the maximum principle; cf.\ Lemmas \ref{lem:Li-L4.4} and \ref{lem:thm2.13}.
Note that the estimate on the oscillation in time follows
from that in space and, for the small region of size of minimum unit such as $r\le \frac{c}N$, 
the polylinearity plays a role.   Collecting these estimates, the modified iteration lemma
 (Lemma \ref{lem:Li-L4.6-N})
implies the integral estimate for oscillation of $\nabla^N\tilde u^N$  on
$Q(r)$ in the $L^2$-sense required for Lemma \ref{lem:Li-L4.3-G}; see Proposition 
\ref{spatial-variation-prop}.

However, as formulated in the assumption in Lemma \ref{lem:Li-L4.3-G}, 
the integral estimate can be shown only with $r_N=r+\frac{c}N$ instead of simply $r\ge 0$
due to the \lq\lq non-local nature'' of our discrete problem. 
As a result, we obtain the uniform H\"older estimate only for two points
$X=(t_1,\frac{x_1}N)$ and $Y=(t_2,\frac{x_2}N)$ such that $|X-Y|\ge \tfrac1{MN}$, 
i.e., excluding the short distance regime, in terms of 
an additional small factor of weighted H\"older seminorm of $\nabla^Nu^N$.

To fill the gap, we need a short distance regime estimate
for $|X-Y|\le \tfrac1{MN}$ for large enough $M$.  This will be shown separately
in Lemma \ref{lem:4.13}.
In addition, to complete the program, we need a discrete version of the interpolation inequality and
an estimate on the time varying norm.

In Section \ref{sec:Schauder-second}, to derive the estimates on 
the second derivatives $\nabla^N \nabla^N u^N(t)$, we consider the equation
\eqref{nondivergenceform} of non-divergence form.  Setting $\bar a(t,\frac{x}N):=
\fa'(u^N(t,\frac{x}N))$ and $g(t,\frac{x}N):= K \bar a(t,\frac{x}N) f(u^N(t,\frac{x}N))$,
the equation \eqref{nondivergenceform} is regarded as a linear discrete PDE with
coefficients $\bar a(t)$ and $g(t)$.  We actually study the system of equations obtained by
applying the gradient operator $\nabla^N$ to \eqref{nondivergenceform}; see
\eqref{eq:we}.  Then, the method of the discrete energy inequality works well for
this system (see Lemma \ref{lem:discreteEnergy-We}), and one can mostly repeat
similar arguments given in Section \ref{sec:Schauder-1} to obtain
Theorem \ref{second_Schauder_thm} and Corollary \ref{alternative_rmk}.

To derive the estimate such as \eqref{eq:main-5} avoiding the singularity
near $t=0$, it turns out to be useful to work at the level of the equation
\eqref{nondivergenceform} for $\psi$, especially to apply the maximum principle;
see  Theorem \ref{second_Schauder_time_thm}.  In other words, the system of
equations \eqref{eq:we} has a gradient structure.  When we apply this result
for the discrete PDE \eqref{eq:2.1-X}, we require the regularity of the
coefficient:  $\fa\in C^5$.  In each statement in Corollaries \ref{cor:K^3}, 
\ref{extended_cor},
\ref{second_Schauder_cor}, \ref{cor:5.7} and \ref{second_Schauder_time_cor},
we will make clear the regularity assumptions on $\fa$ ($\fa\in C^2, C^3$
or $C^5$) by indicating the dependence of the constants $C$ on the 
derivatives of $\fa$.

It is well-known that the Schauder estimate is shown also by gradient 
estimates on the associated fundamental solutions, in particular derived by the E.E.\ Levi's
parametrix method; see Friedman \cite{Fri64}, Il'in et al.\ \cite{IKO}, 
Lady\v{z}enskaja et al.\ \cite{LSU} and Eidel'man \cite{Ei}.  
We discuss the parametrix method in the discrete setting in Section \ref{sec:EELevi}.
To complete this procedure, the H\"older estimate with singularity at $t=0$
obtained in Theorem \ref{thm:Holder-u^N} is not enough and we need to avoid it 
by using the H\"older estimate in Section \ref{sec:Regt=0}.

\subsection{Further comments}

Finally, we give some further comments.
Lady\v{z}enskaja et al.\ \cite{LSU} took a different method to show the
H\"older estimate.  They first derive a local energy inequality with 
a cut-off function.  This implies
the parabolic Harnack inequality which then leads to the H\"older estimate.
We replaced this route by use of the Nash estimate.

Approaches directly applicable to nonlinear PDEs are also known.
Lieberman \cite{Li96} discusses gradient bounds for
quasilinear parabolic PDEs in Chapter XI.  Evans \cite{Evans}
explains the method of (global) energy inequality.  However, the nonlinearity
in our equation impedes in the application of these methods.

Discrete PDEs on complex (random) graphs are well studied in the probability literature,
and the parabolic Harnack inequality, the Gaussian bound called Aronson estimate 
on heat kernels, and H\"older estimates
are known; see \cite{BH}, \cite{Del}, \cite{DD}, \cite{Kuma}, \cite{BB}.  
Semilinear equations with linear discrete Laplacian on general graphs are
discussed by \cite{Gri}.  In our case, the graph is
$\T_N^n$ and it is much simpler than those studied by these authors.
See Section \ref{sec:7} for further comments.  
For $a(t) \in \mathcal{A}(c_-,c_+)$, the  H\"older estimate of the fundamental
solution $p^N$ of $L_{a(t)}^N-\partial_t$ is available,
but we do not have its pointwise gradient estimate in general; see \cite{DD}.

Related to the stochastic homogenization, a weighted integral estimate
on the gradient of $p^N$ in case $a(t)\equiv a$ is obtained in \cite{GNO-15},
Theorem 3.  Especially, their estimate is global in time with optimal decay in $t$.
See \cite{AN}, Proposition 2.1 for the gradient estimate of $p^N$
in a random conductance model and \cite{DG}, Theorem 1.3 for that on
percolation clusters.  See also \cite{GNO-14}, Lemmas 3.2 and 3.3 for estimates in
quenched or annealed sense for the first and second derivatives of
elliptic Green's function on $\T_N^n$.

The convergence as $N\to\infty$ of the solution of linear discrete PDEs 
with boundary condition is discussed in Section 10.5 of \cite{Krylov}, 
in which the time parameter is also discretized.

\section{H\"older estimate for the solution of a discrete PDE}
\label{sec:Nash}

We first define spatial gradients and other notation in Section \ref{sec:2.0}.
Then, in Section \ref{sec:2.1}, recalling Appendix B of Giacomin, Olla and Spohn
\cite{GOS}, we formulate a Nash estimate for the solution of a discrete
linear PDE \eqref{eq:1.2-linear} with $g\equiv 0$ on $\tfrac1N \Z^n$ instead of
$\tfrac1N\T_N^n$.   Based on this,
we derive H\"older regularity of the solution $u^N$ of  \eqref{eq:1.2-linear}
with $g$ and therefore \eqref{eq:2.1-X} on $\tfrac1N\T_N^n$
in Section \ref{sec:HolderForSolution}.  In Section \ref{sec:Regt=0}, we show
that the singularity in H\"older estimate at $t=0$ can be removed under a
regularity assumption on the initial value $u^N(0)$.
In Section \ref{sec:1-comparison}, we show \eqref{eq:1.u_pm} and formulate
the maximum principle.

\subsection{Discrete derivatives and Laplacian}
\label{sec:2.0}

For functions $u$ and $v$ on $\frac1N\T_N^n$,
we define the inner product by
\begin{align}  \label{eq:2.innerproduct}
\lan u,v\ran_N= \tfrac1{N^n} \sum_{x\in \T_N^n} u(\tfrac{x}N)v(\tfrac{x}N).
\end{align}
For $e\in \Z^n: |e|=1$, we define the discrete derivative $\nabla_e^N$ to the 
direction $e$ and $\nabla_e^{N,*}$ by
\begin{align}  \label{eq:2.nablaeN}
\nabla_e^Nu(\tfrac{x}N) := N(u(\tfrac{x+e}N)-u(\tfrac{x}N)), \quad
\nabla_e^{N,*}u(\tfrac{x}N) := N(u(\tfrac{x-e}N)-u(\tfrac{x}N)).
\end{align}
Note that $\nabla_e^{N,*} = \nabla_{-e}^N$ is the dual operator of $\nabla_e^N$
with respect to the inner product $\lan\;,\;\ran_N$.  

We also define the discrete Laplacian by 
\begin{align}  \label{eq:Lap-N}
\De^N := -\tfrac12 \sum_{|e|=1} \nabla_e^{N,*} \nabla_e^N
= -\sum_{|e|=1,e>0} \nabla_e^{N,*} \nabla_e^N
=N\sum_{|e|=1} \nabla_e^N.
\end{align}
Here, $e>0$ for $e:|e|=1$ means that its nonzero component is $1$. 
To see the second and third identities in \eqref{eq:Lap-N}, we note the following simple fact.  
Let $\t_y, y\in \Z^n,$ be the shift operator
defined by $\t_yu(\tfrac{x}N)=u(\tfrac{x+y}N)$.  Then, we have
\begin{align}  \label{eq:2.4}
\nabla_e^{N,*}=\nabla_{-e}^N=-\t_{-e}\nabla_e^N, \quad
[\t_y,\nabla_e^N]=[\nabla_e^N,\nabla_{e'}^N] = [\nabla_e^N,\nabla_{e'}^{N,*}]=0
\end{align}
for every $e,e':|e|=|e'|=1$, where $[A,B] = AB-BA$ denotes the commutator of two
operators $A$ and $B$.  

The second identity in \eqref{eq:Lap-N} follows from
$\nabla_e^{N,*} \nabla_e^N = \nabla_{-e}^{N} \nabla_{-e}^{N,*} = 
\nabla_{-e}^{N,*} \nabla_{-e}^N$, while the third identity $\De^N = N\sum_{|e|=1} \nabla_e^N$
follows from $\De^N u(\tfrac{x}N) 
= N\sum_{|e|=1,e>0}$ $ \big(\nabla_e^N u(\tfrac{x}N)- \nabla_e^N u(\tfrac{x-e}N)\big)$ and
$\nabla_e^N u(\tfrac{x-e}N)= \t_{-e}\nabla_e^N u(\tfrac{x}N) = -\nabla_{-e}^Nu(\tfrac{x}N)$.

Given $a=\{a_{x,e}\}_{x\in \T_N^n, e\in\Z^n:|e|=1} \in \mathcal{A}(c_-,c_+)$
(recall two conditions \eqref{eq:a-1-N} and \eqref{eq:a-2-N}),
we consider the discrete diffusion operator $L^{N}_{a}$ defined by \eqref{eq:1-L_a}.
As an extension of \eqref{eq:Lap-N},
with the relations \eqref{eq:2.4} in hand and by 
\eqref{eq:a-1-N}, it can be rewritten as 
\begin{align}\label{eq:2.6}
L^N_{a} u(\tfrac{x}N) \; \Big(\!\! = N \sum_{|e|=1} a_{x,e}\nabla_e^N u(\tfrac{x}N) \Big)
= - \sum_{|e|=1,e>0} \nabla_e^{N,*}\big( a_{x,e}\nabla_e^N u\big)(\tfrac{x}N),
\end{align}
since we have
\begin{align}  \notag
\nabla_{-e}^{N,*}\big( a_{x,-e}\nabla_{-e}^N u\big)(\tfrac{x}N)
& = \nabla_{-e}^{N,*}\big( a_{x-e,e}\nabla_{-e}^N u\big)(\tfrac{x}N)  \\
& = -\nabla_{-e}^{N,*}\big(\t_{-e} a_{x,e} \t_{-e}\nabla_e^N u\big)(\tfrac{x}N) \notag\\
&=  -\nabla_{-e}^{N,*}\big(\t_{-e}( a_{x,e}\nabla_e^N u)\big)(\tfrac{x}N) \notag \\
& = \nabla_e^{N,*}\big( a_{x,e}\nabla_e^N u\big)(\tfrac{x}N).  \notag
\end{align}
Moreover, $L^{N}_{a}$ is symmetric with respect to $\lan\;,\;\ran_N$
and the corresponding Dirichlet form is given by
\begin{align}  \label{eq:2.5}
\mathcal{E}_a^N(u,v) := - \lan L_{a}^N u, v\ran_N
= \tfrac{1}{2N^n} \sum_{x\in \T_N^n} \sum_{|e|=1}
a_{x,e} \nabla_e^N u (\tfrac{x}N) \nabla_e^N v (\tfrac{x}N).
\end{align}
Given $a(t)\in \mathcal{A}(c_-,c_+)$ for $t\ge 0$, we will consider the time-dependent
operator $L_{a(t)}^N$.

The following will be used in Section \ref {sec:Li-Schauder} and later.
We denote $\De_a^N := L_a^N$ in case that $a=\{a_e\}_{|e|=1}$ does not
depend on $x$ and call $\De_a^N$ the discrete Laplacian with constant
coefficients.  Namely, for given constants $a=\{a_e\}_{|e|=1}$ 
satisfying
\begin{align}  \label{eq:3.5-a}
a_e=a_{-e}, \quad c_-\le a_e\le c_+,
\end{align}
we set
\begin{align}  \label{eq:Lap-a-N}
\De_a^N := - \sum_{|e|=1,e>0} a_e \nabla_e^{N,*}\nabla_e^N 
\ = \  - \tfrac12 \sum_{|e|=1}  \nabla_e^{N,*}a_e\nabla_e^N
\ = \ N\sum_{|e|=1}a_e\nabla_e^N.
\end{align}
The following formulas for discrete derivatives will be useful:
\begin{align}  \label{eq:disc-der-D}
\nabla_e^N(uv)(\tfrac{x}N) &= v(\tfrac{x+e}N)\nabla_e^N u(\tfrac{x}N) + u(\tfrac{x}N) \nabla_e^Nv(\tfrac{x}N) \\
& = \tfrac12(v(\tfrac{x}N)+v(\tfrac{x+e}N)) \nabla_e^Nu(\tfrac{x}N)
+ \tfrac12 (u(\tfrac{x}N)+u(\tfrac{x+e}N))\nabla_e^N v(\tfrac{x}N).
\notag
\end{align}
In particular,
\begin{align}  \label{eq:disc-der-E}
\nabla_e^N u^2(\tfrac{x}N) 
= (u(\tfrac{x}N)+u(\tfrac{x+e}N)) \nabla_e^Nu(\tfrac{x}N).
\end{align}
Note also that
\begin{align}  \label{eq:2.11}
\De_a^N(uv)(\tfrac{x}N)
= v(\tfrac{x}N) \De_a^N u(\tfrac{x}N)
+ u(\tfrac{x}N) \De_a^N v(\tfrac{x}N)
+ \sum_{|e|=1} a_e \nabla_{e}^N u(\tfrac{x}N)\nabla_e^N v(\tfrac{x}N).
\end{align}

\subsection{Nash continuity estimate for linear discrete PDEs on $\tfrac1N\Z^n$}
\label{sec:2.1}

In order to apply the results in \cite{GOS}  and \cite{SZ}, we consider the linear discrete 
PDE \eqref{eq:1.2-linear} with $g\equiv 0$ for $\tfrac{x}N\in \tfrac1N\Z^n$ instead of 
$\tfrac1N\T_N^n$:
\begin{align}  \label{eq:2.7}
\partial_t u = L_{a(t)}^N u, \quad \tfrac{x}N \in \tfrac1N\Z^n.
\end{align}
Here, the coefficient $a(t)=\{a_{x,e}(t)\}_{x\in \Z^n, e\in\Z^n:|e|=1} 
\in \mathcal{A}(c_-,c_+)$ satisfying \eqref{eq:a-1-N} and \eqref{eq:a-2-N}
is given for $t\ge 0$ and for $x\in \Z^n$ instead of $x\in \T_N^n$.
We assume $a(t)$ is continuous in $t$.
Then, the operator $L^{N}_{a(t)}$ is defined by \eqref{eq:1-L_a} (or \eqref{eq:2.6})
for $x\in \Z^n$.
The inner product $\lan\, ,\, \ran_N$ in \eqref{eq:2.innerproduct},
discrete gradients $\nabla_e^N, \nabla_e^{N,*}$ in \eqref{eq:2.nablaeN}
and discrete Laplacian $\De^N$ in \eqref{eq:Lap-N} are easily extended to
$\frac1N\Z^n$.

We take $\a=N^{-1}$ as the scaling parameter $\a$ in \cite{GOS}, p.1167-- 
and in \cite{SZ}.  We note that the factor $N^2 (=\a^{-2})$
from the time change is hidden in $\nabla_e^N$ and $\nabla_e^{N,*}$
in our case.

The  temporally inhomogeneous semigroup generated by $L^{N}_{a(t)}$ is denoted by
$P_{s,t}^N$ for $0\leq s<t<\infty$, that is, $u(t,\tfrac{x}N)=P_{s,t}^N \bar u (\tfrac{x}N)$ 
is the solution of \eqref{eq:2.7} for $t\ge s$ satisfying $u(s,\cdot)=\bar u (\cdot)$.  
The corresponding fundamental solution is given by
$$
p^N(s,\tfrac{y}N;t,\tfrac{x}N) = N^n (P_{s,t}^N 1_{\tfrac{y}N})(\tfrac{x}N).
$$

Note that $p^N$ satisfies the forward equation $\partial_t p^N = L_{a(t), x}^Np^N$ for $t \in 
(s,\infty)$ with $p^N(s,\tfrac{y}N;s,\tfrac{x}N)=N^n \de_{x,y}$ and the backward equation
$\partial_s p^N = - L_{a(s),y}^N p^N$, for $s \in (0,t)$ with $p^N(t,\tfrac{y}N;t,\tfrac{x}N)=N^n \de_{x,y}$;
see Definition \ref{defn:fund} below and recall the symmetry of $L_{a(t)}^N$.

The following proposition is a consequence of the parabolic Harnack inequality.

\begin{prop}  \label{prop:Nash}
(Nash continuity estimate)
There exist $\si\in (0,1)$ and $C=C(n, c_\pm)$
(in particular,
uniform in $N$) such that for every $N\in \N$ and $u_0\in L^\infty(\frac1N\Z^n)$
$$
|P_{0,t}^N u_0(\tfrac{x}N)- P_{0,s}^N u_0(\tfrac{y}N)|
\le C \|u_0\|_{\infty} \Big( \frac{|t-s|^{\frac12}\vee |\tfrac{x}N
- \tfrac{y}N|}{(t\wedge s) ^{\frac12}} \Big)^\si,
$$
for $t,s>0$ and $x,y\in \Z^n$,
where $\|u_0\|_{\infty}= \sup_{x\in \Z^n} |u_0(\tfrac{x}N)|$.  
In particular, we have
$$
\big| p^N(0,\tfrac{y'}N; t', \tfrac{x'}N)
- p^N(0,\tfrac{y}N; t, \tfrac{x}N)\big|
\le C \Big( |t'-t|^{\frac12} + |\tfrac{x}N
- \tfrac{x'}N| + |\tfrac{y}N - \tfrac{y'}N| \Big)^\si 
  (t'\wedge t)^{-\frac{n+\si}2}.
$$
\end{prop}

\begin{proof}
See Proposition B.6 of \cite{GOS} and also Theorem 1.31 of Stroock and Zheng \cite{SZ}.
Note that  \cite{SZ} deals with the temporally homogeneous case.  Note also that the
symmetry of $L_{a(t)}^N$ implies that of $p^N$ in $\tfrac{x}N$ and $\tfrac{y}N$ by
Lemma \ref{lem:13.2} below.
\end{proof}

\begin{rem}  \label{rem:2.sigma}
\rm
The H\"older exponent $\si\in (0,1)$ is implicitly determined.  Its lower bound is 
discussed in \cite{HN} in a continuous setting. 
We note that, in the semilinear case with $\fa(u)=cu$ with $c>0$,  
the bound \eqref{eq:main-4} is shown with $\si=1$ under the assumption  
$\sup_N\|u^N(0)\|_{C_N^1} \le CK$; see \cite{FT}, Proposition 4.3.
See also Remark \ref{rem:3.CH} below.
\end{rem}

\subsection{H\"older estimate for discrete PDEs with an external term on $\tfrac1N\T_N^n$}
\label{sec:HolderForSolution}

In this subsection, we make use of the Nash estimate (Proposition \ref{prop:Nash})
to show H\"older regularity for the solution $u^N(t,\tfrac{x}N)$ of \eqref{eq:1.2-linear} 
with the external term $g=g(t,\tfrac{x}N)$ and accordingly 
\eqref{eq:2.1-X} in $(t,\tfrac{x}N)$, $t>0, \tfrac{x}N\in \tfrac1N\T_N^n$, uniformly in $N$.
Recall that $g$ is bounded.
To apply Proposition \ref{prop:Nash} with respect to $\frac1N\Z^n$, we extend the
solution $u^N(t,\tfrac{x}N)$ to $\frac1N\Z^n$ periodically in $x$, setting  $u^N(t,\tfrac{x}N+e)
=u^N(t,\tfrac{x}N)$ for every $e\in \Z^n$.  Alternatively, we may solve
the discrete PDE \eqref{eq:1.2-linear}  on $\frac1N\Z^n$ with $N$-periodically 
extended coefficients $a(t)=\{a_{x,e}(t)\}$, $g(t,\tfrac{x}N)$ and initial value $u^N(0,\tfrac{x}N)$, 
$x\in \Z^n$. Then, by uniqueness of solution,
$u^N(t,\tfrac{x}N)$ is also periodic. 
 
We now derive a H\"older estimate for the solution $u^N$ of \eqref{eq:1.2-linear}.
The following theorem holds on $\tfrac1N\Z^n$ for non-periodic $a(t)$ and $g$.
But, we state it only on $\tfrac1N\T_N^n$.  We define the distance on $\T^n
\cong [0,1)^n$ and therefore on $\tfrac1N\T_N^n$ as follows.  
For $z_1, z_2\in \T^n$ and for $x_1, x_2\in \T_N^n$,
\begin{align}  \label{eq:distanceTn}
\big| z_1-z_2 \big| := \min_{e\in \Z^n} \big| z_1-z_2+e \big|, \quad
\big| \tfrac{x_1}N - \tfrac{x_2}N \big| :=
\min_{x_2'= x_2 \,{\rm mod} \, N} \big| \tfrac{x_1}N - \tfrac{x_2'}N \big|.
\end{align}

\begin{thm} \label{thm:Holder-u^N}
Let $u^N(t,\tfrac{x}N)$ be the solution of \eqref{eq:1.2-linear} on
$\tfrac1N \T_N^n$ with $a(t)\in \mathcal{A}(c_-,c_+)$ being continuous in $t$
and initial value $u^N(0)$ satisfying $\|u^N(0)\|_\infty <\infty$.
Let $\si \in (0,1)$ be as in Proposition \ref{prop:Nash}.  Then, we have
\begin{align}  \label{eq:thm2.2}
|u^N(t_1,\tfrac{x_1}N) - u^N(t_2,\tfrac{x_2}N)|
\le C(\|g\|_\infty +\|u^N(0)\|_\infty) \frac{ \big| t_2-t_1 \big|^{\frac{\si}2}
+ \big| \tfrac{x_1}N - \tfrac{x_2}N \big|^\si }{(t_1\wedge t_2)^{\frac{\si}2}},
\end{align}
for $0<t_1, t_2\le T$ and $x_1, x_2\in \T_N^n$,
where $C=C(n, c_\pm, T)$
and $\|g\|_\infty :=\|g\|_{L^\infty([0,T]\times \tfrac1N\T_N^n)}$.
\end{thm}

\begin{proof}
{\it Step 1.}
Recall that  $P_{s,t}^N$ is the semigroup  associated to the discrete linear PDE
\eqref{eq:2.7} with periodic $a(t)$.
By Duhamel's formula, the periodically extended solution $u^N$ of \eqref{eq:1.2-linear}
with $g$ satisfies
\begin{align}  \label{eq:Duha-Q}
u^N(t,\tfrac{x}N) &= \big(P_{0,t}^Nu^N(0)\big)(\tfrac{x}N)
+ \int_0^t ds \big(P_{s,t}^Ng(s,\cdot)\big)(\tfrac{x}N) \\
& =: I_1^N(t,\tfrac{x}N) +  I_2^N(t,\tfrac{x}N),
\notag
\end{align}
where $I_1^N$ and $I_2^N$ satisfy
\begin{align*}
& \partial_t I_1^N = L_{a(t)}^NI_1^N, \quad I_1^N(0) = u^N(0),\\
& \partial_t I_2^N = L_{a(t)}^NI_2^N +g(t,\tfrac{x}N), \quad I_2^N(0) = 0.
\end{align*}

Observe that the first term $I_1^N(t,\tfrac{x}N)$ is $(\frac{\si}2,\si)$-H\"older
continuous and the difference $|I_1(t_1, \frac{x_1}{N}) - I_1(t_2, \frac{x_2}{N})|$ is 
bounded, with front factor 
$C(n, c_\pm)\|u^N(0)\|_\infty (t_1\wedge t_2)^{-\frac{\si}2}$,  using Proposition \ref{prop:Nash}.

\vskip .1cm

{\it Step 2.}  The second term $I_2^N(t,\tfrac{x}N)$ can be estimated as follows.
First for $x_1, x_2 \in \T_N^n$ (embedded in $\Z^n$), by Proposition \ref{prop:Nash},
shifting time by $s$,
\begin{align*}
| I_2^N&(t,\tfrac{x_1}N) - I_2^N(t,\tfrac{x_2}N)|  
 \le \int_0^t ds \big| \big(P_{s,t}^Ng(s,\cdot)\big)(\tfrac{x_1}N)
-\big(P_{s,t}^Ng(s,\cdot)\big)(\tfrac{x_2}N)\big|  \\
& \le C(n,c_\pm) \|g\|_\infty \big| \tfrac{x_1}N - \tfrac{x_2}N \big|^\si
\int_0^t (t-s)^{-\frac{\si}2}ds \le C(n,c_\pm, T) \|g\|_\infty \big| \tfrac{x_1}N - \tfrac{x_2}N \big|^\si,
\end{align*}
for $t\in [0,T]$, since $\si\in (0,1)$.

\vskip .1cm
{\it Step 3.} Next, let us show the  H\"older estimate for $I_2^N(t,\tfrac{x}N)$
in $t$.  The proof is similar.  For $0<t_1<t_2$,
\begin{align*}
I_2^N(t_2,\tfrac{x}N) - I_2^N(t_1,\tfrac{x}N)
& = \int_0^{t_1} ds \Big\{ \big(P_{s,t_2}^Ng(s,\cdot)\big)(\tfrac{x}N)
-\big(P_{s,t_1}^N g(s,\cdot)\big)(\tfrac{x}N)\Big\}  \\
& \hskip 10mm  + \int_{t_1}^{t_2} ds \big(P_{s,t_2}^Ng(s,\cdot)\big)(\tfrac{x}N) \\  
& =: I_{2,1}^N+ I_{2,2}^N.
\end{align*}
Here, $P_{s,t}^N1 \equiv 1$ since $u(t)\equiv 1$ solves
$\partial_t u= L_{a(t)}^Nu$ for $t>s$ with $u(s)\equiv 1$.  Also, 
$\big|\big(P_{s,t_2}^Ng(s,\cdot)\big)(\tfrac{x}N) \big| \le \| g\|_\infty
\big|P_{s,t_2}^N 1\big| = \| g\|_\infty$.  Hence,
$$
|I_{2,2}^N| \le \|g\|_\infty |t_2-t_1|.
$$
For $I_{2,1}^N$, by Proposition \ref{prop:Nash} shifting time by $s$ again,
\begin{align*}
|I_{2,1}^N| & \le C(n,c_\pm) \|g\|_\infty \big| t_2-t_1 \big|^{\frac{\si}2}
\int_0^{t_1} (t_1-s)^{-\frac{\si}2}ds  \le C(n,c_\pm, T) \|g\|_\infty \big| t_2-t_1 \big|^{\frac{\si}2}.
\end{align*}
Thus,  we obtain
$$
|I_2^N(t_2,\tfrac{x}N) - I_2^N(t_1,\tfrac{x}N)|
\le C(n,c_\pm,T)\|g\|_\infty \big| t_2-t_1 \big|^{\frac{\si}2}, \quad 0\le t_1,t_2\le T.
$$
The theorem is shown by combining these estimates.  
\end{proof}

The diverging denominator in the estimate in Theorem \ref{thm:Holder-u^N}
comes from $I_1^N(t,x)$.  In the next subsection, we remove this singularity when
$u^N(0, \cdot)$ has $C^2$-regularity.

Theorem \ref{thm:Holder-u^N} shows the following corollary for the equation 
\eqref{eq:2.1-X}.

\begin{cor} \label{cor:2.3}
Let $u^N(t,\tfrac{x}N)$ be the solution of \eqref{eq:2.1-X}  on
$\tfrac1N \T_N^n$ satisfying \eqref{eq:1.u_pm}.  Then, we have, with $C=C(n,c_\pm, T)$ that
\begin{align}  \label{eq:cor2.3-1}
|u^N(t_1,\tfrac{x_1}N) - u^N(t_2,\tfrac{x_2}N)|
\le C(K\|f\|_\infty +\|u^N(0)\|_\infty) \frac{ \big| t_2-t_1 \big|^{\frac{\si}2}
+ \big| \tfrac{x_1}N - \tfrac{x_2}N \big|^\si }{(t_1\wedge t_2)^{\frac{\si}2}},
\end{align}
for $0<t_1,t_2\le T$ and $x_1, x_2\in \T_N^n$, where $\|f\|_\infty = 
\|f\|_{L^\infty([u_-,u_+])}$.

Moreover, $a_{x,e}(t) := a_{x,e}(u^N(t))$
defined by \eqref{eq:a_xe} with $u=u^N(t)$ satisfies
\begin{align}  \label{eq:cor2.3-2}
|a_{x_1,e}(t_1) - a_{x_2,e}(t_2)| \le C(K\|f\|_\infty +\|u^N(0)\|_\infty) \frac{ \big| t_2-t_1 \big|^{\frac{\si}2}
+ \big| \tfrac{x_1}N - \tfrac{x_2}N \big|^\si }{(t_1\wedge t_2)^{\frac{\si}2}},
\end{align}
for $0<t_1,t_2\le T$ and $x_1, x_2\in \T_N^n$ where $C=C(n, c_\pm, T, \|\fa''\|_\infty)$
and $\|\fa''\|_\infty =\|\fa''\|_{L^\infty([u_-,u_+])}$.
\end{cor}

\begin{proof} The first estimate \eqref{eq:cor2.3-1} is immediate from
Theorem \ref{thm:Holder-u^N} by noting \eqref{eq:De-Lta}, \eqref{eq:c_pm-N}
and $\|g\|_\infty \le K \|f\|_{L^\infty([u_-,u_+])}$.  Note that the continuity of
$a(t)$ in $t$ follows from that of $u^N(t)$, \eqref{eq:a_xe} and the next Lemma
\ref{lem:mvt}.  The second estimate \eqref{eq:cor2.3-2} follows from \eqref{eq:a_xe},
Lemma \ref{lem:mvt} and \eqref{eq:cor2.3-1}.
\end{proof}

We state the `mean-value' lemma used in the proof of Corollary \ref{cor:2.3}.

\begin{lem}
\label{lem:mvt}
If $\fa\in C^2([u_-,u_+])$, we have for every $a, b, c, d \in [u_-, u_+]$
such that $a\not= b$, $c\not= d$,
$$
\Big| \tfrac{\fa(a)-\fa(b)}{a-b} - \tfrac{\fa(c)-\fa(d)}{c-d} \Big|
\le C(|a-c|+|b-d|),
$$
and also
$$\Big| \tfrac{\fa(a)-\fa(b)}{a-b} - \fa'(c)\Big|
\le C(|a-c|+|b-c|).
$$
where we can take $C=\tfrac12 \|\fa''\|_{L^\infty ([u_-,u_+])}$ and recall \eqref{eq:1.u_pm}
for $u_\pm$.
\end{lem}

\begin{proof}
We prove the first statement as the second follows from the first by letting $d\to c$.

The left hand side is bounded by
$$
\Big| \tfrac{\fa(a)-\fa(b)}{a-b} - \tfrac{\fa(a)-\fa(d)}{a-d} \Big|
+ \Big| \tfrac{\fa(a)-\fa(d)}{a-d} - \tfrac{\fa(c)-\fa(d)}{c-d} \Big|,
$$
so that we may bound each term by $C|b-d|$ and $C|a-c|$, respectively.
But these are essentially the same by symmetry so that we may prove the first one.
Fix $a$ and $d$, and set
$$
\theta(b):=\tfrac{\fa(a)-\fa(b)}{a-b} - \tfrac{\fa(a)-\fa(d)}{a-d}.
$$
Then, $\theta(d)=0$ and by Taylor's formula
$$
\theta'(b)= \tfrac{-\fa'(b)(a-b)+(\fa(a)-\fa(b))}{(a-b)^2} = \tfrac12 \fa''(A)
$$
for some $A\in [u_-,u_+]$, which implies $|\theta(b)| = |\theta(b)-\theta(d)| \le C |b-d|$.
\end{proof}

\begin{rem}
{\rm  Under the $N$-periodic situation,
the fundamental solution $p^N (s,\tfrac{y}N; t, \tfrac{x}N)$,
$\tfrac{y}N, \tfrac{x}N \in \frac1N\T_N^n,$ on $\frac1N\T_N^n$ 
can be constructed from the fundamental solution 
$\tilde p^N (s,\tfrac{y}N; t, \tfrac{x}N)$ on $\frac1N\Z^n$ as follows:
$$
p^N (s,\tfrac{y}N; t, \tfrac{x}N) := \sum_{x'\in \Z^n: x'=x
\, {\rm mod } \,N} \tilde p^N (s,\tfrac{y}N; t, \tfrac{x'}N).
$$
}
\end{rem}

\subsection{Regularity at $t=0$}
\label{sec:Regt=0}

Here, we improve the H\"older regularity near $t=0$ 
of the solution $u=u^N(t,\tfrac{x}{N})$ for $x\in \T^n_N$
to the equation \eqref{eq:1.2-linear} given in Theorem \ref{thm:Holder-u^N}, 
when the initial condition $u^N(0,\tfrac{x}{N})$ has uniformly bounded second derivatives:  
\begin{align}
\label{discrete-C^2}
\sup_N
\max_{x\in \T^n_N, |e_1|=|e_2|=1}|\nabla^N_{e_1}\nabla^N_{e_2} u^N(0,\tfrac{x}{N})| 
\le C_0<\infty.
\end{align}
A sufficient condition is that $u^N(0,\cdot) = u_0(\cdot)$ for $u_0\in C^2(\T^n)$.
Recall that $a=\{a^N_{x,e}(t)\}$ satisfies \eqref{eq:a-1-N} and \eqref{eq:a-2-N}
for $t\ge 0$.

We show the following theorem.
The idea of the proof is that the regularity of $u^N(0,\cdot)$ allows to extend the solution
backwards in time, say for $t\in [-1,0]$.  Then, after a time-shift by $+1$, one formulates
a new discrete PDE,
whose diffusion coefficient still satisfies symmetry and nondegeneracy/boundedness as
in \eqref{eq:a-1-N} and \eqref{eq:a-2-N}, but which matches equation 
\eqref{eq:1.2-linear} for times $t\geq 1$.  From the H\"older estimate 
(Theorem \ref{thm:Holder-u^N}) to this new formulation for times $t\geq 0$,
one can get a H\"older estimate for the original equation above for times $t\geq 1$
but without a diverging divisor as before.

\begin{thm} \label{extended-Holder-thm}
Let $\si \in (0,1)$ be as in Proposition \ref{prop:Nash} and assume
$u^N(0,\cdot)$ satisfies \eqref{discrete-C^2} with $C_0<\infty$.
Then, for the solution $u^N(t,\tfrac{x}N)$ of \eqref{eq:1.2-linear}
with initial value $u^N(0,\cdot)$, we have
\begin{equation}
\label{extended_Holder_cor}
|u^N(t_1,\tfrac{x_1}N) - u^N(t_2,\tfrac{x_2}N)|
\le C(\|g\|_\infty+\|u^N(0)\|_\infty + C_0)  \Big\{ \big| t_2-t_1 \big|^{\frac{\si}2}
+ \big| \tfrac{x_1}N - \tfrac{x_2}N \big|^\si\Big\},
\end{equation}
for $0\le t_1,t_2\le T$ and $x_1, x_2\in \T_N^n$,
where $C=C(n, c_\pm, T)$.
\end{thm}

\begin{proof}
Consider the discrete heat equation on $\tfrac1N \T_N^n$:
\begin{align}  \label{eq:2.12}
\partial_s v = \De^N v,\quad s\in (0,1],
\end{align}
with initial condition $v(0,\tfrac{x}{N}) = u^N(0,\tfrac{x}{N})$ for $x\in \T^n_N$.
Define $\hat v(t) := v(1-t)$ for $0\leq t< 1$ and $\hat h(t,\tfrac{x}{N}) := -\De^N \hat v(t,\tfrac{x}{N})$.
Note, for $0\leq t< 1$, that $\hat v$ satisfies
\begin{align*}
\partial_t \hat v = -\De^N \hat v
= \De^N\hat v + 2\hat h.
\end{align*}

However, by \eqref{eq:2.12}, $h(t):= \hat h(1-t) = -\De^N v(t)$ satisfies the discrete heat equation
$\partial_t h(t) = \De^N h(t)$ with initial value $h(0)=-\De^Nu^N(0)$ and thus,
by the maximum principle for the discrete heat equation (see Lemma \ref{lem:1.1} below)
and by the condition \eqref{discrete-C^2} for $u^N(0)$, we have
\begin{align*}
|h(t,\tfrac{x}{N})| \le \max_y |\De^N u^N(0,\tfrac{y}N)| \le  C_0,
\end{align*}
which implies $\|\hat h\|_\infty \le C_0$.

Define now 
\begin{align}
\label{hat a defn}
\hat a_{x,e}(t)= \left\{ \begin{array}{rl}
a_{x,e}(t-1) & \ {\rm for \ }t \ge 1\\
1& \ {\rm for \ } 0\leq t<1,
\end{array}\right.
\end{align}
and 
\begin{align}
\label{hat g defn}
\hat g(t,\tfrac{x}{N}) = \left\{ \begin{array}{rl}
g(t-1, \tfrac{x}{N}) & \ {\rm  for  \ }t\geq 1\\
2\hat h(t, \tfrac{x}{N}) & \ {\rm for \ } 0\leq t<1.
\end{array}
\right.
\end{align}
Consider the extended system, for $t\geq 0$,
\begin{align}
\label{extended_equation}
\partial_t \hat u^N = L_{\hat a(t)}^N \hat u^N + \hat g(t, \tfrac{x}N).
\end{align}

Since $\hat a$ satisfies symmetry and nondegeneracy/boundedness, conditions
\eqref{eq:a-1-N} and \eqref{eq:a-2-N} (modify $c_\pm$ if necessary), and 
$\|\hat g\|_{\infty} \leq \|g\|_{\infty} +2 C_0$,
Theorem \ref{thm:Holder-u^N} yields the following statement:
$$
|\hat u(t_1,\tfrac{x_1}N) - \hat u(t_2,\tfrac{x_2}N)|
\le C(\|g\|_\infty+\|u^N(0)\|_\infty + C_0) \frac{ \big| t_2-t_1 \big|^{\frac{\si}2}
+ \big| \tfrac{x_1}N - \tfrac{x_2}N \big|^\si }{(t_1\wedge t_2)^{\frac{\si}2}},
$$
for $0<t_1,t_2\le T+1$ and $x_1, x_2\in \T_N^n$,
where $C=C(n, c_\pm)>0$ depends on $n$, $c_\pm$ given in \eqref{eq:a-2-N},
and $T$.
We obtain, by specializing to times $1\leq t \leq T+1$ and noting that 
$\hat u(t, \cdot) = u^N(t-1, \cdot)$, the theorem for the solution 
$u=u^N$ of \eqref{eq:1.2-linear}.
\end{proof}

Theorem \ref{extended-Holder-thm}  immediately implies the following corollary
for the equation \eqref{eq:2.1-X}.  The proof is 
similar to that of Corollary \ref{cor:2.3}.

\begin{cor} \label{cor:2.6}
Let $u^N(t,\tfrac{x}N)$ be the solution of \eqref{eq:2.1-X}
satisfying \eqref{eq:1.u_pm}.  We assume that the initial value 
$u^N(0,\cdot)$ satisfies \eqref{discrete-C^2}.
Then, we have with $C=C(n, c_\pm, T)$ that
\begin{align}  \label{eq:cor2.6-1}
|u^N(t_1,\tfrac{x_1}N) - u^N(t_2,\tfrac{x_2}N)|
\le C(K\|f\|_\infty +\|u^N(0)\|_\infty + C_0) \Big\{ \big| t_2-t_1 \big|^{\frac{\si}2}
+ \big| \tfrac{x_1}N - \tfrac{x_2}N \big|^\si \Big\},
\end{align}
for $0\le t_1, t_2\le T$ and $x_1, x_2\in \T_N^n$.  Moreover, $a_{x,e}(t) := a_{x,e}(u^N(t))$
 satisfies
\begin{align}  \label{eq:cor2.6-2}
|a_{x_1,e}(t_1) - a_{x_2,e}(t_2)| \le C(K\|f\|_\infty +\|u^N(0)\|_\infty+C_0) \Big\{ \big| t_2-t_1 \big|^{\frac{\si}2}
+ \big| \tfrac{x_1}N - \tfrac{x_2}N \big|^\si \Big\},
\end{align}
for $0 \le t_1, t_2\le T$ and $x_1, x_2\in \T_N^n$ where $C=C(n, c_\pm, T, \|\fa''\|_\infty)$.
\end{cor}

\subsection{Comparison argument and maximum principle}  
\label{sec:1-comparison}

Here, we show \eqref{eq:1.u_pm} for the quasilinear discrete PDE \eqref{eq:2.1-X} and
formulate the maximum principle for linear discrete PDEs.

To show \eqref{eq:1.u_pm}, take $-\infty<u_-<u_+<\infty$ such that $f(u_-)>0, f(u_+)<0$ and
$u_-<u^N(0, \frac{x}N) <u_+$ for all $N, x,K$.  This is possible by our assumptions for
$f(u)$ and $u^N(0)$.  We only show the upper bound in \eqref{eq:1.u_pm}, since the lower
bound is similar.  

For a fixed $K>0$, take $\e>0$ small enough such that
$Kf(u_+)+\e<0$ and let $u_+(t)$ be the solution of the ODE $\frac{du_+}{dt}(t) =
Kf(u_+(t))+\e$ with the initial value $u_+(0) =u_+$.  We easily see that  $u_+(t) \le u_+$
holds for all $t\ge 0$.  

For the solution $u^N(t)$ of \eqref{eq:2.1-X},
set $\t:=\inf \{t>0; u^N(t,\frac{x}N) = u_+(t) \text{ for some }
x\in \T_N^n\}$ and $\t=\infty$ if the set $\{t>0; \cdots\} = \emptyset$.
If $\t=\infty$, by definition, we have $u^N(t,\frac{x}N)<u_+(t) \le u_+$ for all $t\ge 0$
and $x\in \T_N^n$ so that the upper bound in \eqref{eq:1.u_pm} holds.  Therefore, 
we may assume $\t<\infty$ and clearly $\t>0$.  Then, at $t=\t$, 
$u^N(\t,\frac{x}N)\le u_+(\t)$ for all 
$x\in \T_N^n$ and $u^N(\t,\frac{y}N)= u_+(\t)$ for some $y\in \T_N^n$.
However, by the equation \eqref{eq:2.1-X}, we have
\begin{align*}
\partial_t (u_+-u^N)(\t,\tfrac{y}N) 
& = \De^N \{\fa(u_+)-\fa(u^N)\}(\t,\tfrac{y}N) + K f(u_+(\t))+\e - Kf(u^N(\t,\tfrac{y}N)) \\
& = N^2\sum_{|e|=1} \Big\{ (\fa(u_+)-\fa(u^N))(\t,\tfrac{y+ e}N)
- (\fa(u_+)-\fa(u^N))(\t,\tfrac{y}N) \Big\}  +\e \\
& = N^2\sum_{|e|=1} (\fa(u_+)-\fa(u^N))(\t,\tfrac{y+ e}N) +\e \ge \e,
\end{align*}
where we set $u_+(t,\tfrac{x}N)=u_+(t)$ for all $x$.  Note $\De^N\fa(u_+)=0$ and that
we have used the increasing property of $\fa$.  
This, however, implies that $u^N(t,\tfrac{y}N) >u_+(t)$ for some $t\in (0,\t)$ close to
$\t$, which contradicts the definition of $\t$. Thus, the upper bound in \eqref{eq:1.u_pm}
is shown.

Now we state the maximum principle.
Let $a_{x,e}(t)\ge 0$ be given and consider the linear discrete PDEs of divergence form
\begin{align}  \label{eq:1.21}
\partial_t u = L_{a(t)}^N u, \quad \tfrac{x}N \in \tfrac1N\T_N^n,
\end{align}
and also non-divergence form
\begin{align}  \label{eq:1.22}
\partial_t u = \bar L_{a(t)}^N u
:= \sum_{|e|=1} a_{x,e}(t) \nabla_e^{N,*}\nabla_e^N u, \quad \tfrac{x}N \in \tfrac1N\T_N^n.
\end{align}
Note that the symmetry condition \eqref{eq:a-1-N} is not assumed for $a_{x,e}(t)$.
For completeness, we give a proof.  
%A probabilistic proof is also available.

\begin{lem}  \label{lem:1.1}
(1) Let $\La \subset \frac1N \T_N^n$ be given.  Assume that $u(t,\frac{x}N)$ is defined
on $[0,T]\times \bar\La$ and satisfy $(L_{a(t)}^N - \partial_t) u \ge 0$
or  $(\bar L_{a(t)}^N - \partial_t) u \ge 0$ on $(0,T]\times \La$. Then, we have
\begin{align}  \label{eq:1.23}
\max_{[0,T]\times \bar \La} u = \max_{\Ga_T} u,
\end{align}
where $\Ga_T=\{0\}\times \bar\La \cup (0,T]\times \partial_N^+\La$, and
see \eqref{eq:3.ob+c} for $\partial_N^+\La$ and $\bar\La$.\\
(2) For the solutions of the equations \eqref{eq:1.21} or \eqref{eq:1.22}, we have
 $\max_{t\ge 0, x\in\T_N^n} u(t,\frac{x}N) = \max_{x\in\T_N^n} u(0,\frac{x}N)$ and
$\min_{t\ge 0, x\in\T_N^n} u(t,\frac{x}N) = \min_{x\in\T_N^n} u(0,\frac{x}N)$.
\end{lem}

\begin{proof}
(1) First note that, if the function $u$ takes local maximum at $\frac{y}N$ in the sense that
$u(\frac{y}N)\ge u(\frac{y+e}N)$ for every $e: |e|=1$, then we have
$L_{a}^Nu(\frac{y}N)\le 0$ and  $\bar L_{a}^Nu(\frac{y}N)\le 0$.  Indeed, the first
inequality is shown as 
\begin{align*}
L_a^N u(\tfrac{y}N) & = - \tfrac12 \sum_{|e|=1} \nabla_e^{N,*}(a_{y,e}\nabla_e^Nu)(\tfrac{y}N) \\
& = - \tfrac{N}2 \sum_{|e|=1} \Big\{ a_{y-e,e}\nabla_e^Nu(\tfrac{y-e}N) 
- a_{y,e}\nabla_e^Nu(\tfrac{y}N) \Big\} \le 0,
\end{align*}
since $a_{y-e,e}, a_{y,e} \ge 0$ and $\nabla_e^Nu(\tfrac{y-e}N) \ge 0$, $\nabla_e^Nu(\tfrac{y}N) \le 0$.
The second is similar.

Now, set $u^\e=u-\e t$ for $\e>0$ and prove \eqref{eq:1.23}
for $u^\e$ instead of $u$.  Once this is done, \eqref{eq:1.23} is shown for $u$ by
letting $\e\downarrow 0$.  
If \eqref{eq:1.23} does not hold for $u^\e$, there exist
$\frac{y}N\in \La$ and $\t\in (0,T]$ such that $u^\e(\t,\frac{y}N)= \max_{[0,T]\times \bar \La}
 u^\e(t,\frac{x}N)$.
Then, by the above observation, $L_{a(\t)}^Nu^\e(\t,\frac{y}N)\le 0$ and
$\bar L_{a(\t)}^Nu(\t,\frac{y}N)\le 0$ hold.  Thus, if the first condition is satisfied by $u$,
\begin{align}  \label{eq:1.24}
\partial_t u^\e(\t,\tfrac{y}N)
& = \partial_t u(\t,\tfrac{y}N) -\e \le L_{a(\t)}^Nu(\t,\tfrac{y}N) -\e \\
& = L_{a(\t)}^Nu^\e(\t,\tfrac{y}N) -\e \le -\e.  \notag
\end{align}
The same bound holds if the second condition is satisfied.

However, by the definition of $\t$ and $y$, we have $\partial_t u^\e(\t,\frac{y}N)=0$
in case of $0<\t<T$ or $\partial_t u^\e(\t,\frac{y}N)\ge 0$ in case of $\t=T$.
Both contradict \eqref{eq:1.24} so that \eqref{eq:1.23} is shown for $u^\e$.

The assertion (2) follows from (1) by taking $\La=\frac1N\T_N^n$. The statement for the minimum follows 
by considering $-u$.  
\end{proof}

\subsection{Comments on the probabilistic method to show the H\"older estimate}   \label{sec:7}

We consider the equation \eqref{eq:1.2-linear} in the case $g\equiv 0$.
At least in a temporally homogeneous
case and for the generator of the form
\begin{align}\label{eq:2.26}
L u(x) = \mu_x^{-1}\sum_{y:|x-y|=1} \mu_{xy}(u(y)-u(x)),
\quad \mu_x = \sum_y \mu_{xy},
\end{align}
the parabolic Harnack inequality 
is shown (\cite{BH}, \cite{Del}) and, based on it, H\"older continuity, which 
is uniform in $N$, of $u^N(t,\tfrac{x}N)$ in $(t,\tfrac{x}N)$  and the corresponding fundamental
solution is shown; see Proposition 3.2 and below (4.12) of \cite{BH} and
also p.100 below Theorem 8.1.5 of \cite{Kuma}.  Note that they consider the analysis on
a percolation cluster $\mathcal{C}_\infty$ or disordered media, which is quite complicated, 
but if we take the percolation probability $p=1$, we have $\mathcal{C}_\infty=\Z^d$.

The differences from our situation are:\\
\quad (1) Our generator $L^N_{a(t)}$ is temporally inhomogeneous\\
\quad (2)  $L$ in \eqref{eq:2.26} is defined by dividing by $\mu_x$ (cf.\ CSRW and VSRW, 
for example in \cite{Kuma},
for the relations of two generators, one not divided by $\mu_x$).

As we see in Section \ref{sec:EELevi}, if we have the H\"older continuity
in $(t,\tfrac{x}N)$ of the fundamental solution $p^N(s,\tfrac{y}N;t,\tfrac{x}N)$ of 
$L^N_{a(t)}- \partial_t $, by the symmetry of $L^N_{a(t)}$,
it is H\"older continuous also in $(s,\tfrac{y}N)$ (with diverging front factor for $|t-s|$
small as in (4.14) of \cite{BH}).  Then, for the solution of \eqref{eq:1.2-linear}
with $g$, one can apply Duhamel's formula.

A H\"older estimate of the fundamental solution is available in the quenched sense.
However, in general, a gradient estimate of the fundamental solution
is not available--as discussed in \cite{DD} an annealed gradient estimate holds but not
in the quenched sense.  Note that the parabolic Caccioppoli inequality, formulated
in Proposition 4.1 of \cite{DD}, is applied to obtain the gradient estimate.
See also \cite{AKM}, \cite{GNO-20} in a continuous setting.

\section{Preliminary estimates for the discrete heat equation, polylinear interpolation and some norms}
\label{sec:Li-Schauder}

Once the H\"older estimate is established, one can move to the next stage to prove
the Schauder estimate.  We adapt the approach in Chapter IV of \cite{Li96}
originally given for continuous PDEs.  This section summarizes some preparatory facts.

In Section \ref{sec:3.1}, we give bounds on the space-time oscillations of 
the solution of the discrete heat equation \eqref{eq:dHeq}; see Lemmas \ref{lem:Li-L4.4}
and \ref{lem:thm2.13}.  Then, in Section \ref{interior_poly_sec}, we extend 
the approach to continuous space by polylinear interpolations, and rewrite 
these bounds in the continuous setting; see Corollary \ref{cor:Li-L4.4}
and Proposition \ref{lem:Li-L4.5}.

With  Proposition \ref{lem:Li-L4.5} in hand, we can consider versions of the arguments 
in \cite{Li96}.  In Section \ref{sec:DirectAppl}, we introduce 
Campanato-Liberman norms, which are space-time H\"older norms of $u$ 
and $\nabla^N u$ with control of diverging factors near $t=0$.
Then, in Section \ref{sec:three},
we state Lemmas \ref{lem:Li-L4.3-G} and \ref{lem:Li-L4.6},
which correspond to Lemmas 4.3 and 4.6 in \cite{Li96}, respectively,
with a proper modification for Lemma \ref{lem:Li-L4.3-G} in the discrete setting.
Finally, in Section \ref{sec:SbP}, we state
the summation by parts formula, which will be used in the next section.

\subsection{Interior estimates for the discrete heat equation}  \label{sec:3.1}

Let $T>0$ be a fixed time horizon.
Define $\Om = [0,T]\times \T^n$ and $\Om_N=[0,T]\times \tfrac1N\T_N^n$.
For $X=(t,z)\in \Om$, set $|X|= \max\{t^{\frac12}, |z|\}$, where
$|z|$ is usually the $L^2$ (Euclidean) norm on $\T^n \cong [-\tfrac12,\tfrac12)^n$.
In context, to fit in the lattice structure, we will sometimes use the $L^\infty$ norm,
which is denoted by $|z|_{L^\infty}:= \max_{1\le i \le n} |z_i|$ for $z=(z_i)_{i=1}^n$,
instead of $L^2$ norm.

Define the domain $Q(R)= Q(X_0,R)$ in $\Om$ for $X_0=(t_0,z_0)\in\Om$ and 
$0<R \, (<\frac12\wedge \sqrt{t_0})$ as
\begin{align}
\label{Q_defn}
Q(R):=& \{X\in \Om; |X-X_0|<R, t<t_0\}\\
\equiv & (t_0-R^2,t_0)\times \{z\in \T^n; |z-z_0|<R\}.
\notag
\end{align}
We may regard the spatial part of $Q(R)$ as $\T^n$ when $R\ge \tfrac12$
and the temporal part as $[0,T]$ when $R \ge \sqrt{t_0}$.
We also define $Q_N(R) = Q(R) \cap \Om_N$.

We now show a discrete analog of Lemma 4.4 in \cite{Li96} for the solution of
the discrete heat equation \eqref{eq:dHeq}; see Lemma \ref{lem:Li-L4.4}
and Corollary \ref{cor:Li-L4.4}.
If a linear function is defined on a region of width $R$ and takes
values in $[-M,M]$, its slope behaves as $\frac{2M}R$.
The next lemma roughly claims that the solution $\U^N$ of 
\eqref{eq:dHeq} has a similar property.

\begin{lem} (cf.\ Proof of Lemma 4.4 of \cite{Li96})  \label{lem:Li-L4.4}
Let $\U^N=\U^N(t,\tfrac{x}N)$ be a solution of the discrete heat equation \eqref{eq:dHeq}
on $Q_N(R)$ with coefficients $a$ such that \eqref{eq:3.5-a} holds, that is, 
\begin{align}  \label{eq:heat-D}
\partial_t \U^N = \De_a^N \U^N, \quad  (t,\tfrac{x}N) \in Q_N(R),
\end{align}
satisfying $|\U^N| \le M_R$ on $Q_N(R)$ with $R>\frac{c_0}N$ for a constant 
$c_0>\sqrt{n}+1$.  (The condition $R>\frac{c_0}N$ will be removed in the setting
of Corollary \ref{cor:Li-L4.4}).  Then, we have
$$
|\nabla_e^N\U^N(t,\tfrac{x}N)| \le \frac{CM_R}R \quad \text{ in } \; Q_N(r),
$$
for $r\in (0,\tfrac12 c_1R)$ and $e\in\Z^n: |e|=1$, where $c_1=\frac1{\sqrt{n}}
\big( 1-\frac{\sqrt{n}+1}{c_0} \big) \in (0,1)$ and $C=C(n, c_\pm)$.  
Here, we take the spatial center
$z_0$ of $Q(R)$ (and therefore $Q(r)$) in $\frac1N\T_N^n$.
\end{lem}

Note that, for $\De_a^N \U^N$ in \eqref{eq:heat-D} to be defined on $Q_N(R)$,
$v^N$ should be given at least on its closure $\overline{Q_N(R)}$ defined as in
\eqref{eq:4.25-A} below.

\begin{proof}
The main method is to apply the maximum principle for the equation \eqref{eq:heat-D}.
Note that $\nabla_e^Nv^N$ satisfies the same equation, on a smaller domain,
due to the commuting property $[\De_a^N, \nabla_e^N]=0$.
\vskip .1cm

{\it Step 1.}
Let $\mathcal{L}=\De_a^N-\partial_t$ and let $\U=\U^N(t,\tfrac{x}N)$ be the solution of 
$\mathcal{L}v=0$ on $Q_N(R)$.  Then, by a simple computation and writing $\U(\tfrac{x}N)$ for
$\U(t,\tfrac{x}N)$ for simplicity, we have
\begin{align}  \label{eq:L-nablau^2-D}
\mathcal{L}(\nabla_e^N\U(\tfrac{x}N))^2 &= \sum_{|e'|=1} 
a_{e'} (\nabla_{e'}^N\nabla_e^N\U(\tfrac{x}N))^2, \\
\label{eq:L-u^2-D}
\mathcal{L}\U^2(\tfrac{x}N) &= \sum_{|e|=1} a_e (\nabla_e^N\U(\tfrac{x}N))^2,
\end{align}
on $Q_N(R-\tfrac1N)$.
Indeed, by \eqref{eq:Lap-a-N}, \eqref{eq:disc-der-E} and $\mathcal{L}\nabla_e^N v=0$
on $Q_N(R-\tfrac1N)$, which follows by noting that $(t,\tfrac{x+e}N)\in Q_N(R)$ for
$(t,\tfrac{x}N)\in Q_N(R-\tfrac1N)$,
\begin{align*}
\mathcal{L}(\nabla_e^N \U(\tfrac{x}N))^2
&= \De_a^N (\nabla_e^N \U(\tfrac{x}N))^2 - \partial_t (\nabla_e^N \U(\tfrac{x}N))^2 \\
& = N\sum_{|e'|=1} a_{e'}\nabla_{e'}^N (\nabla_e^N \U(\tfrac{x}N))^2 
-2 \nabla_e^N \U(\tfrac{x}N) \, \partial_t \nabla_e^N \U(\tfrac{x}N) \\
& = N\sum_{|e'|=1}a_{e'} (\nabla_e^N \U(\tfrac{x}N) + \nabla_e^N \U(\tfrac{x+e'}N) )
\nabla_{e'}^N \nabla_e^N \U(\tfrac{x}N)\\
& \qquad \quad  
-2 \nabla_e^N \U(\tfrac{x}N) N\sum_{|e'|=1} a_{e'}\nabla_{e'}^N \nabla_e^N \U(\tfrac{x}N)\\
& = N\sum_{|e'|=1} a_{e'}(\nabla_e^N \U(\tfrac{x+e'}N) - \nabla_e^N \U(\tfrac{x}N))
\nabla_{e'}^N \nabla_e^N \U(\tfrac{x}N)\\
&= \sum_{|e'|=1} a_{e'} (\nabla_{e'}^N \nabla_e^N \U(\tfrac{x}N))^2.
\end{align*}
This shows \eqref{eq:L-nablau^2-D}. 
Equation \eqref{eq:L-u^2-D} follows,
again by \eqref{eq:Lap-a-N} and \eqref{eq:disc-der-E}, from
\begin{align*}
\mathcal{L} \U^2(\tfrac{x}N)
&=\De_a^N \U^2(\tfrac{x}N) - \partial_t \U^2(\tfrac{x}N)\\
& = N\sum_{|e|=1} a_e \nabla_e^N \U^2(\tfrac{x}N) - 2 \U(\tfrac{x}N) 
N \sum_{|e|=1} a_e \nabla_e^N \U(\tfrac{x}N) \\
& = N\sum_{|e|=1} a_e (\U(\tfrac{x+e}N) - \U(\tfrac{x}N)) \nabla_e^N \U(\tfrac{x}N) 
= \sum_{|e|=1} a_e(\nabla_e^N \U(\tfrac{x}N))^2.
\end{align*}

\vskip .1cm
{\it Step 2.}
To accommodate the  discrete setting, especially to have a proper discrete boundary
in Step 5, instead of $Q(R)$, we consider a domain
$\tilde Q(R)$, defined similarly to $Q(R)=Q(X_0,R)$ but where we use the $L^\infty$ norm 
$|z|_{L^\infty}$ instead of $|z|$, to define $|X-X_0|$.  Note that $Q_N(R)\subset \tilde Q_N(R)
\subset \overline{\tilde Q_N(R)} \subset Q_N(\sqrt{n}(R+\tfrac1N))$ holds, where
$\tilde Q_N(R) = \tilde Q(R)\cap \Om_N$ and
$\overline{\tilde Q_N(R)} = \tilde Q_N(R) \cup \partial_N^+\tilde Q_N(R)$ is the domain
added the discrete outer boundary; see below.  We will take a smaller $c_1 R$ with
$c_1=\frac1{\sqrt{n}}\big( 1-\frac{\sqrt{n}+1}{c_0} \big)$ instead
of $R$ such that $Q_N(\sqrt{n}(c_1R+\tfrac1N)) \subset Q_N(R-\tfrac1N)$ holds.
In particular, $Q_N(c_1R)\subset \overline{\tilde Q_N(c_1R)} \subset Q_N(R-\tfrac1N)$
holds.

Define a cut off function $\zeta(X) =((c_1R)^2-|z-z_0|_{L^\infty}^2)^+((c_1R)^2-|t_0-t|)^+$
for $X=(t,z) \in \overline{\tilde Q_N(c_1R)}$ and $X_0=(t_0,z_0)$.  Actually for $\zeta$, we only use 
the properties: \eqref{eq:R-A} below, $\zeta\le C(n)R^4$ and $|\partial_t \zeta| \le C(n)R^2$.
Note that  $\zeta=0$  at the discrete outer boundary
$\partial_N^+\tilde Q(c_1R) = \partial\big([0,t_0]\times \frac1N\T_N^n \cap \tilde Q(c_1R)^c\big)$.

By \eqref{eq:L-nablau^2-D}, \eqref{eq:disc-der-E} and \eqref{eq:2.11},
for each $e\in \Z^n: |e|=1$, we have on $\tilde Q_N(c_1R)$
\begin{align}  \label{eq:L-zeta-nabla-u-2}
\mathcal{L}(\zeta\nabla_e^N\U)^2(\tfrac{x}N) 
&= \zeta^2(\tfrac{x}N) \mathcal{L}( \nabla_e^N\U(\tfrac{x}N))^2
+ (\nabla_e^N\U(\tfrac{x}N))^2 \mathcal{L}\zeta^2 (\tfrac{x}N)  \\
& \hskip 10mm
+ \sum_{|e'|=1} a_{e'} \nabla_{e'}^N\zeta^2(\tfrac{x}N) \cdot
\nabla_{e'}^N( \nabla_e^N\U)^2(\tfrac{x}N) \notag \\
&= \zeta^2(\tfrac{x}N)
\sum_{|e'|=1} a_{e'} (\nabla_{e'}^N\nabla_e^N\U(\tfrac{x}N))^2
+ (\nabla_e^N\U(\tfrac{x}N))^2 \mathcal{L}\zeta^2 (\tfrac{x}N) \notag \\
& \hskip 10mm
+ \sum_{|e'|=1} a_{e'} \nabla_{e'}^N\zeta^2(\tfrac{x}N) \cdot
(\nabla_e^N\U(\tfrac{x}N)+ \nabla_e^N\U(\tfrac{x+e'}N))
\nabla_{e'}^N \nabla_e^N\U(\tfrac{x}N) \notag   \\ \notag
&=: I_1+I_2+I_3,
\end{align}
where $I_1\geq 0$.
This is a discrete analog of the formula in line -9 in p.34 (proof of
Lemma 3.18) of \cite{Li96} ($H=\mathcal{L}, W=\zeta$).

\vskip .1cm
{\it Step 3.}  The continuous $\zeta$
(defined with $L^2$-norm for $|z-z_0|$) satisfies $\zeta\le CR^4$,
$|D\zeta|\le CR^3$ and $|D^2\zeta|\le CR^2$; see \cite{Li96}, p.52.
We have a similar property here for $\zeta$ defined above with the $L^\infty$-norm.  Since the
space singularity is not a problem given the discrete setting,
noting $R>\frac{c_0}N$, comparable to the mesh size $\frac1N$,
we have the same estimates for the discrete derivatives of $\zeta$ 
by the mean value theorem:
\begin{align}  \label{eq:R-A}
|\nabla_e^N\zeta|\le C(n)R^3, \quad |\nabla_{e'}^N\nabla_e^N\zeta| \le C(n)R^2.
\end{align}
Therefore, noting $|\partial_t\zeta|\le CR^2$ and \eqref{eq:3.5-a}, we have
\begin{align}  \label{eq:R-B}
|\mathcal{L}\zeta^2(\tfrac{x}N)| 
= \Big| 2\zeta(\tfrac{x}N)\De_a^N\zeta(\tfrac{x}N)
+\sum_{|e|=1} a_e\big(\nabla_e^N\zeta(\tfrac{x}N)\big)^2
- 2\zeta(\tfrac{x}N)\partial_t \zeta(\tfrac{x}N) \Big|
\le C(n,c_\pm)R^6.
\end{align}
In particular, we have
\begin{align*}
I_2 \ge& -C(n,c_\pm)R^6  (\nabla_e^N\U(\tfrac{x}N))^2.
\end{align*}

\vskip .1cm

{\it Step 4.}
For $I_3$ in the right hand side of \eqref{eq:L-zeta-nabla-u-2}, we first note
from \eqref{eq:disc-der-E} that
\begin{align*}  
\nabla_{e'}^N \zeta^2(\tfrac{x}N) 
= (\zeta(\tfrac{x}N)+ \zeta(\tfrac{x+e'}N)) \nabla_{e'}^N \zeta(\tfrac{x}N)
= (2\zeta(\tfrac{x}N) + \tfrac1N\nabla_{e'}^N \zeta(\tfrac{x}N)) \nabla_{e'}^N \zeta(\tfrac{x}N).
\end{align*}
Then, since $\tfrac1N\nabla_{e'}^N F(\tfrac{x}N)= F(\tfrac{x+e'}N)-F(\tfrac{x}N)$,
and using this for $F=\nabla_e^N u$,  we can rewrite $I_3$ as
\begin{align*}  
I_3 =& \sum_{|e'|=1} 2a_{e'} \zeta(\tfrac{x}N)\nabla_{e'}^N\zeta(\tfrac{x}N) \cdot
(\nabla_e^N\U(\tfrac{x}N)+ \nabla_e^N\U(\tfrac{x+e'}N))
\nabla_{e'}^N \nabla_e^N\U(\tfrac{x}N) \\
&+ \sum_{|e'|=1} a_{e'} (\nabla_{e'}^N\zeta(\tfrac{x}N))^2 \cdot
(\nabla_e^N\U(\tfrac{x}N)+ \nabla_e^N\U(\tfrac{x+e'}N))
(\nabla_e^N\U(\tfrac{x+e'}N)- \nabla_e^N\U(\tfrac{x}N)) \\
=:& I_{3,1}+I_{3,2}.
\end{align*}  
Here, by a simple bound $2bc\ge -(b^2+c^2)$ and then by \eqref{eq:R-A},
$I_{3,1}$ is estimated from below as
\begin{align*}
I_{3,1} \ge& - \zeta^2(\tfrac{x}N) \sum_{|e'|=1} 
a_{e'} (\nabla_{e'}^N \nabla_e^N\U(\tfrac{x}N))^2 \\
& - \sum_{|e'|=1} a_{e'}(\nabla_{e'}^N\zeta(\tfrac{x}N))^2
(\nabla_e^N\U(\tfrac{x}N)+ \nabla_e^N\U(\tfrac{x+e'}N))^2\\
\ge& -I_1 -C(n,c_\pm)R^6  \sum_{|e'|=1,0} (\nabla_e^N\U(\tfrac{x+e'}N))^2.
\end{align*}
The other term $I_{3,2}$ is also estimated from below by using \eqref{eq:R-A} as
\begin{align*}
I_{3,2} \ge& - \sum_{|e'|=1} a_{e'} (\nabla_{e'}^N\zeta(\tfrac{x}N))^2
(\nabla_e^N\U(\tfrac{x}N))^2
\ge -C(n,c_\pm)R^6  (\nabla_e^N\U(\tfrac{x}N))^2.
\end{align*}

\vskip .1cm
{\it Step 5.}
Summarizing these estimates, we obtain
\begin{align*}
\mathcal{L}(\zeta\nabla_e^N\U)^2(\tfrac{x}N) \ge -C(n,c_\pm)R^6 
\sum_{|e'|=1, 0} (\nabla_e^N\U(\tfrac{x+e'}N))^2,
\end{align*}
on $\tilde Q_N(c_1R)$.
This combined with \eqref{eq:L-u^2-D} shows
\begin{align}  \label{eq:L-maxPri}
\mathcal{L}\Big( (\zeta\nabla_e^N\U)^2(\tfrac{x}N)+ 
\tfrac{C(n,c_\pm)}{c_-} R^6 \sum_{|e'|=1, 0} \U^2(\tfrac{x+e'}N)\Big)\ge 0,
\end{align}
on $\tilde Q_N(c_1R)$, where $c_-$ is as in \eqref{eq:3.5-a}.

 However, since $\zeta=0$  at the discrete outer boundary
$\partial_N^+\tilde Q(c_1R)$,
this function, that is the function acted upon by $\mathcal{L}$ in the formula \eqref{eq:L-maxPri},
at $\partial_N^+ \tilde Q(c_1R)$ is bounded from above by 
$\tfrac{C}{c_-}(2n+1)R^6M_R^2$ by recalling that
$\overline{\tilde Q_N(c_1R)} \subset Q_N(R-\tfrac1N)$.
Thus, the maximum principle (Lemma \ref{lem:1.1})
for \eqref{eq:heat-D} shows
$$
\big( (\zeta\nabla_e^N\U)(\tfrac{x}N)\big)^2 \le \tfrac{C(n,c_\pm)}{c_-}(2n+1)R^6M_R^2,
$$
for each $e$ on $\tilde Q_N(c_1R)$.
Since $\zeta\ge C(c_1) R^4$ on $Q_N(r)$ from our initial assumption $r<\tfrac12c_1R$, we conclude the proof.
\end{proof}

The next lemma shows that the interior modulus of continuity estimate in 
$\tfrac{x}N$ implies that in $t$.

\begin{lem} \label{lem:thm2.13}
(cf.\ Theorem 2.13 and (2.27) in \cite{Li96})  Let $\U^N(t,\tfrac{x}N)$ be the
solution of the discrete heat equation
\eqref{eq:heat-D} on $Q= (t_1,t_1+ \frac{R^2}{16a_*})\times 
\{|\tfrac{x}N-z_0|<R\} \, (\subset \Om_N)$, $z_0\in \frac1N \T_N^n$, where
$a_*:= \sum_{|e|=1, e>0}a_e$.  Assume that
$R> \frac{c_0}N$ with some $c_0>2$ and
\begin{equation}  \label{eq:thm2.13}
|\U^N(t,\tfrac{x}N)-\U^N(t,z_0)|\le \om
\end{equation}
holds for $|\tfrac{x}N-z_0|<R$ and $t\in [t_1,t_1+ \frac{R^2}{16a_*}]$.
Then, we have
$$
|\U^N(t,z_0)-\U^N(t_1,z_0)|\le 2\om,
$$
for $t\in [t_1,t_1+ \frac{R^2}{16a_*}]$.
\end{lem}

\begin{proof}
The proof is essentially the same as that of Theorem 2.13 of \cite{Li96}
(in our case, $b=c=f=0$ and $\La_0=a_*$).  Set 
\begin{equation}  \label{eq:s-def}
s= \sup_{t\in (t_1,t_1+ \frac{R^2}{16a_*})} |\U^N(t,z_0)-\U^N(t_1,z_0)|
\end{equation}
and define
$$
\nu^\pm(t,\tfrac{x}N) = \tfrac{8sa_*}{R^2}(t-t_1) 
+ \tfrac{4s}{R^2}|\tfrac{x}N-z_0|^2
+ \om \pm (\U^N(t,\tfrac{x}N)-\U^N(t_1,z_0)),
$$
with $L^2$-norm for $|\tfrac{x}N-z_0|$.
Then, recalling $\mathcal{L}=\De_a^N-\partial_t$ and noting $\De_a^N |\tfrac{x}N-z_0|^2 = 2a_*$
and $\mathcal{L}\U^N=0$, we have
$$
\mathcal{L}\nu^\pm(t,\tfrac{x}N) = - \tfrac{8sa_*}{R^2} + \tfrac{8sa_*}{R^2} = 0
$$
in $Q$.  We remark that $\mathcal{L}\nu^\pm \le 0$ would be enough to apply the maximum principle below.

Let $\mathcal{P}_NQ$ be the discrete parabolic inner boundary of $Q$
defined by
$$
\mathcal{P}_NQ := \{t_1\}\times \{|\tfrac{x}N-z_0|<R\}
\bigcup  (t_1,t_1+ \tfrac{R^2}{16a_*})\times \partial_N^- \{|\tfrac{x}N-z_0|<R\},
$$
where $\partial_N^- E$ is the inner boundary of the set $E \subset \frac1N\T_N^n$,
that is, 
$$\partial_N^- E := \{\tfrac{x}N\in E; |\tfrac{x}N-\tfrac{y}N|=\tfrac1N {\rm  \ for \ some\ } 
\tfrac{y}N\notin E\}.$$

On $\mathcal{P}_NQ$, at $t=t_1$, we have 
$$\nu^\pm(t_1,\tfrac{x}N)\ge \om \pm (\U^N(t_1,\tfrac{x}N)
-\U^N(t_1,z_0))\ge 0$$
 by \eqref{eq:thm2.13}.  For $\tfrac{x}N \in
\partial_N^- \{|\tfrac{x}N-z_0|<R\}$, we have 
$$|\tfrac{x}N-z_0|\ge
|\tfrac{y}N-z_0| - |\tfrac{x}N-\tfrac{y}N|\ge R-\tfrac1N>\tfrac12 R,$$
since $R>\frac2N$, so that 
$$\nu^\pm(t,\tfrac{x}N) \ge s + \om \pm \big((\U^N(t,\tfrac{x}N)-\U^N(t,z_0))
+ (\U^N(t,z_0)-\U^N(t_1,z_0)\big) \ge 0$$ 
by \eqref{eq:thm2.13} and \eqref{eq:s-def}.
This shows $\nu^\pm\ge 0$ on $\mathcal{P}Q$.  

We now apply the maximum
principle and see that $\nu^\pm\ge 0$ on $Q$.
Thus, we have obtained 
$$\tfrac{8sa_*}{R^2}(t-t_1) + \om \ge|\U^N(t,z_0)-\U^N(t_1,z_0)|$$
at $\tfrac{x}N=z_0\in \frac1N\T_N^n$.  The conclusion follows
by taking the supremum in $t$.
\end{proof}

\subsection{Interior oscillation estimate for polylinear interpolations}
\label{interior_poly_sec}

We now reformulate the discrete space and time estimates,  Lemmas \ref{lem:Li-L4.4}
and \ref{lem:thm2.13} respectively, to the continuous setting via polylinear interpolation;
see Corollary \ref{cor:Li-L4.4} below.  

Let  $u^N=\{u^N(\tfrac{x}N); x \in \T_N^n\}$ be given and
define $\tilde u^N(z), z=(z_i)_{i=1}^n\in \T^n\, (=[0,1)^n$ or $[-\tfrac12,\tfrac12)^n)$ as 
a polylinear interpolation of $u^N$:
\begin{equation}  \label{eq:poli-A}
\tilde u^N(z) = \sum_{v\in \{0,1\}^n} \vartheta^N(v,z) u^N\big(\tfrac{[Nz]+v}N\big),
\qquad 
\vartheta^N(v,z) = \prod_{i=1}^n \vartheta^N(v_i,z_i),
\end{equation}
where $\vartheta^N(a,b)= \{Nb\} 1_{\{a=1\}}+ (1-\{Nb\}) 1_{\{a=0\}} \in [0,1]$ for
$a=0,1, b\in [0,1)$, $[Nz]=([Nz_i])_{i=1}^n$, and $[Nz_i]$ and
$\{Nb\}$ denote the integer and the fractional
parts of $Nz_i$ and $Nb$, respectively, and $v=(v_i)_{i=1}^n$.

In particular, we have
\begin{equation}  \label{eq:deri-u-F}
\partial_{z_i} \tilde u^N(z) = \sum_{v\in \{0,1\}^n} \vartheta_i^N(v,z)
 \nabla_{e_i}^Nu^N\big(\tfrac{[Nz]+\hat v_i}N\big),
\end{equation}
where $e_i\in \Z^n: |e_i|=1, e_i>0$ is the $i$th unit vector,
$\vartheta_i^N(v,z) = \frac12 \prod_{j\not=i}\vartheta^N(v_j,z_j)$
and $\hat v_i =(v_1,\ldots,v_{i-1},0,v_{i+1},\ldots,v_n)$; see \cite{DGI}.
Since $\sum_{v\in \{0,1\}^n} \vartheta_i^N(v,z) =1$ for each $i$, \eqref{eq:deri-u-F}
implies
\begin{equation}  \label{eq:deri-u-G}
|\partial_{z_i} \tilde u^N(z)| \le
\max_{\hat v_i: v\in \{0,1\}^n} \big| \nabla_{e_i}^Nu^N\big(\tfrac{[Nz]+\hat v_i}N\big)\big|.
\end{equation}

Lemma \ref{lem:Li-L4.4} roughly shows the spatial slope of 
the solution $\U^N$ of the discrete heat equation \eqref{eq:heat-D} is bounded
by $\frac{CM_R}R$ if $R>\frac{c_0}N$
and this combined with Lemma \ref{lem:thm2.13}
controls the oscillation in space and time.
For a function $F$ on $\Om$ and $Q\subset \Om$, we set
\begin{align}  \label{eq:3.osc}
[F]_0\equiv [F]_0^* := \underset{\Om}{\rm osc}\, F,  \quad
\underset{Q}{\rm osc}\, F :=  \sup_{X,Y\in Q} |F(X)-F(Y)|.
\end{align}
Then the following corollary holds.  This will be used later to show
Proposition \ref{lem:Li-L4.5}.

Before stating the corollary, let us examine the definability of the polylinear interpolation 
$\tilde \U$ on a subdomain.  In general for a domain $D\subset \T^n$, we define its 
continuous minimal cover by $\tfrac1N$-boxes (with all vertices in $\tfrac1N\T_N^n$) by
\begin{equation*}
D^*:= \bigcup_{x\in \T_N^n:B(\tfrac{x}N)\cap D \not=\emptyset} B(\tfrac{x}N) \; (\supset D),
\end{equation*} 
where $B(\tfrac{x}N)= \prod_{i=1}^n[\tfrac{x_i}N,\tfrac{x_i+1}N]$.
Set $D_N^*:= D^*\cap \tfrac1N \T_N^n$.  When  $\{\U(\tfrac{x}N); \tfrac{x}N\in D_N^*\}$
are given, its polylinear interpolation $\tilde \U(z)$ is definable on $D$, since
$\U(\tfrac{x}N)$ is defined at every vertex of the $\tfrac1N$-box containing $z\in D$.
Note that $D_N^*$ is slightly wider than $\overline{D_N} = D_N \cup \partial_N^+ D_N$
defined below \eqref{eq:outerb} or \eqref{eq:3.ob+c}, since $D_N^*$ contains
\lq\lq diagonal'' boundary points.

Denote by $D(r)=\{z\in \T^n; |z-z_0|<r\}$ the spatial part of $Q(r)$.
If $r>0$ satisfies $r+\tfrac{\sqrt{n}}N\le R$, then $(D(r))_N^* \subset D_N(R):=
D(R)\cap \tfrac1N\T_N^n$ holds, accordingly, if $\U^N$ is defined on $Q_N(R)$, 
$\tilde \U^N$ is definable on $Q(r)$.  Indeed, if $\tfrac{x}N\in (D(r))_N^*$, then
dist$(\tfrac{x}N,D(r))\le \tfrac{\sqrt{n}}N$ so that $|z_0-\tfrac{x}N|\le r+\tfrac{\sqrt{n}}N \le R$
and thus $\tfrac{x}N \in D_N(R)$.

In other words, the interior estimates on $Q(r)$, given the behavior of $\U^N$ on $Q_N(R)$, 
can be discussed only under the condition $r+\tfrac{\sqrt{n}}N\le R$.  If $r$ and $R$ 
are close, we need some outer condition on $\U^N$ 
(which, in application, is the solution of original PDE \eqref{eq:1.2-linear}
and not that of simpler discrete heat equation \eqref{eq:heat-D})
to have the estimate on $Q(r)$ (though, as noted below Lemma \ref{lem:Li-L4.4},
$v^N$ is given on $\overline{Q_N(R)}$).  We need a band area
to separate $Q(r)$ and $Q(R)$. This comes from the non-local property of the polylinear
interpolation and represents a difference from the continuous case.

We now state the corollary.
Differently from the continuous setting, we need to consider three different ranges
$0<r\le R_1<R$.  Especially, $R_1$ and $R$ should be distinguished with a gap of at least 
$\tfrac{\sqrt{n}}N$  due to the non-local nature of our problem.

\begin{cor}  \label{cor:Li-L4.4}
(cf.\ Lemma 4.4 of \cite{Li96})
Let $\U^N(t,\tfrac{x}N)$ be the solution of the discrete heat equation
\eqref{eq:heat-D} on $Q_N(R) (\not= \emptyset)$.
Assume $R_1+\tfrac{\sqrt{n}}N\le R$ for $R_1>0$ (in particular, $R>\frac{\sqrt{n}}N$)
 so that the polylinear interpolation
$\tilde \U^N(t,z)$ of $\U^N(t,\tfrac{x}N)$ as in \eqref{eq:poli-A} in the spatial variable
is well-defined on $Q(R_1)$.

We assume  $|\tilde\U^N|\le M_{R_1}$ on $Q(R_1)$.  Then, we have 
\begin{equation}  \label{eq:corLi-L4.4}
\underset{Q(r)}{\rm osc}\, \tilde \U^N \le \frac{C rM_{R_1}}{R_1},
\end{equation}
for every $0<r\le R_1 \le R -\tfrac{\sqrt{n}}N$.
Moreover, we have
\begin{equation}
\label{continuous gradient bound}
\sup_{ Q(r)}\max_i |\partial_{z_i} \tilde \U^N(t,z)| \leq \frac{CM_{R_1}}{R_1},
\end{equation}
where $r\in (0,\varepsilon R_1/2)$ and $0<\varepsilon<1/2$ is fixed in the proof.  Here, the constants $C=C(n, c_\pm)>0$.
\end{cor}

Note that the spatial center $z_0$ of $Q(r)=Q(X_0,r), X_0=(t_0,z_0),$ may not be on $\frac1N\T_N^n$.
Note also that we don't require the gap $\tfrac{\sqrt{n}}N$ between $r$ and $R_1$.
The estimate \eqref{continuous gradient bound} will not be used later, but it might be useful for 
the reader.   It holds due to the polylinearity in our setting.

\begin{proof}
The condition $|\tilde\U^N|\le M_{R_1}$ on $Q(R_1)$ implies $\underset{Q(r)}{\rm osc}\, \tilde \U^N\le 2M_{R_1}$, 
so that \eqref{eq:corLi-L4.4} is trivial (with $C=\frac2\e$) if $r\ge \e R_1$
for any $\e>0$.  Therefore, we assume $r<\e R_1$ in the following with a fixed
$\e>0$, but taken arbitrarily small by making $c_0>0$ large enough in the following
arguments.  

We now estimate the oscillation of $\tilde v^N$ on $Q(r)$ as
\begin{align}  \label{eq:tildeu-osc}
\underset{Q(r)}{\rm osc}\, \tilde \U^N
& = \sup_{X=(t,z),X'=(t',z')\in Q(r)} |\tilde \U^N(X)-\tilde \U^N(X')| \\
& \le \sup_{t,z} |\tilde \U^N(t,z)-\tilde \U^N(t,z_0)|
+ \sup_{t,t'} |\tilde \U^N(t,z_0)-\tilde \U^N(t',z_0)| \notag  \\
& \hskip 10mm  
+ \sup_{t',z'} |\tilde \U^N(t',z_0)-\tilde \U^N(t',z')|  \notag \\
& = 2\sup_{(t,z)\in Q(r)} |\tilde \U^N(t,z)-\tilde \U^N(t,z_0)|
+ \sup_{t,t'\in (t_0-r^2,t_0)} |\tilde \U^N(t,z_0)-\tilde \U^N(t',z_0)| \notag \\
& =: 2I_1+I_2.  \notag
\end{align}

We divide the analysis into four cases.
\vskip .1cm
{\it Case 1.}
We prepare two rough estimates when $R_1$ satisfies $R_1>\tfrac{\sqrt{n}+1}N$:
\begin{align}  \label{eq:nabla_e-O}
& |\nabla_e^N \U^N(t,\tfrac{x}N)|\le 2NM_{R_1} \qquad \text{on} \quad Q_N(R_1-\tfrac1N), \\
\label{eq:nabla_e-P}
& |\partial_{z_i} \tilde \U^N(t,z)|\le 2NM_{R_1} \qquad \text{on} \quad Q(R_1-\tfrac{\sqrt{n}+1}N).
\end{align}
Indeed, \eqref{eq:nabla_e-O} is shown from $|\nabla_e^N \U^N(t,\tfrac{x}N)|
= \big|N \big(\U^N(t,\tfrac{x+e}N) - \U^N(t,\tfrac{x}N)\big) \big| \le 2NM_{R_1}$ by noting
$\tfrac{x+e}N\in Q_N(R_1) \subset Q(R_1)$ for $\tfrac{x}N \in Q_N(R_1-\frac1N)$,
and $v^N=\tilde v^N$ on $\tfrac1N\T_N^n$.
Equation \eqref{eq:nabla_e-P} follows from \eqref{eq:nabla_e-O} and 
\eqref{eq:deri-u-G} by noting 
$\big( D(R_1-\tfrac1N-\tfrac{\sqrt{n}}N)\big)_N^* \subset D_N(R_1-\tfrac1N)$.

First consider the case that $R_1\le\frac{c_0}N$ for some large enough
$c_0>\sqrt{n}+1$ determined later and $r+\frac{\sqrt{n}+1}N\le R_1$ is satisfied,
where $c_0$ is the
same constant as in Lemma \ref{lem:Li-L4.4} but note that one can make $c_0$
larger in Lemma \ref{lem:Li-L4.4}.
In this case, since $Q(r)\subset Q(R_1-\tfrac{\sqrt{n}+1}N)$, \eqref{eq:nabla_e-P} is
applicable on $Q(r)$.  Thus, for $I_1$, from
\eqref{eq:nabla_e-P} and noting $|z-z_0|<r$ and then $N\le \frac{c_0}{R_1}$, we have
$$
I_1 \le \sqrt{n} 2NM_{R_1} r \le 2 \sqrt{n} c_0 M_{R_1}\tfrac{r}{R_1}.
$$

For $I_2$, by the discrete heat equation \eqref{eq:heat-D} and 
\eqref{eq:nabla_e-O}, noting $r+\frac{\sqrt{n}}N\le R_1-\tfrac1N$, 
and recalling $a_*=\sum_{|e|=1,e>0}a_e\leq nc_+$, we have
\begin{align}   \label{eq:1stcase}
|\partial_t\U^N(t,\tfrac{x}N)| = \Big| N\sum_{|e|=1}a_e\nabla_e^N \U^N(t,\tfrac{x}N)\Big|
\le 2a_* \cdot 2M_{R_1}N^2 \qquad \text{on} \quad Q_N(r+\tfrac{\sqrt{n}}N).
\end{align}
Further, note $\tilde \U^N(t,z_0)$ is a convex combination
of $\{\U^N(t,\tfrac{x}N)\}$ around $z_0$ as in \eqref{eq:poli-A}.  Observe also that
$X,X'\in Q(r)$ implies $t, t'\in (t_0-r^2,t_0)$ so that
$|t-t'|<r^2$.  These facts show that 
$$
I_2 \le 4a_* M_{R_1}N^2 |t-t'| \le 4a_* M_{R_1}N^2 r^2.
$$
Since $Nr\le \e NR_1\le \e c_0$ and $R_1\le\frac{c_0}N$, we obtain
$$
I_2   \le 4a_* \e c_0^2  \tfrac{M_{R_1} r}{R_1}.
$$
Thus, we obtain 
\eqref{eq:corLi-L4.4} in case that $R_1\le \frac{c_0}N$ and $r+\frac{\sqrt{n}+1}N\le R_1$
is satisfied. 

Also, in this case, \eqref{continuous gradient bound} holds from \eqref{eq:nabla_e-P} as 
$N\le \tfrac{c_0}{R_1}$.

\vskip .1cm
{\it Case 2.}
Second, we consider the case that $R_1> \frac{c_0}N$ and $0<r \le \frac{c_2}N$
with $c_2\ll c_0$, such that $r+\frac{\sqrt{n}+1}N\le R_1$ is satisfied, by choosing $c_2, c_0$ so that
$c_2+\sqrt{n}+1 < c_0$.  In this case, one can apply  Lemma \ref{lem:Li-L4.4} with 
the pair $(r+\frac{\sqrt{n}}N,R_1)$ in place of $(r,R)$ in this lemma, noting 
$r+\frac{\sqrt{n}}N<\frac{1}2 c_1 R_1$ if $c_0c_1>2(c_2+\sqrt{n})$, to get
that $|\nabla_e^N \U^N(t,\tfrac{x}N)|\le \frac{C(n,c_\pm)M_{R_1}}{R_1}$ on $Q_N(r+\tfrac{\sqrt{n}}N)$,
recall that $c_1$ is the constant given in Lemma \ref{lem:Li-L4.4}.
By \eqref{eq:deri-u-G}, this shows \eqref{continuous gradient bound}, 
namely $|\partial_{z_i} \tilde \U^N(t,z)|\le \frac{C(n,c_\pm)M_{R_1}}{R_1}$
on $Q(r)$.  From these estimates, we obtain bounds for $I_1$ and $I_2$  
as above, and therefore \eqref{eq:corLi-L4.4} holds. 

\vskip .1cm
{\it Case 3.}
Third, we consider the case that $R_1> \frac{c_0}N$ and $\frac{c_2}N< r < \e R_1$.
Recall that  $z_0\in \T^n$ is the spatial center of $Q(r)$.
Take $\bar z_0\in \frac1N\T_N^n$ such that $|z_0-\bar z_0|\le 
\frac{\sqrt{n}}N$ (or $\frac{\sqrt{n}}{2N}$ is enough).  Then, we claim
\begin{align*}
 Q(r) = Q_{z_0}(r) \subset Q_{\bar z_0}(\bar r) \subset
Q_{\bar z_0}(\bar R_1) \subset Q_{z_0}(R_1)=Q(R_1)
\end{align*}
where $\bar r:= r+ \frac{\sqrt{n}}N$, $\bar R_1:= R_1-\frac{\sqrt{n}}N$
 and  the subscript $\bar z_0$ in $Q_{\bar z_0}(\cdot)$ and $Q_{N,\bar z_0}(\cdot)$ below
means that the spatial center is $\bar z_0$.
Indeed, note that $\bar r\le (1+\frac{\sqrt{n}}{c_2}) r$.
Note also that $\bar r+\frac{\sqrt{n}}N< \frac12 c_1 \bar R_1$ automatically holds as
$R_1>\frac{c_0}N$, under the choice $c_0>\frac{(4+c_1)\sqrt{n}}{c_1-2\e}$
(here, $\e$ should be taken as $\e\in (0,\tfrac12c_1)$), and $r<\e R_1$.
In particular, $\bar r<\bar R_1$ so that the middle inclusion indicated above holds.

We now investigate the oscillation of $\tilde \U^N$ on the wider space
$Q_{\bar z_0}(\bar r)$.  
Since $|\U^N|\le M_{R_1}$ on $Q_N(R_1)$, this holds also on $Q_{N,\bar z_0}(\bar R_1)$.
Therefore, by Lemma \ref{lem:Li-L4.4} applied for the pair 
 $(r,R)=(\bar r+\frac{\sqrt{n}}N,\bar R_1)$
(recall $\bar r+\frac{\sqrt{n}}N < \frac12 c_1 \bar R_1$), we have $|\nabla_e^N \U^N(t,\tfrac{x}N)|
\le \frac{C(n,c_\pm)M_{R_1}}{\bar R_1}$ on $Q_{N,\bar z_0}(\bar r+\tfrac{\sqrt{n}}N)$.

This shows \eqref{continuous gradient bound}, namely $|\partial_{z_i} \tilde \U^N(t,z)|\le 
\frac{C(n,c_\pm)M_{R_1}}{\bar R_1}$ on $Q_{\bar z_0}(\bar r)$ by \eqref{eq:deri-u-G}, in the present case.
Note that  $\tfrac1{\bar R_1} < \tfrac{c_0}{c_0-\sqrt{n}} \tfrac1{R_1}$ holds
from $R_1 > \tfrac{c_0}N$.  Therefore, we also obtain
$$
I_1 \le \sqrt{n} C(n,c_\pm)M_{R_1}\tfrac{r}{\bar R_1}.
$$

For $I_2$, from Lemma \ref{lem:thm2.13} applied on $Q_{\bar z_0}(\bar r)$
with $\om = \sqrt{n} C(n,c_\pm)M_{R_1}\frac{\bar r}{\bar R_1}$ 
in \eqref{eq:thm2.13}, we have $|\tilde \U^N(t,\bar z_0)-\tilde \U^N(t_1,\bar z_0)|
=|\U^N(t,\bar z_0)-\U^N(t_1,\bar z_0)|\le 2\om = 2\sqrt{n} C(n,c_\pm)M_{R_1}\frac{\bar r}{\bar R_1}$
if $|t-t_1|\le \bar r^2$ since $\bar z_0\in \frac1N \T_N^n$.
Note that $X, X' \in Q_{\bar z_0}(\bar r)$ implies $|t-t'|<
\bar r^2$. Take the smaller one of $\{t,t'\}$ as $t_1$.
Thus, recalling $\bar r\le (1+\frac{\sqrt{n}}{c_2} ) r$, we obtain
\eqref{eq:corLi-L4.4} in case that $R_1>\frac{c_0}N$ and $\tfrac{c_2}N<r<\e R_1$.

\vskip .1cm
{\it Case 4.}
Fourth and finally, we consider the remaining case that $R_1\le \frac{c_0}N$ and
$r+\frac{\sqrt{n}+1}N> R_1$.  It is then sufficient to show 
the conclusion when $R_1<\frac{\sqrt{n}+1}{(1-\e)N}$ and $r<\e R_1$.  Recall, 
as we mentioned, $\e>0$ will be taken small enough.  The polylinearity
plays a role in the following proof in a short distance regime.
 
First, we consider the oscillation in $z$ by
deriving estimates on the slope of
$\tilde \U^N$ on $Q(r)$ in terms of $M_{R_1}$ and $R_1$.  Recall 
$D(r)=\{|z-z_0|<r\}$, which is the spatial part of $Q(r)$.
By a spatial shift, we may assume
$|z_0|_{L^\infty}\le \frac1{2N}$ for the center $z_0\in \frac{1}{N}\T^n_N$ of $D(r)$ and $D(R_1)$.
We divide the box $E_{\frac1N}:= \{|z|_{L^\infty}<\frac1N\}$ as $E_{\frac1N}
=\cup_{v\in\{\pm 1\}^n}E_{v,\frac1N}$ into $2^n$ unit orthants $E_{v,\frac1N}:=
\{z\in E_{\frac1N} ; \text{sgn}\, z_i=v_i, 1\le i \le n\}$, where $v=(v_i)_{i=1}^n$.

From \eqref{eq:deri-u-F}, the slope $\partial_{z_i} \tilde \U^N(z)$ 
of $\tilde \U^N$ is constant in $z_i$ on each unit orthant $E_{v,\frac1N}$ 
and depends only on $\check z_{(i)} := (z_1,\ldots,z_{i-1},z_{i+1}.\ldots,z_n)$.
We need to consider only $v$ such that $E_{v,\frac1N}\cap D(r)\not=\emptyset$,
that is those orthants that $D(r)$ touches.
On such an $E_{v,\frac1N}$, we work on $\tilde D_{v,\frac1N}(R_1) := E_{v,\frac1N} \cap D(R_1)$.

Define $R_i^*(\check z_{(i)}) $ by
\begin{align*}
R_i^*\equiv R_i^*(\check z_{(i)}):= \max\{|z_i^1-z_i^2|\, ; \, &z^1=(\check z_{(i)},z_i^1),
z^2=(\check z_{(i)},z_i^2) \in \tilde D_{v,\frac1N}(R_1) \\
& \text{ and line connecting }
z^1, z^2 \text{ intersects } D(r)\}.
\end{align*}
Then, we see $R_i^*\ge \{(\sqrt{1-\e^2}-\e)R_1\}\wedge \frac1N$.

Indeed, \lq\lq the line intersects $D(r)$'' implies \lq\lq there exists $b$ such that
$(\check z_{(i)},b)\in D(r)$'', that is,
$|z_0 - (\check z_{(i)},b)|^2 = A+|z_{0,i}-b|^2 <r^2$,
where $z_0=(\check z_{0,(i)},z_{0,i})$ and $A:= |\check z_{0,(i)}-\check z_{(i)}|^2$.
By the definition of $R_i^*$, we see $R_i^* \ge |a-b| \wedge \frac1N$,
where $a$ is taken as $|z_0 - (\check z_{(i)},a)|^2 = R_1^2$.  
Then, since
$|z_{0,i}-b|\le \sqrt{r^2-A}$, $0\le A\le r^2$ and also $r<\e R_1$, we see
\begin{align*}
|a-b| & \ge |z_{0,i}-a| - |z_{0,i}-b| \\
& \ge \sqrt{R_1^2-A} - \sqrt{r^2-A} \\
& \ge \sqrt{R_1^2-r^2} -r \\
& \ge (\sqrt{1-\e^2}-\e)R_1.
\end{align*}
Thus, we obtain
$R_i^*\ge \{(\sqrt{1-\e^2}-\e)R_1\}\wedge \frac1N$.  

Now, choosing $\e>0$
small so that $\sqrt{1-\e^2}-\e$ is close to $1$, we see
 $R_i^* \ge \frac1{c_0}R_1$ (note $\frac1N\ge \frac{R_1}{c_0}$ and $c_0$ is taken large) 
 and, since the slope $\partial_{z_i} \tilde \U^N(t,z)$ is
constant in $z_i$, this leads to the bound on $Q(r) \cap E_{v,\tfrac1N}$,
\begin{align}  \label{eq:4thcase}
|\partial_{z_i} \tilde \U^N(t,z)| \le \tfrac{2M_{R_1}}{R_i^*}\le c_0\tfrac{2M_{R_1}}{R_1},
\end{align}
establishing \eqref{continuous gradient bound} in the present case.
Therefore, we also obtain
$$
I_1\le c_0 \sqrt{n} \tfrac{2M_{R_1}}{R_1} r.
$$

For $I_2$, we can use the same idea as in Case 1.  Indeed, note that,
from \eqref{eq:deri-u-F}, $\partial_{z_i} \tilde \U^N(t,z)$ is a convex
combination of $\nabla_{e_i}^N \U^N(t,\tfrac{x}N)$ on $Q_N(r+\frac{\sqrt{n}}N)$
so that, by \eqref{eq:4thcase}, we have $|\nabla_{e}^N \U^N(t,\tfrac{x}N)|
\le c_0\tfrac{2M_{R_1}}{R_1}$ on $Q_N(r+\frac{\sqrt{n}}N)$.
Therefore, one can apply a similar argument to \eqref{eq:1stcase},
and noting $Nr< \tfrac{(\sqrt{n}+1)\e}{1-\e}$, we obtain
$$
I_2\le 2a_* c_0\tfrac{2M_{R_1}}{R_1} N r^2 \le 4a_* c_0\tfrac{M_{R_1}}{R_1} \tfrac{(\sqrt{n}+1)\e}{1-\e} r.
$$
This completes the proof of the corollary.
\end{proof}

We now obtain the following integral estimates
for the polylinear interpolation $\tilde \U^N$ of the solution $\U^N$ of
the discrete heat equation \eqref{eq:heat-D} on $Q_N(R)$.
These estimates will be applied in the proof of Lemma \ref{lem:5.3} below.
As in Corollary \ref{cor:Li-L4.4}, we consider three different ranges $0<\rho<r<R$, especially,
by distinguishing $r$ and $R$.  
Otherwise, we would get estimate \eqref{eq:nabla^Nv}
in the proof of Lemma \ref{lem:5.3} only for $r\ge \tfrac{\sqrt{n}}N$, which is insufficient.

\begin{prop}  \label{lem:Li-L4.5}
(cf.\ Lemma 4.5 of \cite{Li96}) Let $\U^N=\U^N(t,\tfrac{x}N) (\equiv \U^{N,R}(t,\tfrac{x}N))$
be given on $Q_N(R) = Q_N(X_0,R)$ as in Corollary \ref{cor:Li-L4.4} 
for any fixed $X_0=(t_0,z_0)\in \Om$ and $R>\tfrac{\sqrt{n}}N$.
Let $0<\rho< r\le R-\tfrac{\sqrt{n}}N$.
In particular, $\tilde v^N$ is defined on $Q(r)=Q(X_0,r)$.
Then, there is a constant $C=C(n,c_\pm)$ such that
\begin{align}  \label{eq:Li-L4.5-1}
\int_{Q(\rho)} (\tilde \U^N)^2dX & 
\le C\big(\tfrac{\rho}r\big)^{n+2} \int_{Q(r)} (\tilde \U^N)^2dX, \\
\label{eq:Li-L4.5-2}
\int_{Q(\rho)} |\tilde \U^N-\{\tilde \U^N\}_{\rho}|^2dX & \le C\big(\tfrac{\rho}r\big)^{n+4} 
\int_{Q(r)} |\tilde \U^N-\{\tilde \U^N\}_r|^2dX,
\end{align}
where $Q(\rho)=Q(X_0,\rho)$
 and $\{\U\}_{\rho}= \frac1{|Q(\rho)|}
\int_{Q(\rho)}\U\, dX$.
Moreover,  assuming $R> \tfrac{\sqrt{n}+1}N$, for 
$0<\rho< r\le R-\tfrac{\sqrt{n}+1}N$,  we have
\begin{align}
\label{tilde_estimate}
\int_{Q(\rho)} (\nabla_e^N\tilde \U^N)^2dX & 
\le C\big(\tfrac{\rho}r\big)^{n+2} \int_{Q(r)} (\nabla_e^N\tilde \U^N)^2dX,\\ 
\int_{Q(\rho)} | \nabla_e^N \tilde \U^N-\{\nabla_e^N \tilde \U^N\}_{\rho}|^2dX 
& \le C\big(\tfrac{\rho}r\big)^{n+4} 
\int_{Q(r)} |\nabla_e^N \tilde \U^N-\{ \nabla_e^N\tilde \U^N\}_r|^2dX.
\label{eq:Li-L4.5-3}
\end{align}
\end{prop}

The power, for example, $n+4$ in \eqref{eq:Li-L4.5-3} could be understood as follows.
If $\nabla_e^N \tilde \U^N$ would behave as a linear or Lipschitz function uniform in $N$,
its oscillation in
$Q(\rho)$ is like $C\rho$.  Therefore, recalling $|Q(\rho)|=C\rho^{n+2}$,
the integral in the left hand side of \eqref{eq:Li-L4.5-3} would behave as $C\rho^{n+4}$.
Similarly, the integral in the right hand side would behave as $Cr^{n+4}$,
intuitively explaining the bound \eqref{eq:Li-L4.5-3}.

\begin{proof}
The proof follows that of Lemma 4.5 of \cite{Li96}, properly modified in our discrete setting.
Indeed, to establish \eqref{eq:Li-L4.5-1}, we may assume $\rho<r/4$ as otherwise the statement
follows straightforwardly.  Recall that $\tilde v^N$ is well-defined on $Q(R-\tfrac{\sqrt{n}}N)=
Q(X_0,R-\tfrac{\sqrt{n}}N)$ and, in particular, on $Q(r) = Q(X_0,r)$.

Let now $U = \sup_{X\in Q(r/2)} | \tilde \U^N(X)|d_r(X)^{1+n/2}$ 
(the factor $d_r(X)^{1+n/2}$ leads to $r^{-n-2}$ in \eqref{eq:3.46}, \eqref{eq:3.28-A})
and find $X_1=(t_1,z_1)
\in Q(r/2)$ so that $d_r(X_1)^{1+n/2} |\tilde \U^N(X_1)| >U/2$, where 
$Q(r/2)=Q(X_0,r/2)$ and $d_r(X_1)
:= \inf\{|Y-X_1|; Y\in \mathcal{P}Q(\tfrac{r}2)\}$ is the (parabolic) distance 
from $X_1$ to the parabolic boundary of $Q(\tfrac{r}2)$: 
$\mathcal{P}Q(\tfrac{r}2) = \{t_0-(\tfrac{r}2)^2\}\times D_{z_0}(\tfrac{r}2)
\cup (t_0-(\tfrac{r}2)^2,t_0)\times \partial D_{z_0}(\tfrac{r}2)$.
Recall $|X|= \max\{\sqrt{|t|}, |z|\}$ for $X=(t,z)\in \R\times \T^n$
(including $t<0$).  Note that $|\cdot|$ satisfies the triangular inequality
(viewing $\T^n$ as $\R^n$ by periodic extension).

By the oscillation bound \eqref{eq:corLi-L4.4}
in Corollary \ref{cor:Li-L4.4} applied on the region $Q_N(X_1,R-\tfrac{r}2)$
(i.e., take $X_1$ for $X_0$ and $R-\tfrac{r}2$ for $R$, and note
 $Q_N(X_1,R-\tfrac{r}2) \subset Q_N(X_0,R)$), with $r_1=\gamma d_r(X_1)$ 
(take $r_1$ for $r$ in Corollary \ref{cor:Li-L4.4}) and $R_1=d_r(X_1)/2$, we obtain
\begin{align}  \label{eq:3.45}
\underset{Q(X_1, \gamma d_r(X_1))}{\rm osc}\, \tilde \U^N  
\leq 2C(n,c_\pm)\gamma \sup_{Q(X_1, d_r(X_1)/2)}|\tilde \U^N|,
\end{align}
where we take $\gamma \in (0, 1/4)$.  
Note that this bound is applicable, since the condition of Corollary \ref{cor:Li-L4.4},
$R_1 + \tfrac{\sqrt{n}}N\le R-\tfrac{r}2$, holds.
Indeed, from $d_r(X_1)\le \tfrac{r}2$, we see $R_1\le \tfrac{r}4$ so that 
the condition follows from our assumption: $r+\tfrac{\sqrt{n}}N\le R$.

Since $d_r(X)\geq d_r(X_1)/2$ for $X\in Q(X_1, d_r(X_1)/2)$
and  $Q(X_1,d_r(X_1)/2) \subset Q(X_0,\tfrac{r}2)$ holds, the right-hand side of
\eqref{eq:3.45} is bounded by
\begin{align*}
&2C(n,c_\pm)\gamma 2^{1+n/2} d_r(X_1)^{-1-n/2} \sup_{Q(X_1, d_r(X_1)/2)} |\tilde \U^N(X)| d_r(X)^{1+n/2}\\
&\ \ \ \ \leq C(n,c_\pm)2^{2+n/2}\gamma d_r(X_1)^{-1-n/2} U.
\end{align*}
Note, as well consequently, for $X\in Q(X_1, \ga d_r(X_1))$, that
\begin{align*}
|\tilde \U^N(X)| &\geq |\tilde \U^N(X_1)| - \underset{Q(X_1, \gamma d_r(X_1))}{\rm osc}\, \tilde \U^N\\
&\geq |\tilde \U^N(X_1)| - C(n,c_\pm)2^{2+n/2}\gamma d_r(X_1)^{-1-n/2} U\\
&\geq |\tilde \U^N(X_1)|\big(1-2^{3+n/2}C(n,c_\pm)\gamma \big).
\end{align*}
The last line follows by the choice of $X_1$.
We now choose $\gamma>0$ small enough so that $C_1:= 1-2^{3+n/2}C(n,c_\pm)\gamma >0$.  Then,
taking infimum, and by the choice of $X_1$ again,
\begin{align*}
\inf_{Q(X_1,\gamma d_r(X_1))} |\tilde \U^N|^2 \geq C_1^2 |\tilde \U^N(X_1)|^2
\geq C_1^2U^2 (d_r(X_1)^{-2-n}/4).
\end{align*}
Thus, we have
\begin{align}  \label{eq:3.46}
U^2 &\leq 4C_1^{-2} d_r(X_1)^{2+n} \inf_{Q(X_1, \gamma d_r(X_1))} |\tilde \U^N|^2 \\
&\leq C(n,C_1)\gamma^{-n-2}\int_{Q(X_1,\gamma d_r(X_1))} |\tilde \U^N|^2 dX \ 
\leq C(n, c_\pm)\int_{Q(r)}|\tilde \U^N|^2dX,   \notag
\end{align}
since $d_r(X_1)^{2+n}/|Q(X_1,\gamma d_r(X_1))| = C(n)\gamma^{-n-2}$
and then $Q(X_1,\ga d_r(X_1)) \subset Q(r)=Q(X_0,r)$.

Hence, if $\rho\leq r/4$, 
\begin{align}  \label{eq:3.28-A}
\int_{Q(\rho)} |\tilde \U^N|^2dX &\leq C(n)\rho^{n+2} \sup_{Q(\rho)} |\tilde \U^N|^2 \\
&\leq C(n)\rho^{n+2} \sup_{Q(r/4)} |\tilde \U^N(X)|^2 d_r(X)^{2+n} d_r(X)^{-2-n}   \notag\\
& \leq C(n) (\tfrac{\rho}r)^{n+2} U^2,  \notag
\end{align}
since $d_r(X)\geq r/4$ for $X\in Q(r/4)=Q(X_0,r/4)$
as $Q(r/4)\subset Q(r/2)$ belongs to the interior.

Using the bound \eqref{eq:3.46} on $U$, we obtain \eqref{eq:Li-L4.5-1}.
In particular, we have also shown
\begin{align}
\sup_{Q(\rho)}|\tilde \U^N|^2
\leq  \frac{C(n,c_\pm)}{r^{n+2}} \int_{Q(r)} |\tilde \U^N|^2dX.
\label{intermed bound}
\end{align}

To establish \eqref{eq:Li-L4.5-2}, for a fixed $r\le R-\frac{\sqrt{n}}{N}$, we may assume
$\rho<r/4$ as above.
Indeed, 
\begin{align*}
\int_{Q(\rho)} | \tilde \U^N - \{\tilde \U^N\}_\rho|^2dX
\le 2 \int_{Q(\rho)} | \tilde \U^N - \{\tilde \U^N\}_r|^2dX
+ 2 |Q(\rho)| \cdot | \{\tilde \U^N\}_\rho- \{\tilde \U^N\}_r|^2
\end{align*}
and the second term is rewritten and then bounded as
\begin{align*}
2 |Q(\rho)| \Big| \tfrac1{|Q(\rho)| }\int_{Q(\rho)} \big(\tilde \U^N - \{\tilde \U^N\}_r \big) dX\Big|^2
\le 2 \int_{Q(\rho)} | \tilde \U^N - \{\tilde \U^N\}_r|^2dX,
\end{align*}
by applying Schwarz's inequality.

Write now
\begin{align}
\int_{Q(\rho)} | \tilde \U^N - \{\tilde \U^N\}_\rho|^2dX
\ \leq C(n)\rho^{n+2} \sup_{Q(\rho)} |\tilde \U^N - \{\tilde \U^N\}_\rho|^2.
\label{lem3.4:eq1}
\end{align}
Note that
\begin{align*}
\tilde \U^N (X)- \{\tilde \U^N\}_\rho
&= \tfrac1{|Q(\rho)|} \int_{Q(\rho)} \{\tilde \U^N(X)-\tilde \U^N(Y)\} dY\\
&=  \tfrac1{|Q(\rho)|} \int_{Q(\rho)} \{\big(\tilde \U^N(X) - \{\tilde \U^N\}_r\big)
 -\big(\tilde \U^N(Y)- \{\tilde \U^N\}_r\big)\} dY
\end{align*}
and that $\tilde \U^N - \{\tilde \U^N\}_r$, as $\{\tilde \U^N\}_r$ is a constant, also satisfies the discrete heat equation \eqref{eq:heat-D} on $Q_N(R)$.  Then, applying the oscillation estimate Corollary \ref{cor:Li-L4.4} to $\tilde \U^N - \{\tilde \U^N\}_r$ (with there $r=\rho$, $R_1=r/4<r<R-\sqrt{n}{N}$), we obtain for $X\in Q(\rho)$ that
$$|\tilde \U^N (X)- \{\tilde \U^N\}_\rho| \leq C(n,c_\pm) \big(\frac{\rho}{r}\big)\sup_{Q(r/4)} |\tilde \U^N - \{\tilde \U^N\}_r|.$$

Moreover, applying
\eqref{intermed bound} (with $\rho=r/4$) to the discrete heat equation solution 
$\tilde \U^N - \{\tilde \U^N\}_r$, 
the right-hand side of \eqref{lem3.4:eq1}, is bounded by
\begin{align*}
\frac{C(n,c_\pm)\rho^{n+4} }{r^2}\sup_{Q(r/4)}|\tilde \U^N - \{\tilde \U^N\}_r|^2
&\leq C(n,c_\pm)\left(\frac{\rho}{r}\right)^{n+4} \int_{Q(r)} |\tilde \U^N - \{\tilde \U^N\}_r|^2dX,
\end{align*}
finishing the proof of \eqref{eq:Li-L4.5-2}.

The remaining two statements, \eqref{tilde_estimate} and \eqref{eq:Li-L4.5-3}, follow 
from  \eqref{eq:Li-L4.5-1} and  \eqref{eq:Li-L4.5-2}, respectively, by taking
$R-\tfrac1N$ instead of $R$, since $\nabla^N_e \U^N$ also solves  the discrete 
heat equation \eqref{eq:heat-D} on $Q_N(R-\tfrac1N)$ by noting $\tfrac{x+e}N\in Q_N(R)$ for
$\tfrac{x}N\in Q_N(R-\tfrac1N)$.
\end{proof}

\subsection{H\"older norms}
\label{sec:DirectAppl}

To set the stage for the `weighted' norms that we will use, we first define `unweighted' H\"older norms or seminorms analogous to that in \cite{Li96}, but tailored to our setting.  Recall 
$\Om= [0,T]\times \T^n$ and for $X=(t,z)\in \Om$ that $|X|=\max\{t^{\frac12}, |z|\}$.
Consider, for a function $F$ on $\T^n$ (and therefore on $\Om$), the continuous space 
gradient $\nabla^N F(z) = \{\nabla_e^N F(z)\}_{|e|=1,e>0}\in \R^n$
for $z\in \T^n$ where
\begin{align}  \label{eq:3.3-1}
\nabla_e^N F(z):= N\big( F(z+\tfrac{e}N)-F(z)\big).
\end{align}

For  a function $F=F(X)$ on $\Om$, set
\begin{align}  \label{eq:norm-1}
|F|_0 \equiv \|F\|_\infty := \sup_{X\in \Om} |F(X)|.
\end{align}
For $\a\in (0,1]$, define the parabolic H\"older seminorms by
\begin{align}\label{eq:norm-2}
\begin{aligned}
& [F]_{\a} := \sup_{X\not=Y\in \Omega} \frac{|F(X)-F(Y)|}{|X-Y|^\a},\\
& [F]_{1+\a} := \sup_{X\not=Y\in \Omega} \frac{|\nabla^N F(X)-\nabla^NF(Y)|}{|X-Y|^\a},
\end{aligned}
\end{align}
where $|\nabla^N F(X)-\nabla^NF(Y)|:= \max_{|e|=1, e>0}
|\nabla_e^N F(X)-\nabla_e^NF(Y)|$.  For $\a=0$, $[F]_0 := \underset{\Om}{{\rm osc}}(F)$
is the oscillation of $F$ on $\Om$, recall \eqref{eq:3.osc}.
For $\b\in (0,2]$, define
\begin{align} \label{eq:norm-3}
\lan F\ran_{\b} := \sup_{X\not=Y\in \Om, x=y} 
\frac{|F(X)-F(Y)|}{|X-Y|^{\frac{\b}2}},
\end{align}
where the spatial coordinates $x,y$ of $X$ and $Y$ are in common.
Adding all these, we define for $a\in (0,2]$, the unweighted H\"older norm
\begin{align}  \label{eq:norm-4}
|F|_a := [F]_a+\lan F\ran_a + |F|_0.
\end{align}

We now introduce several weighted norms to take care of possible diverging effects
near $t=0$.  
Define the parabolic boundary $\mathcal{P}\Omega$ of $\Om$ by
$$
\mathcal{P}\Omega= \{t=0\}\times \T^n.
$$
Since we work on the torus, for $X=(t,z)\in \Om$, the (parabolic) distance $d(X)$ to the boundary
is defined by
$$
d(X) := \inf\{|X-Y|; Y=(0,y)\in\mathcal{P}\Om\} \equiv \sqrt{t}.
$$
We will sometimes write $d$ instead of $d(X)$ to simplify notation when the context is clear.

For $a=0$ and $b\ge 0$, define
\begin{align}  \label{eq:norm-5}
|F|_{0}^{(b)} := \sup_{X\in \Om} d(X)^b |F(X)|.
\end{align}
For $0<a=k+\alpha \le 2$ where $k=0,1$ and $\alpha\in (0,1]$, and $b\geq 0$,
let
\begin{align}  \label{eq:norm-6}
[ F ]_{a}^{(b)} & := \sup_{X\not=Y\in \Om}
(d(X)\wedge d(Y))^{a+b}
\frac{|(\nabla^N)^k F(X)- (\nabla^N)^k F(Y)|}{|X-Y|^\a}.\, \\
\langle F\rangle_a^{(b)} & = \sup_{X\not= Y\in \Om, x=y} (d(X)\wedge d(Y))^{a+b}
\frac{|F(X)-F(Y)|}{|X-Y|^{\frac{a}{2}}}.  \label{eq:norm-7}
\end{align}
Here, $(\nabla^N)^0$ is the identity operator.  

The norm and seminorms without weights
defined in \eqref{eq:norm-1}, \eqref{eq:norm-2}, \eqref{eq:norm-3} can be expressed as
$|F|_0= |F|_0^{(0)}$, $[F]_a = [F]_a^{(-a)}$ 
and $\lan F\ran_a = \lan F\ran_a^{(-a)}$ in terms of those defined in
\eqref{eq:norm-5}, \eqref{eq:norm-6}, \eqref{eq:norm-7}, respectively,
taking $b=-a$ though we have restricted $b\ge 0$.

We will also have occasion to to use H\"older norms with respect to the 
continuous time/discrete space
$\Om_N=[0,T]\times \frac{1}{N}\T_N^n$.  Replacing $\Om$ by $\Om_N$ in the definitions
\eqref{eq:norm-5}, \eqref{eq:norm-6}, \eqref{eq:norm-7} of seminorms, 
we define, for a function $F$ on $\Om_N$ and $b\ge 0$,
\begin{align}  \label{eq:norm-8}
|F|_0^{(b),N} &:= \sup_{X\in \Om_N} d(X)^{b} |F(X)|,
\end{align}
and also, for $a = k +\a$ where $k=0,1$ and $\a\in (0, 1]$, and $b\geq 0$ that
\begin{align}\label{eq:norm-9}
[F]_a^{(b), N} &:= \sup_{X\not= Y\in \Om_N} (d(X)\wedge d(Y))^{a+b}
\frac{|(\nabla^N)^k F(X)-(\nabla^N)^k F(Y)|}{|X-Y|^\a}\\
\langle F\rangle_a^{(b),N} & := \sup_{X\not= Y\in \Om_N, x=y} (d(X)\wedge d(Y))^{a+b}
\frac{|F(X)-F(Y)|}{|X-Y|^{\frac{a}{2}}}.
\label{eq:norm-10}
\end{align}

In the following, the superscript $*$ for seminorms means $(0)$ so that for $a\in (0,2]$,
\begin{align}\label{eq:norm-11}
& [F]_a^* := [F]_a^{(0)}, \quad 
\lan F\ran_a^* := \lan F\ran_a^{(0)}, \quad 
[F]_a^{*,N} := [F]_a^{(0),N}, \quad 
\lan F\ran_a^{*,N} := \lan F\ran_a^{(0),N}. 
\end{align}
We will use $\|F\|_\infty$ rather than $|F|_0^*$ or $|F|_0^{*,N}$.

The seminorms defined by taking the supremum over $Q\subset \Om$ instead of $\Om$,
such as $\sup_{X\in Q}$ or $\sup_{X\not= Y\in Q}$,
are denoted by adding $Q$ in the subscript. [We actually do not use this notation in this article.]

We finally note that, via polylinear interpolation, seminorms for discrete and continuous functions
are mutually equivalent. 

\begin{lem}  \label{lem:Eq-norms}
For a function $F=F(t,\tfrac{x}N)$ on $\Om_N$, let $\tilde F$ be its polylinear
interpolation in spatial variable defined by \eqref{eq:poli-A} taking $F(t,\cdot)$ instead of $u^N$.
Then, we have the following for some $C=C(n)>0$,
\begin{align}\label{eq:2-norm-1}
& |F|_0^{(b), N} = |\tilde F|_0^{(b)},  \\
\label{eq:2-norm-2}
& [F]_a^{(b), N} \le [\tilde F]_a^{(b)} \le C [F]_a^{(b), N}, \\
& \lan F\ran_a^{(b),N} \le \lan \tilde F\ran_a^{(b)} \le C \lan F\ran_a^{(b),N}.
\label{eq:2-norm-3}
\end{align}
\end{lem}

\begin{proof}
The first inequalities in \eqref{eq:2-norm-2}, \eqref{eq:2-norm-3} and 
$\lq\lq\le$'' in \eqref{eq:2-norm-1} are obvious as
$\tilde F(X)=F(X)$ and $\nabla_e^N \tilde F(X) = \nabla_e^N F(X)$
for $X\in \Om_N$, the latter of which follows from
\begin{align}  \label{eq:tilde-grad}
 \widetilde{\nabla_e^N F}(X)
= \nabla_e^N \tilde F(X), \quad X\in \Om.
\end{align}
This equality is shown from $\vartheta^N(v,z+\tfrac{e}N) = \vartheta^N(v,z)$ in \eqref{eq:poli-A}.
The converse inequality $\lq\lq\ge$'' in \eqref{eq:2-norm-1} follows from
$|\tilde F(t,z)| \le \sup_{x\in \T_N^n}|F(t,\tfrac{x}N)|$ for every $z\in \T^n$ and
$t\in [0,T]$.  

The  second inequality in \eqref{eq:2-norm-2} in the case $k=0$ is shown as follows.
For $X=(t,z), Y=(s,y)\in \Om$, first take 
$Y_0=(s,y_0) \in \Om$ such that $z\equiv y_0$ modulo $\tfrac1N$ and $[Ny]=[Ny_0]$
(i.e.\ $y$ and $y_0$ belong to the same $\tfrac1N$-box) hold both componentwise,
in particular, $|Y_0-Y|\le \tfrac{\sqrt{n}}N$.  Then,
$$
|\tilde F(X)-\tilde F(Y)|\le |\tilde F(X)-\tilde F(Y_0)| + |\tilde F(Y_0)-\tilde F(Y)|,
$$
and the first term is written, noting that the $\vartheta$-parts are in common, as
$$
\Big|\sum_{v\in \{0,1\}^n} \vartheta^N(v,z) \Big\{F\big(t,\tfrac{[Nz]+v}N\big)
- F\big(s,\tfrac{[Ny_0]+v}N\big)\Big\}\Big|.
$$
This is bounded by
$[F]_a^{(b), N} (d(X)\wedge d(Y))^{-(a+b)} |X-Y_0|^{\a}$.  On the other hand,
the second term is rewritten as $|\{\tilde F(Y_0)-F(Z_0)\}-\{\tilde F(Y)-F(Z_0)\}|$
taking $Z_0=(s,z_0)$ with $z_0=\frac1N[Ny]\in \tfrac1N\T_N^n$ so that it is further rewritten as
$$
\Big| \sum_{v\in \{0,1\}^n} \Big\{ \vartheta^N(v,y_0) -  \vartheta^N(v,y) \Big\}
\Big\{F\big(s,z_0+\tfrac{v}N\big)- F\big(s,z_0\big)\Big\}\Big|.
$$
This is bounded by
$$
C(n)N |y_0-y| [F]_a^{(b), N} d(Y)^{-(a+b)} (\tfrac{\sqrt{n}}N)^\a\le C(n)\sqrt{n}
 |Y_0-Y|^\a [F]_a^{(b), N} d(Y)^{-(a+b)}$$
by noting \eqref{eq:poli-A} and $|y_0-y|\le \tfrac{\sqrt{n}}N$.

In the case $|X-Y|\ge \tfrac1{10N}$, since both $|X-Y_0|, |Y_0-Y| \le C(n)|X-Y|$, 
we obtain the second inequality in \eqref{eq:2-norm-2} in the case $k=0$.  

On the other hand, in the case $|X-Y|<\frac1{10N}$, we connect $X_0=(s,z)$ and $Y=(s,y)$ by a line.
Then, it touches the boundary of $\frac1N$-boxes at most $2^n$ times, so that,
the line is divided as $X_0\to Z_1\to \cdots \to Z_\ell\to Y$ with $\ell\le 2^n$.
As we saw above, in the same $\tfrac1N$-box, our estimate picks up factors
$|Z_{i+1}-Z_i|^\a$, all of them are bounded by $|X_0-Y|^\a$.  We also have the
estimate on $|\tilde{F}(X)-\tilde F(X_0)|$ as above; see the first term with $Y_0=X_0$.
From these, we can derive the second inequality in \eqref{eq:2-norm-2} for $k=0$ also
in the case $|X-Y|<\frac1{10N}$.

The  second inequalities in \eqref{eq:2-norm-2} in case $k=1$ (i.e.\
$\nabla_eF$ in place of $F$) and in \eqref{eq:2-norm-3} are similar.
\end{proof}

\subsection{Two basic lemmas}   \label{sec:three}

We now consider and adapt some of the results from Chapter 4 of \cite{Li96}
with respect to our discrete setting,
especially in terms of polylinear interpolations.

The first is Lemma 4.3 of \cite{Li96}
that an integral estimate implies H\"older continuity.  Recall the definition 
of the parabolic ball $Q(R) = Q(X_0,R)$ in \eqref{Q_defn}.  
Due to the discrete nature of the problem in our setting, especially, 
the non-locality of the polylinear interpolation, the assumption
\eqref{eq:Li-lem-4.12} below can be given only with $r_N$ instead of $r$, which is
weaker than that of Lemma 4.3 of \cite{Li96}.  As a result, we obtain 
the H\"older property \eqref{eq:Li-L4.3-G} only for $|Y-Y_1|\ge \tfrac1{MN}$,
excluding the short distance regime, with an additional term
$\de U_{F,1+\a}  \big(d(Y)\wedge d(Y_1)\big)^{-(1+\a)}$, which can be  
made arbitrary small by taking $C=C_{\de,M}$ large.  This will be applied 
in the proofs of Propositions \ref{spatial-variation-prop} and \ref{prop:5.4}
below.  The short distance regime $|Y-Y_1|\le \tfrac1{MN}$ will be covered 
separately by Lemmas \ref{lem:4.13} and \ref{lem:4.13-5} later.

\begin{lem}  \label{lem:Li-L4.3-G}
(cf.\ Lemma 4.3 of \cite{Li96})
Let $F\in L^1(Q(X_0,2R))$ 
and suppose there are constants $\a \in (0,1]$ and $H>0$ along with 
a function $G$ defined on $Q(X_0,2R)\times (0,R)$ such that
\begin{equation}  \label{eq:Li-lem-4.12}
\int_{Q(Y,r)}|F(X) -G(Y,r)|dX \le H r_N^{n+2+\a},
\end{equation}
for any $Y\in Q(X_0,R)$ and any $r\in (0,R)$, where $r_N=r+\tfrac{c}N$ with
some $c>0$.  
Then, for every $\de>0$ and $M>0$, there exists $C=C_{\de,M}(n, \alpha)>0$ such that
\begin{align}  \label{eq:Li-L4.3-G}
|F(Y)-F(Y_1)| \le \Big(CH +\de U_{F,1+\a}  \big(d(Y)\wedge d(Y_1)\big)^{-(1+\a)}\Big)
|Y-Y_1|^\a,
\end{align}
holds if  $Y, Y_1\in Q(X_0,R)$ satisfy $|Y-Y_1|  \ge \tfrac1{MN}$, where
$$
U_{F,1+\a} := \sup_{X\not= Y} \big(d(X)\wedge d(Y)\big)^{1+\a}
\tfrac{|F(X)-F(Y)|}{|X-Y|^\a}.
$$
Note that, for $F=\nabla^N\tilde u^N$ (vector-valued), 
$U_{F,1+\a}= U_{1+\a}$ defined in \eqref{seminorm_defn} below.
[We will choose $\de>0$ small enough such that $\de U_{F,1+\a}$ will be 
eventually absorbed by $U_{F,1+\a}$ itself.]
\end{lem}

\begin{proof}
As in the proof of Lemma 4.3 of \cite{Li96}, we take two nonnegative convolution kernels
$\fa=\fa(z) \in C^1(\R^n)$ supported on $\{|z|\le 1\}$ and $\eta=\eta(\si) \in
C^1(\R)$ supported on $[0,1]$ such that $\int_{\R^n}\fa dz = \int_\R \eta d\si=1$.
Define for $(Y,\t)\in Q(X_0,R)\times (0,R)$
$$
\bar F(Y,\t) =\int_{\R^n}\int_\R F(y+\t z, s-\t^2 \si)\fa(z)\eta(\si) dzd\si,
\quad Y=(s,y).
$$
Then, as in \cite{Li96}, setting $K=\|D\fa\|_{L^\infty}, L=\|\eta'\|_{L^\infty}$,
by the condition \eqref{eq:Li-lem-4.12},
\begin{align*}
|\bar F_y(Y,\t)|
&\le \tfrac{K}{\t^{n+3}}\int_{Q(Y,\t)} |F(X)-g(Y,\t)| \eta(\tfrac{s-t}{\t^2}) dX \\
&\le HKL \frac{\t_N^{n+2+\a}}{\t^{n+3}}.
\end{align*}
Similarly,
\begin{align*}
|\bar F_s(Y,\t)| & \le HKL \frac{\t_N^{n+2+\a}}{\t^{n+4}}, \\
|\bar F_\t(Y,\t)| & \le (n+K+2+2L)H \frac{\t_N^{n+2+\a}}{\t^{n+3}}.
\end{align*}

Now, fix $Y$ and $Y_1$ in $Q(X_0,R)$ and set $\t=|Y-Y_1|$.  If $\e\in (0,\t)$, then
similarly to \cite{Li96} p.51,
\begin{align}  \label{eq:4.42}
|\bar F(Y,\e) -\bar F(Y_1,\e)| 
\le & 2(n+K+2+2L)H \int_\e^\t  \tfrac{\rho_N^{n+2+\a}}{\rho^{n+3}} d\rho\\
& +HKL  \tfrac{\t_N^{n+2+\a}}{\t^{n+3}} |y-y_1|+ HKL  \tfrac{\t_N^{n+2+\a}}{\t^{n+4}} |s-s_1|.
\notag
\end{align}
Recall that $\t\ge \tfrac1{MN}$ by our assumption so that we have
$$
\tfrac{\t_N}\t = 1+\tfrac{c}{\t N}\le 1+cM.
$$
Thus, the sum of the second and third terms is bounded as 
\begin{align*}
 HKL \Big(\tfrac{\t_N^{n+2+\a}}{\t^{n+3}} |y-y_1|+ \tfrac{\t_N^{n+2+\a}}{\t^{n+4}} |s-s_1| \Big)
\le C_M H \t^\a,
\end{align*}
since $|y-y_1| \le \t, |s-s_1|\le \t^2$.

Let us estimate the integral $\int_\e^\t \tfrac{\rho_N^{n+2+\a}}{\rho^{n+3}} d\rho$
in the first term in the right hand side of \eqref{eq:4.42}.
 (This diverges as $\e\downarrow 0$, while the integral 
$\int_\e^\t \tfrac{\rho^\a}\rho d\rho$ in \cite{Li96} in continuous case converges.)
\begin{align*}
\int_\e^\t \tfrac{\rho_N^{n+2+\a}}{\rho^{n+3}} d\rho
 = \int_\e^\t \tfrac{(\frac{\rho_N}\rho)^{n+2+\a}}{\rho^{1-\a}} d\rho
\le  (1+\tfrac{c}{N\e})^{n+2+\a} \tfrac1\a \t^\a,
\end{align*}
since $\frac{\rho_N}\rho= 1+\tfrac{c}{N\rho} \le 1+\tfrac{c}{N\e}$.

Moreover, the convergence rate of $\bar F(Y,\e)$ to $F(Y)$  is estimated as
\begin{align*}
|F(Y)-\bar F(Y,\e)| 
& = \Big| \int_{\R^n}\int_\R \big\{ F(y,s) - F(y+\e z, s-\e^2 \si)\big\} \fa(z)\eta(\si) dzd\si \Big|\\
& \le \int_{\R^n}\int_\R \big( |\e z|^\a+ |\e^2\si|^{\frac{\a}2}\big) \fa(z)\eta(\si) dzd\si 
\times U_{F,1+\a} d(Y)^{-(1+\a)} \\
& \le  2\e^\a U_{F,1+\a} d(Y)^{-(1+\a)}.
\end{align*}

Summarizing these, if $\t\ge \tfrac1{MN}$, we have
\begin{align*}
&|F(Y) -F(Y_1)| 
 \le |F(Y) -\bar F(Y,\e)| 
+ |\bar F(Y,\e) -\bar F(Y_1,\e)| 
+ |\bar F(Y_1,\e) - F(Y_1)| \\
& \ \ \ \ \ \le 4 \e^\a U_{F,1+\a}  \big(d(Y)\wedge d(Y_1)\big)^{-(1+\a)}
+ C_M H \t^\a+ C(n, \alpha)H(1+\tfrac{c}{N\e})^{n+2+\a} \t^\a,
\end{align*}
for every $\e\in (0,\t)$.  We now choose $\e=\tilde\de\cdot \t$ with 
small enough $\tilde\de>0$
such that $4\tilde\de^\a\le \de$.  Then,
since $\tfrac{c}{N\e}=\tfrac{c}{\tilde \de N \t}\le \tfrac{Mc}{\tilde \de}$, 
we obtain \eqref{eq:Li-L4.3-G} for $\t\ge \tfrac1{MN}$.
\end{proof}

\begin{rem}  \label{rem:3.CH}
\rm
To derive the uniform H\"older continuity of $u^N$, we applied the results of \cite{GOS}  
and \cite{SZ} in Section \ref{sec:2.1}.  We will then use Lemma \ref{lem:Li-L4.3-G}
to show the H\"older continuity of $\nabla^Nu^N$ in Section \ref{sec:Schauder-1}.
However, it might be possible to apply Lemma \ref{lem:Li-L4.3-G}, which partially
shows the equivalence between Campanato space and H\"older space, to 
derive the H\"older continuity of $u^N$ as in \cite{HN} in a continuous setting.
\end{rem}

We state the iteration Lemma 4.6 of \cite{Li96}.
This will be modified later in Lemma \ref{lem:Li-L4.6-N} with $\si(r)$
changed to $\si(r_N)$ as in Lemma \ref{lem:Li-L4.3-G} to adjust to our discrete setting
and used to prove Propositions \ref{spatial-variation-prop} and \ref{prop:5.4} below.

\begin{lem}  \label{lem:Li-L4.6}
(cf.\ Lemma 4.6 of \cite{Li96})
Let $\om$ and $\si$ be increasing functions on an interval $(0,R_0]$
and suppose there are positive constants $\bar{\a}, \de$, and $\t$ with 
$\t<1$ and $\de < \bar{\a}$ such that
\begin{equation}  \label{eq:Li-sigma}
r^{-\de}\si(r)\le s^{-\de}\si(s) \qquad \text{if } \; 0<s\le r \le R_0
\end{equation}
and
\begin{equation}  \label{eq:Li-om}
\om(\t r) \le \t^{\bar{\a}}\om(r) +\si(r) \qquad \text{if }\; 0<r\le R_0.
\end{equation}
Then, there is a constant $C = C(\bar\a,\de,\t)$ such that
$$
\om(r) \le C\Big[ \big(\tfrac{r}{R_0}\big)^{\bar{\a}} \om(R_0)+\si(r)\Big] \leq C\Big[ \big(\tfrac{r}{R_0}\big)^\delta \om(R_0)+\si(r)\Big],
$$
for $0<r\le R_0$.
\end{lem}

We remark, as noted in \cite{Li96}, if (4.15)' in \cite{Li96} holds, that is $\omega(\t r) \leq C\t^\beta\omega(r) + \sigma(r)$ with $\beta>\delta$ for all small $\t$, then \eqref{eq:Li-om} is satisfied with $\bar{\a}\in (\delta,\beta)$ and $\t$ chosen so that $C\t^\beta \leq \t^{\bar{\a}}$.

\begin{proof}
Since the argument is short, for the convenience of the reader, we give it here.  The lemma is clear when $r\geq \tau R_0$.  Hence, suppose $0<r<\tau R_0$.  Let $k$ be the smallest integer such that $r<\tau^k R_0$.  Consider the base case $k=1$:  Write
\begin{align*}
\omega(r)& \leq \omega(\tau R_0) \leq \tau^{\bar \alpha} \omega(R_0) + \sigma(R_0)\\
&\leq \tau^{\bar\alpha} \omega(R_0) + \tau^{-\delta} \sigma(\tau R_0).
\end{align*}
Suppose now for $j\leq k-1$ that 
$$\omega(\tau^j R_0) \leq \tau^{j\bar\alpha} \omega(R_0) + \tau^{-\delta} \sigma(\tau^j R_0)\sum_{\ell=1}^{j} \tau^{(\bar \alpha - \delta)(\ell-1)}.$$
Then, when $k>1$, we have
\begin{align*}
\omega(\tau^k R_0) &\leq \tau^{\bar \alpha} \omega(\tau^{k-1}R_0) + \sigma(\tau^{k-1}R_0)\\
&\leq \tau^{\bar \alpha}\big[ \tau^{(k-1)\bar\alpha}\omega(R_0) + \tau^{-\delta}\sigma(\tau^{k-1}R_0)\sum_{\ell=1}^{k-1}\tau^{(\bar\alpha - \delta)(\ell-1)}\big] + \tau^{-\delta}\sigma(\tau^k R_0)\\
&\leq \tau^{k\bar\alpha}\omega(R_0) + \tau^{-\delta}\sigma(\tau^k R_0)\sum_{\ell=2}^{k} \tau^{(\bar\alpha - \delta)(\ell-1)} + \tau^{-\delta}\sigma(\tau^k R_0)\\
&= \tau^{k\bar\alpha}\omega(R_0) + \tau^{-\delta}\sigma(\tau^k R_0)\sum_{\ell=1}^k \tau^{(\bar\alpha - \delta)(\ell-1)}\\
&\leq \tau^{k\bar\alpha}\omega(R_0) + \sigma(\tau^k R_0) \frac{1}{\tau^\delta - \tau^{\bar \alpha}}.
\end{align*}

Since $\tau^k R_0< r/k<R_0$, we have $\sigma(\tau^k R_0)\leq \sigma(r/\tau)\leq \tau^{-\delta}\sigma(r)$.
Then, as $\tau^k\leq r/R_0$,
we obtain
\begin{align*}
\omega(r)& \leq \omega(\tau^k R_0)\leq C\Big[\tau^{k\bar\alpha} \omega(R_0) + \sigma(r)\Big]\leq C\Big[\big(\frac{r}{R_0})^{\bar\alpha}\omega(R_0) + \sigma(r)\Big]
\end{align*}
where
$C=\max\{\tau^{\bar\alpha}, \tau^{-\delta}(\tau^\delta - \tau^{\bar \alpha})^{-1}\}$.
\end{proof}

\subsection{Summation by parts formula}  \label{sec:SbP}

Finally, we prepare a summation by parts formula, that is, Green's identity in discrete setting.
For $\La \subset \frac1N\T_N^n$, we denote the outer boundary and closure of 
$\La$ by 
\begin{equation}  \label{eq:3.ob+c}
\partial_N^+\La= \{\tfrac{x}N\in 
\tfrac1N\T_N^n\setminus\La; \text{dist}(\tfrac{x}N,\La)=\tfrac1N\}, \quad
\overline{\La} = \La\cup\partial_N^+\La,
\end{equation}
respectively.

\begin{lem}\label{lem:DG-Q-2}
Let $\La\subset \frac1N\T_N^n$ and let $e\in \Z^n: |e|=1$ be given.
For functions $F, G$ on $\overline{\La}$ satisfying $F=0$ at $\partial_N^+\La$,
we have
\begin{equation}  \label{eq:SbP-1}
\sum_{\tfrac{x}N\in\La} F(\tfrac{x}N)
\nabla_e^{N,*} G (\tfrac{x}N)
=  \sum_{\tfrac{x}N \text{ or } \tfrac{x+e}N \in \La} 
\nabla_e^N F(\tfrac{x}N) \cdot G(\tfrac{x}N),
\end{equation}
where the sum in the right hand side is taken over $x$ satisfying
$\tfrac{x}N \in \La$  or  $\tfrac{x+e}N \in \La$.
If we replace $\nabla_e^{N,*}$ by $\nabla_e^N$ in the left hand side,
$\nabla_e^N$ and $\frac{x+e}N$ in the right hand side should be
$\nabla_e^{N,*}$ and $\frac{x-e}N$, respectively (since $\nabla_e^{N,*}
=\nabla_{-e}^{N}$). 

 In particular, for $a=\{\bar a_x\}$, we have
\begin{equation}  \label{eq:SbP-2}
\sum_{\tfrac{x}N\in\La} F(\tfrac{x}N)
\nabla_e^{N,*} \big(\bar a_{x} G\big) (\tfrac{x}N)
=  \sum_{\tfrac{x}N \text{ or } \tfrac{x+e}N \in \La} \bar a_{x} 
\nabla_e^N F(\tfrac{x}N) \cdot G(\tfrac{x}N).
\end{equation}

Moreover, for $a=\{a_{x,e}\}$
\begin{align}  \label{eq:SbP-3}
\sum_{\tfrac{x}N\in\La}& F(\tfrac{x}N)
\sum_{|e|=1, e>0}
\nabla_e^{N,*} \big(a_{x,e}\nabla_e^N G\big) (\tfrac{x}N)\\
& =  \sum_{\tfrac{x}N \text{ or } \tfrac{x+e}N \in \La;
|e|=1, e>0} a_{x,e} \nabla_e^N F(\tfrac{x}N) \nabla_e^N G(\tfrac{x}N).
\notag
\end{align}
\end{lem}

\begin{proof}
The left hand side of \eqref{eq:SbP-1} is rewritten as
\begin{align*}
& \sum_{\tfrac{x}N\in\La}  F(\tfrac{x}N)
 N \Big\{G(\tfrac{x-e}N) - G (\tfrac{x}N)\Big\} \\
& \quad = N \sum_{\tfrac{x+e}N\in\La} F(\tfrac{x+e}N)
 G (\tfrac{x}N)
- N \sum_{\tfrac{x}N\in\La} F(\tfrac{x}N) 
 G (\tfrac{x}N)  \\
&\quad = \sum_{\tfrac{x}N,\tfrac{x+e}N\in\La} 
\nabla_e^N F(\tfrac{x}N) G (\tfrac{x}N) 
+ \sum_{\tfrac{x}N\not\in\La,\tfrac{x+e}N\in\La} 
N F(\tfrac{x+e}N) G (\tfrac{x}N)   
- \sum_{\tfrac{x}N\in\La,\tfrac{x+e}N\not\in\La} N F(\tfrac{x}N) 
G (\tfrac{x}N) \\
&\quad = \sum_{\{\tfrac{x}N,\tfrac{x+e}N\in\La\}\cup\{\tfrac{x}N\in\La,\tfrac{x+e}N
\not\in\La\}\cup\{\tfrac{x}N\not\in\La,\tfrac{x+e}N\in\La\}}
\nabla_e^N F(\tfrac{x}N) \cdot G (\tfrac{x}N) .
\end{align*}
The last equality holds since $F(\tfrac{x}N)=0$ if $\tfrac{x}N\not\in\La$
for the second term
and $F(\tfrac{x+e}N)=0$ if $\tfrac{x+e}N\not\in\La$ for the third term. 
This shows \eqref{eq:SbP-1}.  The identity \eqref{eq:SbP-2} is immediate from
\eqref{eq:SbP-1} by taking $\bar G(\tfrac{x}N) := \bar a_x G(\tfrac{x}N)$ instead
of $G(\tfrac{x}N)$.  For \eqref{eq:SbP-3}, first consider $\bar G(\tfrac{x}N) :=
a_{x,e} G(\tfrac{x}N)$ instead of $G(\tfrac{x}N)$ for each fixed $e$, and then take
the sum in $e$.
\end{proof}

\section{Schauder estimate for the first discrete derivatives}  \label{sec:Schauder-1}

We are now ready to show the Schauder estimate for the first discrete derivatives.
We consider the linear discrete PDE \eqref{eq:1.2-linear} under the assumption
that the coefficients $a(t)=\{a_{x,e}(t)\}_{|e|=1}$ satisfy the $\a$-H\"older continuity
condition (A.2) stated below.  By \eqref{eq:cor2.3-2} in Corollary \ref{cor:2.3},
the nonlinear discrete PDE \eqref{eq:2.1-X} falls in this class with $\a=\si$.

We preview the main results.  In Section \ref{sec:4.1}, we study
the equation \eqref{eq:1.2-linear} and state Theorem \ref{thm:Li-T4.8}, which is
an analog of Theorem 4.8 of \cite{Li96}. Then, in Theorem \ref{thm:4.2},
we improve the regularity at $t=0$ for $C^2$-initial values.  In Section
\ref{first_schauder_context_subsec}, these results are applied for the nonlinear equation
\eqref{eq:2.1-X}; see Corollaries \ref{cor:K^3} and \ref{extended_cor}.
The outline of the proof of Theorem \ref{thm:Li-T4.8} is given in
Section \ref{sec:4.1} and the main part of the proof is postponed to later.
In Section \ref{interpolation_sec}, we show the discrete version of the
interpolation inequality.  In Section \ref{sec:6}, we apply 
Proposition \ref{lem:Li-L4.5} and two basic lemmas 
(Lemmas \ref{lem:Li-L4.3-G} and \ref{lem:Li-L4.6}
or its variant \ref{lem:Li-L4.6-N})
together with the summation by parts formula
(Lemma \ref{lem:DG-Q-2}) to show a useful energy inequality and a main H\"older estimate. Finally, Section \ref{sec:4.5} is for an estimate on
the time varying norm.

We will make use of the solution of the discrete heat equation
\eqref{eq:dHeq} or \eqref{eq:heat-D}, which will be denoted by $v^N$ in Sections
\ref{sec:Li-Schauder} and \ref{sec:Schauder-1}
to distinguish it from the solution $u^N$ of \eqref{eq:1.2-linear}.  We will compare
$u^N$ to $v^N$ as in \eqref{eq:w=u-v} or \eqref{eq:w=u-v-B} and will apply the
estimates for $v^N$ obtained in Proposition \ref{lem:Li-L4.5}.

\subsection{Schauder estimate for \eqref{eq:1.2-linear} with H\"older continuous coefficients}
\label{sec:4.1}

We consider the linear discrete PDE \eqref{eq:1.2-linear} 
with coefficients $a(t)$ and $g(t)$ continuous in $t$, that is,
\begin{align}
\label{general-divergence}
\partial_t u^N = L_{a(t)}^{N} u^N + g(t,\tfrac{x}N), \quad \tfrac{x}N \in \tfrac1N\T_N^n.
\end{align}

Our basic assumptions for $a(t)$ and $g(t)$ are:
\begin{itemize}
\item[(A.1)] (symmetry, nondegeneracy, boundedness) \quad $a(t) \in \mathcal{A}(c_-,c_+)$, $t\ge 0$,
\item[(A.2)] (H\"older continuity)\quad  $[a]^{*,N}_\alpha \le A<\infty$, $\a\in (0,1)$,
\item[(A.3)] (boundedness of $g$)\quad $\|g\|_\infty \le G_\infty<\infty$.
\end{itemize}
The assumption (A.2) means that $[a_e]_\a^{*,N} \le A$ holds for each $e$
and for the seminorm defined in \eqref{eq:norm-11} by regarding $a_{x,e}(t)$ as
a function of $(t,\tfrac{x}N)\in \Om_N$.
Note that, by \eqref{eq:cor2.3-2}, the coefficient $a(t)$ determined from 
the nonlinear Laplacian as in \eqref{eq:De-Lta} and \eqref{eq:a_xe} in the 
equation \eqref{eq:2.1-X} satisfies (A.2) with $\a=\si$ and $A=C(K+1)$;
see Corollary \ref{cor:K^3}.  The bound on $\|g\|_\infty$ in (A.3) is used
in the proof of Lemma \ref{lem:4.13}, but a weaker bound on 
$\sup_{t\in [0,T]}\|g(t)\|_{L^1(\T_N^n)}$ is sufficient at most places.

With respect to the solution $u^N(t,\tfrac{x}N)$ of \eqref{general-divergence}, 
define the polylinear spatial interpolation $\tilde u^N(t,z)$ following the prescription in 
\eqref{eq:poli-A}.

Set also, for $\alpha\in (0,1)$,
\begin{align}  
&U_{1+\a}:= [\tilde u^N]_{1+\a}^{*}
= \sup_{X\not= Y\in\Om} (d(X)\wedge d(Y))^{1+\a} 
\frac{|\nabla^N \tilde u^N(X)-\nabla^N \tilde u^N(Y)|}{|X-Y|^\a}, \nonumber\\
& U_1 := 
|\nabla^N \tilde u^N|_0^{(1)} \equiv
\max_{|e|=1, e>0} |\nabla_e^N \tilde u^N|_0^{(1)},      \label{seminorm_defn}   \\
& \mathcal{U}:=U_{1+\a} + \lan u^N\ran_{1+\a}^{*,N},  \notag
\end{align}
and the seminorm
\begin{align}  \label{eq:4.3}
|\tilde u^N|_{1+\a}^* := U_1+\mathcal{U} 
\equiv [\tilde u^N]_{1+\a}^{*}+ \lan u^N\ran_{1+\a}^{*,N} + |\nabla^N \tilde u^N|_0^{(1)}.
\end{align}
Here, we may recall Lemma \ref{lem:Eq-norms} for the consistency of two seminorms
$\lan u^N\ran_{1+\a}^{*,N}$ and $\lan \tilde u^N\ran_{1+\a}^{*}$.

We now come to the main estimate of this section.

\begin{thm}  \label{thm:Li-T4.8}
(First Schauder estimate, cf.\ Theorem 4.8 of \cite{Li96})
For the equation \eqref{eq:1.2-linear}, assume (A.1), (A.2), (A.3) and 
\begin{align}  \label{eq:4.bounded}
\sup_N \|u^N\|_\infty \equiv 
\sup_N \sup_{\Om_N}|u^N(t,\tfrac{x}N)|<\infty.
\end{align}
Then, we have
\begin{align}  \label{eq:4.4}
|\tilde u^N|_{1+\alpha}^* &\leq C\big[(A+1)^{1+\frac{1}{\alpha}} \|u^N\|_\infty + G_\infty\big].
\end{align}
 In particular, we have
\begin{align}  \label{eq:4.5}
|\nabla_e^Nu^N(t,\tfrac{x}{N})|&\leq C t^{-\tfrac{1}{2}}\big[(A+1)^{\frac{1}{\alpha}}\|u^N\|_\infty + (A+1)^{-1}G_\infty\big]
\end{align}
Here, the constants $C=C(n, c_\pm, T, \alpha)$.
\end{thm}

\begin{proof}
We outline the proof of Theorem \ref{thm:Li-T4.8}, referring to results proved later.

We first prove \eqref{eq:4.4}.
Recalling $\mathcal{U}=U_{1+\a} + \lan u^N\ran_{1+\a}^{*,N}$, we show
in Proposition \ref{spatial-variation-prop} that
$$
U_{1+\a} = [\tilde u^N]_{1+\a}^*
\le \bar{C} [(A+1) U_1+ G_\infty ]+ \de \mathcal{U},
$$
for every $\de>0$, where $\bar{C}=\bar{C}(n,c_\pm, \alpha, \delta)$.
Moreover, in Proposition \ref{temporal-prop}, we estimate that
$$
\lan u^N\ran_{1+\a}^{*,N}\le \widehat{C}[AU_1+ U_{1+\a}+G_\infty ]
$$
where $\widehat{C}=\widehat{C}(n,c_+, T, \alpha)$.
Thus, by these two estimates, we have
\begin{align*}
\mathcal{U}&= U_{1+\a}+ \lan u^N\ran_{1+\a}^{*,N}
\le \widehat{C}AU_1+ (\widehat{C}+1) U_{1+\a}+\widehat{C}G_\infty  \\
&\le  [\widehat{C}+\bar{C}(\widehat{C}+1)] (A+1)U_1+ \de (\widehat{C}+1) \mathcal{U}+[\widehat{C}+ \bar{C}(\widehat{C}+1)] G_\infty .
\end{align*}
Choosing $\de>0$ small, this implies
\begin{align}  \label{eq:4.6}
\mathcal{U}\le C[(A+1)U_1+G_\infty ]
\end{align}
and as $|\tilde u^N|_{1+\a}^* = U_1+\mathcal{U}$ that
\begin{align}
\label{eq:4.8}
|\tilde u^N|_{1+\a}^* &\le \big(C(A+1)+1\big)U_1+ CG_\infty
\end{align}
where $C=C(n,c_\pm, T,\alpha)$.

However, by the interpolation
inequality proved in Proposition \ref{prop:Li-P4.2} in Section \ref{interpolation_sec},
\begin{align}  \label{eq:4.7}
U_1\le 3 [\tilde u^N]_0^{\frac{\a}{1+\a}}
\big( [\tilde u^N]_0+ [\tilde u^N]_{1+\a}^*\big) ^{\frac1{1+\a}},
\end{align}
where we recall $[u]_0= \underset{\Om}{\text{osc}}\, u$ in \eqref{eq:3.osc}.  Observe also that
$[\tilde u^N]_0 \le 2 |\tilde u^N|_0= 2 \|u^N\|_\infty.$
Thus, from \eqref{eq:4.7}, we have that
\begin{align}  \label{eq:U_1}
\big(&C(A+1)+1\big)U_1 \\
& \le \Big(2^{\frac1\a} \big(3\big[C(A+1)+1\big]\big)^{\frac{1+\a}\a}
2\|u^N\|_\infty \Big)^{\tfrac\a{1+\a}} \Big( \tfrac12
\big(2\|u^N\|_\infty+ U_{1+\a}\big)\Big)^{\tfrac1{1+\a}}   \notag   \\
& \le 2^{\frac1\a} \big(3\big[C(A+1)+1\big]\big)^{\frac{1+\a}\a}
2\|u^N\|_\infty + \tfrac12
\big(2\|u^N\|_\infty + |\tilde u^N|_{1+\alpha}^* \big),  \notag
\end{align}
where $C=C(n, c_\pm, T, \alpha)$ and we have used a trivial bound: $ab \, (\le \tfrac{a^p}p+\tfrac{b^q}q)
\le a^p+b^q$
for $a,b>0$ and $\tfrac1p+\tfrac1q=1$.  
Inserting this into \eqref{eq:4.8},
we obtain \eqref{eq:4.4}.

The estimate \eqref{eq:4.5} now follows from \eqref{eq:U_1} and \eqref{eq:4.4}.
\end{proof}

When the initial value $u^N(0,\cdot)$ is $C^2$, or more precisely, if
it satisfies the bound
\begin{align}  \label{eq:Th4.2-0}
\sup_N \| u^N(0)\|_{C_N^2} \le C_0<\infty,
\end{align}
we expect a better regularity estimate at $t=0$.  Recall \eqref{eq:norm-uN}
for the norm $\| \cdot \|_{C_N^2}$.
Indeed to discuss the equation \eqref{eq:1.2-linear}, we need the
following additional assumptions for the coefficient $a_{x,e}(t)$:
\begin{align}  \label{eq:Th4.2-B}
& [a]_\a^{(-\a),N} \le B<\infty: \quad
|a_{x_1,e}(t_1) - a_{x_2,e}(t_2)| \le B \Big\{ \big| t_2-t_1 \big|^{\frac{\a}2}
+ \big| \tfrac{x_1}N - \tfrac{x_2}N \big|^\a \Big\},\\
\label{eq:Th4.2-1}
& \sup_{N,x,e} |\nabla_e^{N,*}a_{x,e}(0)| \le C_1<\infty.
\end{align}
The condition \eqref{eq:Th4.2-B} is understood in relation to (A.2) for the seminorm
$[a_e]_\a^{(-\a),N}$ for each $e$ defined by \eqref{eq:norm-9} taking 
$a=\a, k=0, b=-\a$ so that $a+b=0$.  Or $[a_e]_\a^{(-\a),N}$ is the seminorm
$[a_e]_\a$ in \eqref{eq:norm-2} with the supremum taken for $X\not= Y\in \Om_N$.

\begin{thm}  \label{thm:4.2}
We assume (A.1), (A.3), \eqref{eq:4.bounded} in Theorem \ref{thm:Li-T4.8}
and \eqref{eq:Th4.2-0}, \eqref{eq:Th4.2-B}, \eqref{eq:Th4.2-1} stated above.
Then, for the solution $u^N$ of \eqref{eq:1.2-linear}, we have
\begin{align}  \label{eq:4-13}
|\tilde u^N|_{1+\a} 
\le C[(B +1)^{1+\frac1\a}\|u^N\|_\infty+G_\infty + C_2],
\end{align}
where $|\tilde u^N|_{1+\a}$ is the unweighted norm defined in \eqref{eq:norm-4}.
In particular, we have
\begin{align}  \label{eq:4-14}
|\nabla_e^N u^N(t,\tfrac{x}N)|&\leq C\big[(B+1)^{\frac1\a}\|u^N\|_\infty+(B+1)^{-1}\big(G_\infty + C_2]\big)\big]
\end{align}
and the $\a$-H\"older seminorm of $\nabla _e^N u^N$ is uniformly bounded in $t\geq 0$:
\begin{align}  \label{eq:4-15}
[\nabla_e^N \tilde u^N]_\a \le C[(B+1)^{1+\frac1\a} \|u^N\|_\infty +G_\infty +C_2].
\end{align}
Here, $C=C(n, c_\pm, T, \alpha)$
and $C_2 = n(C_1C_0 + c_+C_0)$.
\end{thm}

\begin{proof}
The proof is similar to that of Theorem \ref{extended-Holder-thm}, but to guarantee
the condition (A.2) for the extended system \eqref{extended_equation-B}
below, especially to well-connect at $t=1$, we replace $\De^N$ by $L^{N,0}$ defined by
$$
L^{N,0} =-\sum_{|e|=1, e>0}\nabla_e^{N,*} (a_{x,e}(0)\nabla^N_e \cdot).
$$
Then, consider the discrete linear PDE,
\begin{align}  \label{eq:2.19-Q}
\partial_s v = L^{N,0} v,\quad s \in (0,1],
\end{align}
with initial condition $v(0,\tfrac{x}{N}) = u^N(0,\tfrac{x}{N})$ for $x\in \T^n_N$.
Define $\hat v(t) := v(1-t)$ for $0\leq t< 1$ and $\hat h(t,\tfrac{x}{N}) := 
-L^{N,0} \hat v(t,\tfrac{x}{N})$.  Note, for $0\leq t< 1$, that $\hat v$ satisfies
\begin{align*}
\partial_t \hat v = -L^{N,0} \hat v = L^{N,0}\hat v + 2\hat h.
\end{align*}

However, by \eqref{eq:2.19-Q}, $h(t):= \hat h(1-t) = -L^{N,0} v(t)$ satisfies the discrete PDE
$\partial_t h(t) = L^{N,0} h(t)$ with initial value $h(0)=-L^{N,0}u^N(0)$ and thus,
by the maximum principle (Lemma \ref{lem:1.1}) for this equation, we have
\begin{align*}
|h(t,\tfrac{x}{N})| \le \max_y |L^{N,0} u^N(0,\tfrac{y}N)|.
\end{align*}
But, by the condition \eqref{eq:Th4.2-0} for $u^N(0)$ and by
\eqref{eq:Th4.2-1} for $a_{x,e}(0)$, 
we have, applying \eqref{eq:disc-der-D}, 
\begin{align*}
\|\hat h\|_\infty &= \Big\| \sum_{|e|=1, e>0} \Big\{ \nabla^{N,*}_e a_{\cdot,e}(0) 
\cdot \nabla^N_e u^N(0, \cdot) + a_{\cdot-e,e}(0) \nabla^{N,*}_e \nabla^N_e u^N(0,\cdot)
\Big\}\Big\|_\infty\\
&\leq n(C_1C_0 + c_+C_0) =: C_2.
\end{align*}

Define now 
\begin{align}
\label{hat a defn-B}
\hat a_{x,e} (t)= \left\{ \begin{array}{rl}
a_{x,e}(t-1) & \ {\rm for \ }t\ge 1\\
a_{x,e}(0)& \ {\rm for \ } 0\leq t<1,
\end{array}\right.
\end{align}
and $\hat g(t,\tfrac{x}{N})$ by \eqref{hat g defn} in the present setting.
Consider the extended system, for $t\geq 0$,
\begin{align}
\label{extended_equation-B}
\partial_t \hat u^N = L_{\hat a(t)}^N\hat u^N + \hat g(t,\tfrac{x}N).
\end{align}

Note that $\hat a$ satisfies (A.1) and $[\hat a]_{\a}^{*,N}\le T^{\frac{\alpha}{2}}B<\infty$
from the condition \eqref{eq:Th4.2-B} (which holds also at $t_1, t_2=0$), 
and $\|\hat g\|_{\infty} \leq \|g\|_{\infty} + 2C_2$.
Also, by the maximum principle, $v$ is uniformly bounded by $\|v(s)\|_\infty
\le \|u^N(0)\|_\infty$, $s\in [0,1]$.  Therefore, $\hat u^N$ is bounded by $\|u^N\|_\infty$.
Then, Theorem \ref{thm:Li-T4.8} yields
$$
|\widetilde{\hat u^N}|_{1+\a}^* \le C\big[(B+1)^{1+\frac1\a}\|u^N\|_\infty +G_\infty + 2C_2\big].
$$
where $C=C(n,c_\pm, T, \alpha)$.
We observe, by specializing to times $1\leq t \leq T+1$ and noting
$\hat u(t, \cdot) = u^N(t-1, \cdot)$, that \eqref{eq:4-13} holds for the solution 
$u=u^N$ of \eqref{eq:1.2-linear}. Similarly, \eqref{eq:4-14} follows from
\eqref{eq:4.5}. \eqref{eq:4-15} is immediate from \eqref{eq:4-13}.
\end{proof}

\subsection{Schauder estimate in the context of \eqref{eq:2.1-X}}
\label{first_schauder_context_subsec}

We now consider the solution $u^N(t,\tfrac{x}N)$ of the discrete nonlinear PDE 
\eqref{eq:2.1-X}.   Recall  $u_\pm$ in \eqref{eq:1.u_pm} and $c_\pm$ in \eqref{eq:c_pm-N} 
in this context.

The following is a corollary of Theorem \ref{thm:Li-T4.8}.
\begin{cor} \label{cor:K^3}
For the solution $u^N(t,\tfrac{x}N)$ of \eqref{eq:2.1-X} satisfying \eqref{eq:1.u_pm}, we have 
$$
|\tilde u^N|_{1+\si}^* \le C(K^{1+\frac1\si}+1),
$$
where $\si\in (0,1)$ is as in Theorem \ref{thm:Holder-u^N}, and
$C=C(n, c_\pm, T, \sigma, \|f\|_\infty, \|\fa''\|_\infty, u_\pm)$.
In particular, from the last term in the seminorm $|\tilde u^N|_{1+\si}^*$ in
\eqref{eq:4.3}, noting $d(X)=t^{\frac12}$, we obtain
$$
|\nabla_e^N u^N(t,\tfrac{x}N)| \le C t^{-\frac12}(K^{\frac1\si}+1), \quad t\in (0,T].
$$
From the first term in \eqref{eq:4.3}, that is $U_{1+\a}$ in \eqref{seminorm_defn},
we see that the $\si$-H\"older seminorm of $\nabla_e^N u^N$ has 
the singularity $(t^{-\frac12})^{1+\si}= t^{-\frac12(1+\si)}$ near $t=0$.
\end{cor}

\begin{proof}
We first observe that $g = Kf(u^N)$ is such that $\|g\|_\infty \leq \|f\|_\infty K:= G_\infty$ by \eqref{eq:1.u_pm}.
Next, we have that $a(t):=a(u^N(t))$,
that is $a_{x,e}(u)$ in \eqref{eq:a_xe} taking $u=u^N(t)$, satisfies (A.1) and by Corollary \ref{cor:2.3}, $[a]_{\si}^{*,N} \leq C(n, c_\pm, T, \|u^N(0)\|_\infty, \|f\|_\infty, \|\fa''\|_\infty)(K+1):= A$.
The condition \eqref{eq:4.bounded} follows from \eqref{eq:1.u_pm}.
Therefore, the corollary is a consequence of Theorem \ref{thm:Li-T4.8}.
\end{proof}

We have the following better regularity estimate at $t=0$ assuming \eqref{eq:Th4.2-0}.  
Recall the constant $C_0$ in \eqref{eq:Th4.2-0}.

\begin{cor}
\label{extended_cor}
For the solution $u^N$ of \eqref{eq:2.1-X} satisfying \eqref{eq:1.u_pm}, 
when the initial data satisfies \eqref{eq:Th4.2-0}, we have
\begin{align}  \label{eq:cor4.4-1}
|\tilde u^N|_{1+\si} \le C (K+C_0 +1)^{1+\frac1\si},
\end{align}
where $\si\in (0,1)$ is as in Theorem \ref{thm:Holder-u^N} and 
$C=C(n, c_\pm, T, \sigma, \|f\|_\infty, \|\fa''\|_\infty, u_\pm)$.  
Recall $\|f\|_\infty$ and $\|\fa''\|_\infty$  defined in Corollary \ref{cor:2.3}.
In particular, 
\begin{align}  \label{eq:cor4.4-2}
|\nabla_e^N u^N(t,\tfrac{x}N)|&\leq C (K+C_0+1)^{\frac1\si}
\end{align}
and the $\si$-H\"older seminorm of $\nabla _e^N u^N$ is uniformly bounded in $t\geq 0$,
\begin{align}  \label{eq:cor4.4-3}
[\nabla^N_e \tilde u^N]_\si \leq C (K+C_0+1)^{1+\frac1\si}.
\end{align}
\end{cor}

\begin{proof}
Note that the condition \eqref{eq:Th4.2-B} follows from 
\eqref{eq:cor2.6-2} in  Corollary \ref{cor:2.6} with 
$$B= C(K \|f\|_\infty + \|u^N(0)\|_\infty + C_0)$$ 
where $C=C(n, c_\pm, T, \|\fa''\|_\infty)$.   Moreover,
since $a_{x,e}(0) = a_{x,e}(u^N(0))$ is determined by \eqref{eq:a_xe},
it satisfies (A.1) and the condition \eqref{eq:Th4.2-1}, 
$|\nabla^{N,*}_e a_{x,e}(0)|\leq C_1$ where $C_1= C(\|\fa''\|_\infty) C_0$,
due to the `mean-value' Lemma \ref{lem:mvt} and \eqref{eq:Th4.2-0}.
Thus, with $\|g\|_\infty = K \|f\|_\infty =G_\infty$ and $C_2 =n (C_1C_0 + c_+C_0)
\leq C(n, c_+, \|\fa''\|_\infty)[C_0^2+1]$,
the corollary follows from Theorem \ref{thm:4.2}. 
Note that we obtain the bound
$$
|\tilde u^N|_{1+\si} \le C\big((K+C_0 +1)^{1+\frac1\si} + C_0^2 +1\big)
$$
and the others.  However, since $\si\in (0,1)$, both $C_0^2$ and $1$ are bounded by
$(K+C_0 +1)^{1+\frac1\si}$ so that the desired three estimates 
\eqref{eq:cor4.4-1}--\eqref{eq:cor4.4-3} are shown.
\end{proof}

\subsection{Interpolation inequality}
\label{interpolation_sec}

We adapt the (4.2c) in Proposition 4.1 of \cite{Li96} to our discrete context.
Recall \eqref{eq:3.3-1} for the continuous space gradient $\nabla^N u(t,z)$
defined for a function $u$ on $\Om=[0,T]\times \T^n$. 

\begin{prop}  \label{prop:Li-P4.2}
(cf.\ (4.2c) in Proposition 4.1 of \cite{Li96})
Let $\a\in (0,1)$.
Then, for every function $u$ on $\Om$,
\begin{align}  \label{eq:4.U1}
U_1:= |\nabla^N u|_0^{(1)} \, \Big(\!\! = \max_{|e|=1, e>0} |\nabla_e^N u|_0^{(1)} \Big)
\le 3 [u]_0^{\frac{\a}{1+\a}}
\big( [u]_0+ [u]_{1+\a}^*\big) ^{\frac1{1+\a}}.
\end{align}
Moreover,
\begin{align}\label{eq:4.U2}
U_2 := |\nabla^N u|_0^{(2)}\  \le \ 5( |u|_0^{(1)})^{\frac{\a}{1+\a}}
\big( |u|_0^{(1)}+ [u]_{1+\a}^{(1)}\big) ^{\frac1{1+\a}}.
\end{align}
\end{prop}

\begin{proof}
We argue now the first statement and later discuss the second inequality.  

Consider $ \nabla^N_e u(X)$ for $X=(t,z)\in \Omega$ for fixed $e$.
For simplicity, let us assume that $e$ is the vector $\langle 1,0,\ldots,0\rangle$.  We will write 
$z=\langle z_1,z_2,\ldots,z_n\rangle \in \T^n$.

Let $Y=(t,y)$ be such that $y=(y_1,z_2,z_3,\ldots,z_n\rangle$ and $z_1-y_1$ has the same sign as $\nabla^N_e u(X) \in \R$.
Take $Y$ also such that $|z-y| =|z_1-y_1|= \epsilon d(X)$ for an $\epsilon\in (0,\tfrac12]$.
We note that in this argument that $t>0$ is fixed, and so $d(X)=d(Y) = \sqrt{t}$.

Then, since $z_1-y_1$ has the same sign as $\nabla^N_e u(X)$ and
$|z-y| =|z_1-y_1|$,  we have
\begin{align}  \label{eq:4.21-u}
|\nabla^N_e u(X)| &= \nabla^N_e u(X) \cdot \frac{z_1-y_1}{|z-y|} \\
&= \Big(\frac{1}{|z-y|} \int_{z_1}^{y_1} \nabla^N_e u(t,w)dw_1\Big)\cdot \frac{z_1-y_1}{|z-y|}\nonumber\\
&\ \ \ \ \ \ \  + \Big(\nabla^N_e u(X) - \frac{1}{|z-y|} \int_{z_1}^{y_1} \nabla^N_e u(t,w)dw_1\Big)\cdot \frac{z_1-y_1}{|z-y|}\nonumber,
\end{align}
where $w = \langle w_1,z_2,\ldots,z_n\rangle$.  
However, the above integral term can be rewritten as
\begin{align*} 
\int_{z_1}^{y_1}  \nabla^N_e u(t,w)dw_1 &= N\int_{z_1}^{y_1} [u(t,w+\tfrac{e}N) - u(t,w)] dw_{1}\\
&= N \Big\{\int_{z_1+\tfrac1N}^{y_1+\tfrac1N} u(t,w)dw_1 - \int_{z_1}^{y_1} u(t,w)dw_1\Big\}
  \notag \\
&= N \int_{z_1}^{z_1+\tfrac1N} \big[u(t,w+y-z) -u(t,w) \big]dw_1,
\notag
\end{align*}
Thus, by \eqref{eq:4.21-u}
and recalling $[u]_0 = \underset{\Om}{{\rm osc}}(u)$ and $[u]_{1+\a}^*$ in
\eqref{eq:norm-6}, \eqref{eq:norm-11}, we have
\begin{align}\label{Osc step}
|\nabla^N_e u(X)| &\leq \frac{N}{|z-y|} \int_{z_1}^{z_1+\tfrac1N}\big|u(t,w+y-z)-u(t,w) \big|dw_1\\
&\ \ \ \ \ \ \ \ \ \ \ \ \ \ \ \ + \Big(\frac{1}{|z-y|} \int_{z_1}^{y_1} 
[\nabla^N_e u(X) - \nabla^N_e u(t,w)] dw_1\Big)\cdot \frac{z_1-y_1}{|z-y|}  \notag  \\
&\leq \frac{[u]_0}{|z-y|} + \frac{[u]_{1+\a}^*}{|z-y|} \Big| \int_{z_1}^{y_1} 
(d(X)\wedge d(W))^{-(1+\alpha)} |z-w|^\alpha dw_1\Big|, \nonumber
\end{align}
where $W=(t,w)$.
As remarked earlier, as $t>0$ is fixed, $d(X)=d(W)=\sqrt{t}$ for the points considered above.
Also,
$\big|\int_{z_1}^{y_1} |z-w|^\alpha dw_1\big| \leq \frac{1}{1+\alpha} |z-y|^{1+\alpha} \leq |z-y|^{1+\alpha}$.

From these computations, we have
\begin{align*}
|\nabla^N_e u(X)|
&\leq \frac{[u]_0}{|z-y|} + 
d(X)^{-(1+\alpha)}[u]_{1+\a}^* |z-y|^\alpha\\
&= \frac{[u]_0}{\epsilon d(X)} + \frac{\epsilon^\alpha [u]_{1+\a}^*}{d(X)}.
\end{align*}
Multiply by $d(X)$ and take supremum in $X$ to get
$$U_1 \leq \frac{[u]_0}{\epsilon } + 
\epsilon^\alpha [u]_{1+\a}^*.$$

We would like to choose $\epsilon \in (0,\tfrac12]$ as follows:
If $[u]_0<[u]_{1+\a}^*$,
then take $\epsilon = \tfrac12 ([u]_0/[u]_{1+\a}^*)^{1/(1+\alpha)} (< \tfrac12)$.  
We get in this case
\begin{align*}
U_1 &\leq 2\big([u]_{1+\a}^*\big)^{1/(1+\alpha)}[u]_0^{\alpha/(1+\alpha)} 
+  2^{-\a}[u]_0^{\alpha/(1+\alpha)} \big([u]_{1+\a}^*\big)^{1/(1+\alpha)} \\
& \le 3 \big([u]_{1+\a}^*\big)^{1/(1+\alpha)}[u]_0^{\alpha/(1+\alpha)},
\end{align*}
which is bounded by the right hand side of \eqref{eq:4.U1}.
Otherwise, when $[u]_0\ge [u]_{1+\a}^*$,
choose $\epsilon = 1/2$.  Then, 
\begin{align*}
U_1 & \leq 2[u]_0 + 
2^{-\alpha} [u]_{1+\a}^* 
\leq 3
[u]_0,\end{align*}
which is also bounded by the right hand side of \eqref{eq:4.U1}.
Thus, the first inequality \eqref{eq:4.U1} follows.  

The second inequality \eqref{eq:4.U2} follows the same scheme.  Indeed, in \eqref{Osc step}, 
one bounds $|u(t,w+y-z)- u(t,w)| \leq |u(t,w+y-z)| + |u(t,w)| \le 2 d(X)^{-1} |u|_0^{(1)}$ 
and $|\nabla^N_e u(X) - \nabla^N_eu(t,w)| \leq [u]_{1+\a}^{(1)} d(X)^{-(2+\alpha)}|z-w|^\alpha$,
recall \eqref{eq:norm-5} for $|u|_0^{(1)}$.  Multiplying through by $d^2(X)=(\sqrt{t})^2$ 
at this point, we may follow the derivation of the first inequality to obtain the second statement.
\end{proof}

\subsection{Energy inequalities and estimate on $U_{1+\alpha}=[\tilde u^N]^*_{1+\alpha}$}
\label{sec:6}

For an open domain $D\subset \T^n$, we define its discrete interior by
\begin{equation} \label{eq:outerb}
D_N:= D\cap \tfrac1N \T_N^n.
\end{equation} 
The outer boundary $\partial_N^+ D_N$ and the closure $\overline{D_N}$ of $D_N$ are
defined as in \eqref{eq:3.ob+c} taking $\La=D_N$, respectively.
Recall $\Om = [0,T]\times \T^n$ and $\Om_N =[0,T]\times \tfrac1N\T_N^n$.

Take $Y=(t_1,y)\in \Om$ and $r: 0< r<\frac12 d(Y) = \frac12 t_1^{\frac12}$.
Recall $Q(r)\equiv Q(Y,r) = (t_1-r^2,t_1)\times D(y,r)$,
where $D(y,r) = \{z\in \T^n;|z-y|<r\}$ with the distance $|z-y|$ as in
\eqref{eq:distanceTn}, and set  $Q_N(r) \equiv Q_N(Y,r) = (t_1-r^2,t_1)\times D_N(y,r)
\, (=Q(r) \cap \Om_N)$, where $D_N(y,r)$ is the discrete interior of
$D(y,r)$.  Define the parabolic outer boundary of $Q_N(Y,r)$ by
\begin{equation} \label{eq:4.24-A}
\mathcal{P}_N^+Q_N(Y,r) :=  \{t_1-r^2\}\times \overline{D_N(y,r)}
\cup  (t_1-r^2,t_1]\times \partial_N^+D_N(y,r),
\end{equation} 
where $\partial_N^+D_N(y,r)$ and $\overline{D_N(y,r)}$ are the 
outer boundary and closure of $D_N(y,r)$ defined as above, respectively.
We also denote 
\begin{equation} \label{eq:4.25-A}
\overline{Q_N(Y,r)} :=[t_1-r^2,t_1]\times \overline{D_N(y,r)}.
\end{equation} 
In the following, $Y$ is fixed until it moves in
the proof of Proposition \ref{spatial-variation-prop}.  All constants
will be uniform in $Y$.

Let $u=u^N$ be the solution of the linear discrete PDE 
\eqref{general-divergence} or equivalently \eqref{eq:1.2-linear}
on $\Om_N$:
$$
\mathcal{L}_{a(\cdot)}u \equiv (L_{a(t)}^{N} -\partial_t) u =-g(t)
$$ 
with $a(t)$ and $g(t)$ satisfying the assumptions (A.1), (A.2) and (A.3).

Take the closest point $\tilde y \in \tfrac1N\T_N^n$ to $y$, in particular,
$|y-\tilde y| \le \tfrac{\sqrt{n}}{2N}$ holds, and set $\tilde Y :=(t_1,\tilde y)\in \Om_N$.
Note that $d(\tilde Y) =d(Y)= t_1^{\tfrac12}$.
Let $v=v^N =v^{N,Q(r)}$ be the solution of the discrete heat equation
\eqref{eq:heat-D} or equivalently \eqref{eq:dHeq} on $Q_N(r)=Q_N(Y,r)$:
\begin{align}  \label{eq:4.heat}
\mathcal{L}_av\equiv (\De_a^N -\partial_t)v =0
\end{align}
with constant coefficients
$a_e:= a_e(\tilde Y) \equiv a_{N\tilde y,e}(t_1)\ge c_->0$
under the boundary condition $v=u$ at $\mathcal{P}_N^+Q_N(Y,r)$.
We will consider the case that $r> \tfrac{\sqrt{n}}N$ (in the proof of Lemma
\ref{lem:5.3}), in particular, $Q_N(r)\not=\emptyset$.

Set 
\begin{align}  \label{eq:w=u-v}
w \equiv w^N:= u-v\equiv u^N-v^N \quad \text{ on } \; \overline{Q_N(Y,r)}.
\end{align}
Then, the following discrete energy inequality holds.

\begin{lem} \label{lem:D-energy-a-Holder}
Assume  $0<r<\frac12 d(Y)=\frac12 t_1^{\frac12}$ and $Q_N(r)\not=\emptyset$.
Then, we have
\begin{align}  \label{eq:Nnablawdt}
\int_{t_1-r^2}^{t_1} & N^{-n}\sum_{\tfrac{x}N \text{ or }\frac{x+e}N\in 
D_N(y,r); |e|=1, e>0} |\nabla_e^N w|^2(X) dt \\
& \le CA^2 (\tfrac{r+\tfrac{1+\sqrt{n}}N}{d(Y)} )^{2\a} (r+\tfrac1N)^n r^2 
\sup_{X \text{ or } X+\frac{e}N \in Q_N(r); |e|=1, e>0} |\nabla_e^N u(X)|^2 
\notag \\
& \quad +  C G_\infty  (r+\tfrac1N)^n (\tfrac{r}{d(Y)})^{1+\a}  \sup_{X\in Q_N(r)} |w(X)|,  \notag
\end{align}
where $X=(t,\tfrac{x}N)$ and $X+\frac{e}N = (t, \frac{x+e}N)$, and recall
$Q_N(r) \equiv Q_N(Y,r)$.  Here, the constants $C=C(n,c_-, T)$.
\end{lem}

\begin{proof}
The proof is divided into steps.

\vskip .1cm
{\it Step 1.} First from $\partial_t w= L_{a(t)}^{N}u+g(t)-\De_a^Nv$ on $Q_N(r)$,
and then applying \eqref{eq:SbP-3} in Lemma \ref{lem:DG-Q-2} since
$w(t,\cdot)=0$ at the boundary $\partial_N^+D_N(y,r)$, we have
\begin{align*}
\tfrac12\partial_t & \bigg( N^{-n}\sum_{\tfrac{x}N\in D_N(y,r)} w^2(X) \bigg)\\
& = N^{-n}\sum_{\tfrac{x}N \in D_N(y,r); |e|=1, e>0} 
w(X) \big( -\nabla_e^{N,*} (a_e(\tilde{Y})\nabla_e^N w)\big)(X) \\
& \qquad + N^{-n}\sum_{\tfrac{x}N \in D_N(y,r); |e|=1, e>0} 
w(X) \big( -\nabla_e^{N,*}((a_e(X)-a_e(\tilde{Y}))\nabla_e^N u)\big) (X)  \\
&\qquad + N^{-n}\sum_{\tfrac{x}N\in D_N(y,r)} w(X)g(X) \\
&= - N^{-n}\sum_* a_e(\tilde{Y})|\nabla_e^N w|^2(X)\\
& \qquad - N^{-n}\sum_* (a_e(X)-a_e(\tilde{Y})) \nabla_e^N w(X)
 \cdot \nabla_e^N u(X)  \\
&\qquad + N^{-n}\sum_{\tfrac{x}N\in D_N(y,r)} w(X)g(X) \\
& =: -I_1+I_2+I_3,
\end{align*}
where $X=(t,\tfrac{x}N)$, $a_e(X):= a_{x,e}(t)$ and $\sum\limits_*$ means the sum
over $x$ and $e$ such that $\tfrac{x}N$ or $\frac{x+e}N\in D_N(y,r)$
and $|e|=1, e>0$.  Note that we put a minus sign in $I_1$ so that $I_1\ge 0$.
\vskip .1cm

{\it Step 2.}
Integrate both sides in $t\in [t_1-r^2,t_1]$.  Since $w(t_1-r^2, \cdot)=0$,
the left hand side is
$$
\tfrac12 N^{-n}\sum_{\tfrac{x}N\in D_N(y,r)} w^2(t_1, \tfrac{x}N)\ge 0.
$$

On the other hand, for every $\e>0$,
\begin{align*}
\int_{t_1-r^2}^{t_1} |I_2|dt 
& \le \e \int_{t_1-r^2}^{t_1}  N^{-n}\sum_* |\nabla_e^N w|^2(X) dt \\
& \qquad + \frac1\e \int_{t_1-r^2}^{t_1}  N^{-n}
\sum_* |a_e(X)-a_e(\tilde{Y})|^2 |\nabla_e^N u|^2(X) dt.
\end{align*}
Here, we estimate $ |\nabla_e^N u|^2(X) \le \sup_{X \text{ or } 
X+\frac{e}N\in Q_N(r)} |\nabla_e^N u|^2$, and by $[a]_\a^{*,N}\le A$,
\begin{align*}
|a_e(X)-a_e(\tilde{Y})|\le A(d(X)\wedge d(\tilde{Y}))^{-\a} |X-\tilde{Y}|^\a
\le A 2^\a d(Y)^{-\a}(r+\tfrac{1+\sqrt{n}}N)^\a,
\end{align*}
since $|X-\tilde{Y}|\le |X-Y|+|Y-\tilde{Y}| \le (r+\tfrac1N)+\tfrac{\sqrt{n}}{2N}
\le r+\tfrac{1+\sqrt{n}}N$ 
for $X\in \overline{Q_N(r)}$ (note that, when $X+\tfrac{e}N\in Q_N(r)$,
it happens that $X\in \overline{Q_N(r)}$)
and $r<\frac12 d(Y)= \tfrac12 d(\tilde Y)$ implies $d(X)\ge \frac12 d(Y)$ 
(note the time direction is not enlarged by $\frac1N$) as
\begin{align}  \label{eq:d-Y-X}
d(Y)=t_1^{\frac12}\le t^{\frac12}+|t-t_1|^{\frac12}
\le d(X)+r \le d(X) + \tfrac12 d(Y),
\end{align}
so that $d(X)\wedge d(\tilde{Y}) = d(X) \ge \frac12 d(Y)$.
Moreover,
$$
\int_{t_1-r^2}^{t_1} N^{-n}
\sum_{\tfrac{x}N \text{ or }\frac{x+e}N\in 
D_N(y,r); |e|=1, e>0} 1 dt \le C(n) (r+\tfrac1N)^n r^2.
$$
Note that, even for very small $r>0$, $D_N(y,r)$ may contain
a single point so that we pick up the factor $+\tfrac1N$.
Therefore, recalling $a_e(\tilde{Y})\ge c_->0$, we obtain
\begin{align*}
\int_{t_1-r^2}^{t_1} |I_2|dt 
& \le \frac{\e}{c_-} \int_{t_1-r^2}^{t_1}  I_1dt \\
& \quad + \frac1\e C(n)2^{2\a} A^2(\tfrac{r+\tfrac{1+\sqrt{n}}N}{d(Y)})^{2\a} (r+\tfrac1N)^n r^2 
\sup_{X \text{ or } X+\frac{e}N \in Q_N(r); |e|=1, e>0} |\nabla_e^N u|^2.
\end{align*}

\vskip .1cm
{\it Step 3.}
For $I_3$, since $|Q_N(r)| \le C(n) N^n (r+\tfrac1N)^n r^2$ (we pick up $+\frac1N$
similarly as above), by $\|g\|_\infty \le G_\infty$,
\begin{align*}
\int_{t_1-r^2}^{t_1} |I_3|dt
& \le G_\infty  C(n) (r+\tfrac1N)^n r^2 \sup_{Q_N(r)} |w|\\
&=  c G_\infty  (r+\tfrac1N)^n (\tfrac{r}{d(Y)})^{1+\a}(\tfrac{r}{d(Y)})^{1-\a} 
d(Y)^2  \sup_{Q_N(r)} |w|\\
& \le (\tfrac12)^{1-\a} T C(n) G_\infty  (r+\tfrac1N)^n (\tfrac{r}{d(Y)})^{1+\a}  \sup_{Q_N(r)} |w|,
\end{align*}
since $\frac{r}{d(Y)}\le \frac12$ and $d(Y)^2 \le T$.  Summarizing these estimates,
we have
\begin{align*}
(1-\tfrac{\e}{c_-}) \int_{t_1-r^2}^{t_1}  I_1 dt 
\le &\frac1\e C(n)2^{2\a}A^2 (\tfrac{r+\tfrac{1+\sqrt{n}}N}{d(Y)} )^{2\a}  (r+\tfrac1N)^n r^{2} 
\sup_{X \text{ or } X+\frac{e}N \in Q_N(r); |e|=1, e>0}|\nabla_e^N u|^2\\
&+  (\tfrac12)^{1-\a} TC(n) G_\infty   (r+\tfrac1N)^n (\tfrac{r}{d(Y)})^{1+\a}  \sup_{Q_N(r)} |w|.
\end{align*}
Choosing $\e>0$ small and noting $a_e(\tilde{Y})\ge c_->0$, we have shown \eqref{eq:Nnablawdt}.
\end{proof}

We now give estimates on three terms in \eqref{eq:Nnablawdt}
in terms of polylinear interpolations. First recall that $\tilde u = \tilde u^N(t,z), 
(t,z)\in \Om$ was defined as the polylinear interpolation of $u=u^N(t,\tfrac{x}N)$
in Section \ref{sec:4.1}.  

Next, taking $r_N^1 = r + \tfrac{\sqrt{n}+1}N$, consider the solution $v \equiv v^{N, r}(X)
:= v^{N, Q(Y,r_N^1)}(X)$, $X\in \overline{Q_N(Y,r_N^1)}$
of the discrete heat equation \eqref{eq:4.heat} on $Q_N(Y,r_N^1)$
with the boundary condition $v=u$ at $\mathcal{P}_N^+ Q_N(Y,r_N^1)$.
Then, set as in \eqref{eq:w=u-v}
\begin{align}  \label{eq:w=u-v-B}
w \equiv w^{N, r}:= u- v\equiv u^N- v^{N, r}
\quad \text{ on } \; \overline{Q_N(Y,r_N^1)}.
\end{align}
We consider $v=v^{N,r}$ on a domain $Q_N(Y,r_N^1)$ slightly enlarging
$Q_N(Y,r)$.  By this choice, the polylinear interpolations 
$\tilde v= \tilde{v}^{N,r}(X)$ and $\tilde w= \tilde{w}^{N,r}(X)$
of $v$ and $w$, respectively, are well-defined for $X\in Q(r)\equiv Q(Y,r)$.
Moreover, $\nabla_e^N\tilde v(X)$ and $\nabla_e^N\tilde w(X)$ are also defined
for $X\in Q(r)$ and $\widetilde{\nabla^N_ew} = \nabla_e^N \tilde w$ holds on $Q(r)$
from \eqref{eq:tilde-grad}; recall the discussions above Corollary \ref{cor:Li-L4.4}
and  Proposition \ref{lem:Li-L4.5}.  Furthermore, the estimate 
\eqref{eq:Li-L4.5-3} in Proposition \ref{lem:Li-L4.5} is applicable for $v=v^{N,r}(X)$ by taking
$R=r_N^1$ so that $0<\rho< r\le r_N^1-\tfrac{\sqrt{n}+1}N =r$.

Recall \eqref{seminorm_defn}  for $U_1$ and $\mathcal{U}$
defined from $\tilde u^N$ and $u^N$.

\begin{lem} \label{lem:estimate-D-energy}
Assume  $0<r <\frac12 d(Y)= \tfrac12 t_1^{\tfrac12}$.
Then, we have
\begin{align}  \label{eq:1-Nnablawdt}
& \int_{Q(r)} |\widetilde{\nabla^Nw}|^2(X)dX \le 
\int_{t_1-r^2}^{t_1} N^{-n}\sum_{\tfrac{x}N \in 
D_N(y,r+\tfrac{\sqrt{n}}N); |e|=1, e>0} |\nabla_e^N w|^2(t,\tfrac{x}N) dt, \\
\label{eq:2-Nnablawdt}
& \sup_{X \text{ or } X+\frac{e}N \in Q_N(r); |e|=1, e>0} |\nabla_e^N u(X)|^2
\le \frac{4U_1^2}{d(Y)^2}, \\
\label{eq:3-Nnablawdt}
& \sup_{X\in Q_N(r)} |w(X)| \le C(n) \mathcal{U} (\tfrac{r}{d(Y)})^{1+\a}.
\end{align}

Note that, in the left hand side of \eqref{eq:1-Nnablawdt}, differently from
$X=(t,\tfrac{x}N)$ in its right hand side, that $X=(t,z)\in \Om$ is a continuous
variable.  Also recall $|\widetilde{\nabla^Nw}(X)| := \max_{|e|=1, e>0}
|\widetilde{\nabla_e^Nw}(X)|$.
\end{lem}

\begin{proof}
To show \eqref{eq:1-Nnablawdt}, noting from \eqref{eq:poli-A} (with $\nabla_e^Nw$
instead of $u^N$) that $\widetilde{\nabla_e^Nw}$ is a convex combination
of $\nabla_e^Nw$ at neighboring sites, we observe 
\begin{equation}  \label{eq:4.nablaw}
|\widetilde{\nabla_e^Nw}(t,z)| \le
\max_{v\in \{0,1\}^n} \big| \nabla_{e}^Nw\big(t,\tfrac{[Nz]+v}N\big)\big|.
\end{equation}
However, since $\big| z- \tfrac{[Nz]+v}N \big| \le \tfrac{\sqrt{n}}N$, we have
$\tfrac{[Nz]+v}N\in D_N(y,r+\tfrac{\sqrt{n}}N)$
for $z\in D(y,r)$.  Thus, \eqref{eq:4.nablaw} proves \eqref{eq:1-Nnablawdt}.

To show \eqref{eq:2-Nnablawdt},
since $\nabla_e^Nu(X) = \nabla_e^N \tilde u(X)$ for $X=(t,\tfrac{x}{N})$ 
with $x\in \T^n_N$, recalling \eqref{seminorm_defn}  for $U_1$, bound
for each $e$ that
\begin{align*} 
\sup_{X \text{ or } X+\frac{e}N \in Q_N(r)} |\nabla_e^N u(X)|
& \le \sup_{X\in \Om: d(X)\ge (d(Y)^2-r^2)^{1/2}} |\nabla^N_e \tilde u(X)|  \\
& \le U_1 \sup_{d(X)\ge (d(Y)^2-r^2)^{1/2}} d(X)^{-1}.
\end{align*}
However, as we saw in \eqref{eq:d-Y-X} in the proof of Lemma \ref{lem:D-energy-a-Holder},
$d(X)\ge \frac12 d(Y)$ and this shows \eqref{eq:2-Nnablawdt}.

To demonstrate \eqref{eq:3-Nnablawdt},
we divide $|w(X)|$ into two terms when $D_N(r)\equiv D_N(y,r)$ is non-empty, 
as otherwise the estimate holds trivially.  For $X =(t,\tfrac{x}N), 
Z=(t',\tfrac{z}N) \in Q_N(r)$,
$$
|w(X)| \le |u(X)-u(Z)- \nabla^N u(Z)\cdot \tfrac{x-z}N|
+ |v(X)-u(Z)- \nabla^N u (Z)\cdot \tfrac{x-z}N|,
$$ 
where $\cdot$ means the inner product in $\R^n$.
The first term, writing $u=u^N$, is further bounded from above by
\begin{align*}
& |{u}^N(X)-{u}^N(X')| 
  +|{u}^N(X')- {u}^N(Z)- \nabla^N{u}^N(Z)\cdot\tfrac{x-z}{N}| \\
& \le  \lan u^N\ran_{1+\a}^{*,N} (2r)^{1+\a}\big(d(X)\wedge d(X')\big) ^{-(1+\a)} 
+ \tfrac{1}{N} \sum_j |\nabla^N{u}^N(Z_j)- \nabla^N{u}^N(Z)|   \\
& \le 2^{2+2\a} \lan u^N\ran_{1+\a}^{*,N} \big( \tfrac{r}{d(Y)}\big)^{1+\a} 
+ 4^{1+\a} \sqrt{n} \, [\tilde{u}^N]_{1+\a}^* (r+\tfrac{\sqrt{n}}N)^{1+\a}\cdot d(Y)^{-(1+\a)}
\end{align*}
where $X'$ is taken as $X'=(t',\tfrac{x}N)$ for $X=(t,\tfrac{x}N)$ and $Z=(t',\tfrac{z}N)$, and
$\{Z_j =(t',\frac{z_j}{N})\}$ is a sequence of points in $Q_N(r)$ connecting $X'$ and $Z$ 
by moving along in nearest-neighbor steps. The sum over $j$ consists of 
at most $|x-z|_{L^\infty}$ so that $\sqrt{n}|x-z|$ terms.
Note that $d(X)\wedge d(X') \ge \frac12 d(Y)$, $d(Z_j)\wedge d(Z) \ge \frac12 d(Y)$
from \eqref{eq:d-Y-X} and $|Z_j-Z|^\a \le \big(2(r+\tfrac{\sqrt{n}}N)\big)^\a$ for 
$Z_j, Z \in Q_N(r)$ (or even $\in \overline{Q_N(r)}$).

When there are two distinct space points $\tfrac{x}{N}, \tfrac{z}{N} \in D_N(r)$, necessarily $\frac{1}{N}<r$.  When there is only one space point in $D_N(r)$, we have $x=z$ in the above sequence.  Hence, covering both cases and recalling $\mathcal{U}= [\tilde{u}^N]_{1+\a}^* + \lan u^N\ran_{1+\a}^{*,N}$,
we bound the first term, noting $\alpha<1$, as
\begin{align*}
|u(X)-u(Z)- \nabla^N u(Z) \cdot \tfrac{x-z}N|
\le C(n) \mathcal{U} \big( \tfrac{r}{d(Y)}\big)^{1+\a}.
\end{align*}

The second term
$V(X) := v^N(X)-u^N(Z)- \nabla^N u^N(Z)\cdot \tfrac{x-z}N$ on $Q_N(r)$, 
noting the boundary value is the
same as the first term, by applying the maximum principle for
the discrete heat equation \eqref{eq:heat-D}, it has the same bound as the first term.
Therefore, we have
$$
|w| \le 2C(n) \mathcal{U} (\tfrac{r}{d(Y)})^{1+\a},
$$
on $Q_N(r)$.  This shows \eqref{eq:3-Nnablawdt}.
\end{proof}

We now combine the bounds in Lemmas \ref{lem:D-energy-a-Holder},
\ref{lem:estimate-D-energy} and also Proposition \ref{lem:Li-L4.5} to obtain the following estimate.
Recall $Q(\cdot)= Q(Y,\cdot)$.

\begin{lem}  \label{lem:5.3}
For $\rho\in (0,r)$, when $r>0$ satisfies $r_N 
= r+ \frac{1+2\sqrt{n}}{N} < \frac{1}{2}d(Y)$, we have
\begin{align} \label{eq:om(r)}
\om(\rho) \le & \bar{C}\big(\tfrac{\rho}r\big)^{n+4} \om(r)+ \si(r_N),
\end{align}
where
\begin{align}  \label{eq:om+si}
\begin{aligned}
& \om(r) = \int_{Q(r)} |\widetilde{\nabla^N u}-\{\widetilde{\nabla^N u}\}_r|^2dX,\\
& \si(r) = \widehat{C}[ A^2 U_1^2+ G_\infty  \mathcal{U}]  r^{n+2+2\a}d(Y)^{-(2+2\a)}
\end{aligned}
\end{align}
and $\bar{C}=\bar{C}(n, c_\pm)$, $\widehat{C}=\widehat{C}(n, c_\pm, T)$.
\end{lem}

\begin{proof}
Recall that, setting $r_N^1 = r + \tfrac{\sqrt{n}+1}N$ for each $r>0$, 
we consider the solution $v=v^N\equiv v^{N,r}$ of the discrete heat 
equation \eqref{eq:4.heat} on $Q_N(r_N^1)$ with boundary condition $v^N=u^N$
at $\mathcal{P}_N^+Q_N(r_N^1)$ and set $w\equiv w^{N,r}= u^N-v^{N,r}$
on $\overline{Q_N(r_N^1)}$ as in \eqref{eq:w=u-v-B}.

Then, from \eqref{eq:1-Nnablawdt} in Lemma \ref{lem:estimate-D-energy}, then applying
Lemma \ref{lem:D-energy-a-Holder} (with $r$ replaced by
$r+ \tfrac{\sqrt{n}}N$) and \eqref{eq:2-Nnablawdt}, \eqref{eq:3-Nnablawdt} 
in Lemma \ref{lem:estimate-D-energy} (with $r$ replaced by $r+ \tfrac{\sqrt{n}}N$),
we obtain
\begin{align}  \label{eq:nabla^Nw}
\int_{Q(r)}  |\widetilde{\nabla^N w}|^2(X) dX
 \le \bar{C}[ A^2 U_1^2+ G_\infty  \mathcal{U}]  r_N^n (\tfrac{r_N}{d(Y)})^{2+2\a},
\end{align}
for all $r>0$ such that $r_N= r+ \tfrac{1+2\sqrt{n}}N<\frac12 d(Y)$ where $\bar{C}=\bar{C}(n, c_-, T)$.

Moreover, $v$ solves the discrete heat equation \eqref{eq:heat-D} on $Q_N(r_N^1)$.
Thus, from \eqref{eq:Li-L4.5-3} in Proposition \ref{lem:Li-L4.5} (with $R=r_N^1$), we have
\begin{align}\label{eq:nabla^Nv}
\int_{Q(\rho)} |\widetilde{\nabla_e^N v}-\{\widetilde{\nabla_e^N v}\}_{\rho}|^2dX 
 \le \widehat{C}\big(\tfrac{\rho}r\big)^{n+4} 
\int_{Q(r)} |\widetilde{\nabla_e^N v}-\{\widetilde{\nabla_e^N v}\}_r|^2dX,
\end{align}
for $\rho\in (0,r)$ where $\widehat{C}=\widehat{C}(n, c_\pm)$.

Now, let us show the estimate \eqref{eq:om(r)} for $\om(\rho)$ and $\rho\in (0,r)$.
We first rewrite $u$ in the integrand $|\widetilde{\nabla^N u}-
\{\widetilde{\nabla^N u}\}_\rho|^2$ in $\om(\rho)$ as $u=w+v\equiv w^{N,r}+v^{N,r}$
(note that we take $w$ and $v$ those determined from $r$ and not by $\rho$) and estimate it by
$3\big( |\widetilde{\nabla^N w}|^2 + \{\widetilde{\nabla^N w}\}_\rho^2
+ |\widetilde{\nabla^N v}-\{\widetilde{\nabla^N v}\}_\rho|^2 \big)$.  Then, we have
\begin{align*}
\om(\rho) \le&  
3\int_{Q(\rho)}  |\widetilde{\nabla^N w}|^2 dX
+ 3 \{\widetilde{\nabla^N w}\}_{\rho}^2 \cdot |Q(\rho)|
+ 3 \int_{Q(\rho)} |\widetilde{\nabla^N v}-\{\widetilde{\nabla^N v}\}_{\rho}|^2dX 
  \\
\le& 6 \int_{Q(r)}  |\widetilde{\nabla^N w}|^2 dX
+ 3\widehat{C}\big(\tfrac{\rho}r\big)^{n+4} 
\int_{Q(r)} |\widetilde{\nabla^N v}-\{\widetilde{\nabla^N v}\}_r|^2dX \\
\le & 6 \bar{C}[ A^2 U_1^2+ G_\infty \mathcal{U}]  r_N^n (\tfrac{r_N}{d(Y)})^{2+2\a} \\
& \qquad
+ 3\widehat{C}\big(\tfrac{\rho}r\big)^{n+4} \Big[
3 \int_{Q(r)} |\widetilde{\nabla^N w}|^2dX 
+ 3 \{\widetilde{\nabla^N w}\}_r^2 \cdot |Q(r)|
+ 3 \om(r)   \Big] \\
\le & 6 \bar{C}[ A^2 U_1^2+ G_\infty \mathcal{U}]  r_N^n (\tfrac{r_N}{d(Y)})^{2+2\a}  \\
& \qquad 
+ 3\widehat{C}\big(\tfrac{\rho}r\big)^{n+4} \cdot 6 \bar{C}[ A^2 U_1^2+ G_\infty \mathcal{U}]  r_N^n 
(\tfrac{r_N}{d(Y)})^{2+2\a} +  9\widehat{C}\big(\tfrac{\rho}r\big)^{n+4} \om(r).
\end{align*}
Here, for the second inequality, we have estimated $\{\widetilde{\nabla^N w}\}_{\rho}^2
\le \tfrac1{|Q(\rho)|} \int_{Q(\rho)}|\widetilde{\nabla^N w}|^2 dX$ and then simply enlarged the
domain of integral to $Q(r)$ for the first and second terms.  
We also used \eqref{eq:nabla^Nv} for the last term.
For the third inequality, we have used \eqref{eq:nabla^Nw} for the first term and 
a similar estimate to the first inequality for the second term.  For the fourth inequality,
we use \eqref{eq:nabla^Nw} again.
Finally, estimating $\big(\tfrac{\rho}r\big)^{n+4} \le 1$ in the second term, we obtain 
\eqref{eq:om(r)}. 
\end{proof}

We now apply Lemma \ref{lem:Li-L4.6} in our setting.
This iteration lemma improves $\si((R_0)_N)$, in \eqref{eq:om(r)} with $r=R_0$,
to $\si(r_N)$.

\begin{lem}  \label{lem:Li-L4.6-N}
Assume \eqref{eq:om(r)} in Lemma \ref{lem:5.3}.
Then, we have
$$
\om(r) \le C\Big[ \big(\tfrac{r}{R_0}\big)^{n+2+2\a} \om(R_0)+\si(r_N)\Big],
$$
for $0< r \le R_0$, where $R_0 < \tfrac12 d(Y)-\tfrac{1+2\sqrt{n}}N$, $r_N=r+\tfrac{1+2\sqrt{n}}N$ and $C=C(n, c_\pm)$.
\end{lem}

\begin{proof}
We may check that $\si(r_N)$ satisfies the condition
\eqref{eq:Li-sigma} in place of $\si(r)$, that is,
\begin{equation}  \label{eq:Li-sigma-N}
r^{-\de}\si(r_N)\le s^{-\de}\si(s_N) \qquad \text{if } \; 0<s\le r \le R_0,
\end{equation}
where $\de=n+2+2\a$, $\bar \a \in (\de, \b)$, $\b= n+4$ (note $\a\le 1$), 
$\si(r) = \widehat{C}(n, c_\pm,T)r^\de$ from \eqref{eq:om+si} and $r_N=r+\tfrac{c}N$.
Indeed,
$$
r^{-\de}\si(r_N) = r^{-\de}(r+\tfrac{c}N)^\de= (1+\tfrac{c}{Nr})^\de
$$
is decreasing in $r$.  

Thus, the conclusion follows by applying Lemma \ref{lem:Li-L4.6} and the comment after it:  
In our situation,
(4.15)$'$ of \cite{Li96} holds with $\b=n+4$, that is $\omega(\t r) \leq C\t^\beta\omega(r) + \sigma(r_N)$ for all small $\t$ and $C=C(n, c_\pm)$.  Then, \eqref{eq:Li-om} in Lemma \ref{lem:Li-L4.6}
holds for any $\delta<\bar\a<\b$,
choosing $\t\in (0,1)$ such that $C\t^\b<\t^{\bar\a}$.

Note that the increasing property of $\si(r_N)$ is clear, while that of $\om(r)$ follows
by showing $\om'(r)\ge 0$.  Indeed, $\om$ is increasing in its integral domain
$Q(r)$ and another term coming from the derivative of $\{\nabla_e^N \tilde u\}_r$ in
$r$ vanishes: $\int_{Q(r)} \big(\nabla_e^N \tilde u-\{\nabla_e^N \tilde u\}_r
\big)dX \cdot \partial_r \{\nabla_e^N \tilde u\}_r =0$.
\end{proof}

We are at the position to give  the estimate on $U_{1+\alpha}$.
It is important that the coefficient $\de$ of $\mathcal{U}$ in the estimate
\eqref{spatial-variation-equation} can be made arbitrary small.  
This is because our estimate is shown in terms of $\sqrt{\mathcal{U}}$.

\begin{prop}
\label{spatial-variation-prop}
We have
\begin{align}
\label{spatial-variation-equation}
U_{1+\a} = [\tilde u^N]_{1+\a}^*
\le C [(A+1) U_1+ G_\infty ]+ \de \mathcal{U}
\end{align}
for all small $\de>0$ with some $C=C(n, c_\pm, T, \de)$.
\end{prop}

\begin{proof}
We apply Lemma \ref{lem:Li-L4.6-N} to get
\begin{align}
\label{spatial-step}
\om(r) \le & C\big(\tfrac{r}{R_0}\big)^{n+2+2\a} \om(R_0)
+C[A^2U_1^2+ G_\infty \mathcal{U}] r_N^{n+2+2\a} d(Y)^{-(2+2\a)},
\end{align}
for $0<r\le R_0$, taking $R_0= \tfrac13 d(Y)$, if $\tfrac13 d(Y) <
\tfrac12 d(Y)-\tfrac{1+2\sqrt{n}}N$, that is, if $d(Y) > \tfrac{6(1+2\sqrt{n})}N$ holds.
Recall $r_N=r+\tfrac{c}N$ with $c=1+2\sqrt{n}$ (later in the proof we may take $c$ larger).  Here, $C=C(n, c_\pm, T)$.

However, $\om(R_0)$ is bounded as
\begin{align*}
\om(R_0) & \le 2 \int_{Q(R_0)} |\nabla^N \tilde u|^2dX + 2 |Q(R_0)| \,
 \{\nabla^N \tilde u\}_{R_0}^2  \\
& \le 4 \int_{Q(R_0)} |\nabla^N \tilde u|^2dX \\
&\le C(n) \, U_1^2 R_0^{n+2}\, d(Y)^{-2}.
\end{align*}
Here, the first line is from the definition of $\om(R_0)$ in \eqref{eq:om+si}, 
the second is by Schwarz's inequality applied for the second term, and
the third is from \eqref{eq:2-Nnablawdt} taking $r=R_0$.

Thus, first by Schwarz's inequality and then recalling $R_0=\tfrac13 d(Y)$
so that the first term in the right hand side of \eqref{spatial-step} is bounded
by the second term with $A^2$ replaced by $A^2+1$, we have
\begin{align}  \label{eq:4.41-C}
\int_{Q(r)} |\nabla_e^N \tilde u-\{\nabla_e^N \tilde u\}_r|dX
& \le \sqrt{|Q(r)|} \sqrt{\om(r)}\\
&\le C \sqrt{(A^2+1) U_1^2+ G_\infty \mathcal{U}} \, d(Y)^{-(1+\a)} r_N^{n+2+\a} 
  \notag \\
&\le \big\{C'[(A+1) U_1+ G_\infty] + \de \mathcal{U} \big\} \, d(Y)^{-(1+\a)} r_N^{n+2+\a},
  \notag   
\end{align}
for every $\de>0$ with some $C'=C'(n, c_\pm, T,\de)>0$,
for $0<r\le \tfrac13 d(Y)$, if $d(Y) > \tfrac{6(1+2\sqrt{n})}N$.
We have estimated $C\sqrt{G_\infty \mathcal{U}} \le \de \mathcal{U}+ \tfrac{C^2}\de G_\infty$
as usual.

In the case that $d(Y) \le \tfrac{6(1+2\sqrt{n})}N$, we can apply \eqref{eq:4.42-B}
in Lemma \ref{lem:4.13} stated below by making $c>0$ in $r_N$ 
larger if necessary.  Indeed, 
by \eqref{eq:4.42-B}, we can estimate for $0<r\le \tfrac13 d(Y)$
\begin{align}\label{eq:4.41-D}
\int_{Q(r)} |\nabla_e^N \tilde u-\{\nabla_e^N \tilde u\}_r|dX
& \le \int_{Q(r)} \frac{dX}{|Q(r)|} \int_{Q(r)}|\nabla_e^N \tilde u(X)
-\nabla_e^N \tilde u(Z)|dZ \\
&  \le C(n, c_+) r^{n+2}\cdot r^\a \big(U_{1+\a} +U_1A+G_\infty \big) 
d(Y)^{-(1+\a)}.   \notag
\end{align}
Here, note that $|X-Z|\le 2r \le \tfrac{c_1}N$ with $c_1=4(1+2\sqrt{n})$ and also
$d(X), d(Z) \ge \tfrac{\sqrt{8}}3 d(Y)$ for $X, Z\in Q(r)$.
However, from $r\le \tfrac13 d(Y) \le \tfrac{c_2}N$ with $c_2=2(1+2\sqrt{n})$, 
one can obtain
\begin{align}\label{eq:4.41-E}
r^{n+2+\a} \le (\tfrac{c_2}{c_2+c})^{2+n+\a} r_N^{2+n+\a} \le \tfrac{\de}C \, r_N^{2+n+\a},
\end{align}
for every $\de>0$, since $\tfrac{c_2}{c_2+c}$ can be small by making $c$ large enough.

From \eqref{eq:4.41-C}--\eqref{eq:4.41-E}, recalling $U_{1+\a}\le \mathcal{U}$, we have shown
(without any restriction on $d(Y)$)
\begin{align}  \label{eq:4.45-F}
\int_{Q(Y,r)} |\nabla_e^N \tilde u-\{\nabla_e^N \tilde u\}_r|dX
&\le \big\{C[(A+1) U_1+ G_\infty] + \de \mathcal{U} \big\} 
\, d(Y)^{-(1+\a)} r_N^{n+2+\a},
\end{align}
for $0<r\le \tfrac13 d(Y)$ and any $Y=(t_1,y)\in \Om$, where $C=C(n, c_\pm, T, \de)$.

Now we apply Lemma \ref{lem:Li-L4.3-G}.  We fix any $X_0=(t_0,z_0)\in \Om$ and
take $R=\tfrac14d(X_0)=\tfrac14 \sqrt{t_0}$.  Then, \eqref{eq:4.45-F} shows that
the condition \eqref{eq:Li-lem-4.12} in Lemma \ref{lem:Li-L4.3-G} holds for
any $Y\in Q(X_0,R)$, $r\in (0,R)$ with
\begin{align*}
&F=\nabla_e^N \tilde u, \ U_F=U_{1+\alpha},\ 
G(Y,r)= \{\nabla_e^N \tilde u\}_r\\
&\ \ \   {\rm and \ \ } H= \{C[(A+1) U_1+ G_\infty] + \de \mathcal{U} \} 
\, d(Y)^{-(1+\a)}.
\end{align*}
  Indeed, for $Y=(t_1,y)\in Q(X_0,R)$, we have $0<t_0-t_1<R^2=\tfrac1{16}t_0$
so that $\sqrt{t_1}>\tfrac{\sqrt{15}}4 \sqrt{t_0}$ and, combined with $r<R=\tfrac14 \sqrt{t_0}$,
this shows $r<\tfrac14 \tfrac4{\sqrt{15}} \sqrt{t_1} = \tfrac1{\sqrt{15}} \sqrt{t_1}
< \tfrac13 d(Y)$ so that \eqref{eq:4.45-F}  is applicable.  

Therefore, by Lemma \ref{lem:Li-L4.3-G},
we see that, for every $\de>0$ and $M>0$, there exists 
$C=C_{\de,M}(n, c_\pm, T)>0$ such that
\begin{align}  \label{eq:4.41-F}
|\nabla_e^N\tilde u(Y)-\nabla_e^N\tilde u(Y_1)| \le \Big( &
 C[(A+1) U_1+ G_\infty] + \de \mathcal{U}+\de U_{1+\a} \Big) \\
& \qquad \times \big(d(Y)\wedge d(Y_1)\big)^{-(1+\a)} |Y-Y_1|^\a,
\notag
\end{align}
holds if  $Y, Y_1\in Q(X_0,R)$ satisfy $|Y-Y_1| \ge \tfrac1{MN}$.

The case $|Y-Y_1| \le \tfrac1{MN}$ is covered by \eqref{eq:4.43-B}
in Lemma \ref{lem:4.13} stated below, and we obtain \eqref{eq:4.41-F} 
for any $Y, Y_1 \in Q(X_0,R)$.

Thus, moving $X_0$, we have shown \eqref{eq:4.41-F} for any $Y, Y_1\in \Om$.
This implies the concluding estimate \eqref{spatial-variation-equation} on $U_{1+\a}$
by recalling the definition of $U_{1+\a}$ in \eqref{seminorm_defn}, $U_{1+\a}\le
\mathcal{U}$ and by changing $\de>0$ to $\tfrac12\de$.
\end{proof}

\begin{rem}
At this point, we comment that estimation in terms of the solution $v^N$ to the discrete heat equation \eqref{eq:heat-D} was an important device 
to derive the power $n+4$ in the bounds in Lemma \ref{lem:5.3} and the bound 
\eqref{spatial-variation-equation}, leading to the desired $\a$-H\"older continuity
of $\nabla_e^N\tilde u$.
\end{rem}

The following H\"older estimate \eqref{eq:4.43-B} for $\nabla_e^N \tilde u$ in the short distance
regime $|Y-Y_1|\le \tfrac1{MN}$ was used in the proof of Proposition
\ref{spatial-variation-prop}.  This complements the H\"older estimate for
$|Y-Y_1|\ge \tfrac1{MN}$ obtained
by applying Lemma \ref{lem:Li-L4.3-G}.  It is essential that the front factor
(especially that of $U_{1+\a}$) can be taken arbitrary small by choosing $M\ge 1$ large enough.
By the polylinearity, at least in spatial directions, the function is
Lipschitz continuous and therefore, in view of the H\"older estimate, the front factor
can be made small in the short distance regime.

\begin{lem}  \label{lem:4.13}
{\rm (1)} If  $|Y-Y_1| \le \tfrac{c_1}{N}$, we have
\begin{align}  \label{eq:4.42-B}
|\nabla_e^N \tilde u(Y)-\nabla_e^N \tilde u(Y_1)| \!
\le C \big(U_{1+\a} +U_1A+G_\infty \big) \!
\big(d(Y)\wedge d(Y_1)\big)^{-(1+\a)}  |Y-Y_1|^\a,
\end{align}
for some $C=C(n, c_+,c_1)>0$.\\
{\rm (2)}  Furthermore, for every $\de>0$, there exists $M=M(n,c_+, \alpha)\ge 1$ such that
\begin{align}  \label{eq:4.43-B}
|\nabla_e^N \tilde u(Y)-\nabla_e^N \tilde u(Y_1)| 
\le \de \big(U_{1+\a} +U_1A+G_\infty \big) 
\big(d(Y)\wedge d(Y_1)\big)^{-(1+\a)}  |Y-Y_1|^\a,
\end{align}
holds if  $|Y-Y_1| \le \tfrac1{MN}$.
\end{lem}

\begin{proof}
We now show \eqref{eq:4.43-B}.  For \eqref{eq:4.42-B}, we may take $\tfrac1M=c_1$ in the proof.

{\it Step 1}. (spatial direction)
Let $Y=(t,z^{(1)})$ and $Y_1=(t,z^{(2)})$ have the same $t$ coordinate.  
We will assume first that $z^{(1)} = (z_i^{(1)})_{i=1}^n$ and
$z^{(2)} = (z_i^{(2)})_{i=1}^n$ belong to the same $\frac1N$-box, so that
$\tfrac{[Nz^{(1)}]}N = \tfrac{[Nz^{(2)}]}N =: \bar z =(\bar z_i)_{i=1}^n\in \tfrac1N\T_N^n$ are common,
and also that $z_j^{(1)} = z_j^{(2)} =: z_j$ except $j=i$.  Recall 
$\nabla_e^N \tilde u(Y)= \widetilde{\nabla_e^N u}(Y)$ as in \eqref{eq:tilde-grad} so that 
$$
\nabla_e^N \tilde u(z) = \sum_{v\in \{0,1\}^n} \vartheta^N(v,z) 
\nabla_e^N u\big(\bar z + \tfrac{v}N\big),
$$
from \eqref{eq:poli-A}, where we have now only specified the spatial variable.  Then,
\begin{align*}
& \nabla_e^N \tilde u(z^{(1)}) - \nabla_e^N \tilde u(z^{(2)}) \\
& \qquad  = \sum_{v\in \{0,1\}^n} \big(\vartheta^N(v_i,z_i^{(1)}) - \vartheta^N(v_i,z_i^{(2)})\big)
\prod_{j\not=i} \vartheta^N(v_j,z_j)
\nabla_e^N  u\big(\bar z + \tfrac{v}N\big).
\end{align*}
However,
\begin{align*}
& \sum_{v_i=0,1} \vartheta^N(v_i,z_i^{(1)}) \nabla_e^N u\big(\bar z + \tfrac{v}N\big) \\
& \qquad = N (z_i^{(1)}-\bar z_i) \nabla_e^N u\big(\bar z + \tfrac{e_i+\hat v_i}N\big)
+ (1-N (z_i^{(1)}-\bar z_i)) \nabla_e^N u\big(\bar z+ \tfrac{\hat v_i}N \big),
\end{align*}
where $\hat v_i$ is defined below \eqref{eq:deri-u-F} from $v$.
Therefore, noting $|\vartheta^N(v_j,z_j)|\le 1$, we have
\begin{align*}
& \big|\nabla_e^N \tilde u(z^{(1)}) - \nabla_e^N \tilde u(z^{(2)})\big| \\
& \qquad  \le \sum_{v\in \{0,1\}^{n-1}} \Big| N (z_i^{(1)}-z_i^{(2)}) 
\nabla_e^N u\big(\bar z + \tfrac{e_i+\hat v_i}N\big)
- N (z_i^{(1)}-z_i^{(2)}) \nabla_e^N u\big(\bar z+\tfrac{\hat v_i}N\big) \Big| \\
& \qquad  =
 \sum_{v\in \{0,1\}^{n-1}}  N |z^{(1)}-z^{(2)}| \Big| \nabla_e^N u\big(\bar z + \tfrac{e_i+\hat v_i}N\big)
- \nabla_e^N u\big(\bar z+\tfrac{\hat v_i}N\big) \Big|  \\
& \qquad \le 2^{n-1} N |z^{(1)}-z^{(2)}| \cdot U_{1+\a} (\tfrac1N)^\a d(Y)^{-(1+\a)}.
\end{align*}
Here, recalling $|Y-Y_1| \le \tfrac1{MN}$, we estimate
$$
|z^{(1)}-z^{(2)}|= |z^{(1)}-z^{(2)}|^\a \cdot |z^{(1)}-z^{(2)}|^{1-\a}
\le |z^{(1)}-z^{(2)}|^\a (\tfrac1{MN})^{1-\a},
$$
and obtain
\begin{align}  \label{eq:L.4.11-1}
\big| \nabla_e^N \tilde u(z^{(1)}) - \nabla_e^N \tilde u(z^{(2)}) \big|
\le |z^{(1)}-z^{(2)}|^\a (\tfrac1{M})^{1-\a} \cdot 2^{n-1} 
U_{1+\a} d(Y)^{-(1+\a)}.
\end{align}

When $z^{(1)}$ and $z^{(2)}$ belong to different $\tfrac1N$-boxes, we consider
the segment connecting these two points, divided into pieces belonging to
the same $\tfrac1N$-boxes and apply the above result.
Thus, the desired  H\"older estimate \eqref{eq:4.43-B} in the spatial direction is shown
by taking $M$ large, as $1-\alpha>0$.

\vskip 3mm
{\it Step 2}. (temporal direction) 
Let $Y=(t,z)$ and $Y_1=(s,z)$; we may assume $z=\tfrac{x}N \in \tfrac1N \T_N^n$.
By the equation \eqref{eq:1.2-linear} or equivalently
\eqref{general-divergence}, for $X=(t,z)$, we have
\begin{align*}
\partial_t \nabla_e^N u(X)
=  - \tfrac12\sum_{e':|e'|=1} \nabla_e^N \nabla^{N,*}_{e'} \big(a_{x,e'}(t)\nabla^N_{e'} u \big)(X)
+\nabla_e^Ng(X).
\end{align*}
We apply the following rough estimates to the right hand side:
By \eqref{eq:disc-der-D} applied for $\nabla_{e'}^{N,*}$ and (A.1)--(A.3),
\begin{align*}
|\nabla_e^N \nabla^{N,*}_{e'} \big(a_{x,e'}(t)\nabla^N_{e'} u \big)(X)|
& \le 2N^2 \Big\{ c_+ \sup_{X'}|\nabla^N_{e'} u(X'-\tfrac{e'}N) - \nabla^N_{e'} u(X')|  \\
& \hskip 16mm + U_1 d(X)^{-1} \sup_{X'=(t,x')}|a_{x'-e',e'}(t)-a_{x',e'}(t)| \Big\} \\
& \le 2N^2 \Big\{ c_+ U_{1+\a} + U_1 A\Big\} (\tfrac1N)^{\a} d(X)^{-(1+\a)},
\end{align*}
and $|\nabla_e^Ng(X)| \le 2 N G_\infty$.  Therefore, we have
\begin{align*}
|\nabla_e^N u(t,z)- \nabla_e^N u(s,z)|
\le & 2n N^{2-\a} |t-s| \Big\{ c_+ U_{1+\a} + U_1 A\Big\} \big(d(Y)\wedge d(Y_1)\big)^{-(1+\a)}  \\
&+ 2N|t-s| G_\infty.
\end{align*}
Here, we estimate
$$
|t-s|=|t-s|^{\frac{\a}2} |t-s|^{1-\frac{a}2}
\le |t-s|^{\frac{\a}2} (\tfrac1{MN})^{2-\a},
$$
by noting $|t-s|\le (\tfrac1{MN})^2$ from $|Y-Y_1|\le \tfrac1{MN}$. 
Thus, we obtain
\begin{align*}
& |\nabla_e^N u(t,z)- \nabla_e^N u(s,z)| \\
& \qquad \le |t-s|^{\frac{\a}2} (\tfrac1{M})^{2-\a} \Big[ 2n \Big\{ c_+ U_{1+\a} + U_1 A\Big\} 
\big(d(Y)\wedge d(Y_1)\big)^{-(1+\a)} + 2G_\infty\Big].
\end{align*}
The desired  H\"older estimate \eqref{eq:4.43-B} in the temporal direction is shown by taking $M$ large, as $2-\alpha>0$.

The proof of the lemma is completed by combining the results obtained
in Steps 1 and 2.
\end{proof}

\subsection{Estimate on $\lan u^N\ran_{1+\alpha}^{*,N}$}
\label{sec:4.5}

The purpose of this subsection is to provide an estimate on $\lan u^N\ran_{1+\a}^{*,N}$ used in the proof of Theorem \ref{thm:Li-T4.8}.  The argument, although different in our discrete setting, is inspired by that of Theorem 4.8 (p. 58) of \cite{Li96}.  Recall \eqref{seminorm_defn}  for $U_1, U_{1+\a}$
and the assumptions (A.1)--(A.3) for $a(t)$ and $g(t)$, especially for the constants $A$ and $G_\infty$.

\begin{prop}
\label{temporal-prop}
We have
$$
\langle u^N\rangle^{*,N}_{1+\alpha} \leq C[AU_1 + U_{1+\alpha} + G_\infty ]
$$
where $C=C(n, c_+, T, \alpha)$.
\end{prop}

\begin{proof}
{\it Step 1.}  
Let $X=(s,y/N)$ and $Y=(t,y/N)$ be given with  $y\in \T^n_N$ and $0\le t< s \le T$ 
such that $s-t = r^2$, $0<r<\frac{1}{2}d(Y)$.  Let
 $$
 S_N(r,y) = \big\{x\in \T^n_N: |\tfrac{x}{N} - \tfrac{y}{N}|_{L^\infty} \le  r \big\}
 $$ 
 be a square with center $\tfrac{y}{N}$ with width $r$. Note that $|S_N(r,y)| = (2[ Nr] +1)^n$,
 recall that $[Nr]$ denotes the integer part of $Nr$. 
 Define, for $y\in \T^n_N$ and $t\in[0,T]$,
$$
U(t) = \frac{1}{|S_N(r,y)|}\sum_{x\in S_N(r, y)} u^N(t, \tfrac{x}{N}),
$$
and divide
\begin{align}  \label{eq:4.39}
| u^N(Y) - u^N(X)| \leq |u^N(t,\tfrac{y}{N}) - U(t)| + |U(t)-U(s)| + |U(s) -  u^N(s,\tfrac{y}{N})|.
\end{align}
We will develop bounds for each of these terms.

\vskip .1cm
{\it Step 2.}  First, we  rewrite the second term of \eqref{eq:4.39} as
\begin{align}
\label{step2.0}
& U(s) - U(t) = \frac{1}{|S_N(r,y)|} \sum_{x\in S_N(r,y)} \int_t^s  \partial_q u^N(q,\tfrac{x}{N})  dq \\
& = \frac{1}{|S_N(r,y)|}\sum_{x\in S_N(r,y)} \int_t^s \Big\{-\sum_{|e|=1, e>0} 
\nabla^{N,*}_e\big(a_{x,e}(q)\nabla^N_eu^N\big)\big(q,\tfrac{x}{N}\big) + g\big(q,\tfrac{x}{N}\big)\Big\} dq 
\notag \\
&= \frac{N}{|S_N(r,y)|}\sum_{|e|=1,e>0} \int_t^s \Big\{\sum_{x\in S_N(r,y)} 
a_{x,e}(q)\nabla^N_eu^N\big(q,\tfrac{x}{N}\big)
\notag \\
&\hskip 50mm  
- \sum_{x\in S_N(r,y-e)} 
a_{x,e}(q)\nabla^N_eu^N\big(q,\tfrac{x}{N}\big)\Big\}dq
\notag \\
&\ \ + \frac{1}{|S_N(r,y)|} \sum_{x\in S_N(r,y)} \int_t^s g\big(q,\tfrac{x}{N}\big)  dq
\notag \\
&=: \sum_{|e|=1,e>0} M^1_e + M^2,  \notag
\end{align}
where we use the equation \eqref{general-divergence} for the second line
and then recall $\nabla_e^{N,*}=
\nabla_{-e}^N$ for the third line.

One can write $a_{x,e}(q) = a_{y,e}(q) + \big(a_{x,e}(q) - a_{y,e}(q)\big)$ and insert so that 
\begin{align*}
M^1_e
&=\frac{N}{|S_N(r,y)|}\int_t^s\Big\{ \sum_{x\in S_N(r,y)}  \big(a_{x,e}(q) - a_{y,e}(q)\big)
\nabla^N_eu^N\big(q,\tfrac{x}{N}\big) \\
&\hskip 35mm
- \sum_{x\in S_N(r,y-e)} \big(a_{x,e}(q) - a_{y,e}(q)\big)\nabla^N_eu^N\big(q,\tfrac{x}{N}\big)\Big\} dq\\
&\hskip 7mm
+ \frac{N}{|S_N(r,y)|} \int_t^s a_{y,e}(q)
\Big\{\sum_{x\in S_N(r,y)}  \nabla^N_eu^N\big(q,\tfrac{x}{N}\big) 
 - \sum_{x\in S_N(r,y-e)}  \nabla^N_eu^N\big(q,\tfrac{x}{N}\big)\Big\}dq\\
&=: J_{1,e}+J_{2,e}.
\end{align*}
Further, as the difference of sums of $\nabla^N_e u^N\big( q,\tfrac{y}{N}\big)$, which does not depend on $x$, over $x\in S_N(r,y-e)$ and $x\in S_N(r,y)$ vanishes, we may rewrite $J_{2,e}$ as
\begin{align*}
J_{2,e} &= \frac{N}{|S_N(r,y)|} \int_t^s a_{y,e}(q)\Big\{\sum_{x\in S_N(r,y)} 
\big[\nabla^N_eu^N\big(q,\tfrac{x}{N}\big) - \nabla^N_e u^N\big( q,\tfrac{y}{N}\big)\big]\\
& \hskip 38mm - \sum_{x\in S_N(r,y-e)}  \big[\nabla^N_eu^N\big(q,\tfrac{x}{N}\big) - \nabla^N_e u^N\big( q,\tfrac{y}{N}\big)\big]\Big\}dq.
\end{align*}

\vskip .1cm
{\it Step 3.}
Note that the size of the symmetric difference satisfies
$$|S_N(r,y)\Delta S_N(r,y-e)| = 2\big(2[ Nr] +1)^{n-1}.$$
We now see that, recalling $s-t=r^2$,
\begin{align*}
|J_{1,e}| &\leq \tfrac{2}{2 [ Nr] +1} AU_1 N r^2 \big(\sqrt{n}(r+\tfrac{1}{N})\big)^\alpha d(Y)^{-(1+\alpha)},
\end{align*}
since, for $x\in S_N(r,y-e)\cup S_N(r,y)$, we have 
$|x-y|_{L^\infty} \leq r + \frac{1}{N}$ so that $|a_{x,e}(q) - a_{y,e}(q)| 
\leq A \big(\sqrt{n}(r+\frac{1}{N})\big)^\alpha d(Y)^{-\a}$ noting that $d((q,\tfrac{x}N))\ge d(Y)$ for
$q\in [t,s]$, and also
$$
|\nabla^N_eu^N\big(q,\tfrac{x}{N}\big)| \le |\nabla^N_eu^N|_0^{(1),N} d(Y)^{-1}
=U_1 d(Y)^{-1},
$$
by \eqref{eq:2-norm-1} and \eqref{eq:tilde-grad} in Lemma \ref{lem:Eq-norms}.

On the other hand, for $J_{2,e}$, since $0<a_{y,e}(q)\le c_+$ and also 
$$
\big|\nabla^N_eu^N\big(q,\tfrac{x}{N}\big) - \nabla^N_e u^N\big( q,\tfrac{y}{N}\big)\big|
\le U_{1+\a} d(Y)^{-(1+\a)} |\tfrac{x}{N}-\tfrac{y}{N}|^\a,
$$
by \eqref{eq:2-norm-2} in Lemma \ref{lem:Eq-norms},
we observe that
\begin{align*}
|J_{2,e}| &\leq  \tfrac{2}{2 [ Nr] +1} c_+
U_{1+\alpha} N r^2 \big(\sqrt{n}(r+\tfrac{1}{N})\big)^\a d(Y)^{-(1+\a)}.
\end{align*}

Hence, if $r>\frac{1}{N}$, 
\begin{align*}
|M_e^1| = |J_{1,e}+J_{2,e}|
&\leq 2 \big(AU_1+ c_+U_{1+\alpha}\big) \frac{Nr^2 d(Y)^{-(1+\alpha)}}{[ Nr] +1} 
(2\sqrt{n} r)^\a.
\end{align*}
Since $[ Nr] +1\geq Nr$, further
\begin{align}
\label{M_2^1 estimate}
|M_e^1| \leq  2 (2\sqrt{n})^\a\big(AU_1+ c_+ U_{1+\alpha}\big) r^{1+\alpha} d(Y)^{-(1+\alpha)}.
\end{align}

However, if $r\leq \frac{1}{N}$,
as $[ Nr] +1\geq 1$ and $0<\alpha<1$, then
$$\frac{N \big(\sqrt{n}(r+\frac{1}{N})\big)^\alpha r^2}{[ Nr] +1} 
\leq  (2\sqrt{n})^\a N^{1-\alpha}r^2 \leq (2\sqrt{n})^\a r^{1+\alpha}.$$
Hence, in this case also we recover the same estimate on $M_e^1$ in \eqref{M_2^1 estimate}.

Moreover, as $\a<1$, $d(Y)\leq \sqrt{T}$ and $r^2\le T$
so that $r^{1-\a}\le \sqrt{T}^{1-\a}$, we bound the term $M^2$ by
\begin{align}\label{second_temporal_help}
|M^2| & \le \frac{1}{|S_N(r,y)|} \sum_{x\in S_N(r,y)} \Big| \int_t^s g\big(q,\tfrac{x}{N}\big)  dq \Big| \\
&\leq  G_\infty r^2  \  \leq \ G_\infty r^{1+\a} d(Y)^{-(1+\a)} T^{1+\a}.
\nonumber
\end{align}

Hence, from \eqref{step2.0}, \eqref{M_2^1 estimate} and \eqref{second_temporal_help},
we obtain the bound
\begin{align}  \label{eq:4.43}
|U(t)-U(s)| \leq C(AU_1 + c_+U_{1+\alpha} + G_\infty) \big(\tfrac{r}{d(Y)}\big)^{1+\a}
\end{align}
where $C=C(n, T, \alpha)$.

\vskip 1mm
{\it Step 4.} 
For the first and third terms of \eqref{eq:4.39},
noting that $\sum_{x\in S_N(r,y)} (\tfrac{y}{N}-\tfrac{x}{N}) =0$, write
\begin{align*}
&|u^N(t,\tfrac{y}{N}) - U(t)| \\
&\leq 
\big| \frac{1}{|S_N(r,y)|} 
\sum_{x\in S_N(r,y)} \big| u^N(t,\tfrac{y}{N})- u^N(t,
\tfrac{x}{N}) - \nabla^N u^N(t,\tfrac{y}{N})\cdot (\tfrac{y}{N}-\tfrac{x}{N}) \big|.
\end{align*}
When $r<\tfrac1N$, as there is a single point in $S_N(r,y)$, the above display vanishes.  

When $r\ge\tfrac1N$ and there are multiple points in $S_N(r,y)$, as in the proof 
of \eqref{eq:3-Nnablawdt} in Lemma \ref{lem:estimate-D-energy}, consider a path 
$\{(t, \frac{z_j}{N})\}\subset \{t\}\times \frac{1}{N}S_N(r,y)$ which moves from 
$(t, \tfrac{y}{N})$ to $(t,\tfrac{x}{N})$ in nearest-neighbor steps in $\{t\}\times S_N(r,y)$
with at most $\sqrt{n}|x-y|$ terms.
Then, 
$$|u^N(t,\tfrac{y}{N})- u^N(t,\tfrac{x}{N}) - \nabla^N u^N(t, \tfrac{y}{N})
\cdot  (\tfrac{y}{N}-\tfrac{x}{N})| \leq \tfrac{1}{N}\sum_j 
|\nabla^N u^N(t, \tfrac{z_j}{N}) - \nabla^N u^N(t, \tfrac{y}{N})|$$
and, as $\tfrac1N |x-y| \leq r$, we have
by \eqref{eq:2-norm-2} in Lemma \ref{lem:Eq-norms} that
\begin{align}  \label{eq:4.44}
|u^N(t,\tfrac{y}{N}) - U(t)| \leq C(n)[\tilde u^N]^*_{1+\a} r^\a \cdot r d(Y)^{-(1+\a)} 
= C(n)U_{1+\a}
 (\tfrac{r}{d(Y)})^{1+\a}.
\end{align}

\vskip .1cm

{\it Step 5.}  Finally, combining \eqref{eq:4.43} and \eqref{eq:4.44}, inputting back into \eqref{eq:4.39}, 
we have the concluding estimate on $\langle u^N \rangle_{1+\alpha}^{*,N}$ noting that
$d(X)\wedge d(Y)=d(Y)$.
\end{proof}

\section{Schauder estimate for the second discrete derivatives}
\label{sec:Schauder-second}

The goal of this section is to derive uniform $L^\infty$ bounds and H\"older estimates
for the second discrete derivatives $\nabla_{e_1}^N\nabla_{e_2}^N u^N(t,\tfrac{x}N)$
of the solution of the equation \eqref{eq:2.1-X}.
To derive such estimates, as in the proof of Theorem 4.9 of \cite{Li96},
it is natural to consider the equation
for $\nabla_e^N u^N(t,\tfrac{x}N)$ and repeat the same argument for proving
Theorem \ref{thm:Li-T4.8} for this equation.  However, it turns out
to be more convenient in our situation to consider $\V^N$, instead of $u^N$,
defined by the nonlinear transformation
\begin{equation}  \label{eq:vN-0}
\V^N(t,\tfrac{x}N) :=\fa(u^N(t,\tfrac{x}N))
\end{equation}
and the equation satisfied by its discrete derivatives; see \eqref{eq:we}
below.

In Section \ref{sec:5.1}, we study the system of linear discrete PDEs \eqref{eq:we},
which is obtained as above and has a different form from \eqref{eq:1.2-linear}, 
but as we will see, a similar application of the energy inequality works well for this equation.
In Theorem \ref{second_Schauder_thm}, we give Schauder bounds for 
\eqref{eq:we} which suit our applications; see Remark \ref{alternative_rmk} for
an alternative.  In Section \ref{sec:5.2}, we return to the original setting and
formulate corresponding Schauder estimates there; see Corollary \ref{second_Schauder_cor}.
When the initial value is smooth enough, we also give in Corollary 
\ref{second_Schauder_time_cor} an improved regularity estimate.

\subsection{Schauder bounds for the associated linear discrete PDE \eqref{eq:we}}
\label{sec:5.1}

From \eqref{eq:2.1-X}, $\V^N$ defined by \eqref{eq:vN-0} satisfies
\begin{equation}  \label{eq:vN}
\partial_t \V^N = \bar a \partial_t u^N
= \bar a \{ \De^N \V^N + Kf(u^N)\},
\end{equation}
where the last term can also be written as $Kf(\fa^{-1}(\V^N))$ 
in terms of $\V^N$ and
\begin{equation}  \label{eq:5.3-a}
\bar a \equiv \bar a(t, \tfrac{x}N) := \fa'(u^N (t, \tfrac{x}N)).
\end{equation}

Note that the coefficient $\bar a$ is a site function, while $a_{x,e}(u)$ in
\eqref{eq:a_xe}, which appears in \eqref{eq:2.1-X} by \eqref{eq:1-L_a}
and \eqref{eq:De-Lta}, depends also on the edges or directions $e$.
In the continuous setting, these two functions are the same.

For $e\in \Z^n: |e|=1, e>0$, consider the discrete derivative of
$\V^N$ in the direction $e$:
$$
\W_e\equiv \W_e^N(t, \tfrac{x}N) := \nabla_e^N \V^N(t, \tfrac{x}N).
$$
By acting $\nabla_e^N$ on the equation \eqref{eq:vN} and
recalling \eqref{eq:Lap-N} for $\De^N$, $\{\W_e\}_e$ satisfies the system of equations
\begin{align}  \label{eq:we}
\partial_t \W_e
=  -\nabla_e^N \Big\{ \bar a(X)\sum_{|e'|=1, e'>0}
\nabla_{e'}^{N,*}  \W_{e'} \Big\}
+ \nabla_e^N g,
\end{align}
where $X=(t,\tfrac{x}N)$ and
\begin{align}  \label{eq:5.5-g}
g(X)\equiv g ^N(t, \tfrac{x}N):=K \bar a  (t, \tfrac{x}N) f(u^N (t, \tfrac{x}N)).
\end{align}
One may view the system \eqref{eq:we} as a perturbation of a closed `diagonal' 
or  `decoupled' system over $\{e\}$:  Noting \eqref{eq:disc-der-D}, 
and \eqref{eq:Lap-N} and \eqref{eq:2.4} to write the term with $\Delta^N$, we have
\begin{align}
\label{modified eq:we}
\partial_t \W_e &= \bar{a}(X)\Delta^N \W^N_e -\nabla^N_e\bar{a}(X) \sum_{|e'|=1, e'>0} \nabla^{N,*}_{e'} \W_{e'}(X+\tfrac{e}N) + \nabla^N_e g,
\end{align}
where $X+\frac{e}N= (t, \frac{x+e}N)$ as in Lemma \ref{lem:D-energy-a-Holder}.

In the following, apart from \eqref{eq:vN} and similarly to Section \ref{sec:4.1},
we will study the equation \eqref{eq:we}, where $\W_e$ is in the form of a gradient of 
abstract function $\psi$, that is $\W_e= \nabla^N_e \psi$, with functions $\bar a$ and $g$ satisfying
the following abstract assumptions:

\begin{itemize}
\item[(B.1)] (nondegeneracy, boundedness) \quad $c_- \le \bar a \le c_+$,
\item[(B.2)] (H\"older continuity)\quad  $[\bar a]^{*,N}_\alpha \le \bar A<\infty$, $\a\in (0,1)$,
\item[(B.3)] (boundedness of $\nabla^N g$)\quad  
$|\nabla^N g|_0^{(1),N} \equiv \max\limits_{e:|e|=1}|\nabla_e^N g|_0^{(1),N}\le G_0^1<\infty$.
\end{itemize}

We always assume the continuity of $g(t)$ in $t$.
Note that $\bar a$ defined by \eqref{eq:5.3-a} from the solution $u^N$ of
\eqref{eq:2.1-X} with \eqref{eq:1.u_pm} satisfies the conditions (B.1) and
(B.2) with $\a=\si$ noting \eqref{eq:cor2.3-1} in Corollary \ref{cor:2.3} and 
$\fa\in C^2([u_-,u_+])$ is strictly increasing.  Also, we will observe that $g$ satisfying \eqref{eq:5.5-g} satisfies (B.3) as a consequence of the Schauder estimate proved for $u^N$ in Theorem \ref{thm:Li-T4.8}.  See Section \ref{sec:5.2} for these computations.

One might think that an $L^\infty$ bound on $\nabla^Ng$ could be substituted instead of (B.3).  However, without further regularity, if one only assumes an $L^\infty$ bound on the initial data $u^N(0,\cdot)$, then (B.3) is natural in view of the bound of $\nabla^N u^N$ in Theorem \ref{thm:Li-T4.8}.

To begin, we need to rewrite Lemma  \ref{lem:D-energy-a-Holder}
to adjust in the present setting; see Lemma \ref{lem:discreteEnergy-We}
below.  Take $Y=(t_1,y) \in \Om=[0,T]\times \T^n$ and $r: 0<r<\frac12 
d(Y)=\frac12 t_1^{\frac12}$.  Recall that 
$Q_N(Y,r) = (t_1-r^2,t_1)\times D_N(y,r)$ with the discrete interior
$D_N(y,r)$ of the minimal cover of $D(y,r)$ by $\frac1N$-boxes, and
$\mathcal{P}_N^+Q_N(Y,r)$ is the parabolic outer boundary of $Q_N(Y,r)$.

Take the closest point $\tilde y \in \tfrac1N\T_N^n$ to $y$, in particular,
$|y-\tilde y|\le \tfrac{\sqrt{n}}{2N}$ and set $\tilde Y:=(t_1,\tilde y) \in \Om_N$.
For each $e:|e|=1, e>0$, let $\zeta_e=\zeta_e^N = \zeta_e^{N,Q(r)}(t, \tfrac{x}N)$ be
the solution of the discrete heat equation (with direction-independent coefficient)
on $Q_N(r) = Q_N(Y,r)$:
\begin{align}  \label{eq:zetae}
\partial_t \zeta_e(t, \tfrac{x}N)= \bar a(\tilde Y) \De^N \zeta_e(t, \tfrac{x}N),
\end{align}
satisfying $\zeta_e = \W_e$ at $\mathcal{P}_N^+Q_N(r)$, and set 
$$
W_e \equiv W_e^N:= \W_e-\zeta_e \quad \text{ on } \; \overline{Q_N(r)}.
$$
Recall \eqref{eq:4.24-A} and \eqref{eq:4.25-A} for $\mathcal{P}_N^+Q_N(r)$
and $\overline{Q_N(r)}$, respectively.  Note, as before, by \eqref{eq:Lap-N} and \eqref{eq:2.4}, that
\begin{align}
\label{w_sub_eq}
&-\nabla^N_e \Big\{ \bar a(\tilde Y) \sum_{|e'|=1, e'>0} \nabla^{N,*}_{e'}\W^N_{e'} \Big\}
= - \nabla^N_e \Big\{\bar a(\tilde Y) \sum_{|e'|=1, e'>0} \nabla^{N,*}_{e'}\nabla^{N}_{e'} \V^N \Big\}\\
&\ \ \ \ \ \ \ \ \ \ \ \ \ \ \ \ =  \bar a(\tilde Y) \Delta^N \nabla^N_e \V^N = \bar a(\tilde Y)\Delta^N \W^N_e.
\notag
\end{align}

We have the following discrete energy inequality for $W_e$.

\begin{lem} \label{lem:discreteEnergy-We}
Assume  $0<r<\frac12 d(Y)=\frac12 t_1^{\frac12}$ and $Q_N(r)\not=\emptyset$.
Then, for each $e$, we have
\begin{align}  \label{eq:Nnablawdt-B}
\int_{t_1-r^2}^{t_1} & N^{-n}\sum_{\tfrac{x}N \text{ or }\frac{x+e'}N\in 
D_N(y,r); |e'|=1, e'>0} |\nabla_{e'}^N W_e|^2(X) dt \\
& \le C \bar A^2 (\tfrac{r+\tfrac{1+\sqrt{n}}N}{d(Y)} )^{2\a} (r+\tfrac1N)^n r^2
\sup_{X \text{ or } X+\frac{e}N\in Q_N(r); |e'|=1, e'>0} 
|\nabla_{e'}^N \W_{e'}^N|^2(X) 
\notag \\
& \quad +  C G_0^1(r+\tfrac{1}{N})^n r^2 d(Y)^{-1}
\sup_{X\in Q_N(r)}|W_e(X)|,
\notag
\end{align}
where $X=(t,\tfrac{x}N)$ and $X+\frac{e}N = (t, \frac{x+e}N)$.  
Here, the constants $C=C(n, c_-)$.
\end{lem}

\begin{proof}
The proof is similar to that of Lemma \ref{lem:D-energy-a-Holder}.
Rewriting as $\bar a(X)=\bar a(\tilde{Y})+ (\bar a(X)-\bar a(\tilde{Y}))$ in \eqref{eq:we}, noting
\eqref{w_sub_eq},
and from \eqref{eq:zetae},
we have
\begin{align*}
\partial_t W_e(t, \tfrac{x}N)
= & \bar a(\tilde{Y})\De^N W_e(t, \tfrac{x}N)  \\
& -\nabla_e^N \Big\{ (\bar a(X)- \bar a(\tilde{Y})) \sum_{|e'|=1, e'>0}
\nabla_{e'}^{N,*}  \W_{e'} \Big\} (t, \tfrac{x}N)
+ \nabla_e^N g (t, \tfrac{x}N).
\end{align*}
Noting that $W_e(t,\cdot)=0$ at 
the boundary $\partial_N^+D_N(y,r)$, by \eqref{eq:SbP-2}
in Lemma \ref{lem:DG-Q-2}, we have
\begin{align*}
\tfrac12\partial_t & \bigg(N^{-n}\sum_{\tfrac{x}N\in D_N(y,r)} W_e^2(X)\bigg) \\
& = - N^{-n} \bar a (\tilde{Y})  \sum_{\tfrac{x}N \in D_N(y,r)} 
W_e(X) \sum_{|e'|=1, e'>0} \big(\nabla_{e'}^{N,*} \nabla_{e'}^N W_e\big)(X) \\
& \qquad - N^{-n}\sum_{\tfrac{x}N \in D_N(y,r)} 
W_e(X) \nabla_{e}^{N}\Big\{ (\bar a(X)-\bar a(\tilde{Y}))
\sum_{|e'|=1, e'>0} \nabla_{e'}^{N,*} \W_{e'}^N\Big\} (X)  \\
&\qquad + N^{-n}\sum_{\tfrac{x}N\in D_N(y,r)} W_e(X) \nabla_e^N g(X) \\
&= - N^{-n} \bar a(\tilde{Y}) \sum_* |\nabla_{e'}^N W_e|^2(X)\\
& \qquad - N^{-n}\sum_{**} (\bar a(X)- \bar a(\tilde{Y})) \nabla_{e}^{N,*} W_e(X)
 \sum_{|e'|=1, e'>0}\nabla_{e'}^{N, *} \W_{e'}^N(X)  \\
 &\qquad + N^{-n}\sum_{\tfrac{x}N\in D_N(y,r)} W_e(X) \nabla_e^N g(X)\\
& =: - I_1+I_2+I_3,
\end{align*}
where $X=(t,\tfrac{x}N)$,  $\sum\limits_*$ means the sum over $(x,e')$
such that $\tfrac{x}N$ or $\frac{x+e'}N\in D_N(y,r)$
and $|e'|=1, e'>0$, while $\sum\limits_{**}$ means the sum over $x$
such that $\tfrac{x}N$ or $\frac{x-e}N\in D_N(y,r)$.
Note that $I_1\ge 0$.

Integrate both sides in $t\in [t_1-r^2,t_1]$.  Since $W_e(t_1-r^2, \cdot)=0$,
the left hand side is
$$
\tfrac12 N^{-n}\sum_{\tfrac{x}N\in D_N(y,r)} W_e^2(t_1, \tfrac{x}N)\ge 0.
$$
On the other hand, for every $\e>0$,
\begin{align*}
\int_{t_1-r^2}^{t_1} |I_2|dt 
& \le \e \int_{t_1-r^2}^{t_1}  N^{-n}\sum_{**} |\nabla_{e}^{N,*} W_e|^2(X) dt \\
& \qquad + \frac{n}\e \int_{t_1-r^2}^{t_1}  N^{-n}
\sum_{**} |\bar a(X)- \bar a(\tilde{Y})|^2 \sum_{|e'|=1, e'>0}|\nabla_{e'}^{N,*} \W_{e'}^N|^2(X) dt.
\end{align*}
Here, for the first term, since $\nabla_e^{N,*}= \nabla_{-e}^N$ implies
$\nabla_e^{N,*} W_e(X)=- \nabla_e^N W_e(X-\frac{e}N)$, we have
\begin{align*}
\sum_{**} |\nabla_{e}^{N,*} W_e|^2(X) 
& = \sum_{**} |\nabla_{e}^{N} W_e|^2(X-\tfrac{e}N) \\
& = \sum_{\frac{x}N \text{ or } \frac{x+e}N \in D_N(y,r)} 
|\nabla_{e}^{N} W_e|^2(X),
\end{align*}
by replacing $X-\frac{e}N$ by $X$.
For the second term, by a similar replacement, we estimate 
\begin{align*}
|\nabla_{e'}^{N,*} \W_{e'}^N|^2(X)  \le J_{N,r}:=
\sup_{X \text{ or } X+\frac{e}N\in Q_N(r); |e'|=1, e'>0} 
|\nabla_{e'}^N \W_{e'}^N|^2(X),
\end{align*}
and by $[\bar a]_\a^{*,N} \le \bar A$,
\begin{align*}
|\bar a(X)-\bar a(\tilde{Y})|\le \bar A(d(X)\wedge d(\tilde{Y}))^{-\a} |X-\tilde{Y}|^\a
\le \bar A 2^\a d(Y)^{-\a}(r+\tfrac{1+\sqrt{n}}N)^\a,
\end{align*}
since $|X-\tilde{Y}|\le r+\tfrac{1+\sqrt{n}}N$ for $X\in \overline{Q_N(r)}$ and $r<\frac12 d(Y)
= \tfrac12 d(\tilde Y)$ implies $d(X)\ge \frac12 d(Y)$ as in \eqref{eq:d-Y-X}
so that $d(X)\wedge d(\tilde{Y})\ge \frac12 d(Y)$.
Moreover,
$$
\int_{t_1-r^2}^{t_1} N^{-n}
\sum_{\tfrac{x}N \text{ or }\frac{x-e}N\in 
D_N(y,r)} 1 dt \le C(n) (r+\tfrac1N)^n r^2.
$$
Therefore, recalling $\bar a(\tilde{Y}) \ge c_->0$, we obtain
\begin{align*}
\int_{t_1-r^2}^{t_1} |I_2|dt 
& \le \frac{\e}{c_-} \int_{t_1-r^2}^{t_1}  I_1 dt \\
& \quad + \frac{n^2}\e C(n)2^{2\a} \bar A^2
(\tfrac{r+\tfrac{1+\sqrt{n}}N}{d(Y)})^{2\a} (r+\tfrac{1}N)^n r^2 J_{N,r}.
\end{align*}

For $I_3$, since $|Q_N(r)| \le C(n) N^n (r+\tfrac1N)^n r^2$, we obtain by 
$|\nabla_e^N g|_0^{(1),N}\le G_0^1$,
\begin{align*}
\int_{t_1-r^2}^{t_1} |I_3|dt
&\le G_0^1 C(n)(r+ \tfrac{1}{N})^n r^2 d(Y)^{-1} \sup_{Q_N(r)}|W_e^N|.
\end{align*}
Summarizing these estimates, we have
\begin{align*}
(1-\tfrac{\e}{c_-}) \int_{t_1-r^2}^{t_1}  I_1 dt 
\le & \frac{n^2}\e c2^{2\a}\bar A^2
(\tfrac{r+\tfrac{1+\sqrt{n}}N}{d(Y)} )^{2\a}  (r+\tfrac1N)^n r^{2} J_{N,r} \\
&+  C(n) G_0^1(r+ \tfrac{1}{N})^n r^2 d(Y)^{-1} \sup_{Q_N(r)}|W_e^N|.
\end{align*}
Choosing $\e>0$ small and noting $\bar a(\tilde{Y}) \ge c_->0$ and $\alpha<1$,
we have shown \eqref{eq:Nnablawdt-B} in terms of a $C=C(n, c_-)$.
\end{proof}

We now rewrite the discrete estimate obtained in Lemma \ref{lem:discreteEnergy-We}
into a continuous one via polylinear interpolation. 
We will take $\zeta_e^N = \zeta_e^{N,Q(r_N^1)}$ as the solution of the discrete heat equation
\eqref{eq:zetae} on the wider domain $Q_N(Y, r_N^1)$ with boundary condition
$\zeta_e^N = \W_e^N$ on $\mathcal{P}^+_NQ_N(Y,r_N^1)$, where 
$r_N^1=r+ \frac{\sqrt{n}+1}{N}$.
As before, $W^N_e = \zeta^N_e - \W^N_e$ on $\overline{Q_N(Y, r_N^1)}$. 
Set
\begin{align}  \label{eq:5.U}
\begin{aligned}
& U_{2+\a,e}:=[\tilde \W_e^N]_{1+\a}^{(1)}, \\
& U_2  := \max_{e_1,e_2}|\nabla_{e_1}^N \tilde \W_{e_2}^N|_0^{(2)} 
= \max_{e_1,e_2} \sup_{X\in\Om} d^2(X) |\nabla_{e_1}^N \tilde \W_{e_2}^N(X)|, \\
& \mathcal{U}^1_e := [\tilde \W^N_e]_{1+\a}^{(1)} + \langle \W^N_e\rangle_{1+\a}^{(1),N},
\end{aligned}
\end{align}
and also
\begin{align}  \label{eq:5.U-B}
U_{2+\a}:= \max_{e} U_{2+\a,e},
\qquad
\mathcal{U}^1 := \max_{e} \mathcal{U}^1_e.
\end{align}
Recall \eqref{eq:norm-6} and \eqref{eq:norm-10} for $[\tilde \W^N_e]_{1+\a}^{(1)}$
and $\langle \W^N_e\rangle_{1+\a}^{(1),N}$, respectively.
We have the following estimate extending \eqref{eq:nabla^Nw} in the present setting.

\begin{lem}  \label{lem:5.2}
We have 
\begin{align}
\label{full_r_second}
\int_{Q(r)} |\widetilde{ \nabla^N W_{e}}|^2 (X)dX \leq 
C[\bar A^2 U_2^2 + G_0^1 \mathcal{U}_e^1 ]  r_N^n (\tfrac{r_N}{d(Y)})^{2+2\a} d(Y)^{-2},
\end{align}
for every  $r>0$ such that $r_N=r+\tfrac{1+2\sqrt{n}}N < \tfrac12 d(Y)$, where $C=C(n, c_-, T, \alpha)$.
Recall here $Q(r)=Q(Y, r)$. 
\end{lem}

\begin{proof} 
By the same argument given for \eqref{eq:3-Nnablawdt},
except that we multiply and divide by one more factor of  $d(Y)\leq \sqrt{T}$, we have 
  \begin{align}
 \label{W:10.5}
 \sup_{X\in Q_N(r)} |W_e(X)| \leq C(n, T) \mathcal{U}_e^1 (\tfrac{r}{d(Y)})^{1+\a} d(Y)^{-1},
\end{align}
for $0<r<\tfrac12 d(Y)$.
 
The main link now is given by Lemma \ref{lem:estimate-D-energy}.  In fact, 
by \eqref{eq:1-Nnablawdt} for $\nabla^NW_e$ in place of $\nabla^N w$, then by
Lemma \ref{lem:discreteEnergy-We} (with $r$ replaced by
$r+ \tfrac{\sqrt{n}}N$) and \eqref{eq:2-Nnablawdt} 
in Lemma \ref{lem:estimate-D-energy} and also \eqref{W:10.5} (with $r$ replaced by $r+ \tfrac{\sqrt{n}}N$),
we obtain similarly to \eqref{eq:nabla^Nw} that
\begin{align}  \label{eq:nabla^Nw-2}
\int_{Q(r)}  |\widetilde{\nabla_{e'}^N W_e}|^2(X) dX
 &\le C \big[ \bar A^2 U_2^2 d(Y)^{-2}
 + G_0^1 \mathcal{U}_e^1(\tfrac{r_N}{d(Y)})^{1-\a}\big]  r_N^n (\tfrac{r_N}{d(Y)})^{2+2\a} \\
 &\le C[ \bar A^2 U_2^2 + G_0^1\mathcal{U}_e^1]  r_N^n (\tfrac{r_N}{d(Y)})^{2+2\a} d(Y)^{-2}
 \nonumber
\end{align}
for every $e, e'$: $|e|=|e'|=1, e>0, e'>0$ and 
all $r>0$ if it satisfies $r_N<\frac12 d(Y)$.  For the second line, note that $1-\a>0$
and $2r_N/d(Y) \le 1\le Td(Y)^{-2}$.  Here, $C=C(n, c_-, T)$ changed line to line.
\end{proof}

Since $\nabla_{e'}^N \zeta_e$ solves the discrete
heat equation, analogous to the derivation of \eqref{eq:nabla^Nv}, we obtain
\begin{align}\label{eq:nabla^Nv-2}
\int_{Q(\rho)} |\widetilde{\nabla_{e'}^N \zeta_e}-\{\widetilde{\nabla_{e'}^N \zeta_e}\}_{\rho}|^2dX 
& \le C\big(\tfrac{\rho}r\big)^{n+4} 
\int_{Q(r)} |\widetilde{\nabla_{e'}^N \zeta_e}-\{\widetilde{\nabla_{e'}^N \zeta_e}\}_r|^2dX,
\end{align}
for $\rho\in (0,r)$ and $C=C(n, c_\pm)$.  Indeed, in the present setting, the coefficient $\bar a(\tilde Y)$
is direction-independent and one can directly apply
\eqref{eq:Li-L4.5-3} in Proposition \ref{lem:Li-L4.5}.

Therefore, combined with \eqref{full_r_second} in Lemma \ref{lem:5.2}, 
we have as in Lemma \ref{lem:5.3},
\begin{align} \label{eq:om(r)-2}
\om(\rho) \le & \bar{C}\big(\tfrac{\rho}r\big)^{n+4} \om(r)+ \si(r_N),
\end{align}
for $\rho\in (0,r)$ if $r_N= r+ \tfrac{1+\sqrt{n}}N < \frac12 d(Y)$ is satisfied,
where
\begin{align}  \label{eq:om+si-2}
\begin{aligned}
& \om(r) \equiv \om_e(r) = \int_{Q(r)} |\widetilde{\nabla^N \W_e}-\{\widetilde{\nabla^N \W_e}\}_r|^2dX,\\
& \si(r) \equiv \si_e(r) = \widehat{C}[ \bar A^2 U_2^2+ G_0^1\mathcal{U}_e^1]  r^{n+2+2\a}d(Y)^{-(4+2\a)}.
\end{aligned}
\end{align}
Here, $\bar{C}=\bar{C}(n, c_\pm)$ and $\widehat{C}=\widehat{C}(n, c_\pm, T)$.

To derive the bound on $U_{2+\a}$ analogous to Proposition \ref{spatial-variation-prop},
we need the following H\"older estimates for $\nabla_{e'}^N \tilde \xi_e$ in a short distance
regime as in Lemma \ref{lem:4.13}.

\begin{lem}  \label{lem:4.13-5}
{\rm (1)} If  $|Y-Y_1| \le \tfrac{c_1}{N}$, we have
\begin{align}  \label{eq:4.42-Q}
|\nabla_{e'}^N \tilde \xi_e(Y)-\nabla_{e'}^N \tilde \xi_e(Y_1)| \!
\le C \big(U_{2+\a} +U_2 \bar A+G_0^1 \big) \!
\big(d(Y)\wedge d(Y_1)\big)^{-(2+\a)}  |Y-Y_1|^\a,
\end{align}
for some $C=C(n, c_+,c_1)>0$.\\
{\rm (2)}  Furthermore, for every $\de>0$, there exists $M=M(n,c_+)\ge 1$ such that
\begin{align}  \label{eq:4.43-Q}
|\nabla_{e'}^N \tilde \xi_e(Y)-\nabla_{e'}^N \tilde \xi_e(Y_1)| 
\le \de \big(U_{2+\a} +U_2 \bar A+G_0^1 \big) 
\big(d(Y)\wedge d(Y_1)\big)^{-(2+\a)}  |Y-Y_1|^\a,
\end{align}
holds if  $|Y-Y_1| \le \tfrac1{MN}$.
\end{lem}

\begin{proof}
We only outline the proof following that of Lemma \ref{lem:4.13}.

For the spatial direction, we give the estimate in terms of $U_{2+\a}$
instead of $U_{1+\a}$.  This yields an extra factor of $d(Y)^{-1}$ and,
instead of \eqref{eq:L.4.11-1}, we obtain
\begin{align*}
\big|\nabla_{e'}^N \tilde \xi_e(z^{(1)}) - \nabla_{e'}^N \tilde \xi_e(z^{(2)}) \big|
\le |z^{(1)}-z^{(2)}|^\a (\tfrac1{M})^{1-\a} \cdot 2^{n-1} 
U_{2+\a,e} d(Y)^{-(2+\a)}.
\end{align*}

For the temporal direction,  by the equation \eqref{eq:we}, we have
\begin{align*}
\partial_t \nabla_{e'}^N \xi_e(X)
=  - \sum_{|e_1|=1, e_1>0} \nabla_{e'}^N \nabla^{N}_{e} \big(\bar a(X)\nabla^{N,*}_{e_1} \xi_{e_1} \big)(X)
+\nabla_{e'}^N \nabla_e^N g(X).
\end{align*}
We apply the following rough estimates to the right hand side:
By \eqref{eq:disc-der-D} and (B.1)--(B.3),
\begin{align*}
|\nabla_{e'}^N \nabla^{N}_{e} \big(\bar a(X)\nabla^{N,*}_{e_1} \xi_{e_1} \big)(X)|
& \le 2N^2 \Big\{ c_+ \sup_{X'}|\nabla^{N,*}_{e_1} \xi_{e_1} (X'+\tfrac{e}N) 
  - \nabla^{N,*}_{e_1} \xi_{e_1} (X')|  \\
& \hskip 16mm + U_2 d(X)^{-2} \sup_{X'}| \bar a(X'+\tfrac{e}N)-\bar a(X')| \Big\} \\
& \le 2N^2 \Big\{ c_+ U_{2+\a} + U_2 \bar A\Big\} (\tfrac1N)^{\a} d(X)^{-(2+\a)},
\end{align*}
and $|\nabla_{e'}^N\nabla_e^Ng(X)| \le 2 N G_0^1 d(X)^{-1}$.  Therefore, similarly to 
Step 2 in the proof of Lemma \ref{lem:4.13}, we obtain
\begin{align*}
& |\nabla_e^N u(t,z)- \nabla_e^N u(s,z)| \\
& \qquad \le C |t-s|^{\frac{\a}2} (\tfrac1{M})^{2-\a} \Big[ \Big\{ U_{2+\a} + U_2 \bar A\Big\} 
\big(d(Y)\wedge d(Y_1)\big)^{-(1+\a)} + G_0^1\Big] \big(d(Y)\wedge d(Y_1)\big)^{-1},
\end{align*}
where $Y=(t,z)$, $Y_1=(s,z)$ and $C=C(n, c_+)$.  This concludes the proof of the lemma.
\end{proof}

We are at the position to give the bound on $U_{2+\a}$, analogous to 
Proposition \ref{spatial-variation-prop}.

\begin{prop} \label{prop:5.4}
We have
\begin{align}  \label{eq:5.16}
U_{2+\a}= \max_{|e|=1, e>0} [\tilde \W_e^N]_{1+\a}^{(1)}
 \le C\big[(\bar A+1) U_2 + G_0^1\big]+ \de \mathcal{U}^1,
\end{align}
for every $\de>0$ with some $C=C(n, c_\pm, T, \de)$.
\end{prop}

\begin{proof}
We apply Lemma \ref{lem:Li-L4.6-N}, combining with \eqref{eq:om(r)-2}, to get
\begin{align*}
\om(r) \le & C\big(\tfrac{r}{R_0}\big)^{n+2+2\a} \om(R_0)
+C[\bar A^2U_2^2+ G_0^1 \mathcal{U}_e^1] r_N^{n+2+2\a} d(Y)^{-(4+2\a)},
\end{align*}
for $0<r\le R_0 = \tfrac13 d(Y)$, if $d(Y) > \tfrac{6(1+2\sqrt{n})}N$.  
Here, $C=C(n, c_\pm, T)$ and recall $r_N=r+\tfrac{c}N$ with $c=1+2\sqrt{n}$.

However, $\om(R_0)$ is bounded as
\begin{align*}
\om(R_0) & \le 4 \int_{Q(R_0)} |\nabla^N \tilde \xi_e|^2dX 
\le C(n) \, U_2^2 R_0^{n+2}\, d(Y)^{-4}.
\end{align*}

Therefore, we have
\begin{align}  \label{eq:4.41-C-5}
\int_{Q(r)} |\nabla_{e'}^N \tilde \xi_e-\{\nabla_{e'}^N \tilde \xi_e\}_r|dX
&\le C \sqrt{(\bar A^2+1) U_2^2+ G_0^1 \mathcal{U}_e^1} \, d(Y)^{-(2+\a)} r_N^{n+2+\a} 
  \\
&\le \{C'[(\bar A+1) U_2+ G_0^1] + \de \mathcal{U}_e^1 \} \, d(Y)^{-(2+\a)} r_N^{n+2+\a},
  \notag   
\end{align}
for every $\de>0$ with some $C'=C'(n, c_\pm, T, \de)>0$, 
for $0<r\le \tfrac13 d(Y)$, if $d(Y) > \tfrac{6(1+2\sqrt{n})}N$.

In the case that $d(Y) \le \tfrac{6(1+2\sqrt{n})}N$, we can apply \eqref{eq:4.42-Q}
in Lemma \ref{lem:4.13-5}  by making $c>0$ in $r_N$ larger if necessary.  Indeed, 
we obtain for $0<r\le \tfrac13 d(Y)$
\begin{align}\label{eq:4.41-D-5}
\int_{Q(r)} |\nabla_{e'}^N \tilde \xi_e-\{\nabla_{e'}^N \tilde \xi_e\}_r|dX
\le C r^{n+2+\a} \big(U_{2+\a} +U_2 \bar A+G_0^1 \big) 
d(Y)^{-(2+\a)}
\end{align}
where $C=C(n, c_+)$.

From \eqref{eq:4.41-C-5}, \eqref{eq:4.41-D-5}, making  $c>0$ in $r_N$ large
as in the proof of Proposition \ref{spatial-variation-prop} and recalling
$U_{2+\a} \le \mathcal{U}^1$, we have shown
\begin{align}  \label{eq:4.45-F-5}
\int_{Q(Y,r)} |\nabla_{e'}^N \tilde \xi_e-\{\nabla_{e'}^N \tilde \xi_e\}_r|dX
&\le \{C[(\bar A+1) U_2+ G_0^1] + \de \mathcal{U}^1 \} 
\, d(Y)^{-(2+\a)} r_N^{n+2+\a},
\end{align}
for $0<r\le \tfrac13 d(Y)$ and any $Y=(t_1,y)\in \Om$ where $C=C(n, c_\pm, T, \de)$.

We now fix any $X_0=(t_0,z_0)\in \Om$ and
take $R=\tfrac14d(X_0)=\tfrac14 \sqrt{t_0}$. Then, by Lemma \ref{lem:Li-L4.3-G}
(with a proper modification in $U_{F,1+\a}$), 
we see that, for every $\de>0$ and $M>0$, there exists 
$C=C_{\de,M}(n, c_\pm, T)>0$ such that
\begin{align}  \label{eq:4.41-F-5}
|\nabla_{e'}^N \tilde \xi_e(Y)-\nabla_{e'}^N \tilde \xi_e(Y_1)| \le \Big( &
 C[(\bar A+1) U_2+ G_0^1] + \de \mathcal{U}^1+\de U_{2+\a} \Big) \\
& \qquad \times \big(d(Y)\wedge d(Y_1)\big)^{-(2+\a)} |Y-Y_1|^\a,
\notag
\end{align}
holds if  $Y, Y_1\in Q(X_0,R)$ satisfy $|Y-Y_1| \ge \tfrac1{MN}$.

The case $|Y-Y_1| \le \tfrac1{MN}$ is covered by \eqref{eq:4.43-Q}
in Lemma \ref{lem:4.13-5}, and we obtain \eqref{eq:4.41-F-5} 
for any $Y, Y_1 \in Q(X_0,R)$.
Thus, moving $X_0$, we have shown \eqref{eq:4.41-F-5} for any $Y, Y_1\in \Om$.
This implies the concluding estimate \eqref{eq:5.16} on $U_{2+\a}$.
\end{proof}

Now, by the same argument given for Proposition \ref{temporal-prop}, 
using in place of $U_{1+\alpha}$, $U_1$ and $G_\infty$ the quantities 
$U_{2+\alpha,e}$, $U_2$ and $G_0^1$ (which introduce an extra factor of $d(Y)^{-1}$), 
we have
\begin{align}  \label{eq:weN} 
\langle \W^N_e\rangle_{1+\a}^{(1),N} \leq  C'\big[\bar AU_2 + U_{2+\alpha,e} + G_0^1\big]
\end{align}
and so
\begin{align}  \label{eq:5.17} 
 \mathcal{U}_e^1 =U_{2+\a,e}+ \langle \W^N_e\rangle_{1+\a}^{(1),N}
 \leq U_{2+\a,e} +  C'\big[\bar AU_2 + U_{2+\alpha,e} + G_0^1\big]
\end{align}
where $C'=C'(n, c_+, T, \alpha)$.
 
Hence, inserting \eqref{eq:5.17}, maximized over $e$, into \eqref{eq:5.16} in Proposition \ref{prop:5.4}, we obtain
\begin{align*}
U_{2+\a} &\leq C\Big[ (\bar A+1) U_2 + G_0^1 + 
\de \big\{U_{2+\a}+C'[\bar AU_2 + U_{2+\a}+G_0^1 \big]\big\}\Big]
\end{align*} 
and by choosing $\de>0$ small that
\begin{align}  \label{eq:5.18-B} 
U_{2+\a} \leq C''\big[(\bar A+1) U_2 + G_0^1\big]
\end{align}
where $C''=C''(n, c_\pm, T, \alpha)$.

However, by applying the interpolation inequality \eqref{eq:4.U2}
in Lemma \ref{prop:Li-P4.2} to $U_2$,
\begin{align}  \label{eq:5.19-B} 
U_2\le 5 \max_{e'} (|\tilde \W_{e'}^N|_0^{(1)})^{\frac{\a}{1+\a}}
\big( |\tilde \W_{e'}^N|_0^{(1)}+ [\tilde \W_{e'}^N]_{1+\a}^{(1)}\big) ^{\frac1{1+\a}}.
\end{align}
Denote
$$
|\tilde \W_\cdot^N|_0^{(1)} := \max_{e'} |\tilde \W_{e'}^N|_0^{(1)}.
$$
Then, from \eqref{eq:5.19-B}, we have in view of \eqref{eq:5.18-B} that
\begin{align}  \label{eq:5.AU2}
C''(\bar A+1)U_2 & \le \Big(2^{\frac1\a} \big(5C''(\bar A+1)\big)^{\frac{1+\a}\a}
|\tilde \W_\cdot^N|_0^{(1)} \Big)^{\tfrac\a{1+\a}} \Big( \tfrac12
\big(|\tilde \W_\cdot^N|_0^{(1)} + U_{2+\a}\big)\Big)^{\tfrac1{1+\a}} \\
& \le 2^{\frac1\a} \big(5C''(\bar A+1)\big)^{\frac{1+\a}\a}
|\tilde \W_\cdot^N|_0^{(1)} + \tfrac12
\big(|\tilde \W_\cdot^N|_0^{(1)} + U_{2+\a}\big),  \notag
\end{align}
where we have used a trivial bound: $ab \, (\le \tfrac{a^p}p+\tfrac{b^q}q)
\le a^p+b^q$
for $a,b>0$ and $\tfrac1p+\tfrac1q=1$.  Inserting this into \eqref{eq:5.18-B},
we obtain
\begin{align} \label{eq:5.U2+a}
U_{2+\a} \leq \bar{C} \Big[(\bar A+1)^{1+\tfrac1\a}|\tilde \W_\cdot^N|_0^{(1)}+ G_0^1\Big]
\end{align}
where $\bar{C}=\bar{C}(n, c_\pm, T, \alpha)$.

Reporting these estimates, as an extension of Theorem \ref{thm:Li-T4.8},
we arrive at the following bound for $U_2$ and
$$
|\tilde \W_\cdot^N|_{1+\a}^{(1)} := U_2 + U_{2+\alpha} + \max_e\langle \W^N_e\rangle_{1+\a}^{(1),N}.
$$

\begin{thm} \label{second_Schauder_thm}
Under the assumptions (B.1), (B.2) and (B.3), we have
\begin{align}  \label{eq:5.22-B}
& |\tilde \W_\cdot^N|_{1+\a}^{(1)} 
\leq C \Big[(\bar A+1)^{1+\tfrac1\a}|\tilde \W_\cdot^N|_0^{(1)}+ G_0^1\Big],
\intertext{and}
\label{eq:5.22-C}
& U_2 \le C\Big[(\bar A+1)^{\tfrac1\a}|\tilde \W_\cdot^N|_0^{(1)}+ (\bar A+1)^{-1}G_0^1\Big]
\end{align}
where $C=C(n, c_\pm, T, \alpha)$.
\end{thm}

\begin{proof}
For \eqref{eq:5.22-B},
it is enough to bound each term in $|\tilde \W_\cdot^N|_{1+\a}^{(1)}$ by the right hand side.
For $U_{2+\a}$, the bound follows from \eqref{eq:5.U2+a}.
For $U_2$, we can use \eqref{eq:5.AU2}, where multiplying  
%Actually, this shows a better bound for $U_2$ by
the right hand side of \eqref{eq:5.U2+a} (or \eqref{eq:5.22-B})
by $(\bar A+1)^{-1}$ gives \eqref{eq:5.22-C}.
Finally for $\langle \W^N_e\rangle_{1+\a}^{(1),N}$, we may use
\eqref{eq:weN} together with \eqref{eq:5.AU2} and \eqref{eq:5.U2+a}.
\end{proof}

Theorem \ref{second_Schauder_thm} shows a form of the first Schauder estimate
for the linear discrete PDE \eqref{eq:we}. It is an estimate in terms of norms which 
weight more near the boundary $t=0$.  In fact, the estimate \eqref{eq:5.22-C} on $U_2$
yields the singularity $\frac1t$ for $\nabla_{e_1}^N\xi_{e_2}^N$ near $t=0$. 
However, we note here an analog of Theorem \ref{thm:Li-T4.8} and \eqref{eq:4.8} for 
the equation \eqref{eq:we}, which exhibits a weaker singularity $\frac1{\sqrt{t}}$;
see the second estimate in \eqref{rmk 5.1 bounds}.

\begin{cor}
\label{alternative_rmk}
Under the assumptions (B.1), (B.2) and 
$\|\nabla^Ng\|_\infty<\infty$ (replacing (B.3)), and also $\|\W^N_\cdot\|_\infty <\infty$ (compare with \eqref{eq:1.u_pm}), we have
\begin{align}
\label{rmk 5.1 bounds}
\begin{aligned}
|\tilde \W^N_\cdot|_{1+\alpha}^*& \leq C\big[ (\bar A + 1)^{1+ \frac{1}{\alpha}} \|\W^N_\cdot\|_\infty + \|\nabla^Ng\|_\infty\big],\\
|\nabla^N \W^N_\cdot|_0^{(1)} &\leq 
C\big[(\bar A + 1)^{\frac{1}{\alpha}} \|\W^N_\cdot\|_\infty +(\bar A+1)^{-1} \|\nabla^Ng\|_\infty\big], 
\end{aligned}
\end{align}
where $C=C(n, c_\pm, T, \alpha)$.

\end{cor}

\begin{proof} The argument follows the same proof as for Theorem \ref{second_Schauder_thm}.
Observe that the seminorms $|\tilde \W_\cdot^N|_{1+\a}^{(1)}$, $|\tilde \W_\cdot^N|_0^{(1)}$,
$|\nabla^Ng|_{0}^{(1),N}$ all involve just one more weight $d(X)$ compared to
$|\tilde \W^N_\cdot|_{1+\alpha}^*$, $\|\W^N_\cdot\|_\infty$, $\|\nabla^Ng\|_\infty$,
respectively; see also the proof of Corollary \ref{cor:5.7}.
\end{proof}

Now, as in Theorem \ref{thm:4.2}, when the initial value of $\psi$ and $\W^N_e =\nabla^N_e \psi$ are smooth enough, we will obtain a more regular estimate at $t=0$ without any singularity.
We will suppose that $\psi$ is $C^4$ and so $\W^N_e$ is $C^3$ in the sense that
\begin{align}
\label{psi_smooth_initial}
\sup_N \|\psi(0)\|_{C_N^4} \le \mathcal{C}_0<\infty \ \ \ {\rm and \ so \ \  \ }
\sup_N \|\W_e^N(0)\|_{C_N^3} \le \mathcal{C}_0<\infty.
\end{align}
We will need to following additional assumptions for the coefficient $\bar a$,
similar to \eqref{eq:Th4.2-B} and \eqref{eq:Th4.2-1}:
\begin{align}
\label{a-holder}
&[\bar a]_\alpha^{(-\alpha),N} \leq \bar B<\infty: |\bar a(t_1, \tfrac{x_1}{N})-a(t_2, \tfrac{x_2}{N})|\leq \bar B\big\{|t_2-t_1|^{\frac{\alpha}{2}} + \big|\tfrac{x_2}{N}-\tfrac{x_1}{N}\big|^\alpha\big\},\\
\label{a-grad_bounds}
&\sup_{N}\| \bar a(0,\cdot)\|_{C^4_N} \leq \mathcal{C}_1<\infty.
\end{align}

Also, we assume that $\nabla^N g$ satisfies
\begin{align}
\label{g-bound}
\sup_{N,e} |\nabla^N_e g|_0 \leq G^1_\infty<\infty.
\end{align}

\begin{thm}
\label{second_Schauder_time_thm}
Consider the discrete PDE \eqref{eq:we}, satisfying (B.1), \eqref{a-holder}, \eqref{a-grad_bounds}, \eqref{g-bound}, such that
\begin{align}
\label{W-bound}
\sup_{N,e} |\W^N_e|_0 \leq W_\infty<\infty,
\end{align}
and with initial value given in \eqref{psi_smooth_initial}.  Then, with respect to the unweighted norms,  
\begin{align}  \label{1+alpha bound}
& |\tilde \W_\cdot^N|_{1+\alpha}\leq C\big[ (\bar B  +1)^{1+\frac{1}{\alpha}}\big(W_\infty + \mathcal{C}_6\big) +G^1_\infty + \mathcal{C}_7\big]\\
\label{nabla bound}
&|\nabla^{N} \tilde \W^N_\cdot|_0 \leq C\big[ (\bar B  +1)^{\frac{1}{\alpha}}\big(W_\infty + \mathcal{C}_6\big)+ (\bar B  + 1)^{-1}(G^1_\infty + \mathcal{C}_7)\big].\\
\label{holder-bound}
&[\nabla^N \tilde \W^N_\cdot]_\alpha \leq C\big[(\bar B  +1)^{1+\frac{1}{\alpha}}\big(W_\infty + \mathcal{C}_6\big) + G^1_\infty + \mathcal{C}_7\big],
\end{align}
where $C=C(n, c_\pm, T, \alpha)$ and $\mathcal{C}_6, \mathcal{C}_7$, specified in the proof, depend on $\bar B$, $\mathcal{C}_0$, $\mathcal{C}_1$ and $\alpha$.

In particular, we have polynomial bounds in $\bar B$, $\mathcal{C}_0$ and $\mathcal{C}_1$, in terms of a universal constant $C$,
\begin{align*}
\mathcal{C}_6  \leq C(\bar B+1)^{\frac{1}{\alpha}}\mathcal{C}_0(1+\mathcal{C}_1^5) \ \ {\rm and \ \ }
\mathcal{C}_7  \leq C(\bar B+1)^{\frac{2}{\alpha}}\mathcal{C}_0(1+\mathcal{C}_1^{11}).
\end{align*}
\end{thm}

\begin{proof}
We will apply the same sort of scheme as given for Theorem \ref{thm:4.2}, where now Theorem \ref{second_Schauder_thm} will be invoked with respect to a time extended system.
Recall equation \eqref{eq:we} for $t>0$,
\begin{align}
\label{new_eq}
\partial_t \W^N_e = \nabla^N_e\big\{-\bar a(X) \sum_{|e'|=1, e'>0} \nabla^{N,*}_{e'}  \W^N_{e'}\big\}  + \nabla^N_e g,
\end{align}
where $\{\W^N_e =\nabla^N_e \psi\}_{|e|=1}$ is a gradient of the function $\psi$. 
As before, the idea now is to extend the evolution below time $t=0$.

Consider $\bar a$ evaluated at time $t=0$, and define $b(X):=\bar a\big(0,\tfrac{x}{N}\big)$.
Note that $b$ does not depend on time $t$ or direction $e$, and satisfies $c_-\leq b\leq c_+$ from (B.1).
We will need to consider the reciprocal $1/b$ in the following.  Note by \eqref{a-grad_bounds} that $\|(1/b)\|_{C^4_N} \leq C(n,c_\pm)\mathcal{C}_1^4$.

For $t\geq 0$, define $Z$ by
\begin{align}
\label{eq:Z}
\partial_t Z = -\sum_{|e'|=1, e'>0} \nabla^{N, *}_{e'} \big(b \nabla^N_{e'} Z\big)
\end{align}
with initial condition $Z(0)= \tfrac{1}{b}\psi(0)$.  Note as $\sup_N\|(1/b)\|_{C^4_N}\leq  C(n, c_\pm)\mathcal{C}_1^4$ that we have $\sup_N\|Z(0)\|_{C^4_N}
\leq C(n, c_\pm)\mathcal{C}_0(1+\mathcal{C}^4_1)$.
Note also that equation \eqref{eq:Z} is at the level of \eqref{eq:vN} for $\psi=\psi^N$,
but we divide it by $b$ to make the equation in divergence form so that we can apply
the gradient bounds obtained in Theorem \ref{thm:4.2}.

By the maximum principle (Lemma \ref{lem:1.1}), $\|Z\|_\infty \leq \|Z(0)\|_\infty \leq C(c_-)\mathcal{C}_0$.  Moreover, we may apply Theorem \ref{thm:4.2} to bound gradients of $Z$:  We have $b$ satisfies (A.1) and also \eqref{eq:Th4.2-B} and \eqref{eq:Th4.2-1} with $B=[b]_\alpha^{(-\alpha),N}\leq \bar B$ and $C_1=\sup_{N,e} |\nabla^N_e b_0| \leq \mathcal{C}_1$.  Here, $G_\infty=0$, $C_0=\sup_N\|Z(0)\|_{C^2_N}\leq C(n, c_\pm)\mathcal{C}_0(1+\mathcal{C}^4_1)$, and \eqref{eq:4.bounded} holds with bound $\|Z\|_\infty\leq C(c_-)\mathcal{C}_0$. Then,
by \eqref{eq:4-14},
\begin{align}
\label{claim-grad}
\max_{e}|\nabla^N_{e}Z|_0&\leq C(n,c_\pm, T, \alpha)\big[(B +1)^{\frac{1}{\alpha}}|Z|_0 
+ (B+1)^{-1} \big\{G_\infty+C_2\big\}\big] \\
&\leq C(n, c_\pm, T, \alpha)\mathcal{C}_3
\nonumber
\end{align}
where 
\begin{align}
\label{eq:C_0}
{C}_2 &= n(C_1C_0+c_+C_0) \leq C(n, c_\pm)\mathcal{C}_0(1+\mathcal{C}_1^5),\\
\mathcal{C}_3 &=
 (\bar B+1)^{\frac{1}{\alpha}}\mathcal{C}_0 
 + (\bar B+1)^{-1}\mathcal{C}_0(1+\mathcal{C}_1^5) \nonumber\\
 &\leq C(\bar B+1)^{\frac{1}{\alpha}}\mathcal{C}_0(1+\mathcal{C}_1^5),\nonumber
\end{align}
in terms of a universal constant $C$.

Define now, and noting \eqref{eq:disc-der-D},
\begin{align}
\label{eq: h_1}
h_1 &:= \sum_{|e'|=1, e'>0} \nabla^{N, *}_{e'} \big(b \nabla^N_{e'} Z\big)  \\
&= -b\Delta^N Z + \sum_{|e'|=1, e'>0} \nabla^{N, *}_{e'}b \cdot \nabla^N_{e'}Z(X+ \tfrac{e}{N}).\nonumber
\end{align}
 Observe that $h_1$ satisfies \eqref{eq:Z}.  Then, by the maximum principle, and previous bounds $b\leq c_+$, $\max_e |\nabla^N_e b|_0\leq \mathcal{C}_1$ in \eqref{a-grad_bounds}, and $\|Z(0)\|_{C^4_N}\leq C(n, c_\pm)\mathcal{C}_0(1+\mathcal{C}_1^4)$, we have
\begin{align}
\label{h_1 bound}
\|h_1\|_\infty \leq \|h_1(0)\|_\infty \leq C(n,c_\pm)\mathcal{C}_0(1+ \mathcal{C}^5_1). 
\end{align}

   Hence, noting \eqref{eq: h_1}, by \eqref{claim-grad} and triangle inequality, we also have 
\begin{align}  \label{eq:5.Z-2}
\|\Delta^N Z\|_\infty \leq C(n, c_\pm, T, \alpha)\mathcal{C}_1\mathcal{C}_3 +
C(n,c_\pm)\mathcal{C}_0(1+ \mathcal{C}^5_1).
\end{align}

Further, by Theorem \ref{thm:4.2}, as in the set-up of \eqref{claim-grad}, 
we obtain an estimate for the gradient of $h_1$:
Indeed, as $Z = \tfrac{1}{b}\psi$ and $\|Z(0)\|_{C^4_N} \leq C(n, c_\pm)\mathcal{C}_0(1+\mathcal{C}_1^4)$,
noting \eqref{a-grad_bounds}, we have
$C_0=\sup_N \|h_1(0)\|_{C_N^2} \leq C(n, c_\pm)\mathcal{C}_0(1+\mathcal{C}^5_1)$.  Also, as before, $C_1=\sup_{N, e}|\nabla^N_e b|_0\leq \mathcal{C}_1$, $B=[b]_\alpha^{(-\alpha),N}\leq \bar B$ and $G_\infty=0$.  Hence, in this context (cf. \eqref{eq:C_0}), $C_2 \leq C(n, c_\pm)\mathcal{C}_0(1+\mathcal{C}_1^6)$.
By \eqref{h_1 bound}, $|h_1|_0\leq C\mathcal{C}_0(1+\mathcal{C}_1^5)$.  Then,
\begin{align}
\label{eq h_1 grad}
\max_e|\nabla^N_e h_1|_0 &\leq C(n,c_\pm, T, \alpha)\big[(B +1)^{\frac{1}{\alpha}}|h_1|_0 
+ (B+1)^{-1} \big\{G_\infty+C_2\big\}\big]\\ 
&\leq
C(n, c_\pm, \alpha, T)\mathcal{C}_4\nonumber
\end{align}
where, in terms of a universal constant $C$,
\begin{align*}
\mathcal{C}_4 &= (\bar B +1)^{\frac{1}{\alpha}}\mathcal{C}_0(1+\mathcal{C}_1^5) + (\bar B + 1)^{-1}\mathcal{C}_0(1+\mathcal{C}_1^6)\\
&\leq C(\bar B+1)^{\frac{1}{\alpha}}\mathcal{C}_0(1+\mathcal{C}_1^6).
\end{align*}

By considering the equation for $Z_e := \nabla^N_e Z$,
\begin{align}
\label{eq:grad Z}
\partial_t Z_e = -\sum_{|e'|=1, e'>0} \nabla^{N,*}_{e'} \big(b \nabla^N_{e'} Z_e\big)
 - \sum_{|e'|=1, e'>0} \nabla^{N, *}_{e'} \big((\nabla^N_e b)Z_{e'}(X+ \tfrac{e}{N})\big),
\end{align}
we may bound second gradients $\nabla_{e_1}^N \nabla^N_{e_2} Z$.  
Indeed, $\eta_e:=\nabla_e^N(bZ)$ is more natural object that we consider below, which is parallel
to $\xi_e$, but here we consider $Z_e$
to deduce the second gradient bounds.

Observe here that $C_0= \sup_N\|Z_e(0)\|_{C^2_N} \leq C(n, c_\pm)\mathcal{C}_0(1+\mathcal{C}^4_1)$.  
As before, in this context, $C_1=\sup_{N, e}|\nabla^N_e b|_0\leq \mathcal{C}_1$, $B=[b]_\alpha^{(-\alpha),N}\leq \bar B$ and $C_2 \leq C(n, c_\pm)\mathcal{C}_0(1+\mathcal{C}_1^5)$ (cf. \eqref{eq:C_0}).  However, here $|\nabla^N_{e_2}Z|_0 \leq C(n, c_\pm, T, \alpha)\mathcal{C}_3$ and, by \eqref{claim-grad} and \eqref{eq:5.Z-2},
\begin{align*}
G_\infty &= \big\|\sum_{|e'|=1, e'>0} \nabla^{N, *}_{e'} \big((\nabla^N_e b)Z_{e'}(X+ \tfrac{e}{N})\big)\big\|_\infty \leq C(n)\mathcal{C}_1\max_e \|Z_e\|_\infty + \mathcal{C}_1\|\Delta^N Z\|_\infty\\
&= C(n, c_\pm, T, \alpha)\big\{ \mathcal{C}_1\mathcal{C}_3 + \mathcal{C}_1\big[\mathcal{C}_1\mathcal{C}_3 + \mathcal{C}_0(1+\mathcal{C}^5_1)\big]\big\}.
\end{align*}

Then, by Theorem \ref{thm:4.2} applied to \eqref{eq:grad Z}, we have
\begin{align}
\label{second grad Z}
\max_{e_1, e_2} |\nabla^N_{e_1}\nabla^N_{e_2} Z|_0 &\leq 
C(n,c_\pm, T, \alpha)\big[(B +1)^{\frac{1}{\alpha}}|\nabla^N_{e_1}Z|_0 
+ (B+1)^{-1} \big\{G_\infty+C_2\big\}\big]\\ 
&\leq C(n, c_\pm, T, \alpha) \mathcal{C}_5, \nonumber
\end{align}
where
\begin{align*}
\mathcal{C}_5 &= (\bar B +1)^{\tfrac{1}{\alpha}}\mathcal{C}_3 + (\bar B+1)^{-1}\big\{\mathcal{C}_1\mathcal{C}_3 + \mathcal{C}_1\big[\mathcal{C}_1\mathcal{C}_3 + \mathcal{C}_0(1+\mathcal{C}^5_1)\big]+ \mathcal{C}_0(1+\mathcal{C}_1^5)\big\}\\
&\leq C(\bar B+1)^{\tfrac{2}{\alpha}}\mathcal{C}_0(1+\mathcal{C}_1^7),
\end{align*}
incorporating the expression for $C_2$ and $\mathcal{C}_3$ in \eqref{eq:C_0} in terms of 
a universal constant $C$.

Let now $\hat Z(t) = Z(1-t)$ and $\hat h_1(t) = h_1(1-t)$ for $0\leq t<1$.  Then,
$$\partial_t \hat Z = -\sum_{|e'|=1, e'>0} \nabla^{N, *}_{e'} \big(b \nabla^N_{e'} \hat Z\big) + 2\hat h_1.$$
Define $\hat \Psi = b\hat Z$ to be back to the level of $\psi=\psi^N$.  Then,
by using \eqref{eq:disc-der-D},
\begin{align*}
\partial_t \hat \Psi = b\Delta^N\hat \Psi - b\hat h_2
+ 2b\hat h_1,
\end{align*}
where
$$\hat h_2 = \sum_{|e'|=1, e'>0} \nabla^{N, *}_{e'}\big[ b(\nabla^N_{e'} \tfrac{1}{b})\hat\Psi(X+\tfrac{e'}{N})\big].$$
We have by the previous bounds \eqref{psi_smooth_initial}, \eqref{a-grad_bounds}, \eqref{claim-grad} and $\|(1/b)\|_{C^4_N}\leq C(n, c_\pm)\mathcal{C}_1^4$, noting $\hat\Psi = b\hat Z$, that
\begin{align}
\label{h_2 bound}
\|\hat h_2\|_\infty & \leq C(n, c_\pm) \max_{e'}\Big\{
|\nabla^{N, *}_{e'}\big[b(\nabla^N_{e'}\tfrac{1}{b})(X)b(X+\tfrac{e'}{N})\big]|_0|Z|_0\\
&\ \ \ \ \ \ \ \  + |b(\nabla^N_{e'}\tfrac{1}{b})(X)b(X+\tfrac{e'}{N})|_0|\nabla^{N, *}_{e'} Z|_0\Big\} \nonumber\\
& \leq C(n, c_\pm, T, \alpha)\big\{ \mathcal{C}_0\mathcal{C}^4_1(1+\mathcal{C}_1)+
\mathcal{C}^4_1\mathcal{C}_3\big\}.\nonumber
\end{align}
Analogously, using also \eqref{second grad Z}, we have
\begin{align}
\label{h_2 grad bound}
\|\nabla^N_e \hat h_2\|_\infty & \leq C(n, c_\pm)\max_{e'} \Big\{|\nabla^N_e\nabla^{N, *}_{e'}\big[b(\nabla^N_{e'}\tfrac{1}{b})(X)b(X+\tfrac{e'}{N})\big]|_0|Z|_0\\
&\hskip 30mm   + |b(\nabla^N_{e'}\tfrac{1}{b})(X)b(X+\tfrac{e'}{N})|_0|\nabla^N_e\nabla^{N, *}_{e'} Z|_0
\nonumber\\
&\hskip 30mm + |\nabla^{N, *}_{e'}\big[b(\nabla^N_{e'}\tfrac{1}{b})(X)b(X+\tfrac{e'}{N})\big]|_0|\nabla^N_eZ|_0 
\nonumber\\
& \hskip 30mm
+ |\nabla^N_e\big[b(\nabla^N_{e'}\tfrac{1}{b})(X)b(X+\tfrac{e'}{N})\big]|_0|\nabla^{N, *}_{e'}Z|_0\Big\} \nonumber\\
&\leq C(n, c_\pm, T, \alpha)\big\{ \mathcal{C}_0\mathcal{C}^4_1(1+\mathcal{C}_1^2) + \mathcal{C}^4_1\mathcal{C}_5 + \mathcal{C}_3\mathcal{C}^4_1(1+\mathcal{C}_1)\big\}.\nonumber
\end{align} 

Now define $\hat \eta_e = \nabla^N_e \hat\Psi$ for $|e|=1$ to be at the level of
$\xi_e^N$.  Then,
\begin{align}
\label{eq: eta_e}
\partial_t \hat \eta_e &= \nabla^N_e\Big\{ -b \sum_{|e'|=1,e'>0} \nabla^{N, *}_{e'}\hat \eta_{e'}\Big\} -\nabla^N_e (b \hat h_2) + 2\nabla^N_e (b \hat h_1).
\end{align}
Note that $\hat \eta_e(1) = \nabla^N_e (b Z(0)) = \nabla^N_e \psi(0) = \W^N_e(0)$, matching the initial data to \eqref{new_eq}.
Also, by \eqref{psi_smooth_initial}, \eqref{a-grad_bounds} and \eqref{claim-grad}, noting $Z= \tfrac{1}{b}\psi$ and via \eqref{eq:disc-der-D} $\nabla^N_e (b\hat Z) = (\nabla^N_e b)\hat Z + b \nabla^N_e \hat Z(X+\tfrac{e}{N})$,
\begin{align}
\label{eta_e bound}
\|\hat\eta_e\|_\infty& = \|\nabla^N_e (b\hat Z)\|_\infty \leq C(n, c_\pm, T, \alpha)\big[\mathcal{C}_0\mathcal{C}_1 + \mathcal{C}_3\big]\\
&=:C(n, c_\pm, T, \alpha)\mathcal{C}_6 \nonumber
\end{align}
where, incorporating the expression for $\mathcal{C}_3$ in \eqref{eq:C_0}, in terms of a universal constant $C$,
$$\mathcal{C}_6=\mathcal{C}_0\mathcal{C}_1 + \mathcal{C}_3 
\leq C(\bar B+1)^{\frac{1}{\alpha}}\mathcal{C}_0(1+\mathcal{C}_1^5).$$

Now, as in \eqref{hat a defn-B}, define
$$\hat a = \left\{\begin{array}{rl}
\bar a(t-1)& {\rm for \ }t\geq 1\\
b& {\rm for \ } 0\leq t<1,
\end{array}\right. \ \ {\rm and \ \ }
\hat g = \left\{\begin{array}{rl}
\nabla^N_e g& {\rm for \ } t\geq 1\\
2\nabla^N_e(b\hat h_1) - \nabla^N_e(b\hat h_2) & {\rm for \ }0\leq t<1.
\end{array}\right.
$$
Note that $\hat a$ satisfies condition (B.1), and also (B.2) with 
$$[\hat a]^{*,N}_\alpha \leq (T+1)^{\frac{\alpha}{2}}[\bar a]_\alpha^{(-\alpha),N}
\leq C(T)\bar B  = \bar A$$
since $b=\bar a(0)$ does not depend on time $t$ and $d(X)\leq \sqrt{T+1}$.

Also, (B.3) holds, 
noting the bounds \eqref{a-grad_bounds}, \eqref{g-bound}, \eqref{h_1 bound}, \eqref{eq h_1 grad}, \eqref{h_2 bound} \eqref{h_2 grad bound}, via \eqref{eq:disc-der-D},
 with 
\begin{align*}
|\hat g|_0^{(1)}&\leq (T+1)^{\frac{\alpha}{2}} \big\{ G^1_\infty + 2|\nabla^N_e(b\hat h_1)|_0+|\nabla^N_e(b\hat h_2)|_0 \big\}\\
&= C(n, c_\pm, T,\alpha)\big\{G^1_\infty + \mathcal{C}_7\big\} = G_0^1
\end{align*}
where 
\begin{align*}
\mathcal{C}_7 &= |\nabla^N_e b|_0|\hat h_1|_0 + |\nabla^N_e \hat h_1|_0 + |\nabla^N_e b|_0|\hat h_2|_0 + |\nabla^N_e\hat h_2|_0\\
&\leq \mathcal{C}_1\big[\mathcal{C}_0(1+\mathcal{C}^5_1)\big] + \mathcal{C}_4
+ \mathcal{C}_1\big[ \mathcal{C}_0\mathcal{C}^4_1(1+\mathcal{C}_1) + \mathcal{C}^4_1\mathcal{C}_3\big]\\ 
&\ \ + \mathcal{C}_0\mathcal{C}_1^4(1+\mathcal{C}_1^2) +\mathcal{C}_1^4\mathcal{C}_5  + \mathcal{C}_3\mathcal{C}_1^4(1+\mathcal{C}_1)\\
&\leq C\Big\{\mathcal{C}_0\mathcal{C}_1(1+\mathcal{C}_1^5) + \mathcal{C}_3\mathcal{C}_1^4(1+\mathcal{C}_1) + \mathcal{C}_4 + \mathcal{C}_5\mathcal{C}_1^4\Big\}\\
&\leq C (\bar B+1)^{\frac{2}{\alpha}}\mathcal{C}_0(1+\mathcal{C}_1^{11}).
\end{align*}
The last line follows by incorporating the expressions of $\mathcal{C}_3$ in \eqref{eq:C_0}, 
$\mathcal{C}_4$ below \eqref{eq h_1 grad}, and $\mathcal{C}_5$ below \eqref{second grad Z}.
Here, the constants $C$ are universal.

We now formulate the extended system for $t\geq 0$, which corresponds to \eqref{eq: eta_e} when $0\leq t<1$ and \eqref{new_eq} when $t\geq 1$, as
\begin{align*}
\partial_t V_e =  \nabla^N_e\Big\{-\hat a(X)\sum_{|e'|=1, e'>0}\nabla^{N,*}_{e'}V_{e'}(X) \Big\}+ \hat g(X).
\end{align*}
Observe, by \eqref{W-bound} and \eqref{eta_e bound}, that 
$$
|V_e|_0^{(1)} \leq (T+1)^{\frac{\alpha}{2}} \big\{ \max_e |\W^N_e|_0 + \max_e |\eta_e|_0\big\} \leq C(n, c_\pm, T, \alpha)\big(W_\infty + \mathcal{C}_6\big).
$$
We may now apply Theorem \ref{second_Schauder_thm} to the system $\{V_e\}_{|e|=1}$. 
As $\W_e(t) = V_e(t+1)$ for $t\geq 0$, the desired statements \eqref{1+alpha bound}, \eqref{nabla bound}, and 
\eqref{holder-bound} follow.
\end{proof}

\subsection{Second Schauder estimate for \eqref{eq:2.1-X}}  \label{sec:5.2}

We now specialize to the setting introduced at the beginning of this section
and apply Theorem \ref{second_Schauder_thm}, Corollary \ref{alternative_rmk}, and Theorem \ref{second_Schauder_time_thm} to the equation \eqref{eq:2.1-X}
satisfying \eqref{eq:1.u_pm}.  
See, correspondingly, Corollaries \ref{second_Schauder_cor}, \ref{cor:5.7} 
and \ref{second_Schauder_time_cor} below, in which we obtain estimates of
the second discrete derivatives of $u^N(t,\frac{x}N)$ exhibiting singularities
$\frac1t, \frac1{\sqrt{t}}, 1$, respectively, near $t=0$, without assuming or  assuming
some regularity conditions for the initial value $u^N(0)$.
Recall the correspondences at the beginning of Section \ref{sec:5.1}.

\begin{cor}  \label{second_Schauder_cor}
Consider the nonlinear discrete PDE \eqref{eq:2.1-X}
satisfying \eqref{eq:1.u_pm}. Then, we have
\begin{align}  \label{eq:5.25}
& |\tilde \W_e^N|_{1+\si}^{(1)} \leq C\big( K^{1+\frac{2}{\si}}+1\big),
\intertext{for every $e$, and}
 \label{eq:5.28}
& |\W_e^N(X)| \leq \frac{C(K^{\frac{1}{\si}}+1)}{\sqrt{t}},\\
& |\nabla_{e_1}^N \W_{e_2}^N(X)| \leq \frac{C (K^{\frac{2}{\si}}+1)}t,
\label{eq:5.27}  \\
& |\nabla^N_{e_1}\nabla^N_{e_2} u^N(X)| \leq \frac{C(K^{\frac{2}{\si}}+1)}t,
\label{eq:5.29}
\end{align}
for every $e, e_1, e_2$ and $X=(t,\tfrac{x}N)\in \Om_N$.
%indented
%\noindent
Here, $C=C(n, c_\pm, T, \sigma, \|f\|_\infty, \|\fa^{(i)}\|_\infty, u_\pm: i=1,2)$.

Also, \eqref{eq:5.25} implies a $\sigma$-H\"older estimate for 
$\nabla^N_{e_1}\nabla^N_{e_2} u^N(X)$,
that is
\begin{align}
\label{eq:holder-uu}
[\nabla^N_{e_1}\nabla^N_{e_2} u^N]_{\sigma}^{(3-\sigma)} \leq C\big(K^{\frac{3}{\si}} +1\big)
\end{align}
for every $e_1, e_2$, where $C=C(n, c_\pm, T, \sigma, \|f\|_\infty, \|\fa^{(i)}\|_\infty, 
u_\pm: 1\le i \le 3)$ and $\|h\|_\infty:= \|h\|_{L^\infty([u_-,u_+])}$ for a function $h$.
\end{cor}

The H\"older seminorm $[\,\cdot\,]_\si^{(3-\si)}$ in \eqref{eq:holder-uu} is
weaker than $[\,\cdot\,]_\si^{(2)}$ which is expected from \eqref{eq:5.25}.
The reason is that a larger diverging factor appears from $\nabla_{e_1}^N\nabla_{e_2}^N a_{x,e}(t)$
in the Taylor expansion; see the estimates for $J_1$ in the proof. However, in the next 
two corollaries, the diverging factor matches better.

We remark that we can also get an estimate $\langle \nabla_e^N u^N\rangle_{1+\si}^{(1)} \leq C(K^{1+\tfrac{2}{\si}} +1)$ (and therefore for $\langle \nabla_e^N \tilde u^N\rangle^{(1)}_{1+\si}$ by Lemma \ref{lem:Eq-norms}), by elements of the proof of the bound \eqref{eq:holder-uu}, but this is not detailed here.  Similar estimates would also hold in the next two corollaries.

\begin{proof}
To apply Theorem \ref{second_Schauder_thm},
we estimate the right hand sides of \eqref{eq:5.22-B} and \eqref{eq:5.22-C}
in terms of $K$ in the context of \eqref{eq:2.1-X} or \eqref{eq:vN}.  First,
by the H\"older continuity of $u^N$ given in \eqref{eq:cor2.3-1} of Corollary \ref{cor:2.3}
with $\a=\si$, we have
\begin{align}
\label{bar a estimate}
[\bar a]_\si^{*,N} =[\varphi'(u^N)]_\si^{*,N} 
\leq \|\varphi''\|_\infty [u^N]_\si^{*,N} \leq C'(K+1)
\end{align}
where $C'=C'(n, c_\pm, T, \|f\|_\infty, \|\fa''\|_\infty, \|u^N(0)\|_\infty)$.
Next for $g$ in \eqref{eq:5.5-g}, by noting \eqref{eq:disc-der-D},
we have
\begin{align}
\label{G_0^1 estimate}  
|\nabla^N g|_0^{(1),N} \leq K C(\fa,f) |\nabla^N u^N|_0^{(1),N},
\end{align}
where $C(\fa,f) := \|\fa''\|_\infty\|f\|_\infty + \|\fa'\|_\infty\|f'\|_\infty$.
However, by Corollary \ref{cor:K^3}, 
\begin{align}  \label{eq:5.23-B}
|\nabla^N u^N|_0^{(1),N} \le C''(K^{\frac{1}\si} + 1)
\end{align}
where $C''=C''(n, c_\pm, T, \sigma, \|f\|_\infty, \|\fa''\|_\infty, u_\pm)$.
Thus, one can take $\bar A=C'(K+1)$ and $G_0^1= C(\fa, f)C''(K^{1+ \frac{1}\si} + K)$ in the assumptions (B.2) and (B.3), respectively.

Next, for $|\tilde \W_e^N|_0^{(1)}$, seen in \eqref{eq:5.22-B} and \eqref{eq:5.22-C},
\begin{align}  \label{eq:5.24-B}
|\W_e^N|_0^{(1),N} = |\nabla_e^N \varphi(u^N)|_0^{(1),N} 
\leq \|\fa'\|_\infty  |\nabla_e^N u^N|_0^{(1),N}.
\end{align}

Thus, by Lemma \ref{lem:Eq-norms} and \eqref{eq:5.23-B}, we have
\begin{align*}
|\tilde \W_e^N|_0^{(1)} = |\W_e^N|_0^{(1),N} 
 \leq C'''(K^{\frac{1}{\si}}+1)
\end{align*}
where $C'''=C'''(n, c_\pm, T, \sigma, \|f\|_\infty, \|\fa'\|_\infty, \|\fa''\|_\infty, u_\pm)$.
Therefore, we can apply Theorem \ref{second_Schauder_thm} to obtain the desired
bounds: \eqref{eq:5.25} from \eqref{eq:5.22-B}, \eqref{eq:5.28} from \eqref{eq:5.24-B} 
and \eqref{eq:5.27} from \eqref{eq:5.22-C}, respectively.

We now argue the bound \eqref{eq:5.29}.  As $\varphi$ has a strictly increasing $C^2$ inverse, \eqref{eq:5.29} follows from \eqref{eq:5.28} and \eqref{eq:5.27}.
Indeed, set
\begin{align}
\label{a^0 eq}
a^0_{x,e}(t) = \big(a_{x,e}(t)\big)^{-1} \equiv \big(a_{x,e}(u^N(t))\big)^{-1}.
\end{align}
Note, by the Lipschitz property of $\varphi$, and also the proof of Lemma \ref{lem:mvt} and  \eqref{eq:5.24-B}, that
\begin{align}
\label{a^0 calc}
&\tfrac{1}{c_+}\leq a^0_{x,e}\leq \tfrac{1}{c_-} \ \ {\rm and \ \ }\\
&|\nabla^N_e a^0_{x, e'}(t)| = \Big| \frac{\nabla^N_e a_{x, e'}(t)}{a_{x+e, e'}(t)a_{x, e'}(t)} \Big|
\leq C(c_\pm)\max_z|\nabla^N_e\fa(u^N(t, \tfrac{z}{N}) |\nonumber\\
&\hskip 20mm  \leq C(c_\pm, \|\fa'\|_\infty)\max_z |\nabla^N_eu^N(t, \tfrac{z}{N}|\leq \frac{C(c_\pm, \|\fa'\|_\infty)C'''}{\sqrt{t}}(K^{\frac{1}{\sigma}} + 1). \nonumber
\end{align}
Calculate, noting $\nabla^N_{e_2} u^N(X) =  \frac{\xi^N_{e_2}(X)}{a_{x, e}(u^N(t))}=a_{x, e}^0(t)\xi^N_{e_2}(X)$
and by \eqref{eq:disc-der-D},
 \begin{align} \label{second u derivative}
 \nabla^N_{e_1}\nabla^N_{e_2} u^N(X)
 & = \nabla^N_{e_1} \big( a^0_{x,e_2}(t) \W^N_{e_2}(X) \big) \\
   &= \nabla^N_{e_1}a^0_{x,e_2}(t)\cdot \W^N_{e_2}(X+\tfrac{e_1}{N}) + a^0_{x,e_2}(t)\cdot\nabla^N_{e_1}\W^N_{e_2}(X).  \nonumber
\end{align}
Now, the desired estimate \eqref{eq:5.29} holds by inputting \eqref{a^0 calc}, \eqref{eq:5.28} and \eqref{eq:5.27}
 into \eqref{second u derivative}.

Finally, to show the H\"older estimate in \eqref{eq:holder-uu}, by \eqref{second u derivative}, write
\begin{align}
\label{J decomposition}
&\nabla^N_{e_1}\nabla^N_{e_2}u^N(X) - \nabla^N_{e_1}\nabla^N_{e_2} u^N(Y)\\
&\ \ \ = \big(\nabla^N_{e_1}a^0_{x,e_2}(t) - \nabla^N_{e_1}a^0_{y,e_2}(s)\big) \W^N_{e_2}(X+\tfrac{e_1}{N})+ \nabla^N_{e_1}a^0_{y,e_2}(s)\big(\W^N_{e_2}(X+\tfrac{e_1}{N})-\W^N_{e_2}(Y+\tfrac{e_1}{N})\big)\nonumber\\
&\ \ \ \ + \big(a^0_{x,e_2}(t) - a^0_{y,e_2}(s)\big)\nabla^N_{e_1}\W^N_{e_2}(X)
 + a^0_{y,e_2}(s)\big(\nabla^N_{e_1} \W^N_{e_2}(X) - \nabla^N_{e_1}\W^N_{e_2}(Y)\big)\nonumber\\
&\ \ \  =: J_1+J_2+J_3+J_4.\nonumber
\end{align}
Recall, in the following, $|X-Y|^\si = \max\big\{ |\frac{x}{N}-\frac{y}{N}|^\si, |t-s|^{\frac{\si}{2}}\big\}$.

The term $J_4$, as $|a^0|_0 \leq c_-^{-1}$, is bounded via \eqref{eq:5.25}:
\begin{align}
\label{J_4 eq}
|J_4| = a^0_{y, e_2}(s)\big|\nabla^N_{e_1} \W^N_{e_2}(X) - \nabla^N_{e_1}\W^N_{e_2}(Y)\big| \leq \frac{C}{(t\wedge s)^{1+\frac{\si}{2}}}\big(K^{1+\frac{2}{\si}} +1\big)|X-Y|^\si.
\end{align}

To bound $J_3$, note
by the argument for Corollary \ref{cor:2.3} via Lemma \ref{lem:mvt}, that $[a^0]^{*,N}_\sigma \le C(n, c_\pm, T, \|f\|_\infty, \|\fa''\|_\infty)(K+1)$.  Hence,
\begin{align}
\label{a^0 holder}
\big|a^0_{x,e_2}(t)-a^0_{y, e_2}(s)\big| \leq \frac{C}{(t\wedge s)^{\frac{\si}{2}}}(K+1) |X-Y|^{\si}.
\end{align}
Then, by \eqref{a^0 holder} and also the gradient bound \eqref{eq:5.27}, we obtain
\begin{align}
\label{J_3 eq}
|J_3| = \big|a^0_{x, e_2}(t) - a^0_{y, e_2}(s)\big| |\nabla^N_{e_1}\W^N_{e_2}(X)| \leq \frac{C}{(t\wedge s)^{1+\frac{\si}{2}}} (K+1)(K^{\frac{2}{\sigma}}+1)|X-Y|^\sigma.
\end{align}

Recall the definition of $a_{x,e}(t)$ (cf. \eqref{eq:a_xe}).  To deal with the terms $J_1$ and $J_2$, it will be helpful to observe $\nabla^N_e \fa(u^N(t, \tfrac{x}{N})) = a_{x, e}(t)\nabla^N_e u^N(t, \tfrac{x}{N})$ and so, by the Lipschitz property of $\fa$,
that
\begin{align}
\label{phi to u}
&|\W^N_{e}(X)-\W^N_e(Y)| =|\nabla^N_e\fa(u^N(t, \tfrac{x_1}{N})) - \nabla^N_e \fa(u^N(s, \tfrac{x_2}{N}))| \\
&\ \ \leq |a_{x_1, e}(t)-a_{x_2, e}(s)||\nabla^N_e u^N(t, \tfrac{x_1}{N})| + a_{x_2, e}(s)|\nabla^N_e u^N(t, \tfrac{x_1}{N}) - \nabla^N_e u^N(s, \tfrac{x_2}{N})|.\nonumber
\end{align}
Moreover, by Corollary \ref{cor:2.3}, we have
$|a_{x_1, e}(t)-a_{x_2, e}(s)| \leq C(K+1)(t\wedge s)^{-\tfrac{\si}{2}}|X-Y|^\si$ and, by the Schauder estimate and gradient bound in Corollary \ref{cor:K^3}, we have $|\nabla^N_eu^N(X)-\nabla^N_e(Y)|\leq C(t\wedge s)^{-\tfrac{1+\si}{2}}(K^{1+\frac{1}{\si}} +1)|X-Y|^\si$ and $|\nabla^N_e u^N(X)|\leq Ct^{-\tfrac{1}{2}}(K^{\frac{1}{\si}} +1)$. 

 Then, equation \eqref{phi to u} is bounded
\begin{align}
\label{J_2 help}
&|\W^N_{e}(X)-\W^N_e(Y)| \leq 
 \frac{C}{(t\wedge s)^{\frac{\si}{2}}} (K+1)|X-Y|^{\si} \cdot \frac{C}{\sqrt{t}}(K^{\frac{1}{\si}} +1)\\
 &\ \ \ \ \  + \frac{C}{(t\wedge s)^{\frac{1+\si}{2}}} (K^{1+\frac{1}{\si}} +1)|X-Y|^\si. \nonumber
\end{align}

We now bound $J_2$.  By the bound $|\nabla^N_{e_1} a^0_{y,e_2}(s)|\leq C\max_z|\nabla^N_{e_1} u^N(s, \tfrac{z}{N})|\leq \frac{C}{\sqrt{s}}(K^{\frac{1}{\si}} +1)$
in \eqref{a^0 calc}, and \eqref{phi to u}, \eqref{J_2 help},
we have
\begin{align}
\label{J_2 eq}
&|J_2| = |\nabla^N_{e_1}a^0_{y, e_2}(x)| \big| \W^N_{e_2}(X+\tfrac{e_1}{N}) - \W^N_{e_2}(Y+\tfrac{e_1}{N})\big|\\
&\ \ \leq  |\nabla^N_{e_1}a^0_{y, e_2}(x)| \Big[
|a_{x_1, e}(t)-a_{x_2, e}(s)||\nabla^N_e u^N(t, \tfrac{x_1}{N})| \nonumber\\
&\ \ \ \ \ \ + a_{x_2, e}(s)|\nabla^N_e u^N(t, \tfrac{x_1}{N}) - \nabla^N_e u^N(s, \tfrac{x_2}{N})\Big]\nonumber\\
&\ \ \leq \frac{C}{(t\wedge s)^{1 + \frac{\si}{2}}}(K^{1+\frac{2}{\si}} +1)|X-Y|^\si. \nonumber
\end{align}

We now address the term $J_1$.  We rewrite the quantity $\nabla^N_{e_1}a^0_{x,e_2}(t) - \nabla^N_{e_1}a^0_{y,e_2}(s)$ into two terms, one where time $t$ is fixed, and the other where space $y$ is fixed:
\begin{align*}
\nabla^N_{e_1}a^0_{x,e_2}(t) - \nabla^N_{e_1}a^0_{y,e_2}(s)
&= \big(\nabla^N_{e_1}a^0_{x, e_2}(t) - \nabla^N_{e_1}a^0_{y,e_2}(t) \big)\\
& \ \ \ \ + \big(\nabla^N_{e_1}a^0_{y,e_2}(t) - \nabla^N_{e_1}a^0_{y, e_2}(s) \big)\ =: \ J_{1,1} + J_{1,2}.
\end{align*}
To bound $J_{1,1}$, we observe that it has the same form as $\nabla^N_{e_1}a_{x, e_2}(t) - \nabla^N_{e_1}a_{y, e_2}(t)$ except now the dependent variable is $\fa(u^N)$ in the composition $\fa^{-1}(\fa(u^N))$, instead of simply $u^N$.
Hence, the same proof of the H\"older estimate \eqref{eq:b-R} via Lemma \ref{lem:6.5} and Remark \ref{rem:6.3}, applied to $a^0$, gives 
\begin{align}
\label{J11 eq}
|J_{1,1}|
&\leq C(\|\fa''\|_\infty, \|\fa'''\|_\infty)\Big[\max_{e,z}|\nabla^N_{e}\fa\big(u^N(t, \tfrac{x+z}{N})\big) - \nabla^N_{e}\fa\big(u^N(t, \tfrac{y+z}{N})\big)| \\
&\ \ \ \ + \max_{e, e', z, z'} |\nabla^N_{e} \fa\big(u^N(t, \tfrac{z}{N})\big)||\nabla^N_{e'} \fa\big(u^N(t, \tfrac{z'}{N})\big)||X-Y|^\si\Big].\nonumber
\end{align}
Since $\W^N_e = \nabla^N_e u^N$, we may now bound the first term on the right-hand side of \eqref{J11 eq} by \eqref{phi to u} and \eqref{J_2 help}.  The second term on the right-hand side of \eqref{J11 eq}, by the Lipschitz property of $\fa$, is bounded by
$C\max_{e, e', z, z'} |\nabla^N_{e} u^N(t, \tfrac{z}{N})||\nabla^N_{e'} u^N(t, \tfrac{z'}{N})||X-Y|^\si$ and further bounded using the gradient bound for $|\nabla^N_eu^N(X)|\leq Ct^{-\frac{1}{2}}(K^{\frac{1}{\si}} + 1)$ in Corollary \ref{cor:K^3}. 

Then, together, we have
\begin{align}
\label{J11 description}
|J_{1,1}| &\leq 
C\Big[\max_{e, z} |a_{x,e}(t) - a_{y, e}(t)||\nabla^N_e u^N(t, \tfrac{z}{N})|\\
&\ \ + \max_{e,z} |\nabla^N_e u^N(t, \tfrac{x+z}{N})-\nabla^N_e u^N(t, \tfrac{y+z}{N})|\nonumber\\
&\ \ + \max_{e, e', z, z'} |\nabla^N_e u^N(t, \tfrac{z}{N})||\nabla^N_{e'} u^N(t, \tfrac{z'}{N})||X-Y|^\si\Big]\nonumber\\
&\leq \frac{C}{t^{\frac{1+\si}{2}}} (K^{1+\frac{1}{\si}} +1)|X-Y|^\si + \frac{C}{t}(K^{\frac{2}{\si}} +1)|X-Y|^\si \leq \frac{C}{t}(K^{\frac{2}{\si}} +1)|X-Y|^\si.\nonumber
\end{align}
When $|J_{1,1}|$ is multiplied by $|\W^N_{e_2}(X+ \frac{e_1}{N})|$, since 
\begin{align}
\label{J_2 help 1}
|\W^N_{e_2}(X)|\leq \|\fa'\|_\infty |\nabla^N_{e_2}u^N(X)|\leq \frac{C}{\sqrt{t}} (K^{\frac{1}{\si}}+1),
\end{align}
by the gradient bound in Corollary \ref{cor:K^3}, we have
\begin{align}
\label{J11 final eq}
|J_{1,1}|\cdot |\W^N_{e_2}(X+ \tfrac{e_1}{N})| \leq \frac{C}{t^{\frac{3}{2}}}(K^{\frac{3}{\si}}+1)|X-Y|^\si.
\end{align}

The argument with respect to $J_{1,2}$ is analogous.  By the proof of \eqref{eq:b-R time} discussed in Remark \ref{rem time}, we obtain
\begin{align*}
|J_{1,2}| &\leq C(\|\fa''\|_\infty, \|\fa'''\|_\infty)\Big[ \max_{e,z} |\nabla_e^N \fa(u^N(t, \tfrac{z}{N})) -\nabla_e^N \fa(u^N(s, \tfrac{z}{N}))|\\
&\ \ + \max_z \frac{|\fa(u^N(t, \tfrac{z}{N})) - \fa(u^N(s, \tfrac{z}{N}))|}{\sqrt{|t-s|}} \nonumber\\
&\ \ + \max_{e, e', z, z'}\max_{r\geq t\wedge s} |\nabla^N_{e} \fa(u^N(r, \tfrac{z}{N}))||\nabla^N_{e'} \fa(u^N(r, \tfrac{z'}{N}))||X-Y|^\si\Big].\nonumber
\end{align*}
The first term on the right-hand side is bounded via \eqref{phi to u} and \eqref{J_2 help} again.   By the Lipschitz property of $\fa$, the other two terms can be bounded in terms of $u^N$ and the H\"older estimate $|u^N(t, \tfrac{z}{N})-u^N(s, \tfrac{z}{N})| \leq C(t\wedge s)^{-\frac{1+\si}{2}} (K^{1+\frac{1}{\si}}+1)|t-s|^{\frac{1+\si}{2}}$ and gradient bound $|\nabla^N_e u^N(X)|\leq Ct^{-\frac{1}{2}}(K^{\frac{1}{\si}}+1)$ in Corollary \ref{cor:K^3}.

Then,
\begin{align}
\label{J12 help}
|J_{1,2}|& \leq C\Big[\max_{e, z} |a_{x,e}(t) - a_{y, e}(t)||\nabla^N_e u^N(t, \tfrac{z}{N})|\\
&\ \ + \max_{e,z} |\nabla^N_e u^N(t, \tfrac{x+z}{N})-\nabla^N_e u^N(t, \tfrac{y+z}{N})|\nonumber\\
&\ \ +\max_z \frac{|u^N(t, \tfrac{z}{N}) - u^N(s, \tfrac{z}{N})|}{\sqrt{|t-s|}}\nonumber\\
&\ \   + \max_{e, e', z, z'}\max_{r\geq t\wedge s} |\nabla^N_{e} u^N(r, \tfrac{z}{N})||\nabla^N_{e'} u^N(r, \tfrac{z'}{N})||X-Y|^\si\Big]\nonumber\\
&\leq \frac{C}{(t\wedge s)^{\frac{1+\si}{2}}} (K^{1+\frac{1}{\si}} +1)|t-s|^{\frac{\si}{2}} + \frac{C}{t\wedge s}
(K^{\frac{2}{\si}} + 1)|X-Y|^\si. \nonumber
 \end{align}
 Multiplying the estimate of $J_{1,2}$ in \eqref{J12 help} by $|\W^N_{e_2}(X+ \frac{e_1}{N})|$, noting the bound \eqref{J_2 help 1},
we obtain
\begin{align}
\label{J12 final eq}
|J_{1,2}| \cdot  |\W^N_{e_2}(X+ \tfrac{e_1}{N})| 
\leq \frac{C}{(t\wedge s)^{\frac{3}{2}}} (K^{\frac{3}{\si}} +1)|X-Y|^\si.
\end{align}

Hence, by combining the bounds \eqref{J11 final eq} and \eqref{J12 final eq}, we have
$$|J_1| \leq \frac{C}{(t\wedge s)^{\frac{3}{2}}} (K^{\frac{3}{\si}}+1)|X-Y|^\si =  \frac{C}{(t\wedge s)^{\frac{3-\sigma + \sigma}{2}}} (K^{\frac{3}{\si}}+1)|X-Y|^\si.$$

Collecting the bounds for the terms $J_1, J_2, J_3, J_4$ gives \eqref{eq:holder-uu} with respect to a constant $C=C(n, c_\pm, T, \sigma, \|f\|_\infty, \|\fa'\|_\infty, \|\fa''\|_\infty, \|\fa'''\|_\infty, u_\pm)$.  In fact, we remark that the estimate for $J_1$ is the dominant one.
\end{proof}

A better estimate can be derived when the initial value $u^N(0,\cdot)$ has some smoothness, namely when condition \eqref{eq:Th4.2-0} holds:  
$$
\sup_N \| u^N(0)\|_{C_N^2}\leq C_0.
$$
A sufficient condition for \eqref{eq:Th4.2-0}
is when the initial data is between $u_-$ and $u_+$ such that 
$u^N(0,\cdot) = u_0(\cdot)$ for $u_0\in C^2(\T^n)$.

Consider the weighted norm, for fixed $e$,
$$|\tilde \W_e^N|^*_{1+\si} := [\tilde \W_e^N]_{1+\si}^* + \langle \W^N_e\rangle_{1+\si}^* + |\nabla^N\tilde \W^N_e|_0^{(1)},$$
analogous to that given in \eqref{eq:4.3}.

\begin{cor} \label{cor:5.7}
Consider the solution of the discrete PDE \eqref{eq:2.1-X} satisfying \eqref{eq:1.u_pm}.
When $u^N(0,\cdot)$ satisfies the condition \eqref{eq:Th4.2-0}, we have
\begin{align*}
&|\tilde \W_e^N|^*_{1+\si} \leq C K_0^{1+\frac{2}{\si}},  \\
&|\W^N_e(X)| \leq C K_0^{\frac1\si},
\end{align*}
for every $X=(t,\tfrac{x}N)\in \Om_N$ and $e$, and
\begin{align*}
&|\nabla_{e_1}^N \W^N_{e_2}(X)| \leq \frac{C}{\sqrt{t}} K_0^{\frac{2}{\si}},\\ 
&|\nabla^N_{e_1}\nabla^N_{e_2} u^N(X)| \leq  \frac{C}{\sqrt{t}} K_0^{\frac{2}{\si}},
\end{align*}
for every  $e_1, e_2$, where $K_0=K+C_0+1$ and
$C=$ $C(n, c_\pm, T, \si, \|f\|_\infty, \|\fa^{(i)}\|_\infty, u_\pm: i=1,2)$.

Moreover, we have the H\"older estimate:
\begin{align}
\label{cor 5.9 holder}
[\nabla^N_{e_1}\nabla^N_{e_1} u^N]^{(1)}_{\si} \leq C K_0^{\tfrac{3}{\si}},
\end{align}
for every $e_1, e_2$ where
$C=C(n, c_\pm, T, \si, \|f\|_\infty, \|\fa^{(i)}\|_\infty, u_\pm: 1\le i \le 3)$.
\end{cor}

\begin{proof}
When the condition \eqref{eq:Th4.2-0} is satisfied, 
by Corollary \ref{extended_cor}, \eqref{eq:5.23-B} is improved as
\begin{align}  \label{eq:5.38-U}  
\|\nabla^N u^N\|_\infty = |\nabla^N \tilde u^N|_0^{*} 
 \le C K_0^{\frac1\si}.
\end{align}
Then, the desired bound for $\W^N_\cdot$ follows,
\begin{align}
\label{eq:5.38-A}
\|\W^N_\cdot\|_\infty \leq \|\fa'\|_\infty \|\nabla^N u^N\|_\infty\leq 
C K_0^{\frac1\si}
\end{align} 
and also \eqref{G_0^1 estimate} is improved to 
\begin{align}
\label{nabla_g infinity bound}
\|\nabla^N g\|_\infty 
\leq C(\fa, f)K\|\nabla^Nu^N\|_\infty \le 
CK K_0^{\frac1\si}.
\end{align}
In these bounds, $C=C(n, c_\pm, T, \si, \|f\|_\infty, \|\fa'\|_\infty, \|\fa''\|_\infty, u_\pm)$.

We have $\bar A= C(n, c_\pm, T, \|f\|_\infty, \|\fa''\|_\infty, u_\pm) K_0$ 
(cf.\ \eqref{bar a estimate}) in the estimate of $[\bar a]_\sigma^{*,N}$ by Corollary \ref{cor:2.6}.

By Corollary \ref{alternative_rmk}, and the bound \eqref{eq:5.38-A} shown for $\W^N_e$ above, 
 we obtain the desired estimates for $|\tilde \W_e^N|_{1+\si}^*$ and $|\nabla^N_{e_1} \W^N_{e_2}|_0^*$.  Indeed,  by \eqref{rmk 5.1 bounds}, we have
\begin{align*}
& |\tilde \W_e^N|^*_{1+\si}  \leq C K_0^{1+\frac{2}{\si}}
\intertext{and}
& |\nabla^N_{e_1} \W^N_{e_2}|_0^{(1)} \leq C K_0^{\frac{2}{\si}},
\end{align*}
with $C=C(n, c_\pm, T, \si, \|f\|_\infty, \|\fa'\|_\infty, \|\fa''\|_\infty, u_\pm)$.
 
   The desired bound on $\nabla^N_{e_1}\nabla^N_{e_2}u^N$ now follows as in the proof of Corollary \ref{second_Schauder_cor}. Indeed, in the present context, by \eqref{second u derivative}, we need to bound
   $$ \nabla^N_{e_1}a^0_{x,e_2}(t)\cdot \W^N_{e_2}(X+\tfrac{e_1}{N}) + a^0_{x,e_2}(t)\cdot\nabla^N_{e_1}\W^N_{e_2}(X).$$
   The first term, noting \eqref{a^0 calc}, is bounded $|\nabla^N_{e_1} a^0_{x,e_2}\cdot\W^N_{e_2}|\leq C(c_\pm)\max_e\|\W^N_e\|^2_\infty$.  The second term, also noting \eqref{a^0 calc}, is bounded $|a^0_{x,e_2} \nabla^N_{e_1}\W^N_{e_2}|\leq C(c_-)|\nabla_{e_1}\W^N_{e_2}|$.  The estimate follows now by inserting the already proven bounds for $\W^N_e$ and $\nabla^N_{e_1}\W^N_{e_2}$, with respect to a constant $C=C(n, c_\pm, T, \si, \|f\|_\infty, \|\fa'\|_\infty, \|\fa''\|_\infty, u_\pm)$.
   
   Finally, the H\"older bound \eqref{cor 5.9 holder} follows by scheme given in Corollary \ref{second_Schauder_cor}, using the proven bounds of $|\tilde \W_e^N|_{1+\si}^*$ and $|\nabla^N_{e_1}\W^N_{e_2}|_0^*$ and now Corollaries \ref{extended_cor} and \ref{cor:2.6} instead of Corollaries \ref{cor:K^3} and \ref{cor:2.3}. Indeed, consider the decomposition \eqref{J decomposition} of the difference
   $\nabla^N_{e_1}\nabla^N_{e_2}u^N(X) - \nabla^N_{e_1}\nabla^N_{e_2}(Y)= J_1+J_2+J_3+J_4$.  
   By \eqref{J_4 eq}, \eqref{J_3 eq}, and \eqref{J_2 eq}, we may bound $J_4, J_3$ and $J_2$ respectively.  
   By adding \eqref{J11 description} to \eqref{J12 help}, multiplied by $\max_{e'}|\W^N_{e'}|\leq C(\|\fa'\|_\infty)\max_{e'}|\nabla_{e'}^N u^N|_0$, we may bound $J_1$.

 In particular, noting \eqref{a^0 calc} and \eqref{a^0 holder}, with respect to \eqref{J_4 eq} and \eqref{J_3 eq}, we have
 \begin{align}
 \label{J_4, J_3}
 |J_4| + |J_3| & \leq C\Big[ |\nabla^N_{e_1}\W^N_{e_2}(X) - \nabla^N_{e_1}\W^N_{e_2}(Y)|\\
&\ \  + \max_z |u^N(t, \tfrac{z}{N}) - u^N(s, \tfrac{z}{N})|\cdot |\nabla^N_{e_1}\W^N_{e_2}(X)|\Big].\nonumber
\end{align}
 By Corollary \ref{extended_cor}, we may bound $|u^N(t, \tfrac{z}{N}) - u^N(s, \tfrac{z}{N})| \leq C K_0|t-s|^{\tfrac{\si}{2}}$.  Then, noting the already proven bounds for $[\nabla^N_{e_1}\W^N_{e_2}]_\si^{(1)}\leq |\tilde \W_\cdot^N|_{1+\si}^*$ and for $|\nabla^N_{e_1}\W^N_{e_2}(X)|$, we have the further bound
 $$
 |J_4|+|J_3| \leq \frac{C}{(t\wedge s)^{\frac{1+\si}{2}}} K_0^{\frac{3}{\si}}|X-Y|^\si.
 $$

 Also, note the estimate $|\nabla^N_{e} a^0_{x, e'}(t)|\leq C\max_z|\nabla^N_e u^N(t, \tfrac{z}{N})|$ in \eqref{a^0 calc}, $a_{x,e}\leq c_+$, and $|a_{x, e}(t)-a_{y, e}(s)|\leq C K_0|X-Y|^\si$ in Corollary \ref{cor:2.6}.  Then, by the uniform bounds for $u^N$ in Corollary \ref{extended_cor}, we have, with respect to \eqref{J_2 eq}, that
 \begin{align}
 \label{J2:cor5.9}
 |J_2| &\leq C|\nabla^N_{e_1}u^N|_0\Big[|a_{x_1, e}(t)-a_{x_2, e}(s)||\nabla^N_e u^N(t, \tfrac{x_1}{N})| \\
 &\hskip 25mm
  + a_{x_2, e}(s)|\nabla^N_e u^N(t, \tfrac{x_1}{N}) - \nabla^N_e u^N(s, \tfrac{x_2}{N})|\Big]\nonumber\\
 &\leq C K_0^{\frac{3}{\si}}|X-Y|^\si.\nonumber
 \end{align}
 
 Finally, from adding \eqref{J11 description} to \eqref{J12 help} and then multiplying by $\max_{e'}|\nabla_{e'}^Nu^N|_0$,
 we have
 \begin{align*}
 |J_1|&\leq C\max_{e'}|\nabla^N_{e'} u^N|\big(J_{1,1} + J_{1,2}\big)\\
 &\leq C\max_{e'}|\nabla^N_{e'} u^N|
\Big[\max_{e, z} |a_{x,e}(t) - a_{y, e}(t)||\nabla^N_e u^N(t, \tfrac{z}{N})|\\
&\ \ + \max_{e,z} |\nabla^N_e u^N(t, \tfrac{x+z}{N})-\nabla^N_e u^N(t, \tfrac{y+z}{N})|\nonumber\\
&\ \ +\max_z \frac{|u^N(t, \tfrac{z}{N}) - u^N(s, \tfrac{z}{N})|}{\sqrt{|t-s|}}\nonumber\\
&\ \   + \max_{e, e', z, z'}\max_{r\geq t\wedge s} |\nabla^N_{e} u^N(r, \tfrac{z}{N})||\nabla^N_{e'} u^N(r, \tfrac{z'}{N})||X-Y|^\si\Big].\nonumber
\end{align*}
By the bounds for $u^N$ in Corollary \ref{extended_cor}, we obtain
\begin{align}
\label{J1:cor5.9}
|J_1| \leq C K_0^{\frac{3}{\si}} |X-Y|^\si.
\end{align}

Note that the constant $C=C(n, c_\pm, T, \si, \|f\|_\infty, \|\fa^{(i)}\|_\infty, u_\pm: 1\le i \le 3)$, 
with respect to $J_1$, whereas with respect to $J_2, J_3, J_4$ the constant does not depend on $\|\fa'''\|_\infty$.
Therefore, combining the bounds for $J_1, J_2, J_3, J_4$, we arrive at the desired H\"older estimate for $\nabla^N_{e_1}\nabla^N_{e_2} u^N$, with constant $C=C(n, c_\pm, T, \si, \|f\|_\infty, \|\fa^{(i)}\|_\infty, u_\pm: 1\le i \le 3)$.
\end{proof}

Still more can be said with a more regular initial condition at the level of 
Theorem \ref{second_Schauder_time_thm}.  Consider the initial value 
\begin{align}
\label{initial_data}
\W_e^N(0,\tfrac{x}{N}) = \nabla^N_e \varphi(u^N(0,\tfrac{x}{N})),
\end{align}
where 
\begin{align}  \label{eq:5.40-A}
\sup_N \|u^N(0)\|_{C_N^4}\le \bar C_0<\infty.
\end{align}
Recall \eqref{eq:norm-uN} for the discrete $C^4$-norm.
A sufficient condition, when $\varphi\in C^4$, is that $u^N(0,\tfrac{x}{N})=u_0(\tfrac{x}{N})$ for $x\in \T^n_N$ and $u_0\in C^4(\T^n)$.
 
We now observe that $\sup_N \|\varphi(u^N(0, \cdot))\|_{C^4_N}\leq \mathcal{C}_0$ and so $\sup_N \|\W_e^N(0)\|_{C_N^3} \le \mathcal{C}_0$ where 
\begin{align}
\label{C_1 eq}
\mathcal{C}_0 = C(\|\fa^{(i)}\|_\infty: 1\leq i\leq 4)\big[\bar C_0 + \bar C_0^2 + \bar C_0^3 +\bar C_0^4\big].
\end{align}
Indeed, recalling \eqref{eq:a_xe} and denoting $a_{x,e_3}(0):= a_{x,e_3}(u^N(0))$, by \eqref{eq:disc-der-D},
the second derivative of $\xi_{e_3}^N(0)$ is written as
\begin{align*}
& \nabla^N_{e_1} \nabla^N_{e_2}  \xi_{e_3}^N(0,\tfrac{x}{N}) \\
& \quad = \nabla^N_{e_1}\nabla^N_{e_2}\nabla^N_{e_3} \varphi(u^N(0,\tfrac{x}{N})) \\
& \quad = \nabla^N_{e_1}\nabla^N_{e_2} \big[ a_{x,e_3}(0) \nabla^N_{e_3} u^N(0,\tfrac{x}{N}) \big] \\
& \quad = [\nabla^N_{e_1}\nabla^N_{e_2}a_{x,e_3}(0)]\nabla^N_{e_3}u^N(0,\tfrac{x+e_2+e_1}{N}) 
+ [\nabla^N_{e_2} a_{x,e_3}(0)]\nabla^N_{e_1}\nabla^N_{e_3}u^N(0,\tfrac{x+e_2}{N}) \\
&\qquad + [\nabla^N_{e_1} a_{x, e_3}(0)]\nabla^N_{e_2}\nabla^N_{e_3}u^N(0,\tfrac{x+e_1}{N})
+ a_{x, e_3}(0)\nabla^N_{e_1}\nabla^N_{e_2}\nabla^N_{e_3}u^N(0,\tfrac{x}{N}).
\end{align*}
An analogous but longer expression can be written for $\nabla^N_{e_1}\nabla^N_{e_2}\nabla^N_{e_3}\W^N_{e_4}(0, \tfrac{x}{N})$.

Uniform bounds over $e_1, e_2, e_3$ of $\nabla^N_{e_1} a_{x,e_4}(0)$, $\nabla^N_{e_1}\nabla^N_{e_2}a_{x,e_4}(0)$ and $\nabla^N_{e_1}\nabla^N_{e_2}\nabla^N_{e_3} a_{x, e_4}(0)$, in terms of $\varphi$ and $\bar C_0$, follow now from `mean-value' formulas \eqref{eq:3.nab-fa'} (or Lemma \ref{lem:mvt}) and \eqref{eq:3.nab-nab-fa'}, as well as Remark \ref{third derivative} in the next section.

We remark that we will use that $\fa\in C^5$ in the following corollary.  
This is the only place where the property that
$\fa$ is {\it five} times continuously differentiable is used.

\begin{cor}
\label{second_Schauder_time_cor}
Consider the solution of the discrete PDE \eqref{eq:2.1-X} satisfying \eqref{eq:1.u_pm}.
When $u^N(0,\cdot)$ satisfies the condition \eqref{eq:5.40-A}, we have, for every $e_1, e_2$, that
\begin{align}  \label{eq:5.32}
& |\tilde \W_{e_1}^N|_{1+\si}\leq C\big[\bar K_0^{1+\frac{2}{\si}} (1+\bar C_0^{24}) + 
\bar K_0^{\frac{2}{\si}} \bar C_0^{48}\big],\\
\label{eq:5.33}
&|\nabla^{N}_{e_1} \tilde \W^N_{e_2}|_0 \leq C\big[\bar K_0^{\frac{2}{\si}} (1+\bar C_0^{24}) + 
\bar K_0^{\frac{2}{\si}-1} \bar C_0^{48}\big],\\
\label{eq:5:36}
&|\nabla^N_{e_1}\nabla^N_{e_2} u^N(X)| \leq C\big[\bar K_0^{\frac{2}{\si}} (1+\bar C_0^{24}) + 
\bar K_0^{\frac{2}{\si}-1} \bar C_0^{48}\big],\\
\label{cor 5.10 holder}
&[\nabla^N_{e_1}\nabla^N_{e_2} u^N]_\si \leq C\big[\bar K_0^{1+\frac{2}{\si}} (1+\bar C_0^{24}) + 
\bar K_0^{\frac{2}{\si}} \bar C_0^{48} + \bar K_0^{\frac{3}{\si}}\big].
\end{align}
Here, $\bar K_0 = K + \bar C_0 +1$ and 
$C=C(n, c_\pm, T, \si, \|f\|_\infty, \|\fa^{(i)}\|_\infty, u_\pm: 1\leq i \leq 5)$.
% \|\fa''\|_\infty, \|\fa'''\|_\infty, \|\fa^{(4)}\|_\infty, \|\fa^{(5)}\|_\infty, u_\pm)$.
\end{cor}

\begin{proof}
We apply Theorem \ref{second_Schauder_time_thm} with $\bar a = \varphi'(u^N)$ and $g = K \bar a  (t, \tfrac{x}N) f(u^N (t, \tfrac{x}N))$ as in \eqref{eq:5.3-a} and \eqref{eq:5.5-g}.
First, we see the condition \eqref{a-holder} for this $\bar a$ as
$$
[\bar a]_\sigma^{(-\sigma),N} \leq \bar B = C\big(K+\bar C_0 +1\big)=: C\bar K_0
$$ 
by the proof of Corollary \ref{cor:2.6}, where $C=C(n, c_\pm, T, \|f\|_\infty, \|\fa''\|_\infty, \|u^N(0)\|_\infty)$.  

Next, let us check \eqref{a-grad_bounds}.  We have
$|\nabla^{N}_e \bar a(0)|\leq \|\fa''\|_\infty |\nabla^N_e u^N(0)|$.  Also, we have
$|\nabla^{N}_{e'}\nabla^{N}_e \bar a(0)| =|\nabla^{N}_{e'} \big(a^1_{x, e}(0)\nabla^N_eu^N(0)\big)|$, $|\nabla^N_{e''}\nabla^N_{e'}\nabla^N_e \bar a(0)| = |\nabla^N_{e''}\nabla^N_{e'} \big(a^1_{x,e}(0)\nabla^N_e u^N(0)\big)|$ and in addition $|\nabla^N_{e'''}\nabla^N_{e''}\nabla^N_{e'}\nabla^N_e \bar a(0)| = |\nabla^N_{e'''}\nabla^N_{e''}\nabla^N_{e'} \big(a^1_{x,e}(0)\nabla^N_e u^N(0)\big)|$,
where 
\begin{align}
\label{a^1 eq}
a^1_{x,e} = \left\{\begin{array}{rl}
\frac{\fa'(u^N(X+\frac{e}{N}) - \fa'(u^N(X))}{u^N(X+\frac{e}{N}) - u^N(X)} & {\rm when \ } u^N(X+\frac{e}{N}) \neq u^N(X)\\
\fa''(u^N(X))& {\rm when \ }u^N(X+\frac{e}{N}) =u^N(X).
\end{array}\right.
\end{align}
Note $\|a^1_{x,e}\|_\infty \leq \|\fa''\|_\infty$ and $\|\nabla^{N}_{e'}a^1_{x,e}(0)\|_\infty \leq C(\|\fa'''\|_\infty)\|\nabla^N_{e'} u^N(0)\|_\infty$ by the proof of Lemma \ref{lem:mvt}.  Moreover, $|\nabla^N_{e''}\nabla^N_{e'}a^1_{x, e}(0)|  \leq C(\|\fa'''\|_\infty, \|\fa^{(4)}\|_\infty)\sum_{j=1}^2\|u^N(0)\|_{C^2_N}^j$ by the proof of Lemma \ref{lem:6.5}.  Also, we have $|\nabla^N_{e'''}\nabla^N_{e''}\nabla^N_{e'}a^1_{x, e}(0)|  \leq C(\|\fa^{(i)}\|_\infty: 2\leq i \leq 5)\sum_{j=1}^3\|u^N(0)\|_{C^3_N}^j$, applying Remark \ref{third derivative} to $a^1$.
Hence, by \eqref{eq:5.40-A}, we have \eqref{a-grad_bounds} in the form
\begin{align}
\label{5 times diff}
&\|\nabla^N_e \bar a(0)\|_\infty + \|\nabla^{N}_{e'}\nabla^N_e \bar a(0)\|_\infty + \|\nabla^N_{e''}\nabla^N_{e'}\nabla^N_e \bar a(0)\|_\infty +  \|\nabla^N_{e'''}\nabla^N_{e''}\nabla^N_{e'}\nabla^N_e \bar a(0)\|_\infty \\
&\ \ \ \ \leq C(\|\fa^{(i)}\|_\infty: 2\le i \le 5) \big(\bar C_0 +\bar C_0^2 + \bar C_0^3 + \bar C_0^4\big)\nonumber\\
&\ \ \ \  = \mathcal{C}_1\leq C(\fa)\mathcal{C}_0.\nonumber
\end{align}
 
Also, by \eqref{nabla_g infinity bound}, we see the condition \eqref{g-bound} for $g$:
\begin{align*}
\|\nabla^N g\|_\infty 
\leq C(\fa, f)K\|\nabla^Nu^N\|_\infty \le CK \bar K_0^{\frac1\si} = G^1_\infty
\end{align*}
 and, by \eqref{eq:5.38-A}, we have the condition \eqref{W-bound}:
 $$
 \|\W^N_e\|_\infty \leq \|\fa'\|_\infty \|\nabla_e^N u^N\|_\infty
 \leq C \bar K_0^{\frac1\si}= W_\infty.
 $$
 Here, $C=C(n, c_\pm, T, \si, \|f\|_\infty, \|\fa'\|_\infty, \|\fa''\|_\infty, 
u_\pm)$.
 
Now, by Theorem \ref{second_Schauder_time_thm}, recalling $\bar K_0$,
noting $\alpha=\si$ and  $\mathcal{C}_1\leq C(\fa)\mathcal{C}_0$, we have
\begin{align}  \label{eq:5.32-B}
& |\tilde \W_{e_1}^N|_{1+\si}\leq C\big[ (\bar K_0+1)^{1+\frac{1}{\si}}\big({W}_\infty + \mathcal{D}_6\big) + {G}^1_\infty + \mathcal{D}_7\big],\\
\label{eq:5.33-B}
&|\nabla^{N}_{e_1} \tilde \W^N_{e_2}|_0 \leq C\big[ (\bar K_0 +1)^{\frac{1}{\si}}\big({W}_\infty + \mathcal{D}_6\big)+ (\bar K_0 +1)^{-1}({G}^1_\infty + \mathcal{D}_7)\big],
\end{align}
with $\mathcal{D}_6 \leq (\bar K_0 + 1)^{\frac{1}{\si}}\mathcal{C}_0(1+\mathcal{C}_0^5)$ and
$\mathcal{D}_7 \leq (\bar K_0 +1)^{\frac{2}{\si}}\mathcal{C}_0(1+\mathcal{C}_0^{11}).$
Therefore, bounds \eqref{eq:5.32} and \eqref{eq:5.33} follow, noting $\bar K_0+1 \le 2\bar K_0$
and the forms of $G^1_\infty$ and $W_\infty$,  and $\mathcal{C}_0(1+\mathcal{C}_0^5)
\le C(\bar C_0+\bar C_0^{24})$, $\mathcal{C}_0(1+\mathcal{C}_0^{11})
\le C(\bar C_0+\bar C_0^{48})$.

To derive \eqref{eq:5:36}, as in the proof of Corollaries \ref{second_Schauder_cor} and \ref{cor:5.7}, we have
$$ 
\nabla^N_{e_1}\nabla^N_{e_2} u^N(X)
= \nabla^N_{e_1}a^0_{x,e_2}(t)\cdot \W^N_{e_2}(X+\tfrac{e_1}{N}) 
+ a^0_{x,e_2}(t)\cdot\nabla^N_{e_1}\W^N_{e_2}(X),
$$
where
$|\nabla^N_{e_1}a^0_{x,e_2}\cdot \W^N_{e_2}|\leq C(c_-,\|\fa''\|_\infty)\max_e \|\W^N_e\|^2_\infty$ and 
$|a^0_{x,e_2} \cdot \nabla^N_{e_1}\W^N_{e_2}|\leq C(c_-)|\nabla_{e_1}^N\W^N_{e_2}|_0$.  
In particular, by \eqref{eq:5.33-B}, we have
\begin{align}
\label{eq:5:36-B}
&|\nabla^N_{e_1}\nabla^N_{e_2} u^N(X)| \leq C\big[ {W}_\infty^2 + 
(\bar K_0+1)^{\frac{1}{\si}}\big({W}_\infty + \mathcal{D}_6\big)+ (\bar K_0 +1)^{-1}
({G}^1_\infty + \mathcal{D}_7)\big].
\end{align}
The bound \eqref{eq:5:36} now follows from the form of $W_\infty$ above.

Finally, the H\"older estimate \eqref{cor 5.10 holder} follows as in Corollary \ref{cor:5.7} (through the method given for Corollary \ref{second_Schauder_cor}), using now \eqref{eq:5.32-B} and \eqref{eq:5.33-B} to bound $J_3$ and $J_4$.  Indeed, $J_1, J_2$ have the same estimates \eqref{J1:cor5.9} and \eqref{J2:cor5.9} as in the proof of Corollary \ref{cor:5.7}.

To bound $J_3$ and $J_4$, by \eqref{J_4, J_3} and the estimate $|u^N(t, \cdot) - u^N(s, \cdot)| 
\leq C \bar K_0|t-s|^{\frac{\si}{2}}$ by Corollary \ref{extended_cor}, we have
\begin{align*}
|J_4| + |J_3| \leq C\Big[ |\nabla^N_{e_1}\W^N_{e_2}(X) - \nabla^N_{e_1}\W^N_{e_2}(Y)|
 + \bar K_0|\nabla^N_{e_1}\W^N_{e_2}(X)||X-Y|^\si\Big].
\end{align*}
We now input the H\"older bound for $\nabla^N_e \W^N_{e'}$ in \eqref{eq:5.32-B} and the gradient bound in \eqref{eq:5.33-B}, and note that the term with the H\"older bound is dominant:
\begin{align*}
& \bar K_0 \big[\bar K_0^{\frac{1}{\si}}(W_\infty + \mathcal{D}_6) 
+ \bar K_0^{-1}(G^1_\infty + \mathcal{D}_7)\big].
\end{align*}
Hence,  
\begin{align*}
|J_4|+|J_3| &\leq C\Big[\bar K_0^{1+\frac{1}{\si}}(W_\infty + \mathcal{D}_6) + G^1_\infty 
+ \mathcal{D}_7\Big] |X-Y|^\si.
\end{align*}
The following bound \eqref{cor 5.10 holder-B} now follows by adding the (equal) bounds \eqref{J1:cor5.9} and \eqref{J2:cor5.9} for $J_1$ and $J_2$ to the bound for $|J_4|+|J_5|$:
\begin{align}
\label{cor 5.10 holder-B}
[\nabla^N_{e_1}\nabla^N_{e_2} u^N]_\si 
\leq C\Big\{ \bar K_0^{1+\frac{1}{\si}}\big({W}_\infty + \mathcal{D}_6\big) + {G}^1_\infty + \mathcal{D}_7+(\bar K_0+1)^{\frac{3}{\si}} \Big\}
\end{align}
and this implies the bound \eqref{cor 5.10 holder} similarly as above.
\end{proof}

\section{Construction and estimates of fundamental solutions}
\label{sec:EELevi}

This section presents a different approach to the discrete Schauder estimates.
We rely on the fundamental solutions constructed by the parametrix method of E.E.\ Levi.
This route is also well-known in PDE theory to derive the Schauder
estimates; see Friedman \cite{Fri64} Chapter 1, Il'in,  Kalashnikov and Oleinik
\cite{IKO}, Lady\v{z}enskaja, Solonnikov and Ural'ceva \cite{LSU} p.356-- 
and Eidel'man \cite{Ei}.  Cannizzaro and Matetski \cite{CM} used Schauder estimates 
shown in this direction to study discrete KPZ equation.

We first construct the fundamental solutions of the linear difference operators 
of parabolic type and derive the estimate on its spatial discrete derivatives
in Section \ref{sec:13.1}.  Then, we apply it to the quasilinear discrete PDE 
\eqref{eq:2.1-X} in Section \ref{sec:13.2-Q}.

\subsection{Fundamental solutions of linear parabolic difference operators}
\label{sec:13.1}

Let $L\equiv L_{t,x}$ be a linear (elliptic) discrete difference operator of second order 
with coefficients depending on $(t,\tfrac{x}N)$ and let
$L^*$ be its dual operator with respect to the inner product $\lan f, g\ran = \frac1{N^n}
\sum_{x\in \T_N^n} f(\tfrac{x}N) g(\tfrac{x}N)$. 

\begin{defn}  \label{defn:fund}
{\rm (i)} We call $p = p(s,\tfrac{y}N;t,\tfrac{x}N), s\le t,$ a fundamental solution
(in the forward sense) of $L-\partial_t$, if $(L_{t,x}-\partial_t)p=0$
($L$ is acting on the variable $\tfrac{x}N$) and
$p(s,\tfrac{y}N;s,\tfrac{x}N)= N^n \de_{xy}$, where $\de_{xy}$ is the Kronecker's $\de$.\\
{\rm (ii)} 
We call $p^* = p^*(s,\tfrac{y}N;t,\tfrac{x}N)$, $s\le t$,
 a fundamental solution (in the backward sense) 
of $L^*+\partial_s$, if $(L_{s,y}^* + \partial_s)p^*=0$
and $p^*(t,\tfrac{y}N;t,\tfrac{x}N)= N^n \de_{xy}$.
\end{defn}

We recall the following identity from \cite{Fri64}, Chapter 1, Theorem 15.

\begin{lem} \label{lem:13.2}
$p(s,\tfrac{y}N;t,\tfrac{x}N)= p^*(s,\tfrac{y}N;t,\tfrac{x}N).$
\end{lem}

\begin{proof} For $s\le r \le t$, set
$$
q(r) := \tfrac1{N^n}\sum_z p(s,\tfrac{y}N;r,\tfrac{z}N) p^*(r,\tfrac{z}N;t,\tfrac{x}N).
$$
Then, its derivative in $r$ vanishes:
\begin{align*}
\partial_r q(r)  = \tfrac1{N^n}\sum_z \Big\{
L_{r,z} p(s,\tfrac{y}N;r,\tfrac{z}N) p^*(r,\tfrac{z}N;t,\tfrac{x}N)-
p(s,\tfrac{y}N;r,\tfrac{z}N) L_{r,z}^* p^*(r,\tfrac{z}N;t,\tfrac{x}N)\Big\}=0.
\end{align*}
Thus, we have
$p^*(s,\tfrac{y}N;t,\tfrac{x}N) = q(s) = q(t)= p(s,\tfrac{y}N;t,\tfrac{x}N)$.
\end{proof}

We specifically consider as $L=L_t$ the linear difference operator of the following form:
\begin{align}  \label{eq:L_t}
L_tu(\tfrac{x}N) = - \tfrac12 \sum_{|e|=1}a_e(t,\tfrac{x}N) \nabla_e^{N,*}\nabla_e^N u(\tfrac{x}N)
+ \sum_{|e|=1}b_e(t,\tfrac{x}N) \nabla_e^N u(\tfrac{x}N)
+ c_e(t,\tfrac{x}N) u(\tfrac{x}N).
\end{align}
The coefficients are continuous in $t$ and satisfy the following conditions: 
\begin{align} \label{eq:6.2-A}
\begin{aligned}
& 0<c_-\le a_e(t,\tfrac{x}N) \le c_+<\infty, \qquad a_e(t,\tfrac{x}N)=a_{-e}(t,\tfrac{x}N), \\
& |b_e(t,\tfrac{x}N)|, |c_e(t,\tfrac{x}N)| \le D_0,
\end{aligned}
\end{align}
and H\"older continuity:
\begin{align} 
& | a_e(t,\tfrac{x}N) - a_e(s,\tfrac{y}N)| \le C_H (|t-s|^{\a/2} + |\tfrac{x}N-\tfrac{y}N|^\a), 
\notag \\
& | b_e(t,\tfrac{x}N) - b_e(t,\tfrac{y}N)| \le C_H |\tfrac{x}N-\tfrac{y}N|^\a, 
\label{eq:13:Holder}\\
& | c(t,\tfrac{x}N) - c(t,\tfrac{y}N)| \le C_H |\tfrac{x}N-\tfrac{y}N|^\a,
\notag
\end{align}
for some $\a\in (0,1)$ and $D_0, C_H>0$, recall that $|\cdot|$ denotes the distance in $\T^n$
and defined by \eqref{eq:distanceTn}.
Writing $C_H$ for $C_H\vee D_0$, we assume $D_0=C_H$ for simplicity.
The symmetry of $a_e$ in $e$ is assumed without loss of generality, since we have
\begin{align*}
\sum_{|e|=1}a_e \nabla_e^{N,*}\nabla_e^N = \sum_{|e|=1}a_e \nabla_{-e}^N\nabla_e^N
= \sum_{|e|=1}a_e \nabla_{e}^N\nabla_{-e}^N = \sum_{|e|=1}a_{-e} \nabla_e^{N,*}\nabla_e^N
\end{align*}
so that we may take a symmetrization $\tfrac12 (a_e+a_{-e})$ for $a_e$.

Applying the classical parametrix method due to Levi,
we can construct the fundamental solution $p(s,\tfrac{y}N;t,\tfrac{x}N)$, $0\le s <t$, 
(in the forward sense) of the operator $\partial_t -L_t$ and obtain the following estimate
on its (discrete) derivative in $x$. 

\begin{prop} \label{prop:13.4}
The fundamental solution exists and we have the estimate
\begin{align}  \label{eq:prop-13.4}
\big|\nabla_{e,x}^N p(s,\tfrac{y}N;t,\tfrac{x}N)\big| 
\le \frac{C}{\sqrt{t-s}} \, e^{CC_H^{2/\a}} \, g(c(t-s),\tfrac{x-y}N),
 \quad 0\le s < t \le T,
\end{align}
for some $C=C(n, c_\pm, \alpha, T)$, where $g(t,\tfrac{x}N) := 
t^{-n/2} e^{-|\frac{x}N|^2/t}$ with $|\tfrac{x}N|$ defined in \eqref{eq:distanceTn}.
\end{prop}

\begin{proof}
For each fixed $(s, \tfrac{y}N)$, let $Z(s,\tfrac{y}N;t,\tfrac{x}N), t>s,$ be the 
fundamental solution of $\mathcal{L}_{a(s, \frac{y}N)} := \De^N_{a(s, \frac{y}N)}
-\partial_t$, where $\De^N_{a(s, \frac{y}N)}$ is the discrete Laplacian with
the constant coefficient $a(s, \tfrac{y}N) = (a_e(s, \tfrac{y}N))_{|e|=1}$ defined by
\eqref{eq:Lap-a-N}.  Recall that $\De_a^N$ is defined  for
$a=(a_e)_{|e|=1}\in \R^{2n}$ such that $c_-\le a_e\le c_+$ and $a_e=a_{-e}$.  
Namely, $\mathcal{L}_{a(s, \frac{y}N)}
Z(s,\tfrac{y}N;\cdot,\cdot)=0$ for $t>s$, and $Z(s,\tfrac{y}N;s,\cdot)= N^n \de_{\frac{y}N}(\cdot)$.  
Since $\T_N^n$ is spatially homogeneous, let us define $\bar p$ by
$\bar p(t-s, \tfrac{x-y}N) = Z(s,\tfrac{y}N;t,\tfrac{x}N)$.
Note that $\bar p(t,\tfrac{x}N) = N^{n} \sum_{z\in \T_N^n}p^*(N^2 t,x+z)$ 
with $p^*$ in \cite{DD}.  By (1.9) of \cite{DD}
for $p^* (=p_b)$, we have Aronson estimate:
$$
\bar p(t,\tfrac{x}N) \le c g(kt,\tfrac{x}N).
$$
Moreover, Gaussian bounds are known on the 
first and second discrete derivatives of $Z$:
\begin{align}  \label{eq:Z-der-1-Q}
& |\nabla_{e,x}^N Z(s,\tfrac{y}N;t,\tfrac{x}N)| \le c_1 \frac{g(k_1(t-s),\frac{x-y}N)}{\sqrt{t-s}}, \\
& |\nabla_{e',x}^N\nabla_{e,x}^N Z(s,\tfrac{y}N;t,\tfrac{x}N)| \le c_2 \frac{g(k_2(t-s),\frac{x-y}N)}{t-s},
\label{eq:Z-der-2-Q}
\end{align}
which follow from (1.4) and (1.5) of \cite{DD}; we take $N^2(t-s)$
for $t-s$ and estimate $1 \vee N^2(t-s) \ge N^2(t-s)$.  Note that here we are 
in a very simple situation with the discrete Laplacian $\De_a^N$ with constant 
coefficient $a$, and without a random environment as in \cite{DD}.
The condition $a\in \mathcal{A}(c_-,c_+,n,\Ga)$ in \cite{DD} is satisfied due to
the symmetry assumption $a_e=a_{-e}$ in our case. 
The constants $c, c_1, c_2, k, k_1, k_2$ depend only on $n$ and $c_\pm$.

Let $\Phi(s,\tfrac{y}N;t,\tfrac{x}N)$ be the solution of the integral equation
\begin{align}  \label{eq:Phi-5-5-Q}
\Phi(s,\tfrac{y}N;t,\tfrac{x}N) = \mathcal{L} Z(s,\tfrac{y}N;t,\tfrac{x}N) 
+ \int_s^t \tfrac1{N^n}\sum_z \mathcal{L} Z(r,\tfrac{z}N;t,\tfrac{x}N)
\cdot \Phi(s,\tfrac{y}N;r,\tfrac{z}N) dr,
\end{align}
where $\mathcal{L} = L_t-\partial_t$
acts on the variable $(t,\tfrac{x}N)$ (cf.\ (4.1) of \cite{Fri64}).
Then, from (2.8) of \cite{Fri64} (or (11.13) of \cite{LSU}),
\begin{align}  \label{eq:13.16-p}
p(s,\tfrac{y}N;t,\tfrac{x}N) := Z(s,\tfrac{y}N;t,\tfrac{x}N) + 
\int_s^t \tfrac1{N^n}\sum_z Z(r,\tfrac{z}N;t,\tfrac{x}N)
\cdot \Phi(s,\tfrac{y}N;r,\tfrac{z}N) dr
\end{align}
is the fundamental solution of $\mathcal{L} = L_t-\partial_t$.  

The equation \eqref{eq:Phi-5-5-Q} can be solved as
\begin{align}  \label{eq:13.17}
\Phi(s,\tfrac{y}N;t,\tfrac{x}N) = \sum_{k=1}^\infty (\mathcal{L} Z)_k(s,\tfrac{y}N;t,\tfrac{x}N),
\end{align}
where $(\mathcal{L} Z)_1=\mathcal{L} Z$ and
\begin{align}  \label{eq:13.18}
(\mathcal{L} Z)_{k+1} (s,\tfrac{y}N;t,\tfrac{x}N) = \int_s^t \tfrac1{N^n}\sum_z \mathcal{L} Z(r,\tfrac{z}N;t,\tfrac{x}N)
\cdot (\mathcal{L} Z)_k(s,\tfrac{y}N;r,\tfrac{z}N) dr;
\end{align}
see (4.4) and (4.5) of \cite{Fri64} (or (11.23) and (11.24) of \cite{LSU}).
Then, noting that $\mathcal{L} Z= L_t Z - \De^N_{a(s,\frac{y}N)} Z$, 
from the bounds with $c_\pm, D_0 (=C_H)$ and H\"older continuity \eqref{eq:13:Holder} of
the coefficients of $L_t$ in \eqref{eq:L_t} and using \eqref{eq:Z-der-2-Q}, we have
\begin{align}  \label{eq:13.19-1}
|\mathcal{L} Z(s,\tfrac{y}N;t,\tfrac{x}N)| \le 
C(n,c_\pm)(1+C_H)
\Big\{ \big| t-s \big|^{\frac{\a}2}
+ \big| \tfrac{x-y}N \big|^\a \Big\}
\frac{g(k_2(t-s),\tfrac{x-y}N)}{t-s}.
\end{align}
However, one can estimate $g$ as
\begin{align*}  
g(k_2t,z) = (k_2t)^{-\frac{n}2} e^{-\frac{|z|^2}{k_2t}}
\le C_a t^{-\frac{n}2} \Big(\tfrac{|z|^2}t\Big)^{-a}
\end{align*}
for every $a>0$.  First, taking $a=\tfrac12(d+2-2\mu-\a)$, in terms of a parameter $\mu$,
\begin{align}\label{eq:13.19-2}
\big| t-s \big|^{\frac{\a}2}
\frac{g(k_2(t-s),\tfrac{x-y}N)}{t-s}
& \le C_a(t-s)^{\frac{\a}2-1-\frac{n}2}
\Big(\tfrac{|\tfrac{x-y}N|^2}{t-s}\Big)^{-\tfrac12(n+2-2\mu-\a)}\\
&= C_a\frac1{(t-s)^\mu} \frac1{|\tfrac{x-y}N|^{n+2-2\mu-\a}}.   \notag
\end{align}
Next,  taking $a=\tfrac12(d+2-2\mu)$,
\begin{align}\label{eq:13.19-3}
\big|\tfrac{x-y}N\big|^\a
\frac{g(k_2(t-s),\tfrac{x-y}N)}{t-s}
& \le C_a\big|\tfrac{x-y}N\big|^\a (t-s)^{-1-\frac{n}2}
\Big(\tfrac{|\tfrac{x-y}N|^2}{t-s}\Big)^{-\tfrac12(n+2-2\mu)}\\
&= C_a\frac1{(t-s)^\mu} \frac1{|\tfrac{x-y}N|^{n+2-2\mu-\a}}.   \notag
\end{align}
Thus, from the estimates \eqref{eq:13.19-1}, \eqref{eq:13.19-2} and \eqref{eq:13.19-3},
we have shown
\begin{align}  \label{eq:tildeLZ-Q}
|\mathcal{L} Z(s,\tfrac{y}N;t,\tfrac{x}N)| \le \frac{C(n, c_\pm, \mu)(1+C_H)}
{(t-s)^\mu}
\frac1{|\tfrac{x}N-\tfrac{y}N|^{n+2-2\mu-\a}}
\qquad (1-\tfrac{\a}2<\mu<1),
\end{align}
which corresponds to (4.3) of \cite{Fri64}.  

Then, by induction using \eqref{eq:13.18}, one can get for some $\la_0^*, C_1, C_2>0$,
depending on $n, c_\pm, \alpha$:
\begin{align*} 
|(\mathcal{L} Z)_m(s,\tfrac{y}N;t,\tfrac{x}N)| \le C_1 \frac{(C_2C_H)^m}{\Ga(\tfrac{m\a}2)}
(t-s)^{\frac{m\a}2-\frac{n}2-1} 
\exp\Big\{ -\frac{\la_0^*|\tfrac{x}N-\tfrac{y}N|^2}{4(t-s)}\Big\};
\end{align*}
see (11.25) of \cite{LSU} p.362 and (4.14) of \cite{Fri64}.
Indeed, we replace the integral estimates such as Lemmas 2 and 3
in \cite{Fri64} by those on Riemann sums uniformly in $N$.  Also the estimates
on $\frac{\partial Z}{\partial x_i}, \frac{\partial^2 Z}{\partial x_i \partial x_j}$
in (4.10), (4.11) of \cite{Fri64} are replaced by those on discrete derivatives
$\nabla_{e_i}^N Z, \nabla_{e_i}^N \nabla_{e_j}^N Z$ given in \eqref{eq:Z-der-1-Q},
\eqref{eq:Z-der-2-Q}.  We note that two $m\a$ in (4.14) of \cite{Fri64}
 should be $\frac{m\a}2$ as in \cite{LSU}.  In fact, the estimate on
$(\mathcal{L} Z)_2$ above (4.14) is shown with $\frac{2\a}2$.  In our computation,
we are especially concerned with the dependence of the estimate on 
the constant $C_H$, which should be $C_H^m$ as above from \eqref{eq:13.18}.
Therefore, from \eqref{eq:13.17}, we obtain (cf.\ (4.15) of \cite{Fri64}),
in terms of $C_3, C_4$ depending on $n, c_\pm, \alpha, T$,
\begin{align} \label{eq:13.20}
|\Phi(s,\tfrac{y}N;t,\tfrac{x}N)| \le \frac{C_3 e^{C_4 C_H^{2/\a}} }{(t-s)^{(n+2-\a)/2}}
\exp\Big\{ -\frac{\la_0^*|\tfrac{x}N-\tfrac{y}N|^2}{4(t-s)}\Big\},
\end{align}
by estimating $(t-s)^{\frac{m\a}2} \le C_T^m (t-s)^{\frac{\a}2}$ for $m\ge 1$ and
$0\le t-s\le T$, and
\begin{align*} 
\sum_{m=1}^\infty \frac{(C_2C_H)^m}{\Ga(\tfrac{m\a}2)}
\le C \sum_{n=1}^\infty \frac{(C_2C_H)^{\frac2{\a} n}} {n!}
\le C e^{(C_2C_H)^{2/\a}}.
\end{align*} 

Further estimates are similar to \cite{Fri64}. We only note that
the estimates on the Riemann sums are uniform in $N$ as noted above,
and by \eqref{eq:13.16-p}, we get the estimate \eqref{eq:prop-13.4} on $\nabla_{e,x}^N p$. 
\end{proof}

\begin{rem}
A weaker estimate than \eqref{eq:prop-13.4} under the average in $x$ 
(i.e., under $N^{-n}\sum_x$) might be available due to the result of \cite{DD}. Indeed,
such an estimate can be shown at least if the coefficient $a$ is independent of $t$
so that it is temporally homogeneous, because the spatial shift acts ergodically.
\end{rem}

\begin{rem}
If the coefficients have a singularity $t^{-\si/2}$ at $t=0$ in \eqref{eq:13:Holder} 
like in \eqref{eq:cor2.3-2}, we face the following difficulty.  Estimates \eqref{eq:13.19-1}
and \eqref{eq:tildeLZ-Q} can be shown with an additional factor $s^{-\si/2}$ in
the right hand side.  Thus, from \eqref{eq:Phi-5-5-Q},
$\Phi(s,\tfrac{y}N;r,\tfrac{z}N)$ has a singularity $s^{-\si/2}$ near $s=0$.
This is inherited by $p^N$ so that the estimates diverge at $s=0$.
\end{rem}

\subsection{Application to the nonlinear problem}
\label{sec:13.2-Q}

We now apply Proposition \ref{prop:13.4} to show the Schauder estimate
for the solution $u^N=u^N(t,\tfrac{x}N), x \in \T_N^n$ of the
quasilinear discrete PDE \eqref{eq:2.1-X}.  One can rewrite
the operator $L_{a(u^N(t))}^N$ in \eqref{eq:1-L_a},  determined from 
$\De^N\fa(u^N(t))$ as in \eqref{eq:De-Lta}, into the form of \eqref{eq:L_t}.  
Indeed, recalling $\nabla^{N,*}_e= \nabla^{N}_{-e}$ and then using 
\eqref{eq:disc-der-D}, we have
\begin{align}  \label{eq:6.33-Q}
L_{a(t)}^Nu(\tfrac{x}N) &
= -\tfrac12 \sum_{|e|=1} \nabla^{N}_{-e} \big(a_{x,e}(t) \nabla^N_e u\big)(\tfrac{x}N) \\
& = -\tfrac12 \sum_{|e|=1} \Big\{ a_{x-e,e}(t) \nabla^{N}_{-e}\nabla^N_e u 
(\tfrac{x}N)+ \nabla^{N}_{-e} a_{x,e}(t) \cdot \nabla^N_e u (\tfrac{x}N)\Big\}.
\notag
\end{align}
Therefore, we may take 
$a_e(t,\tfrac{x}N)=\tfrac14(a_{x-e,e}(t) + a_{x+e,-e}(t))$, $b_e(t,\tfrac{x}N)=
\tfrac12 \nabla^{N}_{-e} a_{x,e}(t)$ and $c_e(t,\tfrac{x}N)=0$ in \eqref{eq:L_t}.
Unfortunately, as $b_e$ is given as a discrete derivative of $a_{x,e}(t)$, we need to use the first Schauder estimate for $u^N(t)$ to estimate it.

In the following, assuming the first Schauder estimate, we illustrate how one can
derive the second Schauder estimate based on the estimates for the fundamental solutions.
Constants $C$ may change line to line, but these are referenced to earlier results.

We will assume the initial value $u_0\in C^2(\T^n)$.  Then, $\|u^N(0)\|_{C_N^2}$ is 
uniformly bounded in $N$ and $K$.  By \eqref{eq:cor2.6-2} in 
Corollary \ref{cor:2.6}, $a_e(X)$ satisfies
\begin{equation}  \label{eq:a-H-Q-R}
|a_e(t_1,\tfrac{x_1}N) - a_e(t_2,\tfrac{x_2}N)|
\le CK \Big\{ \big| t_1-t_2 \big|^{\frac{\si}2}
+ \big| \tfrac{x_1}N - \tfrac{x_2}N \big|^\si \Big\}, \quad t_1, t_2 \ge 0,
\end{equation}
where we simply write $K$ instead of $K+1$ by assuming $K\ge 1$.
Furthermore, by Corollary \ref{extended_cor}, under the assumption that
$u_0\in C^2(\T^n)$, we have
$[\tilde u^N]_{1+\si} (\le |\tilde u^N|_{1+\si}) \le CK^{1+\frac1\si}$ and $|\nabla^N_e u^N|\leq CK^{\frac{1}{\si}}$.
From this, as we will indicate in Remarks \ref{rem:6.3} and \ref{rem time},
$b_e$ satisfies the H\"older continuity
\begin{align}  \label{eq:b-R}
&|b_e(t,\tfrac{x_1}N) - b_e(t,\tfrac{x_2}N)|
\le CK^{\frac2\si} \big| \tfrac{x_1}N - \tfrac{x_2}N \big|^\si, \quad t \ge 0, \\
\label{eq:b-R time}
&|b_e(t,\tfrac{x}N) - b_e(s,\tfrac{x}N)|
\le CK^{\frac2\si} |t-s|^{\tfrac{\si}{2}}, \quad x\in \T^n_N.
\end{align}
The constant $D_0$ in \eqref{eq:6.2-A} can be taken as $D_0=CK^{\frac1\si}$ by 
\eqref{eq:3.nab-fa'} below and Corollary \ref{extended_cor}.
Note that \eqref{eq:b-R time} is unnecessary when we apply Proposition \ref{prop:13.4}.
It is used to estimate $J_{1,2}$ in the proof of Corollary \ref{second_Schauder_cor}.

Thus, $a_e$ and $b_e$ satisfy the conditions \eqref{eq:6.2-A} and \eqref{eq:13:Holder}
with $C_H=CK^{\frac2\si}$ and $\a=\si$.  Therefore, by Proposition \ref{prop:13.4}
and also by Lemma \ref{lem:13.2} noting that $L_{a(u^N(t))}^N$ is
symmetric, we have the estimates on the discrete derivatives
both in $x$ and $y$ of the fundamental solution $p(s,\tfrac{y}N;t,\tfrac{x}N)$
of $L_{a(u^N(t))}^N-\partial_t$:
\begin{align}  \label{eq:3.35-A}
\big|\nabla_{e,x}^N p(s,\tfrac{y}N;t,\tfrac{x}N)\big|,
\big|\nabla_{e,y}^N p(s,\tfrac{y}N;t,\tfrac{x}N)\big|
\le \frac{C}{\sqrt{t-s}} \, e^{CK^{4/\si^2}} \, g(c(t-s),\tfrac{x-y}N),
\end{align}
for $0\le s < t \le T$. 

With the gradient estimate on $p$ in $y$ in hand, we can derive
the estimates on $\nabla_{e_1}^N\nabla_{e_2}^N u^N$ for the solution $u^N$ of \eqref{eq:2.1-X}
based on Duhamel's formula and computations on commutators
of difference operators.  The gradient estimate of $p$ in $x$ would be useful to show the first 
Schauder estimate, but here we have assumed it.

Our result is the following.  Note that this bound is much worse compared to that obtained in Corollary 
\ref{second_Schauder_time_cor}, though we assume only $u_0\in C^2$ not in $C^5$ here.

\begin{prop}  \label{prop.6.4}
Suppose the condition \eqref{eq:1.u_pm} holds at $t=0$ and $u_0\in C^2(\T^n)$.  Then, we have
\begin{align}
|\nabla^N_{e_1}\nabla_{e_2}^Nu^N(t,\tfrac{x}N)|\le C(\|\nabla^N\nabla^N u^N(0)\|_{L^\infty}
+1) \exp\{C e^{CK^{4/\si^2}}\}, \label{eq:lem3.3-2}  
\end{align}
for all $t\in [0,T]$, $x\in \T_N^n$, and unit vectors $e_1$, $e_2$.
%$|e_1|=|e_2|=1$.  

\noindent 
Here, the constant $C=C(n, c_\pm, T, \sigma, \|f\|_\infty, \|\fa''\|_\infty, \|\fa'''\|_\infty, u_\pm)>0$.
\end{prop}

We defer the proof of Proposition \ref{prop.6.4} until after we develop some estimates.
We now compute the commutator $[\nabla_e^N,L_a^N]$ and see that it
has a gradient form.

\begin{lem}  \label{lem:3.4}
We have
\begin{equation}  \label{eq:3.nabla-De}
\nabla_e^N L_a^Nu(\tfrac{x}N) = L_a^N \nabla_e^Nu(\tfrac{x}N)
+ \sum_{|e'|=1, e'>0} \nabla_{e'}^N \left\{ \nabla_e^Na_{x,e'}(u) \t_{e-e'}\nabla_e^N u(\tfrac{x}N)\right\},
\end{equation}
and the bound
\begin{equation}  \label{eq:3.nab-fa'}
|\nabla_e^Na_{x,e'}(u)| \le C \{ |\nabla_e^N u(\tfrac{x}N)| 
+ |\nabla_e^Nu(\tfrac{x-e'}N)|\},
\end{equation}
where $C=\frac12 \|\fa''\|_{L^\infty([u_-,u_+])}$.
\end{lem}

\begin{proof}
The identity \eqref{eq:3.nabla-De} is shown as
\begin{align*}
&\nabla_e^NL_a^Nu(\tfrac{x}N) 
 = \sum_{|e'|=1, e'>0} \nabla_e^N\nabla_{e'}^N 
   \left\{ a_{x,e'}(u) \t_{-e'}\nabla_{e'}^N u(\tfrac{x}N)\right\} \\
&\ \  = \sum_{|e'|=1, e'>0} \nabla_{e'}^N\nabla_e^N 
   \left\{ a_{x,e'}(u) \t_{-e'}\nabla_{e'}^N u(\tfrac{x}N)\right\} \\
&\ \  = \sum_{|e'|=1, e'>0} \nabla_{e'}^N \left\{  
   a_{x,e'}(u) \nabla_e^N \t_{-e'}\nabla_{e'}^N u(\tfrac{x}N)
  + \t_{e} \t_{-e'}\nabla_{e'}^N u(\tfrac{x}N) \cdot \nabla_e^N a_{x,e'}(u) \right\} \\
&\ \  = L_a^N \nabla_e^N u(\tfrac{x}N) + \sum_{|e'|=1, e'>0} \nabla_{e'}^N\{
  \t_{e -e'}\nabla_{e'}^N u(\tfrac{x}N) \cdot \nabla_e^N a_{x,e'}(u)\}.
\end{align*}
Next, to show the bound \eqref{eq:3.nab-fa'}, we write
\begin{align}  \label{eq:6.38}
\nabla_e^N a_{x,e'}(u) 
& = N \left\{ \frac{\fa(u(x+e))-\fa(u(x+e-e'))}{u(x+e)-u(x+e-e')}
-  \frac{\fa(u(x))-\fa(u(x-e'))}{u(x)-u(x-e')}  \right\},
\end{align}
where we omit $\tfrac1N$ for the variables of $u$ in the right hand side.
Then, applying Lemma \ref{lem:mvt} with $a= u(x+e)$, $b=u(x+e-e')$,
$c=u(x)$, $d=u(x-e')$, we obtain the bound \eqref{eq:3.nab-fa'}.
\end{proof}

For the solution $u^N=u^N(t,\frac{x}N)$ of the discrete PDE \eqref{eq:2.1-X}, using
\eqref{eq:3.nabla-De} with $e=e_2$, we have
\begin{align}  \label{eq:3.t-nab-nab}
\partial_t \nabla_{e_1}^N \nabla_{e_2}^N u^N
& = \nabla_{e_1}^N \nabla_{e_2}^N L_a^N u^N 
+ K \nabla_{e_1}^N \nabla_{e_2}^N f(u^N)    \\
& = \nabla_{e_1}^N L_a^N \nabla_{e_2}^N u^N
+ \sum_{e'} \nabla_{e_1}^N \nabla_{e'}^N \{ \nabla_{e_2}^N a_{x,e'}(u^N)  \t_{e_2-e'} \nabla_{e_2}^N u^N\} 
\notag\\ & \qquad
+ K \nabla_{e_1}^N \{g_{x,e_2}(u^N) \nabla_{e_2}^N u^N\}.   \notag
\end{align}
Here, $g_{x,e}(u)$ is defined from $f(u)$ in a similar way to $a_{x,e}(u)$ in \eqref{eq:a_xe}. 
For the second term, by \eqref{eq:disc-der-D}, we expand
\begin{align}  \label{eq:3.18}
& \nabla_{e_1}^N \{ \nabla_{e_2}^N a_{x,e'}(u^N)  \t_{e_2-e'} \nabla_{e_2}^N u^N\} \\
& =\nabla_{e_1}^N \nabla_{e_2}^N a_{x,e'}(u^N)  \t_{e_2-e'+e_1} \nabla_{e_2}^N u^N
+ \nabla_{e_2}^N a_{x,e'}(u^N)  \t_{e_2-e'} \nabla_{e_1}^N\nabla_{e_2}^N u^N.  \notag
\end{align}
%Similar to the argument for \eqref{eq:6.38}, 
We now state the next lemma,
in which $e_1\in \Z^d$ does not need to have norm $1$.

\begin{lem}  \label{lem:6.5}
We have
\begin{align}  \label{eq:3.nab-nab-fa'}
\nabla_{e_1}^N \nabla_{e_2}^N a_{x,e'}(u) 
=  & \tfrac12 \fa''(u(\tfrac{x}N)) \big( \nabla_{e_1}^N \nabla_{e_2}^N u(\tfrac{x}N) 
- \nabla_{e_1}^N \nabla_{e_2}^N u(\tfrac{x-e'}N)\big) \\
& + \fa''(u(\tfrac{x+e_1}N)) \nabla_{e_1}^N \nabla_{e_2}^N u(\tfrac{x}N)
+ R^N(x).  \notag
\end{align}
where $R^N(x)$ is a sum of quadratic functions of  $\{\nabla_i^Nu^N(\frac{z}N); i=e_1, e_2, 
e_1+e_2, -e', z=x, x+e_1, x+e_2, x+e_1+e_2\}$, which is explicitly given in the proof.
In particular, in $R^N(x)$, the squares of $\nabla_{e_1}^Nu^N(\frac{z}N)$
and $\nabla_{e_1+e_2}^Nu^N(\frac{z}N)$ do not appear.
\end{lem}

\begin{proof}
Omitting $\frac1N$ in the variables again,
 the left hand side of \eqref{eq:3.nab-nab-fa'} is rewritten as
\begin{align*} 
& 
N^2 \Big\{ 
\frac{\fa(u(x+e_1+ e_2))-\fa(u(x+e_1+ e_2-e'))}{u(x+e_1+ e_2)-u(x+e_1+ e_2-e')}
 - \frac{\fa(u(x+e_2))-\fa(u(x+e_2-e'))}{u(x+e_2)-u(x+e_2-e')}     \\
& \qquad\qquad   -  \frac{\fa(u(x+e_1))-\fa(u(x+e_1-e'))}{u(x+e_1)-u(x+e_1-e')}
+ \frac{\fa(u(x))-\fa(u(x-e'))}{u(x)-u(x-e')}  
\Big\}   \\
&  = 
N^2 \Big\{
 \fa'(u(x+e_1+ e_2)) \\
 &\hskip 15mm + \tfrac12 \fa''(u(x+e_1+ e_2)) (u(x+e_1+ e_2)-u(x+e_1+ e_2-e'))  \\
& \hskip 15mm
+ \tfrac16 \fa'''(u_1^*) (u(x+e_1+ e_2)-u(x+e_1+ e_2-e'))^2   \\
& \qquad - \fa'(u(x+e_2)) - \tfrac12 \fa''(u(x+e_2)) (u(x+e_2)-u(x+ e_2-e'))  \\
& \hskip 15mm- \tfrac16 \fa'''(u_2^*) (u(x+ e_2)-u(x+ e_2-e'))^2  \\
& \qquad- \fa'(u(x+e_1)) - \tfrac12 \fa''(u(x+e_1)) (u(x+e_1)-u(x+e_1-e')) \\
& \hskip 15mm- \tfrac16 \fa'''(u_3^*) (u(x+e_1)-u(x+e_1-e'))^2    \\
 &\qquad+ \fa'(u(x)) + \tfrac12 \fa''(u(x)) (u(x)-u(x-e'))
+ \tfrac16 \fa'''(u_4^*) (u(x)-u(x-e'))^2
\Big\},
\end{align*}
by the mean value theorem for some $u_1^*,.., u_4^*\in \R$.

The sum of the terms containing $\frac16 \fa'''$, written as $R_1^N(x)$, is given by
\begin{align*} 
R_1^N(x)= &  \tfrac16 \fa'''(u_1^*) (\nabla_{-e'}^{N} u(x+e_1+ e_2))^2 
- \tfrac16 \fa'''(u_2^*) (\nabla_{-e'}^{N} u(x+e_2))^2  \\
&  - \tfrac16 \fa'''(u_3^*) (\nabla_{-e'}^{N} u(x+e_1))^2 
+ \tfrac16 \fa'''(u_4^*) (\nabla_{-e'}^{N} u(x))^2.
\end{align*}

The terms containing $\fa'$ are summarized as
\begin{align*} 
& N^2\{ \fa'(u(x+e_1+ e_2))  - \fa'(u(x+e_1)) - \fa'(u(x+e_2))+ \fa'(u(x)) \}\\
& = N^2\Big\{ 
\fa''(u(x+e_1)) (u(x+e_1+ e_2)-u(x+e_1)) \\
&\qquad\qquad + \tfrac12 \fa'''(u_5^*)(u(x+e_1+ e_2)-u(x+e_1))^2 \\
& \qquad\qquad - \fa''(u(x)) (u(x+e_2)-u(x)) - \tfrac12 \fa'''(u_6^*)(u(x+ e_2)-u(x))^2\Big\} \\
& = N\Big\{ \fa''(u(x+e_1)) \nabla_{e_2}^N u(x+e_1) - \fa''(u(x)) \nabla_{e_2}^N u(x) \Big\} \\
&\qquad\qquad + \tfrac12 \fa'''(u_5^*)(\nabla_{e_2}^N u(x+e_1))^2 
- \tfrac12 \fa'''(u_6^*)(\nabla_{e_2}^N u(x))^2.
\end{align*}
However, the first term is rewritten as
\begin{align*} 
& \fa''(u(x+e_1)) \nabla_{e_1}^N\nabla_{e_2}^N u(x) 
+ N \big(\fa''(u(x+e_1)) - \fa''(u(x))\big) \nabla_{e_2}^N u(x) \\
&\qquad = \fa''(u(x+e_1)) \nabla_{e_1}^N\nabla_{e_2}^N u(x) 
+ \fa'''(u_7^*) \nabla_{e_1}^N u(x) \nabla_{e_2}^N u(x).
\end{align*}
Thus, the terms containing $\fa'$ are rewritten as
\begin{align*} 
 \fa''(u(x+e_1)) \nabla_{e_1}^N\nabla_{e_2}^N u(x)  + R_2^N(x),
\end{align*}
where 
\begin{align*} 
R_2^N(x)= \tfrac12 \fa'''(u_5^*)(\nabla_{e_2}^N u(x+e_1))^2 
- \tfrac12 \fa'''(u_6^*)(\nabla_{e_2}^N u(x))^2
+ \fa'''(u_7^*) \nabla_{e_1}^N u(x) \cdot \nabla_{e_2}^N u(x).
\end{align*}

Finally, the terms containing $\frac12\fa''$ are summarized as
\begin{align*} 
& \tfrac12  N^2 \Big\{ \fa''(u(x+e_1+ e_2)) (u(x+e_1+ e_2)-u(x+e_1+ e_2-e'))  \\
& \qquad\qquad  - \fa''(u(x+e_2)) (u(x+e_2)-u(x+ e_2-e'))  \\
& \qquad\qquad -  \fa''(u(x+e_1)) (u(x+e_1)-u(x+e_1-e')) 
+ \fa''(u(x)) (u(x)-u(x-e'))  \Big\}  \\
&= \tfrac12  N^2\fa''(u(x)) \Big\{ (u(x+e_1+ e_2)-u(x+e_1+ e_2-e'))\\
&\hskip 30mm
- (u(x+e_2)-u(x+ e_2-e')) \\
& \hskip 30mm
- (u(x+e_1)-u(x+e_1-e')) + (u(x)-u(x-e'))  \Big\} + R_3^N(x)  \\
& = \tfrac12 \fa''(u(x))\big\{ \nabla_{e_1}^N\nabla_{e_2}^N u(x) - \nabla_{e_1}^N\nabla_{e_2}^N u(x-e')\big\}
 + R_3^N(x),
\end{align*}
where $R_3^N(x)$ is given and rewritten as
\begin{align*}
R_3^N(x) = 
& - \tfrac12  N \big\{ \fa''(u(x+e_1+ e_2)) - \fa''(u(x))\big\} \nabla_{-e'}^Nu(x+e_1+ e_2)\\
&\quad + \tfrac12  N \big\{ \fa''(u(x+ e_2)) - \fa''(u(x))\big\} \nabla_{-e'}^Nu(x+ e_2)\\
& \quad + \tfrac12  N \big\{ \fa''(u(x+e_1)) - \fa''(u(x))\big\} \nabla_{-e'}^Nu(x+e_1)\\
=
& - \tfrac12  \fa'''(u_8^*) \nabla_{e_1+ e_2}^N u(x) \nabla_{-e'}^Nu(x+e_1+ e_2)\\
& \quad + \tfrac12  \fa'''(u_9^*) \nabla_{e_2}^N u(x) \nabla_{-e'}^Nu(x+ e_2)\\
& \quad + \tfrac12  \fa'''(u_{10}^*) \nabla_{e_1}^N u(x) \nabla_{-e'}^Nu(x+ e_1).
\end{align*}
This completes the proof of \eqref{eq:3.nab-nab-fa'} with $R^N(x) = R_1^N(x)+R_2^N(x)+R_3^N(x)$.
\end{proof}

\begin{rem} \label{rem:6.3}\rm
First recall that \eqref{eq:3.nab-nab-fa'} holds for any $e_1\in \Z^d$, i.e., not necessarily
$|e_1|=1$.  Taking $e_1=x_2-x_1$, $e_2=-e$, $e'=e$, $x=x_1$ and multiplying by $\frac1N$, 
\eqref{eq:3.nab-nab-fa'} implies that
\begin{align*}
b_e(t,\tfrac{x_2}N) - b_e(t,\tfrac{x_1}N)
=  & \tfrac14 \fa''(u(\tfrac{x_1}N)) \big\{\nabla_{-e}^N u(\tfrac{x_2}N) 
- \nabla_{-e}^N u(\tfrac{x_1}N) \} \\ 
& - \tfrac14 \fa''(u(\tfrac{x_1}N)) \big\{\nabla_{-e}^N u(\tfrac{x_2-e}N) 
- \nabla_{-e}^N u(\tfrac{x_1-e}N) \} \\ 
& + \tfrac12\fa''(u(\tfrac{x_2}N)) \big\{\nabla_{-e}^N u(\tfrac{x_2}N) 
- \nabla_{-e}^N u(\tfrac{x_1}N) \} 
+ \tfrac1{2N} R^N(\tfrac{x_1}N).
\end{align*}
For the first three terms, we recall \eqref{eq:1.u_pm}.  Then,
we can apply the first Schauder estimate
$[\tilde u^N]_{1+\si} (\le |\tilde u^N|_{1+\si}) \le CK^{1+\frac1\si}$
shown in Corollary \ref{extended_cor} under the assumption that
$u_0\in C^2(\T^n)$, and see that these three terms are bounded by 
$CK^{1+\frac1\si} \big| \tfrac{x_1}N - \tfrac{x_2}N \big|^\si$.
The last term is a sum of quadratic functions of $\{\nabla_i^Nu^N(\tfrac{z}N)\}$
as indicated in Lemma \ref{lem:6.5}.  The term with $i=e_1=x_2-x_1$ is bounded as
\begin{align*}
|\nabla_{e_1}^N u^N(\tfrac{z}N)| 
& = \Big|\sum_{k=1}^{|x_2-x_1|} \nabla_{\tilde e_k}u^N(\tfrac{z_k}N)\Big|
\le CK^{\frac1\si} |x_2-x_1|,
\end{align*}
by Corollary \ref{extended_cor}   
by choosing $\tilde e_k:|\tilde e_k|=1$ and $z_k$ properly, where
$|x_1-x_2|= |x_1-x_2|_{L^\infty}$.  Similar for the term
with $i=e_1+e_2$.  These terms do not appear in squares as we noted.
From these observation, for $x_1\not= x_2$,
\begin{align*}
\tfrac1N |R^N(\tfrac{x_1}N)| \le \tfrac1N C \{K^{\frac1\si} |x_1-x_2|
\cdot K^{\frac1\si} + (K^{\frac1\si})^2 \}
\le 2CK^{\frac2\si} |\tfrac{x_1}N - \tfrac{x_2}N|.
\end{align*}
Noting $\si<1$ so that $K^{1+\frac1\si}\le K^{\frac2\si}$,
we obtain \eqref{eq:b-R} when the time variable is fixed.
\end{rem}

\begin{rem}  \label{rem time}
\rm
We comment, by the same style of proof given for Lemma \ref{extended-Holder-thm}, we obtain
\begin{align}
\label{eq:rem1}
&\big|\nabla^N_{e_2}a^N_{x, e_1}(t) - \nabla^N_{e_2}a^N_{x, e_1}(s)\big|
\leq C(\|\fa''\|_\infty) \max_{e, y} \big| \nabla^N_{e} u^N(t, \tfrac{y}{N}) - \nabla^N_{e}u^N(s, \tfrac{y}{N})\big| \\
&\hskip 35mm 
+ \frac{C(\|\fa'''\|_\infty)}{N}\max_{e,e', z, z'}\max_{r\geq t\wedge s}\big|\nabla^N_{e}u^N(r, \tfrac{z}{N})\big|\big|\nabla^N_{e'}u^N(r, \tfrac{z'}{N})\big|.\nonumber
\end{align}
When $\tfrac{1}{N} \leq |t-s|^{\frac{\si}{2}}$, \eqref{eq:rem1} and Corollary \ref{extended_cor} give the formula \eqref{eq:b-R time} when the space variable is fixed.
On the ther hand, when $\tfrac{1}{N} \geq |t-s|^{\frac{\si}{2}}$, 
we reduce, by the estimate on $\langle u^N\rangle_{1+\si}$ in Corollary \ref{extended_cor}
and also Lemma \ref{lem:mvt}, that
\begin{align}
\label{eq:rem2}
\big|\nabla^N_{e_2}&a^N_{x, e_1}(t)  - \nabla^N_{e_2}a^N_{x, e_1}(s)\big|\\
&\leq C(\|\fa'\|_\infty)N\big[ |u^N(t, \tfrac{x+e_1}{N}) - u^N(s, \tfrac{x+e_1}{N})| + |u^N(t, \tfrac{x}{N})-u^N(s, \tfrac{x}{N})|\big]\nonumber\\
&\leq C(K^{1+\frac{1}{\si}} +1)\cdot N |t-s|^{\frac{1+\si}{2}} \leq C(K^{1+\frac{1}{\si}} +1) |t-s|^{\tfrac{1}{2}}  \notag \\
& \leq C(K^{1+\frac{1}{\si}} +1) |t-s|^{\frac{\si}{2}},\nonumber
\end{align}
yielding \eqref{eq:b-R time} in this case also.
\end{rem}

\begin{rem}
\label{third derivative}
\rm
Moreover, recalling $\|\fa^{(i)}\|_\infty= \|\fa^{(i)}\|_{L^\infty([u_-,u_+])}$ as in 
Corollary \ref{cor:2.3}, we claim that
\begin{align}
\label{third derivative eq}
|\nabla^N_{e_1}\nabla^N_{e_2}\nabla^N_{e_3} a_{x, e_4}(t)| &\leq
C(\|\fa^{(i)}\|_\infty: 1\leq i\leq 4) \sum_{j=1}^3\|u^N(t, \cdot)\|_{C_N^3}^j.
\end{align}
Indeed, omitting $\tfrac{1}{N}$ in the variables again, write
\begin{align*}
a_{x, e_4}(t) &= \fa'(u^N(t,x)) + \frac{1}{2N}\fa''(u^N(t,x))\nabla^N_{e_4}u^N(t,x)\\
& + \frac{1}{3! N^2}\fa'''(u^N(t,x))\big[\nabla^N_{e_4}u^N(t,x)\big]^2 + \frac{1}{4! N^3} \fa^{(4)}(z^N(x; t, e_4))\big[\nabla^N_{e_4} u^N(t,x)\big]^3\\
&= I_1+I_2+I_3 + I_4
\end{align*}
where
$z^N(x; t, e_4): \T^n_N\rightarrow \R$ is the value between $u^N(t, x)$ and $u^N(t, x+e_4)$ (closest to $u^N(t,x)$ if more than one) with respect to the mean-value theorem.
Then, we need to bound $\nabla^N_{e_1}\nabla^N_{e_2}\nabla^N_{e_3} I_i$ for $1\leq i\leq 4$.

Recall the notation for $a^1$ in \eqref{a^1 eq}.  Note, by \eqref{eq:disc-der-D},
\begin{align*}
\nabla^N_{e_1}\nabla^N_{e_2}\nabla^N_{e_3} I_1 &= \nabla^N_{e_1}\nabla^N_{e_2} (a^1_{x, e_3} 
\nabla^N_{e_3} u^N(t, x))\\
& = \nabla^N_{e_1}\big[\nabla^N_{e_2} a^1_{x, e_3} \cdot \nabla^N_{e_3} u^N(t, x+e_2) + a^1_{x, e_3}\nabla^N_{e_2}\nabla^N_{e_3} u^N(t,x)\big].
\end{align*}
By the proof of Lemma \ref{lem:6.5}, 
$|\nabla^N_{e_1}\nabla^N_{e_2}a^1_{x, e_3}| \leq C\sum_{j=1}^2 \|u^N(t, \cdot)\|_{C_N^2}^j$, with $a^1_{\cdot, \cdot}$ in place of $a_{\cdot, \cdot}$.  Also, by the proof of Lemma \ref{lem:mvt}, $|\nabla^N_{e_1}a^1_{x, e_3}| \leq C\|\nabla^N u^N\|_\infty  \equiv C \max_e \|\nabla_e^N u^N\|_\infty$.
Hence, using \eqref{eq:disc-der-D} again, 
$|\nabla^N_{e_1}\nabla^N_{e_2} \nabla^N_{e_3}I_1| \leq C\sum_{j=1}^3\|u^N(t,\cdot)\|_{C_N^3}^j$.

To same bound for $\nabla^N_{e_1}\nabla^N_{e_2}\nabla^N_{e_3}I_2$, because of the prefactor $N^{-1}$ and $\nabla^N_{e_1} u(x) = N\big(u(x+e_1)-u(x)$, would follow from the estimate
$|\nabla^N_{e_2}\nabla^N_{e_3}\big[ \fa''(u^N(t,x))\nabla^N_{e_4}u^N(t,x)\big]| \leq C\sum_{j=1}^3 \|u^N(t,\cdot)\|_{C_N^3}^j$.  
We may write $\nabla^N_{e_3} \fa''(u^N(t,x)) = a^2_{x, e_3} \nabla^N_{e_3} u^N(t,x)$ where $a^2$ is defined with $\fa''$ instead of $\fa'$ as in the definition of $a^1$.  Then, the same type of argument with respect to the term $I_1$ gives the desired bound.
 
The same bound for $\nabla^N_{e_1}\nabla^N_{e_2}\nabla^N_{e_3} I_3$, due to the prefactor $N^{-2}$, follows from the estimate $|\nabla^N_{e_3}\big\{\fa'''(u^N(t,x)) \big[\nabla^N_{e_4} u^N(t,x)\big]^2\big\}| \leq C\sum_{j=1}^3 \|u^N(t,\cdot)\|_{C_N^3}^j$, shown by analogous arguments used for the term $I_2$.  

The same bound for $\nabla^N_{e_1}\nabla^N_{e_2}\nabla^N_{e_3} I_3$, due to the prefactor is $N^{-3}$, follows from the estimate $|\fa^{(4)}(z^N(x))\big[\nabla^N_{e_4} u^N(t,x)\big]^3|\leq \|\fa^{(4)}\|_\infty \|\nabla^N u^N\|_\infty^3$.

Hence, \eqref{third derivative eq} is established.
We comment that the estimate \eqref{third derivative eq} was used to show \eqref{C_1 eq}.
\end{rem}

\begin{proof} [Proof of Proposition \ref{prop.6.4}]
We apply \eqref{eq:3.nabla-De} again taking $e=e_1$ and $u=\nabla_{e_2}^Nu^N$ for the
first term of \eqref{eq:3.t-nab-nab}.  
Then, we have
$$
\partial_t \nabla_{e_1}^N \nabla_{e_2}^N u^N = L_a^N \nabla_{e_1}^N \nabla_{e_2}^N u^N
+ Q^N(x),
$$
where the remainder term $Q^N(x)$ is the sum of the second and third terms in the right 
hand side of \eqref{eq:3.t-nab-nab} and the last term in \eqref{eq:3.nabla-De} with
$e=e_1$ and $u=\nabla_{e_2}^Nu^N$.

Then, by Duhamel's formula, we obtain
for $v_{e_1,e_2}(t,\tfrac{x}N) := \nabla_{e_1}^N \nabla_{e_2}^N u^N(t,\tfrac{x}N)$,
\begin{align*}
&v_{e_1,e_2}(t,\tfrac{x}N) 
=   N^{-n}\sum_y v_{e_1,e_2}(0,\tfrac{y}N) p(0,\tfrac{y}N;t,\tfrac{x}N)  \\
& \ \ + \int_0^t ds \, N^{-n}\sum_y 
\sum_{e'} \{\t_{e_1-e'}\nabla_{e_1}^N \nabla_{e_2}^N u^N(s,\tfrac{y}N)\cdot 
\nabla_{e_1}^N a_{y,e'}(u^N(s)) \} \nabla_{e',y}^{N,*} p(s,\tfrac{y}N;t,\tfrac{x}N)  \\
& \ \ + \int_0^t ds \, N^{-n}\sum_y \sum_{e'}
\nabla_{e_1}^N  \{ \nabla_{e_2}^N a_{y,e'}(u^N(s)) \t_{e_2-e'}\nabla_{e_2}^N u^N\}(s,\tfrac{y}N)
\nabla_{e',y}^{N,*}  p(s,\tfrac{y}N;t,\tfrac{x}N) \\
& \ \ + K \int_0^t ds\, N^{-n}\sum_y g_{y,e_2}(u^N(s)) \nabla_{e_2}^N u^N(s,\tfrac{y}N)
\nabla_{e_1}^{N,*}p(s,\tfrac{y}N;t,\tfrac{x}N).
\end{align*}

Let $v(t,\tfrac{x}N) = \sum_{e_1,e_2} |v_{e_1,e_2}(t,\tfrac{x}N)|$ and note that
$|\nabla^Nu^N(s,\tfrac{y}N)| \le CK^{\frac1\si}$ from the first Schauder estimate.  
Also, note that $N^{-n}\sum_y p =1$ by the symmetry of $p$ for
the first term.  In addition, note that $g_{y,e_2}(u)$ is bounded and $K\cdot K^{1/\si}\le K^{3/\si}$
for the last term.  Then,
we have from \eqref{eq:3.35-A}, \eqref{eq:3.nab-fa'},
\eqref{eq:3.18} and \eqref{eq:3.nab-nab-fa'} that
\begin{align}  \label{eq:3.20}
\|v(t)\|_{L^\infty} 
\le & C \|v(0)\|_{L^\infty} 
+ CK^{\frac1\si} \cdot e^{CK^{4/\si^2}}
\Big( \int_0^t \frac{ds}{\sqrt{t-s}} \|v(s)\|_{L^\infty} + K^{\frac{2}{\si}} \sqrt{t}\Big),
\end{align}
for $t\in [0,T]$.   This implies, by the argument in \cite{F83}, p.\ 144 (letting $n\to\infty)$,
that
\begin{align}  \label{eq:3.21}
\|v(t)\|_{L^\infty} 
\le  C (\|v(0)\|_{L^\infty} +K^{3/\si}e^{CK^{4/\si^2}}) 
\exp \Big\{Ct\,\big(K^{\frac1\si} \cdot e^{CK^{4/\si^2}}
\big)^2\Big\}
\end{align}
concluding the proof of \eqref{eq:lem3.3-2} by absorbing $K^{\frac1\si}$ 
and also $K^{3/\si}e^{CK^{4/\si^2}}$ in the higher order exponential by changing $C$.
\end{proof}

\begin{rem}\rm
For the linear Laplacian (i.e., $\fa(u)=cu$), we have $\nabla_x^N p^N=\nabla_y^Np^N$
due to $[\nabla_e^N,\De^N]=0$ or $p^N=p^N(t-s,x-y)$ so that the computations made in 
Lemma \ref{lem:3.4} are unnecessary.  In Lemma \ref{lem:3.4}, 
especially \eqref{eq:3.nabla-De}, the second term, which is the
error term obtained by computing $[\nabla_e^N, L_a^N]$,
is in the form of a gradient.  This is important to make the
summation by parts in $y$ and move the discrete derivative
$\nabla_{e'}^N$ to $p$.
\end{rem}

\section*{Acknowledgements}

The authors thank Jean-Dominique Deuschel, Takashi Kumagai and Stefan Neukamm
for their advice on related references.
T.\ Funaki was supported in part by JSPS KAKENHI, Grant-in-Aid for Scientific Researches (A) 18H03672 and (S) 16H06338.
S.\ Sethuraman was supported by grant ARO W911NF-181-0311.

\end{document}